\documentclass{amsart}

\usepackage{latexsym,exscale,enumitem}
\usepackage{mathtools,amsmath,mathscinet,amsfonts,amssymb,amsthm,bbm,euscript}
\usepackage{amscd, stmaryrd}
\usepackage{ulem}
\usepackage[all]{xy}
\usepackage{graphicx}
\usepackage{wrapfig}
\usepackage{xcolor}
\usepackage{comment}
\usepackage{epsfig}
\usepackage{psfrag}
\usepackage{rotating}
\usepackage{lscape}
\usepackage{amsbsy}
\usepackage{verbatim}
\usepackage{moreverb}
\usepackage{color}
\usepackage{mathrsfs}
\usepackage{tikz, tikz-cd}
\usepackage{eqparbox}

\newcommand{\eqmathbox}[2][N]{\eqmakebox[#1]{$\displaystyle#2$}}

\numberwithin{equation}{section}

\newcounter{listcounter}

\RequirePackage{color}
\definecolor{references}{rgb}{.7,.1,.6}

\RequirePackage[pdftex,
colorlinks = true,
urlcolor = references, 
citecolor = references, 
linkcolor = references, 
]
{hyperref}
\usepackage[capitalise]{cleveref}

\usetikzlibrary{math}
\usetikzlibrary{decorations.markings}
\usetikzlibrary{decorations.pathreplacing}
\usetikzlibrary{arrows,shapes,positioning}
\tikzstyle directed=[postaction={decorate,decoration={markings,
    mark=at position #1 with {\arrow{>}}}}]
\tikzstyle rdirected=[postaction={decorate,decoration={markings,
    mark=at position #1 with {\arrow{<}}}}]

\tikzset{
    anchorbase/.style={baseline={([yshift=-0.5ex]current bounding box.center)}},
    tinynodes/.style={font=\tiny,text height=0.75ex,text depth=0.15ex},
    smallnodes/.style={font=\scriptsize,text height=0.75ex,text depth=0.15ex},
    >={Latex[length=1mm, width=1.5mm]},
    rd/.style={red, line width=.5mm},
    bl/.style={blue, line width=.5mm},
    gr/.style={green, line width=.5mm},
    pu/.style={purple, line width=.5mm},
    or/.style={orange, line width=.5mm},
    web/.style={very thick},
    dweb/.style={double},
    wh/.style={white,line width=.15cm}
  }
\tikzcdset{arrow style=tikz, diagrams={>=stealth}}
\tikzcdset{scale cd/.style={every label/.append style={scale=#1},
    cells={nodes={scale=#1}}}}

\newcommand{\pr}{to [out=0,in=180]}

\colorlet{green}{black!30!green}

\definecolor{CQG}{RGB}{0,153,76}
\definecolor{FS}{RGB}{0,76,153} 

\usepackage[backend=biber,style=alphabetic,doi=false,isbn=false,url=false,minnames=6,maxnames=6]{biblatex}
\renewbibmacro{in:}{%
  \ifentrytype{article}{}{\printtext{\bibstring{in}\intitlepunct}}}
\renewbibmacro*{volume+number+eid}{%
  \printtext{vol.}
  \printfield{volume}%
  \setunit*{\addnbspace}
  \printfield{number}%
  \setunit{\addcomma\space}%
  \printfield{eid}}
\DeclareFieldFormat[article]{number}{\mkbibparens{#1}}
\DeclareUnicodeCharacter{0301}{\'{e}}
\usepackage{silence}
\newcommand{\googlebooks}[1]{(preview at \href{https://books.google.com/books?id=#1}{google books})}

\newcommand{\numdam}[1]{}

\def\emph#1{{\sl #1\/}}

\let\hat=\widehat
\let\tilde=\widetilde

\newcommand*{\hte}{\stackrel{\text{h.e.}}{\simeq}}
\newcommand*{\htc}{\sim}
\renewcommand{\to}{\rightarrow}
\newcommand{\To}{\Rightarrow}

\renewcommand{\k}{k}

\def\N{\mathbb{N}}
\def\Q{\mathbb{Q}}
\def\R{\mathbb{R}}
\def\T{\mathbb{T}}

\def\Z{\mathbb{Z}}

\def\cA{\mathcal A}\def\cB{\mathcal B}\def\cC{\mathcal C}\def\cD{\mathcal D}
\def\cE{\mathcal E}\def\cF{\mathcal F}\def\cH{\mathcal H}
\def\cI{\mathcal I}\def\cK{\mathcal K}\def\cL{\mathcal L}
\def\cM{\mathcal M}\def\cO{\mathcal O}\def\cP{\mathcal P}
\def\cQ{\mathcal Q}\def\cR{\mathcal R}\def\cS{\mathcal S}
\def\cV{\mathcal V}\def\cW{\mathcal W}\def\cX{\mathcal X}
\def\cY{\mathcal Y}\def\cZ{\mathcal Z}
\def\mbbK{\mathbb{K}}

\def\fS{\mathfrak S}

\def\mbbZ{\mathbb{Z}}

\renewcommand{\phi}{\varphi}
\renewcommand{\theta}{\vartheta}

\newcommand{\Sbimcl}{\mathrm{Sbim}}
\newcommand{\Sbim}{\mathrm{Sbim}} 
\newcommand{\SBim}{\Sbim}

\newcommand{\BSbimcl}{\mathrm{BSbim}}
\newcommand{\BSbim}{\mathrm{BSbim}}
\newcommand{\BSBim}{\mathrm{BSbim}}
\newcommand{\BSbimp}{\mathrm{BSbim}}

\newcommand{\Set}{\mathrm{Set}}

\newcommand{\op}{\mathrm{op}}

\newcommand{\ev}{\mathrm{ev}}
\newcommand{\F}{{\sf{F}}}
\newcommand{\id}{\mathrm{id}}
\newcommand{\Id}{\mathrm{Id}}

\newcommand{\one}{\mathbf{1}}

\newcommand{\hatCat}{\hat{\Cat}}

\newcommand{\End}{\mathrm{End}}

\newcommand{\Hom}{\mathrm{Hom}}
\newcommand{\eHom}{\underline{\Hom}}

\newcommand{\Braidg}{\operatorname{Br}}

\newcommand{\Ind}{\operatorname{Ind}}

\newcommand{\rk}{\operatorname{\rk}}

\newcommand{\inv}{^{-1}}

\newcommand{\vcomp}{\circ_2}
\newcommand{\Rouq}{F} 
\newcommand{\ub}{\underline{\beta}} 
\newcommand{\Ag}{\sigma} 

\newcommand{\cabledcross}{X}

\newcommand{\swap}{\mathrm{swap}}
\newcommand{\slide}{\mathrm{slide}}

\newcommand{\RR}{\mathbf{R}}

\newcommand{\hcomp}{\circ_1}

\theoremstyle{plain}
\newtheorem{thm}{Theorem}[subsection]
\newtheorem{theorem}[thm]{Theorem}

\newtheorem{maintheorem}{Theorem}[section]

\newtheorem{prop}[thm]{Proposition}
\newtheorem{proposition}[thm]{Proposition}
\newtheorem{cor}[thm]{Corollary}
\newtheorem{corollary}[thm]{Corollary}

\newtheorem{lemma}[thm]{Lemma}

\theoremstyle{definition}

\newtheorem{definition}[thm]{Definition}
\newtheorem{construction}[thm]{Construction}
\newtheorem{rem}[thm]{Remark}
\newtheorem{remark}[thm]{Remark}
\newtheorem{obs}[thm]{Observation}
\newtheorem{observation}[thm]{Observation}

\newtheorem{warning}[thm]{Warning}

\newtheorem{example}[thm]{Example}

\newtheorem{nota}[thm]{Notation}
\newtheorem{notation}[thm]{Notation}
\newtheorem*{exa-nono}{Example}

\makeatletter
\newcommand{\mylabel}[2]{#2\def\@currentlabel{#2}\label{#1}}
\makeatother

\def\AA{\mathbb A}
\def\EE{\mathbb E}

\def\NN{\mathbb N}
\def\RR{\mathbb R}\def\TT{\mathbb T}
\def\VV{\mathbb V}\def\WW{\mathbb W}

\def\bD{\mathbf D}

\def\bK{\mathbf K}
\def\bN{\mathbf N}

\def\bW{\mathbf W}

\newcommand{\bit}[1]{\textit{#1}}

\newcommand{\conn}{\textup{-}\mathrm{conn.}}
\newcommand{\trunc}{\textup{-}\mathrm{trunc.}}
\newcommand{\surj}{\textup{-}\mathrm{surj.}}
\newcommand{\faith}{\textup{-}\mathrm{faithful.}}
\newcommand{\CatInfty}[1]{\Cat_{(\infty, {#1})}}
\newcommand{\Catnk}[2]{\Cat_{({#1}, {#2})}}
\newcommand{\kHom}{\mathrm{kHom}}

\newcommand{\D}{\bD}
\newcommand{\K}{{\bK}}
\newcommand{\oldD}{\mathrm{D}}
\newcommand{\oldK}{\mathrm{K}}
\newcommand{\Chb}[1]{\mathrm{Ch}^b(#1)}

\newcommand{\Kb}{\K^b}
\newcommand{\Kbloc}{\Kb_{\mathrm{loc}}}
\newcommand{\oldDb}[1]{\oldD^b(#1)}
\newcommand{\oldKb}[1]{\oldK^b(#1)}

\newcommand{\Cat}{\mathrm{Cat}}

\newcommand{\Ab}{{\sf Ab}}

\newcommand{\MoritakZ}{\mathrm{Mor}^{\mathrm{flat}, \mathrm{gr-proj}}(\mathrm{mod}_{k}^{Z})}
\newcommand{\Morita}{\mathrm{Morita}}
\newcommand{\Moritacp}{\Morita^{\cp}}
\newcommand{\Moritac}{\Morita^{\mrc}}
\newcommand{\DMoritakZ}{\DerMor^{\mathrm{flat}, \mathrm{gr-perf}}(\mathrm{mod}_{k}^{Z})}
\newcommand{\DerMor}{\mathrm{DMor}}
\newcommand{\Hloc}{H_{\mathrm{loc}}}
\newcommand{\MoritaS}{\mathrm{Mor}^{\mathrm{flat}, \mathrm{gr-proj}}(\mathrm{mod}_{k}^{\mathbb{Z}})}
\newcommand{\DMoritaS}{\DerMor^{\mathrm{flat}, \mathrm{gr-perf}}(\mathrm{mod}_{k}^{\mathbb{Z}})} 
\newcommand{\addkZloc}[1]{\mathrm{Lin}_{k, \mathrm{loc}}^{\mbbZ} (#1)}
\newcommand{\Lin}{\mathrm{Lin}}
\newcommand{\DMorPoly}{\mathrm{DMor}^{\mathrm{poly}, \mathrm{gr-perf}}(\mod_k^{\mbbZ})}
\newcommand{\MorPoly}{\mathrm{Mor}^{\mathrm{poly}, \mathrm{gr-proj}}(\mod_k^{\mbbZ})}

\newcommand{\pt}{{\sf pt}}

\renewcommand{\hom}{\Hom}

\newcommand{\ra}{\rightarrow} 

\newcommand{\xra}{\xrightarrow}

\newcommand{\xhookra}{\xhookrightarrow}

\newcommand{\longra}{\longrightarrow}

\newcommand{\xlongra}[1]{\stackrel{#1}{\longra}}

\makeatletter
\providecommand{\leftsquigarrow}{%
  \mathrel{\mathpalette\reflect@squig\relax}%
}
\newcommand{\reflect@squig}[2]{%
  \reflectbox{$\m@th#1\rightsquigarrow$}%
}
\makeatother

\newcommand{\hookra}{\hookrightarrow}

\newcommand{\Fun}{\mathrm{Fun}} 

\newcommand{\FunL}{\mathrm{Fun^L}} 
\newcommand{\FunR}{\mathrm{Fun^R}} 
\newcommand{\Alg}{\mathrm{Alg}} 
 \newcommand{\Braid}{\mathrm{Braid}}
\newcommand{\Funex}{\mathrm{Fun}^{\mathrm{ex}}}
\newcommand{\Funadd}{\mathrm{Fun}^{\mathrm{add}}}

\newcommand{\Mod}{\mathrm{Mod}}
\newcommand{\BMod}{\mathrm{BMod}}
\newcommand{\RMod}{\mathrm{RMod}}
\newcommand{\Derived}{\cD}

\newcommand{\Deltaop}{\Delta^{\mathrm{op}}}

\newcommand{\AlgCat}{\Alg_{\Cat}}
\newcommand{\loc}{\mathrm{loc}}

\newcommand{\LMod}{\mathrm{LMod}}
\newcommand{\Op}{\mathrm{Op}}

\newcommand{\fin}{\mathrm{fin}}
\newcommand{\Fin}{\mathrm{Fin}}
\newcommand{\Fact}{\mathrm{Fact}}

\newcommand{\CProj}{\mathrm{CProj}}

\renewcommand\mod{\mathrm{mod}}

\newcommand{\add}{\mathrm{add}}
\newcommand{\addR}{\add_{\mbbK}}
\newcommand{\addRG}{\addK^{B\Monoid}}

\newcommand{\addK}{\addR}

\newcommand{\catZ}{\mathrm{Cat}_{\infty}^{B\Z}}

\newcommand{\SetZ}{\Set^{B\Z}}
\newcommand{\addkZ}{\add_{\k}^{B\Z}}
\newcommand{\addkBZ}{\addkZ}

\newcommand{\modkZ}{\mod_{k}^Z}
\newcommand{\grmod}{\mathrm{grmod}}

\newcommand{\st}{\mathrm{st}}
\newcommand{\stR}{\st_{\mbbK}}
\newcommand{\stRG}{\st_{\mbbK}^{B\Monoid}}
\newcommand{\stk}{\st_{\k}}
\newcommand{\stkZ}{\st^{B\Z}_{\k}}
\newcommand{\stkBZ}{\stkZ}

\newcommand{\mrL}{\mathrm{L}}
\newcommand{\mrPr}{\mathrm{Pr}}
\newcommand{\PrL}{\mrPr^\mrL}
\newcommand{\PrLc}{\mrPr^{\mrL,\mrc}}
\newcommand{\PrLcp}{\mrPr^{\mrL,\cp}}
\newcommand{\PrLst}{{\mrPr^{\mrL}_{\st}}}
\newcommand{\PrLstR}{{\mrPr^{\mrL}_{{\stR}}}}
\newcommand{\Monoid}{\cZ}
\newcommand{\PrLstRG}{{\mrPr^{\mrL}_{{\Mod_{\mbbK}^{\Monoid}}}}}
\newcommand{\PrLstc}{\mrPr^{\mrL, \mrc}_{\st}}
\newcommand{\PrLstRGc}{\mrPr^{\mrL, \mrc}_{\Mod_{\mbbK}^{\Monoid}}}
\newcommand{\PrLadd}{{\mrPr^{\mrL}_{\add}}}
\newcommand{\PrLaddR}{{\mrPr^{\mrL}_{{\add_{\mbbK}}}}}
\newcommand{\PrLaddRG}{\mrPr^{\mrL}_{{\Mod_{\mbbK}^{\geq 0,\Monoid}}}}
\newcommand{\PrLaddcp}{\mrPr^{\mrL, \cp}_{\add}}
\newcommand{\PrLaddRGcp}{\mrPr^{\mrL, \cp}_{{\Mod_{\mbbK}^{\geq 0,\Monoid}}}}
\newcommand{\Spectra}{\mathrm{Sp}}
\newcommand{\ConnSpectra}{\Spectra_{\geq 0}}
\newcommand{\Spaces}{\cS}
\newcommand{\cat}{\mathrm{Cat}_{\infty}}
\newcommand{\largecat}{\widehat{\mathrm{Cat}}_{\infty}}
\newcommand{\idem}{\mathrm{idem}}
\newcommand{\catidem}{\cat^\idem}
\newcommand{\catprod}{\cat^{\sqcup, \idem}}
\newcommand{\catrex}{\cat^{\mathrm{rex}, \idem}}
\newcommand{\Map}{\mathrm{Map}}
\newcommand{\Perf}{\mathrm{Perf}}
\newcommand{\Pres}{\cP}
\newcommand{\PresSigma}{\cP^{\Sigma}}
\newcommand{\CAlg}{\mathrm{CAlg}}
\newcommand{\mrc}{\mathrm{c}}
\newcommand{\cp}{\mathrm{cp}}
\newcommand{\conep}{\mrc1\mathrm{p}}

\newcommand{\arrow}{\mathrm{ar}}

\newcommand{\Triv}{\mathrm{Triv}}

\newcommand{\Mul}{\mathrm{Mul}}

\renewcommand{\Im}{\mathrm{Im}}
\newcommand{\PreBraid}{\mathrm{PreBraid}}
\newcommand{\PreBraidid}{\mathrm{Braid}}

\newcommand{\faithful}{\tiny{\text{f}}\,}

\renewcommand{\Funadd}{\mathrm{Fun}^{\sqcup}}

\newcommand{\adj}{\dashv}

\newcommand{\Comm}{\text{Comm}}
\newcommand{\fs}{\text{f.s.}}

\newcommand{\fgt}{{\textup{fgt}}}
\newcommand{\codisc}{{\textup{codisc}}}
\newcommand{\gpd}{{\textup{gpd}}}
\renewcommand{\Bar}{{\textup{Bar}}}
\newcommand{\coCart}{{\textup{coCart}}}
\newcommand{\Cart}{{\textup{Cart}}}
\newcommand{\uno}{\mathbbm{1}}
\newcommand{\Assoc}{{\textup{Assoc}}} 
\newcommand{\lax}{{\textup{lax}}} 

\newcommand{\CRoverd}[3]{{#1}_{\small{/^{#2}}{#3}}}
\makeatletter
\newcommand{\colim@}[2]{%
  \vtop{\m@th\ialign{##\cr
    \hfil$#1\operator@font colim$\hfil\cr
    \noalign{\nointerlineskip\kern1.5\ex@}#2\cr
    \noalign{\nointerlineskip\kern-\ex@}\cr}}%
}
\newcommand{\colim}{%
  \mathop{\mathpalette\colim@{\rightarrowfill@\scriptscriptstyle}}\nmlimits@
}
\renewcommand{\varprojlim}{%
  \mathop{\mathpalette\varlim@{\leftarrowfill@\scriptscriptstyle}}\nmlimits@
}
\renewcommand{\varinjlim}{%
  \mathop{\mathpalette\varlim@{\rightarrowfill@\scriptscriptstyle}}\nmlimits@
}

\makeatother

\renewcommand{\colim}{\mathrm{colim}}

\newcommand{\cabledcrossfig}{
    \vcenter{\hbox{\includegraphics{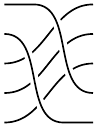}}}   
        \, , \quad
        \vcenter{\hbox{\includegraphics{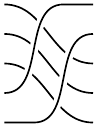}}}   
        \, , \quad
        \vcenter{\hbox{\includegraphics{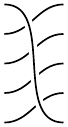}}}   
        \, , \quad
        \vcenter{\hbox{\includegraphics{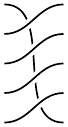}}}   
}

\newcommand{\slidefig}{
    \vcenter{\hbox{\includegraphics{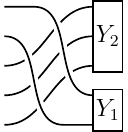}}} 
    \;\;
    \xrightarrow{\slide_{Y_1,Y_2}}
    \;\;
    \vcenter{\hbox{\includegraphics{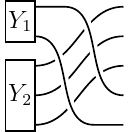}}} 
}

\newcommand{\FigSlide}{
    \vcenter{\hbox{\includegraphics{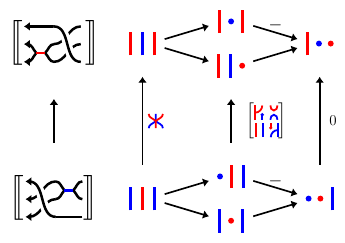}}}
}

\newcommand{\FigSlidee}{
    \vcenter{\hbox{\includegraphics{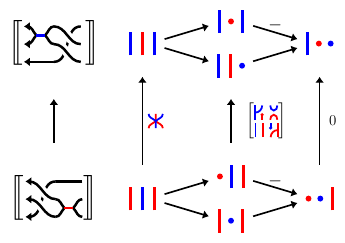}}}
}

\newcommand{\FigSlidei}{
    \vcenter{\hbox{\includegraphics{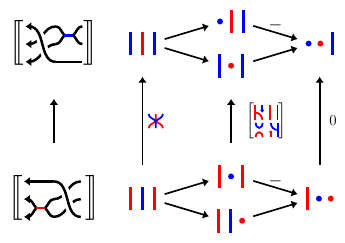}}}
}

\newcommand{\FigSlideei}{
   \vcenter{\hbox{\includegraphics{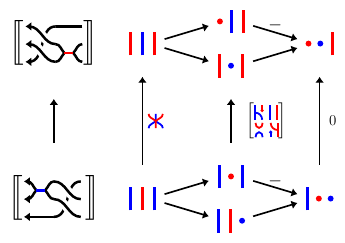}}}
}

\newcommand{\FigSlidec}{
    \vcenter{\hbox{\includegraphics{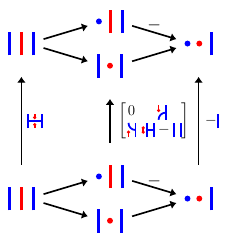}}}
}

\newcommand{\FigSlideh}{
    \vcenter{\hbox{\includegraphics{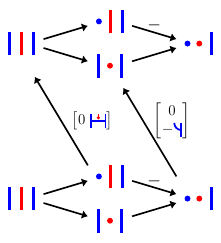}}}
}

\newcommand{\twoproofs}{
    \vcenter{\hbox{\includegraphics{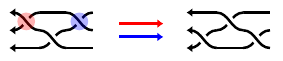}}}
}

\newcommand{\genlist}{
    \vcenter{\hbox{\includegraphics{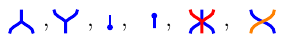}}}
}

\newcommand{\diagcxs}{
    \vcenter{\hbox{\includegraphics{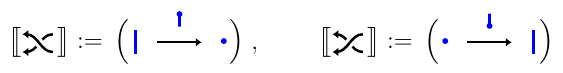}}}
}

\makeatletter
\def\l@subsection{\@tocline{2}{0pt}{4pc}{5pc}{}}
\makeatother

\author{Yu Leon Liu}
\address{Harvard University, Department of Mathematics, 
1 Oxford Street
Cambridge, MA 02138 USA
\href{https://leon2k2k2k.github.io}{leon2k2k2k.github.io}}
\email{yuleonliu@math.harvard.edu}

\author{Aaron Mazel-Gee}
\address{Caltech - The Division of Physics, Mathematics and Astronomy
1200 E California Blvd, Pasadena CA 91125, USA}
\email{aaron@etale.site}

\author{David Reutter}
\address{Fachbereich Mathematik, Universit\"at Hamburg, 
Bundesstra{\ss}e 55, 
20146 Hamburg, Germany
\href{https://www.davidreutter.com/}{davidreutter.com}}
\email{david.reutter@uni-hamburg.de}

\author{Catharina Stroppel}
\address{Mathematisches Institut, Universit\"at Bonn, Endenicher Allee 60, 53115 Bonn, Germany
\href{http://www.math.uni-bonn.de/ag/stroppel/}{catharina.stroppel.de}}
\email{stroppel@math.uni-bonn.de}

\author{Paul Wedrich}
\address{Fachbereich Mathematik, Universit\"at Hamburg, 
Bundesstra{\ss}e 55, 
20146 Hamburg, Germany
\href{https://paul.wedrich.at/}{paul.wedrich.at}}
\email{paul.wedrich@uni-hamburg.de}

\title{A braided monoidal \texorpdfstring{$(\infty,2)$}{infinity-2}-category of Soergel bimodules}

\begin{document}

\maketitle
\vspace{-25pt}
\begin{abstract}
The Hecke algebras for all symmetric groups taken together form a braided
monoidal category that controls all quantum link invariants of type A and, by
extension, the standard canon of topological quantum field theories in dimension
3 and 4. Here we provide the first categorification of this Hecke braided monoidal category, which takes the
form of an $\EE_2$-monoidal $(\infty,2)$-category whose hom-$(\infty,1)$-categories are $k$-linear, stable, idempotent-complete, and equipped with $\mathbb{Z}$-actions. This categorification is designed to control homotopy-coherent link homology theories and to-be-constructed topological quantum field theories in dimension 4 and 5. 

Our construction is based on chain complexes of Soergel bimodules, with monoidal structure given by parabolic induction and braiding
implemented by Rouquier complexes, all modelled homotopy-coherently. This is part of a framework which allows to transfer the
toolkit of the categorification literature into the realm of $\infty$-categories
and higher algebra.
Along the way, we develop families of factorization systems for $(\infty,n)$-categories, enriched $\infty$-categories, and $\infty$-operads, which may be of independent interest.

As a service aimed at readers less familiar with homotopy-coherent mathematics, we include a brief introduction to the necessary $\infty$-categorical technology in the form of an appendix. 
\end{abstract}



\tableofcontents

\newcommand\oldKbloc[1]{\mathrm{K}^b_{\mathrm{loc}}(#1)}
\newcommand{\hincl}{h_1 K_{\mathrm{loc}}}

\newcommand{\Gr}{{\textup{Gr}}}
\newcommand{\shuffle}{{\textup{sh}}}

\newcommand{\blueref}{ {\color{blue}[refs]}~}
\section{Introduction}

\subsection{Overview}
\label{subsection.intro.overview}

Braided monoidal categories of representations, particularly quantum group representations, serve as a powerful framework for understanding invariants of knots, links, and tangles.
Via skein theory \cite{kw:tqft,MR1797619,MR2978449} or (relatedly) factorization homology \cite{MR3431668,MR3818096}, such braided monoidal categories are also at the heart of 3- and 4-dimensional topological quantum field theories, such as those by Witten--Reshetikhin--Turaev \cite{MR990772,MR1091619} and Crane--Yetter--Kauffman \cite{MR1452438}.

Following the paradigm of \cite{MR1295461}, our work is motivated by the desire to
construct higher-dimensional TQFTs by \emph{categorifying} these theories.
To this end, we will provide a construction of a ``categorified braided monoidal category'', which we believe encompasses all the necessary data to form the basis for future constructions of 4- and 5-dimensional analogs, respectively, of Witten--Reshetikhin--Turaev and Crane--Yetter--Kauffman theories of Lie type A, see also \cite{MWW, ICM} and references therein.

We begin our discussion in the decategorified context. The Reshetikhin--Turaev link invariants \cite{MR1036112} give an interpretation of the Jones polynomial \cite{MR766964} in terms of morphisms in the braided monoidal category of representations of quantum $\mathfrak{sl}_2$.
In fact, all of the Reshetikhin--Turaev invariants of type A are controlled via Schur--Weyl duality by a single braided monoidal category $H$, 
just as the $\mathfrak{sl}_N$ link polynomials are specializations of the HOMFLYPT link invariant. Recall for $n \in \NN_0 \coloneqq \{ 0 , 1, 2, \ldots \}$ the \textit{Hecke algebra} $H_n$, which is a quotient of the group algebra over $\Z[q^{\pm 1}]$ of the Artin braid group $\Braidg_n$ for the symmetric group $S_n$.
The braided monoidal category $H$ consists of the following data:

\begin{enumerate}

\item

The objects are given by natural numbers $n \in \NN_0$.

\item

The endomorphism algebra of each object $n \in H$ is $H_n$. All other hom-sets are trivial.

\item The monoidal structure is given on objects by addition, i.e. $m \otimes n
\coloneqq m + n$, and on morphism as the map $H_m \times H_n
\to H_{m+n}$ corresponding to the parabolic subgroup $S_m\times
S_n\hookra S_{m+n}$.

\item

The braiding---a natural isomorphism $m \otimes n \xra{\sim} n \otimes m$ for each $m,n \in H$---is given by the image in $\End_H(m+n) \coloneqq H_{m+n}$ of the positive $(m,n)$-shuffle braid in $\Braidg_{m+n}$.
\end{enumerate}

In the pioneering work \cite{Kho}, Khovanov defined a homology theory for knots and links that categorifies the Jones polynomial, leading to a plethora of categorifications of many other polynomials invariants; see e.g. \cite{MR2275632} for a survey.
Despite the diverse flavors and contexts of these constructions (e.g. involving perverse or coherent sheaves, matrix factorizations, symplectic geometry, Lie theory, and diagrammatic calculus), these categorifications are all---whether explicitly, implicitly, or a posteriori---shadows of a universal categorification involving \emph{Soergel bimodules}~\cite{Soergel}; see \cite{ICM}. Specifically, there is a monoidal additive category $\Sbim_n$ of Soergel bimodules, which categorifies the Hecke algebra $H_n$, i.e. whose split Grothendieck ring is $H_n$. Concretely, $\Sbim_n$ is a certain full additive subcategory of the category of graded bimodules of the polynomial algebra $R_n \coloneqq k[x_1, \ldots, x_n]$ over a $\Q$-algebra $k$ with each $x_i$ in degree $2$. The $\mathbb{Z}$-action by grading shift categorifies the $\mathbb{Z}[q^{\pm 1}]$-action on $H_n$. 
The image of braids in $H_n$ can only be categorified to objects in the (monoidal) bounded chain homotopy category $\oldKb{\Sbim_n}$ of $\Sbim_n$, the so called \emph{Rouquier complexes}~\cite{0409593}, see \cref{subsec:Rouquier}.

The natural next step is to assemble these categorifications into a ``categorified braided monoidal category'', in which categorified braid invariants can be seen as morphisms.
Our main Theorems \ref{thm:main-M} and~\ref{intro.maintheorem}  yield this as a  consequence. Namely, they provide a braided monoidal 2-category $\cH$---more precisely, an $\EE_2$-monoidal $(2,2)$-category---which decategorifies to $H$ upon taking Grothendieck groups.

\begin{enumerate}

\item

The objects are given by natural numbers $n \in \NN_0$.

\item

The endomorphism monoidal category of each object $n \in \cH$ is $\oldKb{\Sbim_n}$. All other hom-categories are trivial.

\item The monoidal structure is given on objects by addition, i.e. $m \otimes n
\coloneqq m + n$, and on morphism by a parabolic induction functor $\oldKb{\Sbim_m}\times \oldKb{\Sbim_n} \to \oldKb{\Sbim_{m+n}}$.

\item

The braiding isomorphism $m \otimes n \xra{\sim} n \otimes m$ is given by the Rouquier
complex in $\oldKb{\Sbim_{m+n}}$ corresponding to the $(m,n)$-shuffle braid in
$\Braidg_{m+n}$.

\end{enumerate}
\vspace{-1mm}

In fact, we go much further: we work in a homotopy-theoretic context based on higher algebra in the sense of Lurie \cite{HTT,HA} and
construct a \textit{braided monoidal $(\infty,2)$-category}\footnote{We refer the reader to \Cref{sec:appendix-recollections} for a brief introduction to $\infty$-categories.} $\Kbloc(\Sbim)$;
a truncation of which recovers the braided monoidal 2-category $\cH$ described
above. Explicitly, the hom triangulated categories $\oldKb{\Sbim_n}$ are replaced by
stable $\infty$-categories $\Kb(\Sbim_n)$ of chain complexes, chain maps and chain homotopies with
the full hierarchy of higher homotopies. 
In contrast to $\oldKb{\Sbim_n}$, using the stable $\infty$-categories $\Kb(\Sbim_n)$ has an essentially advantage, which we crucially use in the construction of $\Kbloc(\SBim)$ and its braided monoidal structure. Namely, $\Kb$ has a universal property: It constructs the free idempotent-complete stable $\infty$-category on an idempotent-complete additive category, see~\cref{subsec:additive-to-stable}.

The braided monoidal $(\infty,2)$-category $\Kbloc(\Sbim)$  yields, in the spirit of Rouquier
\cite{MR3611725}, a suitable derived setting for defining invariants not just of braids but also of braid cobordisms with their isotopies and higher isotopies. Hints of such (coherent) braided monoidal 2-categories from link homology have appeared in the literature, see e.g.~\cite{ManRou}\footnote{Indeed, as mentioned in~\cite{ManRou}, such a construction ``requires working in a suitable homotopical setting ($\mathbb{A}_{\infty}$- or $\infty$-categories), and this creates technical complications for the full construction of a braided monoidal $2$-categorical structure'', also see~\cite{RouESI}.}.

A further motivation for such a derived setting, and for working $\infty$-categorically comes from TQFTs: passing from the triangulated categories $\oldKb{\Sbim_n}$ to their underlying stable $\infty$-categories $\Kb(\Sbim_n)$ yields better finiteness properties. These finiteness properties can be essential for extending TQFTs to manifolds of higher dimensions, and for ensuring that the resulting invariants are small enough to have well-defined decategorifications. See \cite[Thm. 1.5 and \S~4.7]{MR4589588} for an example of this phenomenon.

\subsection{Monoidality theorem}
Our first result is an upgrade of  $\cH$ to a monoidal $(\infty,2)$-category whose hom-$\infty$-categories are $k$-linear, stable, idempotent-complete, and equipped with a $\mbbZ$-action given by the grading shift action on $\Kb(\SBim_n)$.

To formulate the result, we let $\st_k$ denote the $\infty$-category of \textit{small stable idempotent-complete $k$-linear $\infty$-categories}\footnote{Such $\infty$-categories can be modelled by small pretriangulated idempotent-complete $k$-linear dg-categories~\cite{Cohn}.}. Let $\stkBZ$ be the $\infty$-category $\Fun(B\mbbZ, \st_k)$ of such $\infty$-categories equipped with an additional compatible $\mbbZ$-action. 
Then, Day convolution\footnote{We stress the distinction from the pointwise symmetric monoidal structure; Day convolution uses the symmetric monoidal structure of $B\Z$.
}
induces a symmetric monoidal structure on $\stkBZ$, and in turn on the $\infty$-category $\Cat[\stkBZ]$ of small $\infty$-categories enriched in $\stkBZ$ in the sense of Gepner-Haugseng~\cite{GH13}.

\begin{maintheorem}[{\cref{prop:Kbloc-Sbim-explicit}}]
\label{thm:main-M}
There is a monoidal $(\infty,2)$-category $\Kbloc(\SBim)$ with objects labelled by natural numbers $n \in \mathbb{N}_0$ and whose endomorphim $\infty$-categories are the $k$-linear, stable, idempotent-complete $\infty$-categories $\Kb(\SBim_n)$ of chain complexes of Soergel bimodules, with a $\mbbZ$-action by grading shift. More precisely, $\Kbloc(\Sbim)$ defines an $\EE_1$-algebra in the symmetric monoidal $\infty$-category $\Cat[\stkBZ]$.
\end{maintheorem}

The $(\infty,2)$-category $\Kbloc(\Sbim)$ from \cref{thm:main-M} is  constructed in \cref{sec:SBim} via certain enriched variants of Morita $(\infty,2)$-categories developed in \cref{sec:morita-categories} following~\cite{HA}, see also \cite{ HaugsengMor, FrS}.

\subsection{Braiding theorem}
\label{subsection.intro.braiding.thm}

To motivate our main theorem, recall that braided categories of quantum group representations do not exist in isolation: Rather, they come equipped with forgetful \textit{fiber functors} to the category of vector spaces, carrying representations to their underlying vector spaces. Likewise, our monoidal $(\infty,2)$-category $\Kbloc(\Sbim)$ from \cref{thm:main-M} comes equipped with a monoidal `fiber functor', see~\eqref{eq:fiber-functor}, more precisely with a morphism in $\Alg_{\EE_1}(\Cat[\stkBZ])$ of the form
\[\Hloc \colon \Kbloc(\Sbim) \to \stkBZ.
\]
In particular, here $\stk$ plays the role of a $2$-categorical version of the category of vector spaces. Hence $\stkBZ = \Fun(B \mbbZ, \st_k)$ may be thought of as $2$-vector spaces equipped with a $\mbbZ$-action. 
The Day convolution symmetric monoidal structure on $\stkBZ$ may be interpreted as being the natural convolution structure induced by thinking of these as $2$-vector spaces graded by the abelian Lie group $U(1) \simeq B\mbbZ$.\footnote{This is directly analogous to the classical situaton in which $\mathrm{Rep}(\mathbb{Z}) \simeq \Fun(U(1), \mathrm{Vec})$ has two canonical symmetric monoidal structures: a pointwise one, and one arising from Day convolution.}

Homwise,  the functor $\Hloc$ is induced by the functors $\Kb(\SBim_n) \to \D({}_{R_n} \mathrm{grbmod}_{R_n})$ which send chain complexes of Soergel bimodules to their corresponding objects in the derived $\infty$-category of the abelian category of graded $R_n$-bimodules. They then act as morphisms in $\stkBZ$ between certain stable $\infty$-categories of graded $R_n$-modules. 

\begin{maintheorem}[{\cref{cor:main-corollary}}]
\label{intro.maintheorem}
There exists a unique braided monoidal (i.e., $\EE_2$-algebra) structure on $\Kbloc(\Sbim) \in \Cat[\stkBZ]$ that enhances its monoidal structure and satisfies the following conditions.
\begin{enumerate}

\item\label{intro.maintheorem.fiberfctr.is.braided}

The fiber functor $\Hloc\colon \Kbloc(\Sbim) \ra \stkBZ$ is braided monoidal.

\item\label{intro.maintheorem.braiding.is.Rouquiercx}

The braiding $1 \otimes 1 \xra{\sim} 1 \otimes 1$ in $\Kbloc(\Sbim)$ admits an equivalence with the Rouquier complex $F(\sigma) \in  \End_{\Kbloc(\Sbim)}(2) \coloneqq \Kb(\Sbim_2)$ corresponding to the braid group generator $ \sigma \in \Braidg_2$.
\end{enumerate}
\end{maintheorem}

\noindent We emphasize that the uniqueness statement of \Cref{intro.maintheorem} must be interpreted in the $\infty$-categorical sense, where the notion of ``unique up to unique isomorphism'' from classical category theory is generalized to the notion of being parametrized by a contractible $\infty$-groupoid.

We prove \Cref{intro.maintheorem} in \emph{obstruction-theoretic} terms, as explained further in \Cref{subsection.key.ideas.for.braiding.thm}. In fact, \Cref{intro.maintheorem} is a special case of a more general result, namely \Cref{thm:main-theorem-last-sec}, that applies to a general homwise additive 2-category (in place of $\Sbim$) and a general $\stkBZ$-enriched $\infty$-category (in place of $\stkBZ$).

The braided monoidal $2$-category $\cH$ is recovered in \cref{rmk:finally-cH} from the braided monoidal $(\infty,2)$-category $\Kbloc(\SBim)$ by taking its homotopy $2$-category $h_2(\Kbloc(\SBim))$ obtained by quotienting out all $3$-morphisms, see \cref{subsec:hom-cat-of-inf-k-cat}, i.e.\! by applying $H_0$ to its $2$-hom-objects (which are objects of the derived $\infty$-category of $k$). This is the sense in which \cref{thm:main-M} and \cref{intro.maintheorem} enhance the $2$-categorical statement from \cref{subsection.intro.overview}.

\subsection{Context and proof outline for the braiding theorem}
\label{subsection.key.ideas.for.braiding.thm}

Here we further contextualize \Cref{intro.maintheorem} while discussing key aspects of its setup and proof.

\subsubsection{Braided monoidal 2-categories and their generalizations}
\label{subsubsection.brmon2cat}

The problem of finding the right definition of braided monoidal (even ordinary) $2$-categories has a long history, going back to~\cite{MR1266348,KapVoe,MR1402727}, see~\cite{schommer-pries-thesis} for a survey. One motivation~\cite{MR1355899} was the construction of surface invariants in $4$-manifolds, in particular $2$-knots, just like braided monoidal $1$-categories are related to link invariants. 
A challenge hereby was the specification of the \emph{extra data} required for a braiding on a $2$-category, for example with respect to the naturality of the braiding---which is \emph{not a property}, but \emph{extra structure}.

An even larger challenge is the construction of interesting concrete examples of braided monoidal $2$-categories, at least as rich as the theory of braided monoidal categories built from quantum groups. 
First hints of such a landscape came into view with the invention of Khovanov homology and its various associated invariants of tangles and tangle cobordisms \cite{MR2171235}. However, the idea of extracting a braided monoidal $2$-category from these invariants, see e.g. \cite{MR2752518}, turned out to be hard.
Indeed, this idea drove the study of functoriality properties of categorified link and tangle invariants with respect to tangle cobordisms. At the state of the art, the most well-behaved link homology theories assign homotopy classes of chain maps to isotopy classes of tangle cobordisms. As discussed in \cite[\S~6]{MWW}, this is sufficient to satisfy the (classical) axioms for a braided monoidal $2$-category in a toy model that is combinatorial and discrete up to the level of $1$-morphisms and truncated at the level of $2$-morphisms.

It is however desirable to go beyond such a toy model, so that the 2-hom-objects are no longer just sets or vector spaces, but are of a more homological nature: \emph{higher} homotopies give invariants of \emph{higher} isotopies between braids and links. Indeed, systematically incorporating such higher homotopies in link homology theories is also highly relevant beyond the goal of constructing an enhancement of a braided monoidal $2$-category, e.g. towards skein algebra categorification \cite[Discussion after Conjecture 1.8]{1806.03416}, cabling operations \cite[\S~1.3]{GHW}, as well as wrapping and flatting functors \cite[\S~6.4]{MR4526088}, \cite{elias2018gaitsgorys,MR4669324}.

To formulate answers to such questions, the classical axioms for braided monoidal $2$-categories are no longer sufficient: For example the Rouquier complexes corresponding to braid group generators satisfy the braid relation up to homotopy equivalence, \cite{0409593}. These equivalences must be provided as additional data, which then should be subject to further coherence conditions, requiring higher homotopies ad infinitum. Even worse, the compatibilities to be checked at each level quickly explode---in both complexity and in number,---and hence become unfeasible beyond the first two well-known stages of Reidemeister moves and Carter--Saito movie moves.

To be able to address these problems, we find it essential to work within the homotopy-theoretic context of $(\infty,2)$-categories, as explained further in the following. 

The modern formalism of higher algebra, as in~\cite{HA}, offers a robust answer to the problem of modelling braided monoidal structures in such a homotopical context: Namely in terms of the notion of an \bit{$\EE_2$-algebra}~\cite{MR420609}.
 This notion applies in any symmetric monoidal $\infty$-category $\cV$, and in the $(2,1)$-category of ordinary categories recovers the notion of a braided monoidal category by a folklore result that we recover in \cref{rem:classicalbraidingisEtwo}. Namely, an $\EE_2$-algebra structure on an object $V \in \cV$ consists of a suitably compatible system of maps $\mathrm{Conf}_n(\R^2) \ra \Hom_\cV(V^{\otimes n},V)$ from configuration spaces of points in $\R^2$. For example, a chosen basepoint of $\mathrm{Conf}_2(\R^2)$ selects a multiplication map $V \otimes V \xra{\mu} V$, and the generator of $\pi_1(\mathrm{Conf}_2(\R^2)) \simeq \Z$ selects a braiding isomorphism $\mu \xra{\sim} \mu \circ \tau$ (where $\tau$ denotes the symmetry isomorphism in $\cV$). More generally, the requisite compatibilites as $n$ varies collectively encode the homotopy coherent associativity of $\mu$ as well as its compatibility with the braiding. 

To apply this formalism of $\EE_2$-algebras to describe a braiding on $\Kbloc(\Sbim)$, we use $\cV= \Cat[\stkBZ]$. Spelled out, $\Cat[\stkBZ]$ is the $\infty$-category of $\stkBZ$-enriched $\infty$-categories, i.e. $(\infty,2)$-categories whose hom-$(\infty,1)$-categories are stable, idempotent complete, and equipped with a $k$-linear structure and an action by $\mathbb{Z}$, all appropriately compatible with the composition operations. Indeed, a $\stkBZ$-enriched $\infty$-category $\cC \in \Cat[\stkBZ]$ has composition morphisms $\Hom_\cC(c_0,c_1) \otimes \Hom_\cC(c_1,c_2) \ra \Hom_\cC(c_0,c_2)$ in $\stkBZ$, where $\otimes$ denotes the Day convolution symmetric monoidal structure on $\stkBZ$, rather than merely functors $\Hom_\cC(c_0,c_1) \times \Hom_\cC(c_1,c_2) \ra \Hom_\cC(c_0,c_2)$: The tensor product in $\stkBZ$ implicitly enforces our desired compatibility.

\cref{thm:main-M} and \Cref{intro.maintheorem} construct $\Kbloc(\Sbim)$ as an $\EE_2$-algebra\footnote{Although $\Cat[\stkBZ]$ can be seen as an $(\infty,3)$-category (in fact it is self-enriched), the notion of an $\EE_2$-algebra therein only makes references to its underlying $(\infty,1)$-category.} object in $\Cat[\stkBZ]$.  In particular, its monoidal structure and braiding are automatically homotopy-coherently compatible with the $k$-linearity and $\Z$-actions on its hom-$(\infty,1)$-categories.

To construct this $\EE_2$-algebra structure, we take inspiration from the homotopy-theoretic machinery of \emph{obstruction theory}, in which context one often finds it (somewhat paradoxically) easier to prove a stronger theorem. Specifically, we prove \Cref{intro.maintheorem} by establishing not just the existence of an $\EE_2$-algebra structure on the $(\infty,2)$-category $\Kbloc(\SBim)$, but also its \emph{homotopy-theoretic uniqueness along with its compatibility} with $k$-linearity, $\mathbb{Z}$-action, and with the fiber functor $\Hloc \colon \Kbloc(\Sbim) \to \stkBZ$.

\subsubsection{Higher-categorical notions of faithfulness and homotopy categories}
\label{subsubsection.faithful.htpycats.factsys}

To prove our main theorems, we develop machinery to reduce the construction of algebraic structures on $(\infty,2)$-categories to corresponding structures on their underlying ordinary homotopy categories, as well as higher-dimensional variants. We outline some of these results, which  might be of independent interest.

As a toy case, observe that given an $\infty$-category $\cC$, a full subcategory $\cC'$ is determined entirely by the subset $h_0\cC'$ of the set $h_0\cC$ of equivalences classes of objects  in $\cC$ that are contained in $\cC'$. Furthermore, if $\cC$ is endowed with some multiplicative structure (e.g.\! monoidal, braided monoidal, or symmetric monoidal), then $\cC'$ inherits such a structure if and only if $h_0\cC'$ inherits the resulting structure from $h_0\cC$.

Given an $(\infty,2)$-category $\cC$, we write $h_1\cC$ for the ordinary category obtained by first discarding its noninvertible 2-morphisms and then passing to the homotopy category, i.e.\! taking $\pi_0$ of its hom-spaces. Moreover, we say that a functor between $(\infty,2)$-categories is \textit{faithful} if it is fully faithful on hom-$(\infty,1)$-categories.\footnote{Beware that these are \textit{not} generally inclusions of subobjects, i.e.\! monomorphisms, in $\Cat_{(\infty,2)},$ see \cref{warn:nfaithful-neq-ntruncated}.} 
Then, analogously to the above situation, we show that a faithful functor $\cC' \ra \cC$ is determined entirely by their corresponding faithful functor $h_1\cC' \ra h_1\cC$. Moreover, we show that if $\cC$ is braided monoidal equipped with the faithful functor $\cC' \ra \cC$, then endowing $\cC'$ together  with a (compatible) braided monoidal structure is equivalent to endowing $h_1\cC'$ and $h_1\cC' \ra h_1\cC$ with such a structure, see \cref{cor:alg-equivalence-of-over-categories}. The task of defining an $\EE_2$-algebra structure on $\cC'$ therefore reduces in such a situation to the task of defining a braiding (in the classical sense!) on its homotopy 1-category $h_1\cC'$.

In fact, we prove the above results in much broader generality: In \cref{sec:inf-n-cats}, we study the notion of $n$-faithfulness for functors between $(\infty,k)$-categories, and show in \cref{subsec:fact-for-inf-k-cat} that they define the right classes in factorization systems on $\Cat_{(\infty,k)}$, whose corresponding left classes are given by the $n$-surjective functors, which are surjective on objects, and on parallel morphisms up to level $(n+1)$, see \cref{def:nsurj-and-nfaith}.
Using the iterative definition of higher categories,  $\Cat_{(\infty,k)} \coloneqq \Cat[\Cat_{(\infty,k-1)}]$, we deduce this from general results that we prove regarding factorization systems on enriched $\infty$-categories in \cref{subsec:fs-for-Vcats}, also see \cite{haugseng2023tensor} for related recent result.
Generalizing the above, we establish in \cref{cor:equivalence-of-over-categories} that for any $n$ and any $k$, $(n-1)$-faithful functors to an $(\infty,k)$-category $\cC$ are equivalently determined by $(n-1)$-faithful functors to its homotopy $n$-category.

\subsubsection{The fiber functor}
\label{subsubsection.fiber.fctr}

The fiber functor $\Hloc \colon \Kbloc(\Sbim) \to \stkBZ$ on which \Cref{intro.maintheorem} is built, can be viewed as a categorified and graded analog of the forgetful functor from a category of quantum group representations to $\mathrm{Vec}$. Since its target $\stkBZ$ is braided monoidal (in fact symmetric monoidal), it is tempting to apply the machinery of \cref{subsubsection.faithful.htpycats.factsys} to obtain a braiding on $\Kbloc(\Sbim)$. 

However, this does not work because the analogy with the classical situation breaks down in an important way. In the classical setting, the fiber functor is \emph{faithful} and monoidal, but \emph{in general not braided}. By contrast, our categorified fiber functor will be \emph{braided} monoidal but  \emph{not faithful} (in the sense of \Cref{subsubsection.faithful.htpycats.factsys}). This lack of faithfulness prevents us from directly using the reduction results from \Cref{subsubsection.faithful.htpycats.factsys}. In \Cref{subsubsection.intro.discuss.prebraidings}, we will discuss a restricted version of the fiber functor that \textit{is} faithful, which allows us to leverage the results indicated in \Cref{subsubsection.faithful.htpycats.factsys}. 

In the end, the non-faithfulness turns out to be an essential feature of our fiber functor. This feature is what allows $\Hloc$ and its source category $\Kbloc(\Sbim)$ to be equipped with a non-symmetrically braided monoidal structure, even though the target category $\stkBZ$ is symmetric monoidal. A faithful braided fiber functor would force the source category to be symmetric!

In the classical situation, the existence of a fiber functor   arises from the fact that quantum group representations are ultimately categories of modules for a quasitriangular Hopf algebra, which can be recovered from the fiber functor via Tannakian reconstruction. We would be very interested to see an application of Tannakian reconstruction to our categorified fiber functor.

We expect our fiber functor to be important for future applications and computations. For instance, the fact that it is braided may give a means of recursively computing all the algebraic data implicit in the braided monoidal structure on $\Kbloc(\Sbim)$ on a cell-by-cell basis.

\subsubsection{The prebraiding}
\label{subsubsection.intro.discuss.prebraidings}
Another key ingredient of our proof of \Cref{intro.maintheorem} is the following: $\Kbloc(\Sbim)$ is generated, in an appropriate sense, by a monoidal sub-2-category, restricted to which the fiber functor to  $\stkBZ$ \emph{is} faithful.

We begin by considering the sub-2-category\footnote{The reader should be warned that the inclusion $\Sbim \to \Kbloc(\Sbim)$ is merely a faithful functor but not a monomorphism in $\Cat_{(\infty,2)}$, see~\cref{warn:nfaithful-neq-ntruncated}.} $\Sbim \subset \Kbloc(\Sbim)$ described as follows: it contains all the same objects, but on hom-objects we pass to the subcategories $\Sbim_n \subseteq \Kb(\Sbim_n)$ of Soergel bimodules (seen as complexes concentrated in degree 0). This is a monoidal sub-2-category: Soergel bimodules are closed under parabolic induction.

Consider now the \emph{restricted fiber functor}, i.e. the composite $\Sbim \hookra \Kbloc(\Sbim) \ra \stkBZ$ of the inclusion followed by the fiber functor. Here, we arrive at an interesting tension. While the fiber functor $\Kbloc(\Sbim) \ra \stkBZ$ is not faithful but (will be) braided, this composite \emph{is} faihtful  but not braided: The braiding of $\Kbloc(\Sbim)$ does not restrict to one on $\Sbim$. Indeed, Rouquier complexes are typically genuine complexes, not concentrated in degree 0.

However, we now make a key observation: the $(\infty,2)$-category $\Kbloc(\Sbim)$ is obtained by applying the left adjoint $\Kb$ to the hom-objects of the $(2,2)$-category $\Sbim$. In this sense, $\Sbim$ generates $\Kbloc(\Sbim)$, which suggests that the braiding of $\Kbloc(\Sbim)$ might be uniquely determined by its values on $\Sbim$. In fact, for each $m$, there is a full additive monoidal subcategory $\BSbim_m \subseteq \Sbim_m$ of \bit{Bott--Samelson bimodules}~\eqref{eq:BS} which generates $\Sbim_m$ under sums, retracts, and grading shifts. These Bott--Samelson bimodules  assemble into a sub-2-category $\BSbim \subset \Sbim$, which suggests that the braiding of $\Kbloc(\Sbim)$ might even be uniquely determined by its values on $\BSbim$.

In order to explore this idea further, let us imagine constructing the braiding on the monoidal $(\infty,2)$-category $\Kbloc(\Sbim)$ of \Cref{intro.maintheorem} by hand. This certainly requires, for every pair of objects $m,n \in \N_0$, a braiding isomorphism $m \otimes n \xra{\sim} n \otimes m$. Up to equivalence, this 1-morphism is already determined by condition \eqref{intro.maintheorem.braiding.is.Rouquiercx} of \Cref{intro.maintheorem}: it must be given by the Rouquier complex for the positive $(m,n)$-shuffle braid in $\Braidg_{m+n}$ (see also \cref{fig:cabledcross}). Of course, the braiding has to be \textit{natural}: any complex $C \in \Kb(\Sbim_m)$, seen as an endomorphism of $m \in \Kbloc(\Sbim)$, should ``slide'' along the $m$ parallel strands through the braiding by means of a specified homotopy equivalence (see \cref{fig:slide}), and likewise for any $C' \in \Kb(\Sbim_n)$. In fact, the discussion in the previous paragraph suggests that it suffices to specify these ``slide'' homotopy equivalences merely for objects $C \in \BSbim_m \subseteq \Kb(\Sbim_m)$. 

We formalize such a specification through the key notion of a \bit{prebraiding} on the monoidal functor $\BSbim \hookra \Kbloc(\Sbim)$. We first describe such data in the simple setting of ordinary categories, i.e.\! in the $(2,1)$-category $\Cat$: a prebraiding on a monoidal functor $F \colon \cA \ra \cB$ between ordinary categories consists of equivalences $F(x) \otimes F(y) \xra{\sim} F(y) \otimes F(x)$ that are natural in $x,y \in \cA$ and satisfy two appropriate analogs of the hexagon axioms for braidings; see \cref{def:prebraiding}. In particular, a prebraiding on the identity functor $\id_\cA$ is equivalent to a braiding on $\cA$. On the other hand, as we will see in \cref{subsubsection.intro.prebraiding.via.inf.op}, in the $\infty$-categorical context, a prebraiding is a much sparser structure.
Still, a prebraiding on a monoidal functor between $(\infty,2)$-categories involves a substantial amount of coherence data.

Inspired by the results of \Cref{subsubsection.faithful.htpycats.factsys}, as a core input to our proof of \Cref{intro.maintheorem} we construct a prebraiding on the ordinary monoidal functor $h_1\BSbim \hookra h_1\Kbloc(\SBim)$ in \cref{sec:prebraiding} whose prebraiding equivalence is given by the homotopy equivalence classes of the Rouquier complexes of positive $(m,n)$-shuffle braids. This is based on explicit computations using the diagrammatic formulation of Soergel bimodules, \cite{EW, EK}. Furthermore, this prebraiding is compatible with the fiber functor $h_1\Hloc\colon h_1\Kbloc(\Sbim) \to \stkBZ$.

Finally, having constructed such a prebraiding at the level of homotopy categories, it remains to show that it lifts uniquely to a fully homotopy coherent \emph{braiding} on $\Kbloc(\Sbim)$ and on its fiber functor $\Hloc\colon \Kbloc(\Sbim) \to \stkBZ$.

\subsubsection{Prebraidings via $\infty$-operads}\label{subsubsection.intro.prebraiding.via.inf.op}

To lift our $1$-categorical prebraiding to braided monoidal structures on $(\infty,2)$-categories, we first generalize prebraidings themselves to this $(\infty,2)$-categorical setting. 

We implement the notion of a prebraiding in the context of $\infty$-operads as indicated in the diagram 
\[
\TT_2 \otimes \EE_1
\longra
\AA_2 \otimes \EE_1
\longra
\EE_1 \otimes \EE_1
\xlongra{\sim}
\EE_2
\]
thereof, as we now explain. For more details see \cref{appendix:operads} and \cref{sec:prebraidings-to-E2-str}. 

\begin{itemize}

\item As indicated in \Cref{subsubsection.brmon2cat}, we formalize the notion of a braided monoidal structure via the notion of an $\EE_2$-algebra (with $k$-ary operations parametrized by $\mathrm{Conf}_k(\R^2)$).

\item Much simpler is the notion of an $\EE_1$-algebra, which has $k$-ary operations parametrized by $\mathrm{Conf}_k(\R^1)$. These spaces are discrete, and this $\infty$-operad is equivalent to the ordinary associative operad.

\item The equivalence $\EE_1 \otimes \EE_1 \xra{\sim} \EE_2$ is an instance of \textit{Dunn additivity}.  Hence, an $\EE_2$-algebra structure is equivalent to two compatible $\EE_1$-algebra structures: $\Alg_{\EE_2}(\cV) \simeq \Alg_{\EE_1}(\Alg_{\EE_1}(\cV))$.

\item Whereas $\EE_1$ parametrizes (homotopy-coherently) associative and unital binary operations, by forgetting associativity one arrives at the operad $\AA_2$, which parametrizes binary operations that are merely unital.

\item The $\infty$-operad $\TT_2$ is a two-colored ordinary operad, and parametrizes pairs of pointed objects $\cA,\cB \in \Alg_{\EE_0}(\cV)$\footnote{A pointed object $\cA$ is an object together with a morphism $\eta_{\cA} \colon \uno_{\cV} \to \cA$, equivalently, an $\EE_0$-algebra. A pointed morphism is a morphism that respect that pointing, equivalently, an $\EE_0$-algebra map.} together with pointed morphisms $\cA \xra{F} \cB$, and $\cA \otimes \cA \xra{\mu} \cB$, along with pointed homotopies $\mu(\eta_\cA \otimes (-)) \simeq F \simeq \mu((-) \otimes \eta_\cA)$ in $\Hom_{\Alg_{\EE_0}(\cV)}(\cA,\cB)$. 
In particular, we view a $\TT_2$-algebra as having not an underlying object but having an underlying \textit{morphism}, namely the morphism $F$. The map of $\infty$-operads $\TT_2 \ra \AA_2$ corepresents the operation carrying an $\AA_2$-algebra to the evident $\TT_2$-algebra with $\cA = \cB$ and $F= \id_{\cA}$.

\end{itemize}

We define a prebraiding on an $\EE_1$-algebra morphism $F$ in a symmetric monoidal $\infty$-category $\cV$ to be a $\TT_2$-algebra structure on the correponding morphisms in $\Alg_{\EE_1}(\cV)$, or equivalently a lift of the corresponding $[1] \otimes \EE_1$-algebra in $\cV$ to a $\TT_2 \otimes \EE_1$-algebra, see \cref{subsec:space-of-braidings}. In the context of ordinary categories, we show in \cref{ex:prebraidon1cat} that this notion coincides with the one described in \Cref{subsubsection.intro.discuss.prebraidings}.

\subsubsection{From prebraidings to braidings}
\label{subsubsection.intro.from.prebraidings.to.braidings}

In \Cref{subsubsection.intro.discuss.prebraidings} we outlined the construction of a prebraiding on the monoidal functor $h_1\BSbim \hookra h_1\Kbloc(\SBim)$ between ordinary monoidal $1$-categories, which is compatible with the fiber functor $h_1 \Hloc: h_1\Kbloc(\SBim) \to h_1 \stkBZ$. To complete the proof of \Cref{intro.maintheorem}, it remains to show that this admits a unique lift to a braiding on $\Kbloc(\SBim)$ such that the fiber functor is braided. 

We prove this in steps, combining the machinery described above, as follows, see \cref{sec:main-theorem}. 

Applying a version of the techniques described in \Cref{subsubsection.faithful.htpycats.factsys}, it follows that our prebraiding on $h_1\BSbim \ra h_1\Kbloc(\Sbim)$ over $h_1\stkBZ$ lifts uniquely to a prebraiding on the $(\infty,2)$-functor $\BSbim \ra \Kbloc(\Sbim)$ over $\stkBZ$.

In order to proceed, we note a crucial property of prebraidings, see \cref{cor:step0}: given an adjunction $F \colon \cC \rightleftarrows \cD \colon G$ in which the left adjoint is symmetric monoidal, the data of a prebraiding on an $\EE_1$-algebra morphism $c \ra G(d)$ is equivalent to the data of a prebraiding on its adjunct $F(c) \ra d$.
We use this to extend our prebraiding over $\stkBZ$ from one on $\BSbim \hookra \Kbloc(\Sbim)$ to one on $\Sbim \hookra \Kbloc(\Sbim)$, and then we use it again to extend the latter to one on the defining equivalence $\Kbloc(\Sbim) \xra{\sim} \Kbloc(\Sbim)$.

So, we have obtained a prebraiding, i.e. a $\TT_2$-structure on the identity morphism of $\Kbloc(\Sbim) \in \Alg_{\EE_1}(\Cat[\stkBZ]_{/\stkBZ})$.
Because the identity morphism is invertible, this is equivalent to an $\AA_2 \otimes \EE_1$-structure on $\Kbloc(\Sbim) \in \Cat[\stkBZ]_{/\stkBZ}$.

A final task is to lift this $\AA_2 \otimes \EE_1$-structure to an $\EE_2 \simeq \EE_1 \otimes \EE_1$ structure. Recall that prebraidings on identity functors between ordinary monoidal $1$-categories, i.e.  $\AA_2\otimes \EE_1$-structures in $\Cat_1$, are precisely the same as braidings, i.e. $\EE_2$-structures. More generally, we show in \cref{subsec:from-A2-E1-to-E-2} that $\AA_2 \otimes \EE_1$- and $\EE_2$-algebras agree in general $(2,1)$-operads.

This observation applies to our situation due to a crucial truncatedness result regarding the endomorphism $\infty$-operad of the object $\Kbloc(\Sbim) \in \Cat[\stkBZ]_{/\stkBZ}$. Namely, while we expect it to be quite complicated in general, we prove in \cref{cor:collection} that the maximal sub-$\infty$-operad in the image of its $\EE_1$-structure is in fact just a $(2,1)$-operad. This suffices for our purposes since the map of $\infty$-operads $\EE_1 \to \EE_2$ is surjective on path components of mapping spaces.

Thus, our prebraiding on $\Kbloc(\Sbim)$ over $\stkBZ$ extends uniquely to an $\EE_2$-algebra structure establishing our main goal.

After now having introduced the key ideas and concepts, we finish this
introduction by giving an outline of the organization of the paper.

\subsection{Organization of the paper}
In \cref{sec:2} we start by reviewing the basics of Soergel bimodules,
Bott--Samelson bimodules, and their diagrammatics. Next we proceed to the
bounded homotopy category of Soergel bimodules and discuss Rouquier complexes
for braids. The homotopy equivalence classes of bounded chain complexes of
Soergel bimodules can then be organized into a monoidal $1$-category
$h_1\oldKbloc{\Sbimcl}$. In \cref{sec:prebraiding}, we define the notion of a
prebraiding and construct such a prebraiding on the functor $h_1\BSBim \to h_1\oldKbloc{\Sbimcl}$ that includes
isomorphism classes of Bott--Samelson bimodules into $h_1\oldKbloc{\Sbimcl}$.
This requires explicit diagrammatic computations. Finally, we relate the notion
of prebraiding to centers and centralizers in \cref{thm:classbraidings}, which will be important for
$\infty$-categorical aspects later on.

From \cref{sec:stableLA} we move into the world of $\infty$-categories. Section \cref{sec:stableLA} mostly recalls and collects results from \cite{HA}. We begin in \Cref{sec:presentable} by reviewing various colimit completion procedures, such as the $\Ind$-completion under filtered and the $\PresSigma$-completion under sifted colimits, and symmetric monoidal structures on $\infty$-categories with certain colimits (and functors preserving those colimits). Next, we recall the notion of compactly/projectively generated $\infty$-categories in \cref{subsec:compact-gen-and-ind}. Furthermore, in \cref{prop:PrLcp}, 
we show that $\Ind$ gives an equivalence between small idempotent-complete $\infty$-categories with finite colimits, and presentable compactly generated $\infty$-categories. Similarly, $\PresSigma$ gives an equivalence between small idempotent-complete $\infty$-categories with coproducts, and presentable projectively generated $\infty$-categories.
We review the theory of additive and stable $\infty$-categories in \cref{sec:stable} and give an $\infty$-categorical interpretation of a process that is ubiquitous in the categorification literature: passing from an additive category  to its bounded chain homotopy category, in \cref{subsec:additive-to-stable}. In particular we show as \cref{cor:universalch} that for a small idempotent-complete additive $1$-category $\cA$, the stable $\infty$-category $\Kb(\cA)$ of chain complexes in $\cA$ is the free stable $\infty$-category on $\cA$.
We introduce the $\infty$-category of graded $k$-modules and review their Day convolution symmetric monoidal structures in \cref{subsec:Z-graded-k-linear-modules}. Lastly, in \cref{subsec:derived-inf-cat-graded-mod} we review the theory of derived $\infty$-categories developed in \cite[\S~1.3]{HA} and prove as \cref{prop:gradedmodules} that the module $\infty$-category of a discrete graded algebra is the derived $\infty$-category of graded modules.

In \cref{sec:morita-categories}, we define the $\infty$-categories $\addkBZ$ and $\stkBZ$ of additive and stable $k$-linear $\infty$-categories with a $\mbbZ$-action, as well as their self-enrichment. Furthermore, we construct full symmetric monoidal enriched subcategories $\MoritakZ$ and $\DMoritakZ$ of $\addkBZ$ and $\stkBZ$ respectively, whose objects are labelled by flat discrete graded $k$-algebras, and whose enriched homs are given by (chain complexes) of discrete graded bimodules satisfying suitable finiteness condition. These Morita categories facilitate our homotopy coherent
construction of the Soergel $(2,2)$-category in \cref{sec:SBim}.  
We begin in \cref{subsec:presentable-enriched-infty-cats} with a review of morphism objects in various module categories as well as the self-enrichment of presentably symmetric monoidal $\infty$-categories. 
In \cref{z-graded-k-linear-cats}, we define  $\addkBZ$ and $\stkBZ$ as well as the symmetric monoidal left adjoint $\Kb$ in the graded-linear context. 
Concerning gradings,
we prove in \cref{subsubsec:gradingstoactions} an
$\infty$-categorical version of the familiar equivalence between
categories enriched in graded modules and categories with an action of a
monoid.\footnote{On the level of ordinary $1$-categories, this statement is
usually implicitly assumed in the literature on Bott-Samelson and
Soergel bimodules and crucial when passing between algebraic or Lie
theoretic definitions and diagrammatical versions of these important
categories.} 
To finish the section, we  construct the desired enriched Morita categories $\MoritakZ$ and $\DMoritakZ$ in \cref{subsec:morita-cat-for-discrete-alg}.

In \cref{sec:inf-n-cats}, we enter the world of $(\infty, k)$-categories. 
The purpose of
this section is to introduce and study various factorization systems on the
$\infty$-category of $(\infty, k)$-categories. We
prove as \cref{thm:nsurj-nfaithful-fs-for-infty-k-cats} the existence of the ($n$-surjective, $n$-faithful) factorization system on $\CatInfty{k}$, generalizing the familiar (surjective-on-objects, fully-faithful) factorization
system on $1$-categories. Furthermore, we define the homotopy $n$-category functor and prove as \cref{thm:inf-n-fancy-pullback} that $(n-1)$-faithful functors into a $(\infty, k)$-category $\cC$ are controlled by its homotopy $n$-category $h_n \cC$. This is crucial in both the construction of the Soergel $(2,2)$-category in \cref{sec:SBim} as well as reducing braidings to prebraidings in \cref{sec:main-theorem}.
We begin in \cref{subsec:basics-of-inf-k-cat} with a quick recollection on the basics of $(\infty, k)$ category theory. In \cref{subsec:trun-and-conn}, we recall the familiar ($n$-connected, $n$-truncated) factorization system on the $\infty$-category of spaces. We then proceed in \cref{subsec:fact-for-inf-k-cat} to inductively define the classes of ($n$-surjective, $n$-faithful)-morphism and prove as \cref{thm:nsurj-nfaithful-fs-for-infty-k-cats} that they form factorization system. Lastly, we define the $n$-homotopy category functor in \cref{subsec:hom-cat-of-inf-k-cat}, state \cref{thm:inf-n-fancy-pullback} (regarding faithful functors and homotopy categories) and its consequences in \cref{subsec:n-1-faithful-and-hom-n-cat}, and prove \cref{thm:inf-n-fancy-pullback} in \cref{subsec:proof-factorization}.

In \cref{sec:SBim}, we define the relevant higher categories of Bott--Samelson bimodules and (chain complexes of) Soergel bimodules.
We construct the monoidal $(2,2)$-categories $\BSbimp$ in \cref{subsec:definining-BSBim} and $\SBim$ in \cref{subsec:definining-SBim}, whose hom categories are the categories of type A Bott--Samelson and Soergel bimodules, respectively. By construction, $\SBim$ carries all the desired extra structure such as local 
$\k$-linearity and $\mathbb{Z}$-action. In \cref{subsec:checking-SBim}, we verify that the hom-categories of $\SBim$ indeed agree with the categories $\SBim_n$. To finish the section, and to prove \cref{thm:main-M}, 
we construct in \cref{sec:defKSBim} the monoidal $(\infty, 2)$-category $\Kbloc(\SBim)$ whose hom categories are stable $\infty$-categories of chain complexes of type A Soergel bimodules, together with all its structure, and define the fiber functor $\Hloc$ from $\Kbloc(\SBim)$ to $\stkBZ$ in \cref{subsec:fiber-functor-on-KbSBim}.

In \cref{sec:prebraidings-to-E2-str}, we introduce the operadic machinery necessary for constructing the braiding on the monoidal $(\infty, 2)$-category $\Kbloc(\SBim)$. In particular, we formulate prebraidings using the language of $\infty$-operads and prove as \cref{cor:2operadA2} that braidings and prebraidings coincide when the ambient symmetric monoidal $\infty$-category (and more generally $\infty$-operad) is suitably truncated, generalizing the well-known statement that an $\EE_2$-algebra structure on a 1-category is the same as a braided monoidal structure. We extensively use Lurie's theory of $\infty$-operads \cite[\S~2]{HA} and refer the reader to \cref{appendix:operads} for an overview.

After a quick recollection on unital $\infty$-operads in \cref{subsec:operad-recollection}, we construct the $\T_2$ and $\AA_2$ operads, which are the main players of this section, in \cref{subsec:A2-T2}. After defining a relative version of $\T_2$ structure in \cref{subsec:relative-T2},
we review the $\infty$-categorical versions of centralizers and center \cite[\S~5.2]{HA} in \cref{subsec:centralizers-and-centers}.
Using the theory of centralizers, we prove as \cref{cor:classicalbraidings} that $\T_2$-structures are the $\infty$-categorical generalization of the notion of prebraiding. We construct the ($n$-surjective, $n$-faithful) factorization system on $\infty$-operads in \cref{subsec:fact-sys-on-operads} and prove a surjectivity result in \cref{subsec:lifting-operadic-structure}. In \cref{subsec:from-A2-E1-to-E-2} we prove  \cref{cor:2operadA2} (concerning prebraidings and braidings), which is the main result of this section. Finally, in \cref{subsec:lifting-maps-of-algebras} we end with an easy but useful result regarding lifting maps of algebras.

In \cref{sec:main-theorem}, we finally prove \cref{intro.maintheorem}. 
After defining the spaces of braidings and prebraidings in \cref{subsec:space-of-braidings}, in \cref{subsec:statement-of-main-thm} we state our main  \cref{thm:main-theorem-last-sec}, an abstract theorem about lifting prebraidings on homotopy $1$-categories to braidings, assuming the existence of a fiber functor. We expect this theorem can be used to build braidings on other interesting $(\infty,2)$-categories. 
Applying \cref{thm:main-theorem-last-sec} to the prebraiding on $h_1\BSBim \to h_1\oldKbloc{\SBim}$ constructed in \cref{sec:2}, we prove a precise version of \cref{intro.maintheorem} as \cref{cor:main-corollary}.
The remainder of the section is concerned with the proof of \cref{thm:main-theorem-last-sec}. 
The space of braidings is shown to be
equivalent, through a series of reduction steps, to spaces of prebraidings in
progressively less-structured situations, terminating in a set (rather than
a space) of prebraidings between the homotopy $1$-categories.
Assuming various truncatedness conditions, we prove those reduction steps are equivalences in \cref{subsec:from-prebraiding-to-braiding}.  Finally, in \cref{subsec:proof-of-main-thm} we prove those truncatedness hypothesis and complete the proof of \cref{thm:main-theorem-last-sec}.

In \cref{sec:appendix-recollections}, we provide a leasurely introduction to various aspects of $\infty$-categories and higher algebra. 

In \cref{app:f-s-for-enriched-cats}, we review the theory of factorization systems on $\infty$-categories and 
prove a number of ways of obtaining new factorization systems from existing ones, with a focus
on presentable $\infty$-categories. 
In particular, we prove two important results for constructing 
factorization systems, \Cref{thm:fs-and-alg} (concerning algebras over
$\infty$-operads) in \Cref{subsec:fs-for-algebras-over-opds} and
\Cref{thm:fs_and_enriched_cat} (concerning enriched $\infty$-categories) in \Cref{subsec:fs-for-Vcats}. 

\vspace{10pt}
\noindent \textbf{Acknowledgements.}
We would like to thank Markus Zetto for helpful comments on an early draft of this paper. YLL would like to thank 
David Ben-Zvi,
Tom Gannon,
Rune Haugseng,
Theo Johnson-Freyd,
Sam Raskin,
Tomer Schlank, and
Reuben Stern
for helpful discussions, and Mike Hopkins for his continuous support and encouragement.
AMG gratefully acknowledges inspiration from
David Ayala,
Clark Barwick,
David Ben-Zvi,
Justin Campbell,
Tom Gannon,
Tyler Lawson,
Sam Raskin,
Nick Rozenblyum,
Matt Stoffregen,
and
Mike Willis.
DR would like to thank Christopher Douglas, Theo Johnson-Freyd, and Kevin Walker for many illuminating discussions on higher categories, link homologies and TQFTs. 
PW would like to acknowledge the pioneering work of Scott Morrison and Kevin Walker towards TQFTs based on link homology theories and thank them for many illuminating conversations on the topic. 
CS would like to thank Gustavo Jasso and Stefan Schwede for many very useful discussions around higher categories.

\vspace{10pt}

\noindent \textbf{Funding.}
YLL is supported by the Simons Collaboration on Global Categorical Symmetries. 
AMG acknowledges the support of NSF grant DMS-2105031.
DR acknowledges support by the Emmy Noether program of the  Deutsche Forschungsgemeinschaft (DFG, German Research Foundation) – 493608176. 
DR and PW acknowledge support from the Deutsche
Forschungsgemeinschaft (DFG, German Research Foundation) under Germany's
Excellence Strategy - EXC 2121 ``Quantum Universe'' - 390833306 and the Collaborative Research Center - SFB 1624 ``Higher structures, moduli spaces and integrability''.
CS is supported by the Gottfried Wilhelm Leibniz Prize of the German Research Foundation. 

Finally, the authors acknowledge the important role of the Spring 2020 MSRI programs ``Higher Categories and Categorification'' and ``Quantum Symmetries'' supported by the National Science Foundation grant DMS-1440140 and the 2019 Erwin--Schr\"{o}dinger institute workshop ``Categorification in quantum topology
and beyond'' in catalyzing their collaboration.

\renewcommand{\bN}{\mathbb{N}}
\section{A prebraiding on the homotopy category of Soergel bimodules}
\label{sec:2}
\subsection{Review of Soergel bimodules and diagrammatics}
\label{sec:diagrammatics}

We let $\k$ denote the rationals $\Q$ or, more generally, a commutative
$\Q$-algebra. We consider the $\k$-linear monoidal categories
$\Sbimcl_n$ of Soergel bimodules for the symmetric group $S_n$ acting on its
natural representation. 
In this section, if not specified otherwise,  \emph{categories} mean ordinary
categories (in contrast to $\infty$-categories used later) and \emph{functors}
mean ordinary functors. For a fixed nonnegative integer $n$, let
$R_n=\k[x_1,x_2,\ldots, x_n]$ denote the polynomial ring over $\k$ in $n$
variables viewed as polynomial functions on $\mathfrak{h}^*=(\k^n)^*$ in the
standard way. Permuting the basis vectors of $\k^n$ induces a left action of the
symmetric group $W=S_n$ on $R_n$ such that the simple transposition
$s_i=(i,i+1)$ acts by swapping the variables $x_i$ and $x_{i+1}$. Denote
$\check\alpha_i=x_{i}-x_{i+1}$ for $1\leq i\leq n-1$. Then restriction to the
span of the $\check\alpha_i$'s gives the usual geometric representation of $W$
viewed as the Coxeter group generated by the simple transpositions. For any subgroup
$G$ of $W$ let $R_n^G$ be the subalgebra of $G$-invariants in $R_n$. In case
$G=\langle s_i\rangle$ for some $1\leq i\leq n-1$ we abbreviate $R_n^G=R_n^i$.
We will view $R_n$ as a graded (by which we mean $\Z$-graded) algebra by putting
the generators $x_i$ in degree $2$. Note that $R_n^i$ is a graded subalgebra and we
have a canonical, grading-preserving decomposition 
\begin{equation}
    \label{eq:decomp}
    R_n= R_n^i\oplus \check\alpha_iR_n^i\simeq R_n^i\oplus R_n^i\langle 2\rangle
\end{equation} as
graded $R_n^i$-bimodules. Here and in the following we denote for $j\in \Z$ and
a graded (bi)module $M=\oplus_{i\in\Z} M_i$ by $M\langle j\rangle$ the graded
(bi)module which equals $M$ as (bi)module but with the grading shifted up by
$j$, i.e. $M\langle j\rangle_i=M_{i-j}$. The grading shifting functors $\langle
j\rangle$, $j\in\Z$ equip the category of graded $(R_n,R_m)$-bimodules for fixed
$n,m$ with an action of the group $\Z$. 

By a graded $\k$-linear category\footnote{Note that this is not the same concept as a
$\k$-linear category that is also graded by a group.} we mean a category
enriched in $\Z$-graded $\k$-modules. As an example we can take as objects
graded $R_n$-bimodules with all $R_n$-bimodule maps, denoted
$\Hom^{\mathrm{gr}}$. In this case, the grading shift functors are compatible
with the grading on morphisms as follows:
\begin{equation}
    \label{eq:grading}
    \Hom^{\mathrm{gr}}(M\langle k\rangle,N\langle l\rangle) = \Hom^{\mathrm{gr}}(M,N)\langle l-k\rangle
\end{equation} 

\begin{definition}
    \label{def:BSbimclcl}
The \emph{graded $\k$-linear category of Bott--Samelson bimodules} for $R_n$ is the graded $\k$-linear full
subcategory $\BSbimcl^{\mathrm{gr}}_n$ of $R_n$-bimodules given by all graded $R_n$-bimodules 
of the form:
\begin{equation}
\label{eq:BS}
B_{\mathbf{i}}\langle
j\rangle:=R\otimes_{R^{s_{i_k}}}R\otimes_{R^{s_{i_{k-1}}}}\cdots\otimes_{R^{s_{i_1}}}R\langle
j-k\rangle
\end{equation} 
for $R=R_n$, $j\in\Z$ and some $\mathbf{i}=(i_k,i_{k-1},\ldots, i_1)\in
\{1,2,\ldots n-1\}^k$ with $k \in \N_0$, including the bimodules $R\langle j\rangle$ in case $k=0$. In the case $k=1$ we also abbreviate:
\[
B_i:=     R\otimes_{R^{s_{i}}}R\langle-1\rangle
\]
\end{definition}

Using this shorthand, \eqref{eq:BS} may also be expressed as:
\begin{equation}
    \label{eq:BSB}
    B_{\mathbf{i}}\langle
    j\rangle \simeq B_{i_k} \otimes_{R} B_{i_{k-1}}\otimes_R \cdots\otimes_R B_{i_1}\langle
    j\rangle
    \end{equation} 

    \begin{definition}
        \label{def:BSbimn}
There are two common variations of \cref{def:BSbimclcl}:
\begin{eqnarray}
    \label{eq:BSBimn}
\BSbimcl_n\quad{\leftsquigarrow}  
&\BSbimcl^{\mathrm{gr}}_n&
{\rightsquigarrow} \quad
\overline{\BSbimcl}^{\mathrm{gr}}_n
\end{eqnarray}
\begin{itemize}
    \item Namely, $\BSbimcl_n$ is the $\k$-linear (but no longer graded
    $\k$-linear) category obtained by restricting to the degree zero part of the
    morphism spaces. We call this the \emph{degree zero subcategory}. (In the
    language of enriched category theory, this is the \emph{underlying category}
    of the category enriched in graded vector spaces; it inherits the
    linear structure.) By remembering the $\Z$-action by grading shift functors,
    all other homogeneous components of morphism spaces in
    $\BSbimcl^{\mathrm{gr}}_n$ can be recovered from \eqref{eq:grading}. 
    \item Alternatively, one can consider the graded $\k$-linear full
    subcategory $\overline{\BSbimcl}^{\mathrm{gr}}_n$  on \emph{unshifted} Bott--Samelson bimodules $B_{\mathbf{i}}$,
    i.e. where $j=0$ in \eqref{eq:BS}. From this category one can reconstruct the morphism
    spaces between shifted Bott--Samelsons, that is all objects in $\BSbimcl^{\mathrm{gr}}_n$, again via \eqref{eq:grading}. 
\end{itemize}
We refer to \cite[(2.1)]{MazStr} for a discussion of these essentially
equivalent ways of handling graded $\k$-linear categories. 
\end{definition}
 
The version $\BSbimcl^{\mathrm{gr}}_n$ in Definition~\ref{eq:decomp} is the most
flexible one, but with one caveat: when making statements about isomorphism,
idempotents, and categorical constructions such as (co)products, we tacitly
require that the structure morphisms are of degree zero, see e.g.
\eqref{eq:decomp}, i.e. we work in the underlying category. For this reason, we
will henceforth almost exclusively work with $\BSbim_n$. The only exception is
\cref{sec:prebraiding}, where we use the version
$\overline{\BSbimcl}^{\mathrm{gr}}_n$ to connect to the diagrammatic Hecke
category.

\begin{remark}
    \label{rem:projective}
As defined, $\BSbimcl_n$ is a monoidal full subcategory of the $\k$-linear
category of graded $R_n$-bimodules and grading-preserving bimodule maps,
with tensor product $-\otimes_{R_n}-$. More precisely, since each $B_i$ is free
of rank $2$ as a graded $R_n$-module from the left and from the right, all
objects of $\BSbimcl_n$ are finitely generated
graded-projective\footnote{\label{ftn:graded-projective}By a finitely generated graded-projective $A$-module
for a graded algebra $A$ we mean an $A$-module that is a retract of a finite
direct sum of grading shifts of the rank 1 free $A$-module along grading
preserving maps.} $R_n$-modules from both sides and the tensor product coincides
with the derived tensor product. 
\end{remark}

\begin{definition}
    \label{def:Sbimcl}
The monoidal $\k$-linear category $\Sbimcl_n$ of \emph{Soergel bimodules for
$R_n$} is the Karoubian closure, that is the smallest additive idempotent-complete full subcategory of graded $R_n$-bimodules containing $\BSbimcl_n$.
\end{definition}

For later use, we also record how Bott--Samelson and Soergel bimodules for
various $n$ can be related.

\begin{definition}
    \label{def:parind}
Given $a,b,c\in\N_0$ let $j_{a|c}=j_{a|c}^b\colon R_b\hookrightarrow R_{a+b+c}$ be the algebra
homomorphism given by $x_i\mapsto x_{i+a}$. Given an $R_m$-bimodule $M$ and an
$R_n$-bimodule $N$, the tensor product $M\otimes_\k N$ is an $R_m\otimes
R_n$-bimodule, hence an $R_{m+n}$-bimodule via the isomorphism $j_{0|n}\otimes
j_{m|0}$. We call this functorial operation {\it parabolic induction}. 
\end{definition}

It is straightforward to see directly from Definition~\ref{def:BSbimclcl} that Bott--Samelson bimodules are sent to (bimodules isomorphic to) Bott--Samelson bimodules under parabolic induction. To distinguish the two kinds of tensor product, we will use the convention:
\begin{align*}
    \hcomp := \otimes_{R_n} \colon \BSbimcl_n\times \BSbimcl_n&\to \BSbimcl_n\\
    \boxtimes := \otimes_{\k} \colon \BSbimcl_m\times \BSbimcl_n&\to \BSbimcl_{m+n}
\end{align*}
and write $f\vcomp g \colon M\rightarrow P$ for the composition of morphisms
$f\colon M\rightarrow N$, $g\colon N\rightarrow P$ in $\BSbimcl_n$. We use the
same notation for $\Sbimcl_n$. 

The symbols $\vcomp$, $\hcomp$ and $\boxtimes$ are meant to foreshadow that
these operations should form the $2$- and $1$-morphism composition and the tensor
product in a monoidal bicategory, see \cref{rem:monbicat}.

\subsection{Review of Rouquier complexes}
\label{subsec:Rouquier}

\begin{nota}
    \label{conv:oldKb}
    For any $\k$-linear  category $\cC$ with zero object, we write $\Chb{\cC}$ for the
    $\k$-linear category of \emph{bounded (on both sides) chain complexes} in $\cC$, with chain
    maps as morphisms. If $\cC$ is additive (and thus has a zero object) or equipped with
    a $\Z$-action, then so is $\Chb{\cC}$. If $\cC$ is additive and equipped
    with a monoidal structure compatible with $\oplus$, then this is
    inherited by $\Chb{\cC}$. There is a natural notion of homotopy between
    chain maps and the nullhomotopic chain maps form a $\k$-linear (monoidal)
    ideal. 

    \begin{definition}\label{def:classicalK}
    The quotient of $\Chb{\cC}$ by the nullhomotopic chain maps is the 
    \emph{chain homotopy category}  $\oldKb{\cC}$. An isomorphism between objects of
    $\oldKb{\cC}$ is called a \emph{chain homotopy equivalence}. \footnote{The chain homotopy category $\oldKb{\cC}$ is often just called `homotopy category'; we here reserve the latter term for the more general construction of a $1$-category from a higher category, see \Cref{def:homotopy-cat-definition}.}
    \end{definition} 
\end{nota}

\begin{definition} For  $n\geq 2$ we denote by $\Braidg_n$ the braid group with
(Artin) generators  $\Ag_i$, $1\leq i\leq n-1$.  Given a generator or its
inverse, we will consider the following complexes in $\Chb{\Sbimcl_n}$:
\begin{equation}
    \label{eqn:Rouq-alg}
    \Rouq(\Ag_i) 
    := \left( 0 \xrightarrow{} \uwave{B_i} \xrightarrow{m} R\langle-1\rangle \xrightarrow{} 0 \right)\, , \quad
    \Rouq(\Ag_i^{-1}) 
    := \left( 0 \xrightarrow{} R\langle 1\rangle \xrightarrow{\Delta} \uwave{B_i} \xrightarrow{} 0 \right)
\end{equation}
Here the $\uwave{\textup{underlined}}$ part is in homological degree zero, $m$ is induced by the multiplication map $B_i=R \otimes_{R^{s_i}} R \langle
-1 \rangle \to  R\langle-1\rangle$, and $\Delta$ is the bimodule map determined
by $1 \mapsto x_i\otimes 1 - 1\otimes x_{i+1}$.
An expression
$\ub=\Ag_{i_1}^{\epsilon_1}\cdots\Ag_{i_r}^{\epsilon_r}$ with
$\epsilon_j\in\{\pm\}$  is called a \emph{braid word} with corresponding
braid element $\beta\in \Braidg_n$. The word is \emph{positive} if $\epsilon_j=1$ for
$1\leq j\leq r$. Given  $\ub$ define
\begin{equation*}
    \label{eqn:Rouq-algb}
    \Rouq(\ub) 
    := 
    \Rouq(\Ag_{i_1}^{\epsilon_1}) \hcomp \cdots \hcomp \Rouq(\Ag_{i_r}^{\epsilon_r})
\end{equation*}
where we make use of the horizontal composition $\Chb{\Sbimcl_n}$ (given by the obvious extension of $\hcomp=\otimes_R$). By convention, the empty braid word gives $\Rouq(\emptyset)=R$. 
\end{definition}
The complexes  $ \Rouq(\ub) $ are called \emph{Rouquier complexes}. They
were first thoroughly studied by Rouquier who proved in
\cite{0409593} that,  up to canonical homotopy equivalence, these complexes are
independent of the chosen braid word representing $\beta$. More precisely the
following holds:

\begin{thm}[Rouquier canonicity]
    \label{thm:Rouquier-canonicity}
    Let $\ub_1$ and $\ub_2$ be braid words representing the same braid $\beta$,
    then there exist homotopy equivalences
    \[\psi_{\ub_1,\,\ub_2}\colon \Rouq(\ub_1) \to \Rouq(\ub_2)\] which form a
transitive system, i.e. if $\ub_3$ is a third braid word representing the same
braid, then 
\[\psi_{\ub_2,\,\ub_3}\vcomp \psi_{\ub_1,\,\ub_2} \htc \psi_{\ub_1,\,\ub_3}.\]
If moreover $\ub_1'$, $\ub_2'$ are braid words representing another braid $\beta'$ then we have
\begin{equation}
\label{eqn:Rouquier-can-tensor}
  \psi_{\ub_1\ub_1',\ub_2\ub_2'} \htc \psi_{\ub_1,\ub_2}\hcomp \psi_{\ub_1',\ub_2'}.
\end{equation}
\end{thm}
This rephrasing of Rouquier's results from \cite{0409593} is a slight strengthening of \cite[Prop.~2.19]{elias2017categorical}.
As a consequence we may abuse notation and write $\Rouq(\beta)$ instead of $\Rouq(\ub)$.

\begin{corollary}
    The Rouquier complexes $\Rouq(\beta)$ are invertible objects in the monoidal
    category $\oldKb{\Sbimcl_n}$.
\end{corollary}

\begin{rem}
    \label{rem:Rouquierswap}
    For $n \geq 2$ and $R=R_n$ as above and $w\in S_n$ we let $R_{\circlearrowleft w}$ denote the
    graded $R$-bimodule which is isomorphic to $R$ as left $R$-module and with
    right-action twisted by $w$: i.e. $r\in R$ acts on $R_w$ from the right as
    multiplication by $w(r)$. We emphasize that for non-trivial $w$, this $R_n$-bimodule $R_{\circlearrowleft w}$ is \emph{not} an object of $\SBim_n$.

   However, for $1\leq i\leq n-1$, the bimodule morphism $R_{\circlearrowleft s_i}\langle
    1\rangle \to B_i$ determined by $1\mapsto x_i\otimes 1 - 1\otimes x_i$
    induces a quasi-isomorphism $R_{\circlearrowleft s_i}\langle 1\rangle \to F(\sigma_i)$.
    Likewise, the multiplication map $B_i \to R_{\circlearrowleft s_i}\langle -1 \rangle$
    determined by $1\otimes 1 \mapsto 1$ induces a quasi-isomorphism
    $F(\sigma^{-1})\to R_{\circlearrowleft s_i}\langle -1\rangle$.     

    Up to a grading shift, the generating Rouquier complexes, and more
    generally, the Rouquier complexes of positive resp. negative permutation
    braids, can hence be identified with permutation bimodules upon proceeding
    to the derived category $\oldDb{{}_R\mathrm{grbmod}_{R}}$ of graded
    $R$-$R$-bimodules. To obtain an interesting (non-symmetric) braiding, it is thus essential to work up-to-chain-homotopy, rather than up-to-quasi-isomorphism. Nevertheless, the comparison with permutation bimodules is important in this paper and the grading shifts in the following definition are motivated by it.
\end{rem}

\begin{definition}
\label{def:cabledcross}
For $m,n\geq 0$, we define the \emph{braiding complexes}:
\begin{align}
    \cabledcross_{m,n}&:=\Rouq( (\Ag_{n}\cdots\Ag_{1})\cdots (\Ag_{i+n-1}\cdots\Ag_{i}) \cdots (\Ag_{m+n-1}\cdots\Ag_{m}) )\langle -m n \rangle\\
    \cabledcross'_{m,n}&:=\Rouq( (\Ag\inv_{n}\cdots\Ag\inv_{m+n-1}) \cdots (\Ag\inv_{i}\cdots\Ag\inv_{i+m-1}) \cdots (\Ag\inv_{1}\cdots\Ag\inv_{m}) )\langle m n \rangle
\end{align}
The braiding complexes $\cabledcross_{m,n}$
(resp. $\cabledcross'_{m,n}$) will be called positive (resp. negative)
\emph{cabled crossings complexes} or just  \emph{cabled crossings}, since the underlying braids are cabled crossings.

The braids appearing in the special cabled crossings $\cabledcross_{m,1}$,
$\cabledcross'_{m,1}$, $\cabledcross_{1,n}$, and $\cabledcross'_{1,n}$ will be
called \emph{Coxeter braids}, since they are braid versions of Coxeter
words, see the illustrations in \Cref{fig:cabledcross} (as for functions, we compose functors and read diagrams from right to left. The Artin generator $\Ag_{i}$ acts on the $i$-th and $(i+1)$-th strand from the bottom.).
\end{definition}

\begin{figure}[ht]
    \[
\cabledcrossfig
    \]
    \caption{Cabled crossings $\cabledcross_{2,3}$, $\cabledcross'_{3,2}$, and 
    Coxeter braids $\cabledcross_{1,4}$, $\cabledcross'_{1,4}$.}
    \label{fig:cabledcross}
    \end{figure}

\begin{lemma}
    \label{lem:cabled-crossing-from-coxeter}
    Cabled crossings are built from Coxeter braids. For $m,n\geq 0$, we have 
\begin{align*}
\cabledcross_{m,n} 
&\simeq 
(\cabledcross_{1,n}\boxtimes \one_{m-1}) 
\hcomp \cdots \hcomp 
(\one_{m-1-i} \boxtimes \cabledcross_{1,n}\boxtimes\one_{i})
\hcomp \cdots \hcomp
(\one_{m-1}\boxtimes \cabledcross_{1,n})\\
& \hte
(\one_{n-1}\boxtimes \cabledcross_{m,1})
\hcomp \cdots \hcomp 
(\one_{i} \boxtimes \cabledcross_{m,1}\boxtimes\one_{n-1-i})
\hcomp \cdots \hcomp
(\cabledcross_{m,1}\boxtimes \one_{n-1} )
\\
\cabledcross'_{m,n} 
&\simeq
(\one_{n-1}\boxtimes \cabledcross'_{m,1})
\hcomp \cdots \hcomp 
(\one_{i} \boxtimes \cabledcross'_{m,1}\boxtimes\one_{n-1-i})
\hcomp \cdots \hcomp
(\cabledcross'_{m,1}\boxtimes \one_{n-1} )\\
& \hte
(\cabledcross'_{1,n}\boxtimes \one_{m-1}) 
\hcomp \cdots \hcomp 
(\one_{m-1-i} \boxtimes \cabledcross'_{1,n}\boxtimes\one_{i})
\hcomp \cdots \hcomp
(\one_{m-1}\boxtimes \cabledcross'_{1,n})
\end{align*}
\end{lemma}

\begin{proof} The isomorphisms hold by associativity of $\hcomp$. The
homotopy equivalences come from applying braid relations, see
\Cref{thm:Rouquier-canonicity}.
\end{proof}

\subsection{Isomorphism classes of Soergel bimodules}\label{subsec:h1BSbimcl} We
have already seen concrete hints that Soergel bimodules for symmetric groups
form a monoidal bicategory and in \cref{sec:prebraiding} we will see elements of
a braiding on the homotopy category. The rigorous construction of the braiding
will be carried out in an $\infty$-categorical setting later, but it requires a
coupling to the classical setting here to enable a few critically important
computations. The optimal handover point between these two worlds turns out to
be one categorical dimension lower than we have been working in so far. In this
section, we prepare the descent to this common ground from the classical side.

\begin{definition}
    \label{def:oldhoneBSbim}
    We let $h_1\BSbimcl$ denote the  $1$-category, whose set of object
    is $\bN_0$ and whose morphism sets between objects $n, m$ are
    \begin{equation}
        \Hom_{h_1\BSbimcl}(n, m) = \left\{
            \begin{array}{ll}
                h_0\BSbimcl_n & \quad n = m \\
              \{0\}  & \quad n \neq m
            \end{array}
        \right.
    \end{equation}
    where $h_0 \BSbimcl_n$ denotes the set of isomorphism classes of objects in $\BSbimcl_n$, and whose composition of morphisms $n \to n$ is induced by the monoidal structure $\otimes_{R_n}$ of $\BSbimcl_n$. 
    The category $h_1\BSbimcl$ admits a monoidal structure  with monoidal product $\boxtimes\colon h_1\BSbimcl \times h_1\BSbimcl \to h_1\BSbimcl$ defined on objects by $n \boxtimes m = n+m$ and on morphisms using parabolic induction:
    \begin{equation}
       h_0 \boxtimes \colon h_0\BSbimcl_n \times h_0\BSbimcl_m \to h_0\BSbimcl_{n+m}.
    \end{equation}
\end{definition}

The currently ad-hoc notation $h_1$ will be justified later in \cref{cor:uniqueBSBim} when we pass to $\infty$-categories.

\begin{definition}
    \label{def:oldhoneKbloc} Let $h_1\oldKbloc{\Sbimcl}$ be the monoidal
    $1$-category whose objects are $n \in \bN_0$ and the morphisms between objects
    $n, m$ are
    \begin{equation}\label{def:h1BSbimcl}
        \Hom_{h_1\oldKbloc{\Sbimcl}}(n, m) = \left\{
            \begin{array}{ll}
                h_0\oldKb{\Sbimcl_n} & \quad n = m \\
                \{0\} & \quad n \neq m
            \end{array}
        \right.
    \end{equation}
    The monoidal product is induced by parabolic induction on chain complexes of
    Soergel bimodules and again we have $\Z$-actions on the morphism sets,
    inherited from grading shifts of bimodules. 
    Since the inclusion $\BSbimcl_n
    \to\oldKb{\Sbimcl_n}$ of Bott--Samelson bimodules as chain complexes
    concentrated in homological degree zero is compatible with parabolic induction, this
    defines a monoidal functor 
    \begin{equation}\label{eq:defh1K}
        \hincl \colon h_1\BSbimcl \to h_1\oldKbloc{\Sbimcl}.
    \end{equation}
    In fact, this intertwines the $\Z$-actions on morphism sets.
    \end{definition}
    In \cref{cor:h1KSBim}, we match \eqref{eq:defh1K} with its $\infty$-categorical version.

\begin{remark}\label{rem:monbicat} We leave it to the reader to check that the
   involved categories are monoidal as claimed. The conceptual reason behind
   this is that these categories are the shadow of monoidal bicategories in the
   sense of \cite{Benabou} (the objects are the same, but morphism categories
   are given for instance by $\oldKb{\Sbimcl_n}$ instead of the sets
   $h_0 \oldKb{\Sbimcl_n}$  in \eqref{def:h1BSbimcl}). Although these monoidal
   bicategories play an important conceptual role in represention theory and
   quantum topology, see e.g. \cite{EW,HRW1,HRW2}, the construction of the
   monoidal structure has not yet appeared in full detail in the literature. In
   the world of $\infty$-categories we will obtain an analogous construction in
   \cref{sec:SBim}.
\end{remark}

\begin{definition}
    \label{def:oldhoneMorita}
    Given a graded algebra $A$, we call an object of the derived category $\oldD(\mathrm{grmod}_A)$ of graded $A$-modules \emph{graded-perfect} if it is quasi-isomorphic to a finite chain complex of finitely generated graded-projective $A$-modules (see \cref{ftn:graded-projective}). 
    
    Let $h_1\DMorPoly$ be the symmetric monoidal $1$-category whose objects are the graded algebras $R_n= k[x_1, \ldots, x_n]$ for $n \in \bN_0$ and whose morphism sets between algebras $R_n$ and $R_m$ are given by the set 
    \[h_0 \oldD\left( {}_{R_n} \mathrm{grbmod}_{R_m}\right)^{\mathrm{gr-perf}}
    \]
    of isomorphism classes of objects in the derived category of graded $R_n$--$R_m$ bimodules which are graded-perfect as right (i.e. $R_m$-)modules; composition is the derived graded  tensor product over the respective polynomial algebras. 
   Similar to \cref{def:parind}, the monoidal structure is given by the derived graded  tensor product $\otimes_k^L$ over the ground ring $k$, under the identification $R_n \otimes^{L}_k R_m \simeq R_n \otimes_k R_m \simeq R_{n+m}$. \end{definition}

\newcommand{\honeloc}{h_1 H_{\loc}}

\begin{definition}
We define the monoidal functor 
\begin{equation}
    \label{eq:h1Hloc} 
    h_1 H_{\loc}\colon
h_1 \oldKbloc{\Sbimcl}  \to h_1\DMorPoly,
\end{equation} which takes an object $n$ to the polynomial algebra $R_n$ and a chain homotopy equivalence class of a chain complex of Soergel bimodules to the corresponding quasi-isomorphism class of complexes of graded bimodules.\end{definition} 
The notation $h_1 H_{\loc}$ will be justified in~\cref{handover2}, the notation $H_{\loc}$ indicates `taking homology at $1$-morphism level'.

Our first main result is the construction of a \emph{prebraiding structure} on \eqref{eq:defh1K}.

\begin{remark}\label{rk:flattarget}
In the $\infty$-categorical setting, we will replace $h_1\DMorPoly$ with  a more natural target category with less restrictions on objects.  Namely, the category $h_1\DMorPoly$ is a full symmetric monoidal subcategory of the category $h_1\DMoritaS$ whose objects are arbitrary \emph{flat} graded algebras, and morphims are isomorphism classes of right-graded-perfect derived bimodules between them. By passing to module categories over these algebras, this can in turn be realized as a full subcategory of the category $h_1 \stkBZ$ of stable $k$-linear categories with a $\mbbZ$-action and equivalence classes of $k$-linear exact $\mbbZ$-equivariant functors between them.  In the next sections, we will lift the composite functor $h_1 \oldKbloc\SBim \to h_1 \DMorPoly \to h_1\stkBZ$ to a functor of $(\infty,2)$-categories.
    \end{remark}
 
\subsection{Prebraidings}\label{sec:prebraidCat}

As we will see, the cabled crossing complexes from Definition~\ref{def:cabledcross} supply part of the data of a braiding on (an $\infty$-categorical version of) chain complexes of Soergel bimodules. What these complexes themselves do not yet encode is the {\it naturality} of the braiding---informally speaking, how chain complexes of Soergel bimodules \emph{slide through} cabled crossings up to coherent homotopy. To capture the naturality of the braiding, we start with a threefold simplified situation: we only aim to slide bimodules (instead of complexes thereof!) through cabled crossings and in fact only Bott--Samelson bimodules, and even simpler, it will be enough to do this on the level of isomorphism classes. To this end, we introduce the crucial notion of a prebraiding.

\begin{definition}
    \label{def:prebraiding}
Let $\cA$ and $\cB$ be monoidal $1$-categories, with monoidal product denoted by
$\boxtimes$ in both cases and with associators $b_{x,y,z}$ in $\cB$. A \emph{prebraiding} $\beta$ on a monoidal functor $F\colon \cA \to \cB$ consists of the data
of isomorphisms
\[ F(x)\boxtimes F(y) \xrightarrow{\beta_{x,y}} F(y) \boxtimes F(x)\qquad
\forall x,y\in \cA \] that form a natural transformation $\boxtimes\circ (F\times
F) \Rightarrow \boxtimes^{\mathrm{op}}\circ (F\times F)$ and satisfy the
following two \emph{hexagon axioms} 
 for all $x,y,z\in \cA$:
    \begin{equation}
    \label{eq:fakehexagon}
        \begin{tikzcd}[scale cd=.75]
       (F(x)\boxtimes F(y))\boxtimes F(z) \arrow[d,"b"] \arrow[r,"\beta_{x,y} \boxtimes \id"]
       & (F(y)\boxtimes F(x))\boxtimes F(z) \arrow[r,"b"] 
       &  F(y)\boxtimes (F(x)\boxtimes F(z)) \arrow[r,"\id\boxtimes \beta_{x,z}"] 
       &   F(y)\boxtimes (F(z)\boxtimes F(x)) \arrow[d,"b^{-1}"]
       \\ 
       F(x)\boxtimes (F(y)\boxtimes F(z))\arrow[r,"\simeq"]
        & F(x)\boxtimes F(y\boxtimes z) \arrow[r,"\beta_{x,y\boxtimes z}"]
        & F(y \boxtimes z)\boxtimes F(x) \arrow[r,"\simeq"] 
        & (F(y)\boxtimes F(z)) \boxtimes F(x)
        \\
       F(x)\boxtimes (F(y)\boxtimes F(z)) \arrow[d,"b^{-1}"] \arrow[r,"\id\boxtimes \beta_{y,z}"]
       & F(x)\boxtimes (F(z)\boxtimes F(y)) \arrow[r,"b^{-1}"] 
       &  (F(x)\boxtimes F(z))\boxtimes F(y) \arrow[r," \beta_{x,z}\boxtimes\id"] 
       &   (F(z)\boxtimes F(x))\boxtimes F(y) \arrow[d,"b"]
       \\ 
       (F(x)\boxtimes F(y))\boxtimes F(z)\arrow[r,"\simeq"]
        & (F(x\boxtimes y))\boxtimes F(z) \arrow[r,"\beta_{x\boxtimes y,z}"]
        & F(z) \boxtimes F(x\boxtimes y) \arrow[r,"\simeq"] 
        & F(z)\boxtimes (F(x) \boxtimes F(y))
        \end{tikzcd} 
     \end{equation}
     where the isomorphisms $\simeq$ are part of the data of $F$. (Supressing them provides the hexagon shapes.) 
\end{definition}

We have the following trivial observation:

\begin{corollary}\label{cor:prebraidisbraid}
    Let $\cA$ be a monoidal $1$-category. A prebraiding $\beta$ on the identity functor $\Id\colon \cA \to \cA$ is a braided monoidal structure on $\cA$ in the sense of \textup{\cite[Def.~8.1.1.]{EGNO}}.
\end{corollary}

\begin{rem}
\label{rem:noRthree}
Observe however that a prebraiding is not (!) required to satisfy an analog of the
braid relation (a.k.a. third Reidemeister move) of the form 
\begin{equation}
\begin{gathered}\label{braid}
b_{F(z),F(y),F(x)}\circ (\beta_{y,z}\boxtimes \id)\circ  b^{-1}_{F(y),F(z),F(x)}\circ(\id \boxtimes \beta_{x,z}) \circ b_{F(y),F(x),F(z)}\circ (\beta_{x,y}\boxtimes \id)\\
=
(\id \boxtimes \beta_{x,y})\circ  b_{F(z),F(x),F(y)}\circ(\beta_{x,z}\boxtimes \id) \circ b^{-1}_{F(x),F(z),F(y)}\circ (\id \boxtimes \beta_{y,z})\circ b_{F(x),F(y),F(z)}
\end{gathered}
\end{equation}
\end{rem}
\begin{rem}
    \label{rem:onecatprebraid}
 In the situation of Corollary~\ref{cor:prebraidisbraid}, the braid relation \eqref{braid}
    holds, because it can be proven using the naturality of $\beta$. There are in
    fact two distinct proofs, namely by
    sliding either of the two highlighted crossings under the remaining strand:
        \[
        \twoproofs
        \]        
    In a higher-categorical version of a prebraiding, these two witnesses for the braid relation need not be realized by the same 2-morphism. However, in the axiomatics of braided monoidal 2-categories, the cells witnessing these two proofs are equated by the so-called $S_+=S_-$ relation of \cite{MR1402727} (which was omitted in \cite{KapVoe}). For a monoidal higher category, a prebraiding on the identity functor therefore does not imply the braid relations \eqref{braid}, see also \Cref{rem:twoproofs}, and in particular does not encode a braided monoidal (i.e. $\EE_2$-)structure. In \cref{sec:prebraidings-to-E2-str}, we will revisit this point and show that a prebraiding on the identity functor always encodes an $\AA_2 \otimes \EE_1$-structure, which differ in general from $\EE_2$-structures. 
    \end{rem}

We also want to introduce a relative notion of prebraiding, over a braided monoidal $1$-category:
\begin{definition} \label{def:relative-prebraiding}
    Let $\cD$ be a braided monoidal $1$-category. Assume $\cC_1$ is a monoidal $1$-category, and $\cC_2$ is a monoidal
    category \emph{over} $\cD$, i.e. equipped with a monoidal functor 
    $g\colon \cC_2\to \cD$. 
 Let now $F\colon \cC_1 \to \cC_2$ be a monoidal functor. Then we can consider $C_1$ as a monoidal $1$-category over $\cD$, namely with respect to $f \coloneqq g\circ F \colon \cC_1 \to \cD$, and $F$
    becomes a monoidal functor \emph{over} $\cD$.  
    
    A \emph{prebraiding over $\cD$ on $F \colon \cC_1 \to \cC_2$}
    is then defined to be a prebraiding on $F$ as in Definition~\ref{def:prebraiding}, satisfying the additional condition that $g$ maps the prebraiding isomorphisms in $\cC_2$ to the given braiding isomorphisms in $\cD$.
   \end{definition} 

 \begin{remark} More explicitly, with $\cD,\cC_1,\cC_2,F,g,f$ as in \Cref{def:relative-prebraiding}, a prebraiding on $F$ with components $\beta_{x,y}$ is a \emph{prebraiding on $F$ over $\cD$} if the isomorphism \[g \circ \beta_{x,y}\colon  g(F(x))
    \boxtimes g(F(y))\simeq g(F(x) \boxtimes F(y)) \to g(F(y) \boxtimes F(x))
    \simeq g(F(y)) \boxtimes g(F(x))\]  coincides with the given braiding isomorphism on $\cD$, i.e. with $f(x)
    \boxtimes f(y) \to f(y)\boxtimes f(x)$ for all pairs of objects $x,y\in\cC_1$. 
 \end{remark}
     \begin{corollary}\label{cor:prebraidisbraid2}
    Let $\cA$ with $g \colon \cA\rightarrow \cD$  be a monoidal $1$-category over $\cD$. Then a
    prebraiding over $\cD$ on the identity functor $\Id\colon \cA \to \cA$ is a braided monoidal structure on $\cA$ with the property that $g$ is braided monoidal in the sense of \textup{\cite[Def.~8.1.7.]{EGNO}}.
\end{corollary}

\begin{definition}\label{def:notationprebraiding}
For a monoidal functor $F\colon \cA \to \cB$, we denote by $\PreBraid(F)$ the set of prebraidings of $F$ and for a monoidal category $\cA$, we denote by $\Braid(\cA)\coloneqq\PreBraid(\Id\colon \cA \to \cA)$ the set of compatible braidings on $\cA$. The relative versions of \Cref{def:relative-prebraiding} are denoted by $\PreBraid_{/\cD}(F\colon \cA \to \cB)$ and $\PreBraidid_{/\cD}(\cA)$ respectively.
\end{definition}

\subsection{Prebraiding for Soergel bimodules}\label{sec:prebraiding} We are now
prepared to construct a prebraiding on the functor $\hincl \colon h_1 \BSbimcl  \to
h_1 \oldKbloc{\Sbimcl} $ from \eqref{eq:defh1K}.

\begin{definition}
\label{const:prebraid}
For $m,n\in \N$ let
$\beta_{m,n}\colon m+n\to n+m$ be the morphism in $h_1 \oldKbloc{\Sbimcl} $ given by the chain homotopy class of the shifted Rouquier complex $(X_{m,n})$ defining the cabled crossing in \Cref{def:cabledcross}. 
\end{definition}

\begin{theorem}[Prebraiding for Soergel bimodules]
\label{thm:prebraidingSbimcl}
    The family of morphisms $\beta_{m,n}$ from
    \Cref{const:prebraid}  constitute a prebraiding on the functor
    $\hincl\colon h_1 \BSbimcl  \to h_1 \oldKbloc{\Sbimcl} $.
\end{theorem}

Considering $\hincl\colon h_1 \BSbimcl  \to h_1 \oldKbloc{\Sbimcl} $ as a functor over $h_1\DMorPoly$ via the functor $h_1H_{\loc}\colon h_1 \oldKbloc{\Sbimcl} \to h_1\DMorPoly$, we obtain the following refined version:
\begin{corollary}[Relative prebraiding on $\hincl$]\label{cor:Rouquier-correct-prebraiding}
    The Rouquier complexes of cabled crossings define a prebraiding on
    $\hincl\colon h_1 \BSbimcl  \to h_1 \oldKbloc{\Sbimcl} $ over $h_1\DMorPoly$.
\end{corollary}
\begin{proof}
By Remark~\ref{rem:Rouquierswap}, the braiding complexes are quasi-isomorphic to
the associated permutation bimodules, concentrated in homological degree $0$. 
As the permutation bimodules implement the symmetric braiding, this means that
the braiding complexes constructed as (shifted) Rouquier complexes of the cabled
crossings become the canonical symmetric braiding in $h_1\DMorPoly$. 
\end{proof}

The following result is the crucial naturality part for the proof of \Cref{thm:prebraidingSbimcl}. The result is in fact stronger than needed, since it is a statement on the chain level.
\begin{theorem}
    \label{prop:sliding-objects}
    For any Bott--Samelson bimodule of the form $Y=Y_1\boxtimes Y_2$ in $\BSbimcl_m\boxtimes \BSbimcl_n\subset \Sbimcl_{m+n}$, there are
    homotopy equivalences of chain complexes in $\Chb{\Sbimcl_{m+n}}$:
    \begin{align*}
    \slide_{Y_1,Y_2}\colon & \cabledcross_{m,n} \hcomp Y \longrightarrow \swap_{m,n}(Y)\hcomp \cabledcross_{m,n} 
    \end{align*}
    \end{theorem}
    We also use the notation $\slide_{Y}\coloneqq\slide_{Y_1,Y_2}$, when $m$ and
    $n$ are clear from the context. 

\begin{figure}[ht]
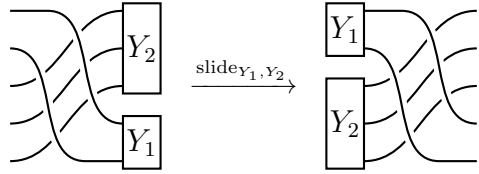

    \[
    \slidefig
    \]
    \caption{Graphical illustration of 
$\slide_{Y_1,Y_2}$---here in case $(m,n)=(2,3)$.}
\label{fig:slide}
\end{figure}

\begin{proof} [Proof of \cref{thm:prebraidingSbimcl}]
    \cref{prop:sliding-objects} indeed implies \cref{thm:prebraidingSbimcl}. Namely it follows from Theorem~\ref{thm:Rouquier-canonicity} that the $\beta_{m,n}$
    are invertible and satisfy the hexagon axioms \eqref{eq:fakehexagon} and thus form the components of a natural
    transformation $\boxtimes\circ (\hincl\times \hincl) \Rightarrow
    \boxtimes^{\mathrm{op}}\circ (\hincl\times \hincl)$ by \Cref{prop:sliding-objects}.
\end{proof}

       To establish \cref{prop:sliding-objects} we construct, after some preparation, the chain maps, which we then call \emph{slide maps}, explicitly. For this we work with  the \emph{Hecke category} $\mathcal{DS}_n$, i.e.~the diagrammatical presentation of the monoidal $1$-category $\BSbimcl_n$ from \cite{EW}, \cite{EK}. The construction of the chain maps $\slide_{Y_1,Y_2}$ 
       proceeds in two steps. The first step is specific to the setting of Bott--Samelson and
Soergel bimodules and uses  $\mathcal{DS}$. It establishes the existence of
\emph{atomic slide chain maps}, namely $\slide_{\one_1,B_1}$ for $(m,n)=(1,2)$
and $\slide_{B_1,\one_1}$ for $(m,n)=(2,1)$; see Lemma~\ref{lem:atomicslide}.
The second step uses that every Bott--Samelson bimodule is a  composition
(monoidal and horizontal) of Bott--Samelson bimodules on two strands, and so
knowing the atomic slide chain maps is sufficient to construct general slide
maps along the following scheme:
  \[
    \begin{tikzpicture}[anchorbase,xscale=-.5,yscale=.5]
        \draw[thick] (1,2) \pr (4,0);
        \draw[thick] (1,3) \pr (4,1);
        \draw[thick] (1,4) \pr (4,2);
        \draw[wh] (1,1) \pr (3,4) to  (4,4);
        \draw[thick] (1,1) \pr (3,4) to (4,4);
        \draw[wh] (1,0) to (2,0) \pr (4,3);
        \draw[thick] (1,0) to (2,0) \pr (4,3);
        \draw[thick] (0,-.2) rectangle (1,1.2);
        \draw[thick] (0,1.8) rectangle (1,4.2);
        \node at (.5,.5) {$Y_1$};
        \node at (.5,3) {$Y_2$};
    \end{tikzpicture}
    \;\; \leftarrow \;\;
    \begin{tikzpicture}[anchorbase,xscale=-.5,yscale=.5]
        \draw[thick] (1,2) \pr (4,0);
        \draw[thick] (1,3) \pr (4,1);
        \draw[thick] (1,4) \pr (4,2);
        \draw[wh] (1,1) \pr (3,4) to  (4,4);
        \draw[thick] (0,1) to (1,1) \pr (3,4) to (4,4);
        \draw[wh] (2,0) \pr (4,3);
        \draw[thick, gray] (0,0) to (1,0) to (2,0) \pr (4,3);
        \draw[thick] (1,1.8) rectangle (0,4.2);
        \node at (.5,3) {$B_i$};
    \end{tikzpicture}
    \;\; \leftarrow \;\;
    \begin{tikzpicture}[anchorbase,xscale=-.5,yscale=.5]
        \draw[thick] (1,2) \pr (4,0);
        \draw[thick] (1,3) \pr (4,1);
        \draw[thick, gray] (0,4) to (1,4) \pr (4,2);
        \draw[wh] (1,1) \pr (3,4) to  (4,4);
        \draw[thick] (0,1) to (1,1) \pr (3,4) to (4,4);
        \draw[wh] (2,0) \pr (4,3);
        \draw[thick, gray] (0,0) to (1,0) to (2,0) \pr (4,3);
        \draw[thick] (1,1.8) rectangle (0,3.2);
        \node at (.5,2.5) {$B_i$};
    \end{tikzpicture}
\]
Essentially the same argument would work for any monoidal bicategory generated
by a single object and one endomorphism of its tensor square.

To formulate the statements, we need at least a rough description of the Hecke category $\mathcal{DS}_n$ and the fact that it is equivalent to $\overline{\BSbimcl}^{\mathrm{gr}}_n$ as graded monoidal $\k$-linear category. For details we refer to \cite{EW}, \cite{EK}. Each simple reflection $s_i\in S_n$ is encoded by a colour. Objects in $\mathcal{DS}_n$ are finite ordered sequences of such colours and they encode the Bott--Samelson bimodules in the form \eqref{eq:BSB} (with $j=0$ since we consider $\overline{\BSbimcl}^{\mathrm{gr}}_n$). For instance if $n=2$ and we encode $s_1$ as red and $s_2$ as blue, then the Bott--Samelson bimodule $B_{\mathbf{i}}$ from \eqref{eq:BSB} is encoded as a sequence of colors red and blue according to ${\mathbf{i}}$, for instance $
{\mathbf{i}}=(1,1,2,1,1)$ corresponds to the object given by the color sequence $(red, red, blue, red, red)$.
Morphisms in $\mathcal{DS}_n$ are $k$-linear combinations of isotopy classes of certain decorated
graphs embedded in the plane.  A morphism from  $B_{\mathbf{i}}$ to
$B_{\mathbf{i}'}$ will have the colour sequence for ${\mathbf{i}}$ as bottom
boundary and that for ${\mathbf{i}}'$ at the top boundary; for instance the
first two diagrams in \eqref{eq:gens} represent morphism from $(2,2)$ to $(2)$ and
vice versa, the third goes from the unit to $(2)$, etc. 

The monoidal structure is given on objects by concatenating sequences and, on morphisms, by (the bilinear extension of) placing diagrams horizontally next to each
other. The composition of morphisms is, likewise, given by (the bilinear extension of) stacking diagrams on top of each other. The empty
sequence is the unit object. 

Apart from multiplication with polynomials, the generating morphisms (in the monoidal sense) are exactly the following, where blue represents any color/number which is neighbored to red and not neighbored to orange.
\begin{equation}\label{eq:gens}
\genlist
\end{equation}

For the rest of this section we identify  $\BSbimcl_n$ with $\mathcal{DS}_n$ as
monoidal $1$-categories (via $\overline{\BSbimcl}^{\mathrm{gr}}_n$) and perform
computations using the diagrammatic calculus.

The Rouquier complexes \eqref{eqn:Rouq-alg} are translated into the diagrammatics as 
\begin{eqnarray}\label{eq:diagcxs}
\diagcxs
\end{eqnarray}
where we encode $s_i$ by blue. (Instead of remembering the grading shifts from
\eqref{eqn:Rouq-alg}, it is more convenient in the diagrammatic setting to
consider the maps as homogeneous of degree one). To keep track of the monoidal
unit (corresponding to $R$) appearing in the complexes, we mostly indicate them
by a colored dot as shown in \eqref{eq:diagcxs}. 

        \begin{lemma}\label{lem:atomicslide} 
      There are slide chain maps 
\[\slide_{\one_1,B_1}:= \hspace{-.5cm}\FigSlide,\quad \slide_{B_1,\one_1}:=\hspace{-.5cm} \FigSlidee \] 
which are invertible up to homotopy. The inverses are given by the chain maps
\[\slide^{-1}_{\one_1,B_1}:= \hspace{-.5cm}\FigSlidei, \quad
\slide^{-1}_{B_1,\one_1}:= \hspace{-.5cm}\FigSlideei \] 
\end{lemma}
\begin{proof}
The proof is given by an explicit calculation. As an example (the
remaining cases are checked analogously) we show that
$\slide^{-1}_{\one_1,B_1}\circ\; \slide_{\one_1,B_1}$ is homotopic to the
identity by computing their difference and exhibiting an explicit null-homotopy:
\[\slide^{-1}_{\one_1,B_1}\circ\; \slide_{\one_1,B_1} - \Id = \left(\FigSlidec\right) = \left[d, \FigSlideh \right]\] 
  \end{proof}      
  Observe that the relevant chain complexes and chain maps are for the two cases
  are related by swapping the colours red and blue. 
  \begin{remark} Readers familiar with the chain maps between Rouquier complexes
      associated to a Reidemeister III move will recognize the atomic slide chain maps
      as filtrations-preserving pieces of the former, see e.g. \cite[(3.3) and (3.4)]{MWW}.
  \end{remark}

Now that we have obtained the atomic slide chain maps in \cref{lem:atomicslide}, we can construct all remaining slide chain maps in an essentially formal way.
        
\begin{proof}[Proof of \Cref{prop:sliding-objects}]
     We will focus on the version for the positive cabled crossing,
     since the other one is analogous. First, we reduce to the case when the
     object $Y=Y_1\boxtimes Y_2$ is a generating object of $\BSbimcl_m\boxtimes
     \BSbimcl_n$. Otherwise, we can decompose into generators:
    \[ Y_1\boxtimes Y_2= (Y_1\boxtimes \one) \hcomp (\one \boxtimes Y_2) =
    (B_{i_1}\boxtimes \one) \hcomp \cdots \hcomp (B_{i_a}\boxtimes \one) \hcomp (\one
    \boxtimes B_{j_1})\hcomp \cdots \hcomp (\one
    \boxtimes B_{j_b}) \] 
    and define 
    \begin{align*}
        \slide_{Y_1,\one}&:= 
        (\id_{\swap_{m,n}(B_{i_1}\hcomp \cdots \hcomp B_{i_{a-1}} \boxtimes \one)}\hcomp \slide_{B_{i_{a}},\one}) 
        \vcomp \cdots \vcomp
        (\slide_{B_{i_{1}},\one} \hcomp \id_{B_{i_1}\hcomp \cdots \hcomp B_{i_{a-1}} \boxtimes \one})
        \\
        \slide_{\one,Y_2}&:= 
        (\id_{\swap_{m,n}(\one\boxtimes B_{j_1}\hcomp \cdots \hcomp B_{j_{b-1}})}\hcomp \slide_{\one, B_{j_{b}}}) 
        \vcomp \cdots \vcomp
        (\slide_{\one, B_{j_{1}}} \hcomp \id_{\one\boxtimes B_{j_1}\hcomp \cdots \hcomp B_{j_{b-1}}})
        \\
        \slide_{Y_1,Y_2} &:= (\id_{\swap_{m,n}(Y_1\boxtimes \one)} \hcomp \slide_{\one,Y_2})  \vcomp (\slide_{Y_1,\one} \hcomp \id_{Y_2}) 
    \end{align*}

    Now we turn to defining $\slide_{B,\one_n}$ and $\slide_{\one_m,B}$, where
    $B$ is one of the generating Bott--Samelson bimodules. Here we place subscripts to distinguish the
    identity bimodules. 
    We first consider the latter situation and reduce it to the case
    $m=1$, where the cabled crossing is a Coxeter braid. Indeed, suppose that
    $m>1$, then we use the first equality from
    \Cref{lem:cabled-crossing-from-coxeter} to define $\slide_{\one_m,B}$ to be
    the composite:
    \begin{gather}
        \nonumber
        \big((\slide_{\one_1,B}\boxtimes \id_{\one_{m-1}})
        \hcomp \cdots \hcomp 
        \id_{\one_{m-1-i} \boxtimes \cabledcross_{1,n}\boxtimes\one_{i}}
        \hcomp \cdots \hcomp
        \id_{\one_{m-1}\boxtimes \cabledcross_{1,n}}\big)
        \vcomp \cdots
        \\
        \label{eqn:slidecabledfromcoxeter}
        \vcomp \big(\id_{\cabledcross_{1,n}\boxtimes \one_{m-1}}
        \hcomp \cdots \hcomp 
        (\id_{\one_{m-1-i}} \boxtimes \slide_{\one_1,B}\boxtimes\id_{\one_{i}})
        \hcomp \cdots \hcomp
        \id_{\one_{m-1}\boxtimes \cabledcross_{1,n}}\big)
        \vcomp \cdots\\ \nonumber
        \vcomp \big(\id_{\cabledcross_{1,n}\boxtimes \one_{m-1}}
        \hcomp \cdots \hcomp 
        \id_{\one_{m-1-i} \boxtimes \cabledcross_{1,n}\boxtimes\one_{i}}
        \hcomp \cdots \hcomp
        (\id_{\one_{m-1}}\boxtimes \slide_{\one_1,B}) \big)
    \end{gather}
For the other case, we first choose chain maps $\phi$ and $\phi^{-1}$ realising
 the first homotopy equivalence in \Cref{lem:cabled-crossing-from-coxeter}.
 Then we define $\slide_{B,\one_n}$ as the composition:
 \begin{gather*}
    \nonumber
    \phi^{-1} \vcomp \big((\one_{n-1}\boxtimes \slide_{B,\one_1})
    \hcomp \cdots \hcomp 
    \id_{\one_{i} \boxtimes \cabledcross_{m,1}\boxtimes\one_{n-1-i}}
    \hcomp \cdots \hcomp
    \id_{\cabledcross_{m,1}\boxtimes \one_{n-1}}\big)
    \vcomp \cdots
    \\
    \vcomp \big(\id_{\one_{n-1}\boxtimes \cabledcross_{m,1}}
    \hcomp \cdots \hcomp 
    (\id_{\one_{i}} \boxtimes \slide_{B,\one_1}\boxtimes\id_{\one_{n-1-i}})
    \hcomp \cdots \hcomp
    \id_{\cabledcross_{m,1}\boxtimes \one_{n-1}}\big)
    \vcomp \cdots
    \\ \nonumber
    \vcomp \big(\id_{\one_{n-1}\boxtimes \cabledcross_{m,1}}
    \hcomp \cdots \hcomp 
    \id_{\one_{i} \boxtimes \cabledcross_{m,1}\boxtimes\one_{n-1-i}}
    \hcomp \cdots \hcomp
    (\slide_{B,\one_1}\boxtimes \one_{n-1}) \big)\vcomp \phi 
 \end{gather*}
It remains to construct $\slide_{\one_1,B_i}$ and $\slide_{B_j,\one_1}$ where
$B_i$ is a generating object of $\BSbimcl_n$ and $B_j$ is a generating object of
$\BSbimcl_m$. Now we reduce this problem to the cases when $n=2$ and $m=2$
respectively. We define $\slide_{\one_1,B_i}$ as the composite:
\begin{align*}
    \Rouq(\Ag_{n}\cdots\Ag_{1}) \hcomp B_i =& \Rouq(\Ag_{n}\cdots\Ag_{i+1})\hcomp \Rouq(\Ag_{i}\Ag_{i-1})\hcomp \Rouq(\Ag_{i-2}\cdots\Ag_{1})\hcomp B_i \\
    \to 
    &\Rouq(\Ag_{n}\cdots\Ag_{i+1})\hcomp \Rouq(\Ag_{i}\Ag_{i-1}) \hcomp B_i \hcomp\Rouq(\Ag_{i-2}\cdots\Ag_{1}) 
   \\ \xrightarrow{\slide}
    &\Rouq(\Ag_{n}\cdots\Ag_{i+1})\hcomp B_{i-1} \hcomp \Rouq(\Ag_{i}\Ag_{i-1}) \hcomp\Rouq(\Ag_{i-2}\cdots\Ag_{1}) 
\\ \to&
 B_{i-1} \hcomp \Rouq(\Ag_{n}\cdots\Ag_{i+1})\hcomp \Rouq(\Ag_{i}\Ag_{i-1}) \hcomp\Rouq(\Ag_{i-2}\cdots\Ag_{1})\\ 
    =& B_{i-1} \hcomp \Rouq(\Ag_{n}\cdots\Ag_{1})
\end{align*}
where the unlabelled maps are far-commutativity isomorphisms and the labelled
arrow is given by $\id \hcomp \slide_{\one_1,B_1} \hcomp \id $, which is
determined by the $n=2$ case. The reduction of $\slide_{B_j,\one_1}$ to the case
$m=2$ is completely analogous. 

Thus we reduced the problem to the statement from  \Cref{lem:atomicslide}. By
construction, all slide maps constructed in this proof are homotopy
equivalences.
\end{proof}

\begin{rem}
\label{rem:twoproofs} In Remarks~\ref{rem:noRthree} and \ref{rem:onecatprebraid}
we have observed that a prebraiding need not satisfy the braid
relation, and that the braid relation in prebraidings on identity functors can
be verified in two ways using naturality. On the level of the homotopy category
of Soergel bimodules, this is reflected in the fact, that the chain maps
implementing the homotopy equivalence corresponding to a braid relation live in
a 2-dimensional space \cite[Section 3, Reidemeister 3 generators]{MR2721032}.
One dimension is encoding an overall non-zero scaling. Even after fixing this,
there exists however a 1-dimensional affine subspace of representatives of the
same homotopy class of chain maps. Two distinct (and thus spanning) points in
this subspace can be built from slide chain maps, analogously to the two ways of
establishing the braid relation in \Cref{rem:onecatprebraid}. 
\end{rem}

\begin{remark}\label{rk:prebraidDavidswish} We used Rouquier canonicity to
    ensure here the independence (up to canonical isomorphism) of our choices in
    the construction of the prebraiding on $\hincl$ in
    \Cref{thm:prebraidingSbimcl}. For the rest of the paper it would be enough
    to make one of the choices and establish the prepraiding using only the
    monoidality of $\hincl$. The desired independence of choices (and in fact also Rouqiuer canonicity itself) could then be
    deduced from the existence statement in \cref{cor:main-corollary}.
    \end{remark}

\subsection{Centralizers and prebraidings}
In the following, given a monoidal $1$-category $\cC$, we will identify  $1\boxtimes x$ and $x\boxtimes 1$ with $x$ for any object $x\in\cC$, see \cite[Rk. 2.2.9]{EGNO}, and will suppress writing associators, as they can be recovered from context.

\begin{definition}\label{def:centralizer}
Let $\cA$ and $\cB$ be monoidal $1$-categories and let
$F\colon\cA\rightarrow\cB$ be a monoidal functor. The \emph{centralizer} $Z(F)$
of $F$ is the following category: 
\begin{itemize}
\item 
Its objects are pairs $(b,\gamma)$ of an object $b\in\cB$ and a natural
isomorphism \[\gamma=\gamma_-\colon b\boxtimes F(-)\rightarrow F(-)\boxtimes
b\] of functors from $\cA$ to $\cB$, called \emph{half-braiding}, which satisfies
the following two compatibility conditions with respect to the monoidal structure of $\cA$.
Firstly, for any $a_1,a_2\in\cA$, the isomorphism $\gamma_{a_1 \boxtimes a_2} \colon b \boxtimes F(a_1 \boxtimes a_2) \to F(a_1 \boxtimes a_2) \boxtimes b$ equals the composite
\[b\boxtimes F(a_1\boxtimes a_2)\xrightarrow{\simeq}
b\boxtimes  F(a_1)\boxtimes F(a_2)\xrightarrow{(\id\boxtimes \gamma_{a_2})\circ
(\gamma_{a_1}\boxtimes\id )} F(a_1)\boxtimes F(a_2)\boxtimes
b\xrightarrow{\simeq} F(a_1\boxtimes a_2) \boxtimes b,\]
and secondly \[\gamma_{1_{\cA}} = \left(b\boxtimes
F(1_{\cA})\xrightarrow{\simeq} b\boxtimes 1_{\cB}=1_{\cB}\boxtimes b\xrightarrow{\simeq}F(1_{\cA})\boxtimes b\right).\] 
\item Its morphisms are morphisms in $\cB$ that are compatible with the
half-braidings as follows: 
$$\Hom_{Z(F)}((b,\gamma),(b',\gamma'))=\{f\in\Hom_{\cB}(b,b')\mid \gamma'_a\circ (f\boxtimes\id)=(\id\boxtimes f)\circ\gamma_a
\colon b\boxtimes F(a)\rightarrow F(a)\boxtimes b\}.$$ 
\item The composition is inherited from $\cB$.
\end{itemize}
\end{definition}
The category $Z(F)$ is monoidal with $(b,\gamma)\boxtimes (b,\gamma')\coloneqq (b\boxtimes b',\gamma\boxtimes\id\circ\id\boxtimes\gamma')$ on objects, and with the tensor product from $\cB$ on morphisms.  The unit object $1_{Z(F)}\in Z(F)$  is $(1_{\cB}\in\cB, \gamma)$ with $\gamma_a\colon 1\boxtimes F(a)=F(a)=F(a)\boxtimes 1$. 
We leave the coherence isomorphisms and their compatibility to the reader. 

\begin{remark}
    \label{rk:Drinfeldcenter}
    In case $\cA=\cB$ and $F= \id_{\cA}$ is the identity functor, the centralizer $Z(\id_{\cA})$ is the \emph{Drinfeld center} $Z(A)$ of $\cA$, \cite[Def.~7.13.1]{EGNO}. The generalized hexagon axiom above turns then into the familiar hexagon diagram \cite[(7.41)]{EGNO}. Centralizers, as generalizations of Drinfeld centers, appear already in \cite[\S3]{MR1151906}. 
\end{remark}

\begin{definition} \label{def:evaluation}
For a monoidal functor $F\colon \cA \to \cB$, we define the monoidal \emph{evaluation functor} 
\[\ev\colon Z(F) \times \cA \to \cB
\]
to send an object $((b,\gamma),a)$ to 
$b\boxtimes F(a)$, and a morphism $(f,g)$ to $f\boxtimes F(g)$. The monoidal structure isomorphisms $$\ev\left(((b,\gamma),a)\boxtimes((b',\gamma'),a')\right)\simeq 
\ev\left((b,\gamma),a\right)\boxtimes \ev\left((b',\gamma'),a')\right) \quad(\text{for } a,a'\in\cA,(b,\gamma),(b',\gamma')\in Z(F))$$ are given by 
$b\boxtimes b'\boxtimes F(a\boxtimes a')\simeq b\boxtimes b'\boxtimes F(a)\boxtimes F(a')
\xrightarrow{\id\boxtimes\gamma'_a\boxtimes\id}
b\boxtimes  F(a)\boxtimes b'\boxtimes F(a')$. 
The defining properties of  $\gamma'$ and the monoidality of $F$ ensure that the
necessary compatibilities hold, so that we indeed get a monoidal functor. 
\end{definition}

Together with the unit $1_{Z(F)} \in Z(F)$ and $F \colon \cA \to \cB$, the evaluation functor fits into the following commuting diagram of monoidal functors:
\[
        \begin{tikzcd}
            & Z(F) \times \cA  \ar[dr, "\ev"] & \\ 
            \{* \} \times \cA \ar[ru, "1_{Z(F)} \times \id_\cA"] \ar[r, "="] & \cA \ar[r, "F"]   &\cB
        \end{tikzcd}
   \]
   
   In \cref{exm:universal-property-centralizer}, we will discuss the universal property satisfied by $Z(F)$ with its monoidal evaluation functor. 
   
The functor $\ev$ is completely determined by the monoidal functor 
 \[\ev_{1_{\cA}} \coloneqq \ev(-, 1_{\cA}) \colon Z(F) \to \cB,\]
which sends an object $(b, \gamma)\in Z(F)$ to the underlying object $b$ and a morphism in $Z(F)$ to the underlying morphism in $\cB$. In these terms, $\ev(-, ?) = \ev_{1_{\cA}}(-) \boxtimes F(?)$.

We obtain a classification of prebraidings on a monoidal functor $F$, Definitions~\ref{def:prebraiding} and \ref{def:notationprebraiding}, in terms of the centralizer $Z(F)$ of $F$:
\begin{theorem}\label{thm:classbraidings}
Let $F\colon  \cA\rightarrow\cB$ be a monoidal functor between monoidal $1$-categories. Then the following are equivalent:
\begin{enumerate}
\item\label{itm:prebraid} the set $\PreBraid(F)$ of prebraidings on $F$; 
\item\label{itm:strictlifts} the set of \emph{strict monoidal factorizations of $F$ through $\ev_{1_{\cA}} \colon Z(F) \to \cB$}, i.e. the set of monoidal functors $s\colon \cA\rightarrow Z(F)$ such that $\ev_{1_{\cA}} \circ s=F$. 
\item \label{itm:weaklifts}the 1-groupoid of \emph{weak monoidal factorizations of $F$ through $\ev_{1_{\cA}} \colon Z(F) \to \cB$}, i.e. the groupoid whose objects are pairs $(s,\eta)$ of a monoidal functor $s \colon \cA \to Z(F)$ and a monoidal natural isomorphism $\eta \colon \ev_{1_{\cA}} \circ s \To F$ and whose morphisms  $(s, \eta) \to (s', \eta')$ are monoidal natural isomorphisms $\mu \colon s\To s'$ such that \[\left( \ev_{1_{\cA}} \circ s \xRightarrow{\ev_{1_{\cA}} \circ \mu} \ev_{1_{\cA}} \circ s' \xRightarrow{\eta'} F \right)=\left( \ev_{1_{\cA}} \circ s \xRightarrow{\eta}F\right).\]
\end{enumerate}
\end{theorem}

\begin{proof}
For the equivalence between \eqref{itm:prebraid} and \eqref{itm:strictlifts}, note that a factorization $s$ must send, on the level of objects, $x$ to $(F(x),\gamma)$ for some $\gamma$, and on the level of morphisms $f$ to $F(f)$. The isomorphisms used for a prebraiding $\beta$ uniquely define the isomorphisms encoded in a possible $\gamma$. The second  hexagon axiom from prebraidings \eqref{eq:fakehexagon} translates into the required properties of $\gamma$, whereas the first hexagon translates into the monoidality of $s$. 

The equivalence between \eqref{itm:strictlifts} and \eqref{itm:weaklifts} follows from abstract-nonsense: Recall that a monoidal functor $F:\cX \to \cZ$ between monoidal $1$-categories is called an \emph{isofibration} if for all isomorphisms $\gamma \colon z \to z'$ in $\cZ$ and $x \in \cX$ with $F(x) = z$, there exists an isomorphism $\mu \colon x\to x'$ with $F(\mu) = \gamma$. It is then an exercise to show that if $F \colon \cX \to \cZ$ is a monoidal functor which is a faithful isofibration and $G\colon \cY \to \cZ$ is another monoidal functor, the groupoid of weak monoidal factorizations, i.e. of pairs $(s, \eta)$ of a monoidal functor $s \colon \cY \to \cX$ and a monoidal natural isomorphism $F \circ s \simeq G$ is equivalent to a discrete groupoid isomorphic to the set of strict monoidal factorizations, i.e. the set of monoidal functors $s$ such that $F \circ s = G$. The equivalence between \eqref{itm:strictlifts} and \eqref{itm:weaklifts} then follows since  $\ev_{1_{\cA}} \colon Z(f) \to \cB$ is indeed a faithful isofibration.
\end{proof}
\begin{corollary}
The special case $F=\id_\cA$ for a monoidal category $\cA$ gives a bijection
\begin{eqnarray*} 
\PreBraidid(\cA)
&\simeq&
\{\text{Monoidal sections }\cA \rightarrow Z(\cA) \text{ of }\ev_{1_{\cA}}\colon Z(\cA)\to\cA\}.
\end{eqnarray*}
\end{corollary}

\section{Stable linear algebra}
\label{sec:stableLA}
This section introduces the $\infty$-categorical foundations of our work.
Throughout, we adopt standard conventions and notations in higher category theory, which we recall in detail in \cref{sec:appendix-recollections}. 
 We follow~\cite{HTT} and use Grothendieck universes --- here called \emph{small}, \emph{large}, and \emph{huge} ---  to deal with set-theoretic issues (see~\cref{subsec:set-theory-Grothendieck-universes} for more details).  In particular, we write $\cat$ for the large $\infty$-category of small $\infty$-categories and $\Spaces$  for the (large) $\infty$-category of small spaces, both of which are objects of the huge $\infty$-category $\widehat{\Cat}_{\infty}$ of large $\infty$-categories.

\subsection{Completions of \texorpdfstring{$\infty$}{infinity}-categories} \label{sec:presentable}
\subsubsection{The Yoneda embedding}

Any small $\infty$-category $\cC$ has a Yoneda embedding into its $\infty$-category of ($\Spaces$-valued) presheaves $\Pres(\cC)=\Fun(\cC^\op, \Spaces)$, \cite[Prop.~5.1.3.1]{HTT}, see also \cite{Cisinski}. It is characterized by the universal property that $\Pres(\cC)$ has all small colimits (i.e. is \emph{cocomplete})~\cite[Cor. 5.1.2.4]{HTT},  and that for any cocomplete $\infty$-category $\cD$ the restriction along the Yoneda embedding induces an equivalence
\begin{equation}
    \label{eq:Yoneda}
    \FunL(\Pres(\cC),\cD) \to \Fun(\cC,\cD)
\end{equation}
where $\FunL$ denotes the full subcategory of the $\infty$-category of functors on those functors which preserve all small colimits (i.e. the \emph{cocontinuous functors})~\cite[Thm.~5.1.5.6]{HTT}, see also  \cite[Thm.~6.3.13]{Cisinski}.
An $\infty$-category $\cC$ is called \emph{idempotent complete} if its image under the Yoneda embedding $\cC \to \Pres(\cC)$ is closed under retracts (see~\cite[Proof of Prop.~5.1.4.2]{HTT}). We refer to \cite[\S~4.4.5]{HTT} for a discussion of retracts and idempotents in $\infty$-categories.

\begin{notation} \label{not:smallcats}
We write
\begin{itemize}
    \item $\catidem$ for the full subcategory of $\cat$ on the idempotent complete small $\infty$-categories,
    \item $\catprod$ for the subcategory of $\catidem$ on the idempotent complete small $\infty$-categories that admit finite coproducts and functors which preserve finite coproducts, and
    \item $\catrex$ for the subcategory of $\catprod$ on the idempotent complete small $\infty$-categories that admit finite colimits and functors which preserve finite colimits.
\end{itemize}
\end{notation}

\subsubsection{Presentable \texorpdfstring{$\infty$}{infinity}-categories}
In general, the presheaf category $\Pres(\cC)$ of a small $\infty$-category $\cC$ is a large category. The sense in which $\Pres(\cC)$ is nevertheless still controlled by a small amount of data is formalized by the notion of a presentable $\infty$-category. We recall the definition and some basic facts from {\cite[Sec. 5.4 and 5.5]{HTT}}.

\begin{definition}
Let $\cD$ be a (possibly large) $\infty$-category.
\begin{enumerate}
\item Let $\cK$ be a collection of $\infty$-categories and $S$ a small set of objects of $\cD$. Then $\cD$ is \emph{generated by $S$ under $\cK$-indexed colimits} if $\cD$ has all colimits indexed by categories in $\cK$ and is the smallest full subcategory of $\cD$ which contains the objects in $S$ and is closed under $\cK$-indexed colimits.
\item Let $\kappa$ be an infinite regular cardinal and assume $\cD$ admits $\kappa$-filtered colimits. Then an object $d\in \cD$ is called \emph{$\kappa$-compact} if the functor $\Hom_{\cD}(d,-)\colon \cD \to \Spaces$ preserves $\kappa$-filtered colimits. 
\item The $\infty$-category $\cD$ is called \emph{accessible} if it is locally small and there exists a regular cardinal $\kappa$ and a small set $S$ of $\kappa$-compact objects in $\cC$ that generates $\cC$ under $\kappa$-filtered colimits.
\item The $\infty$-category $\cD$ is called \emph{presentable} if it has all small colimits and is accessible. 
\end{enumerate}
\end{definition}

\begin{example}
By Simpson's characterisation of presentable $\infty$-categories as localizations of presheaf categories, {\cite[Thm. 5.5.1.1]{HTT}}, we obtain presentability of $\Pres(\cC)$ for any small $\infty$-category $\cC$~\cite[Ex. 5.4.2.7,  Ex. 5.5.1.8.]{HTT}, and more generally the presentability of $\Fun(\cC, \cD)$ for a small $\infty$-category $\cC$ and a presentable $\infty$-category $\cD$.
\end{example}

A main application of the notion of presentable $\infty$-category is the \emph{adjoint functor theorem}:

\begin{proposition}[{\cite[Cor.~5.5.2.9 and Rem.~5.5.2.10]{HTT}}]\label{adjfuncthm} A functor from a presentable $\infty$-category to a locally small $\infty$-category preserves small colimits if and only if it is a left adjoint. 
\end{proposition}

\begin{nota} 
We denote by
\begin{itemize} 
\item $\PrL$ the
$\infty$-category of presentable $\infty$-categories and small colimit preserving functors, i.e left adjoint functors by the adjoint functor theorem.
\item $\FunL(\cC,\cD)$, for $\cC, \cD\in \PrL$, the full subcategory of $\Fun(\cC, \cD)$ of left adjoint (equivalently cocontinuous) functors.
Dually, full subcategories of right adjoint functors will be denoted $\FunR(-,-)$.
\item When denoting an adjunction 
\begin{tikzcd}[column sep=1.5cm]
\cC
\arrow[yshift=0.9ex]{r}{L}
\arrow[leftarrow, yshift=-0.9ex]{r}[yshift=-0.2ex]{\bot}[swap]{R}
&
\cD
\end{tikzcd}
between $\infty$-categories, we use the convention that the top arrow is the left adjoint and the bottom arrow the right adjoint.
\end{itemize}
\end{nota}
For more details, we refer to \cite[\S~5.5.3]{HTT} and \cite[\S~7]{Cisinski}.

\subsubsection{The monoidal structure on \texorpdfstring{$\PrL$}{PrL} and presentably symmetric monoidal categories}
\begin{proposition}[{\cite[Prop.~4.8.1.15, Prop.~4.8.1.10]{HA}}]
\label{prop:prlsym}
    The $\infty$-category $\PrL$ can be equipped with a symmetric monoidal structure for which the Yoneda embedding defines a symmetric monoidal functor $$\Pres\colon \cat \to \PrL$$ where $\cat$ is equipped with its Cartesian symmetric monoidal structure.
\end{proposition}

The tensor unit of this symmetric monoidal structure is given by the presentable $\infty$-category $\Spaces$ of spaces. 
The tensor product $\cC_1\otimes \cC_2$ of two presentable $\infty$-categories $\cC_1, \cC_2$ comes equipped with a 
functor  $\cC_1 \times \cC_2 \to \cC_1\otimes \cC_2$ which preserves small colimits separately in both variables and is characterized by the universal property that for any presentable $\infty$-category $\cD$, the induced functor $$\FunL(\cC_1\otimes \cC_2, \cD) \to \mathrm{Fun^{L \times L}}(\cC_1 \times \cC_2, \cD)$$
is an equivalence. Here, $\mathrm{\Fun^{L \times L}}(\cC_1 \times \cC_2,\cD)$ denotes the full subcategory of $\Fun(\cC_1\times \cC_2, \cD)$ on those functors which preserve small colimits separately in both variables.
Abusing notation, given objects $c_1\in \cC_1$ and $c_2\in \cC_2$ we denote the image of $(c_1, c_2) \in \cC_1 \times \cC_2$ under the functor $\cC_1\times \cC_2 \to \cC_1\otimes \cC_2$ by $c_1\boxtimes c_2\in \cC_1 \otimes \cC_2$ and call it their \emph{external tensor product}.

It follows from~\cite[Prop.~4.8.1.17, Prop.~4.8.1.16]{HA} after taking adjoints, that the tensor product of presentable $\infty$-categories can be expressed as the following functor category (which is in particular presentable): 
\begin{equation}
\label{eq:prtensor}
\cC \otimes \cD \simeq \FunL(\cD, \cC^\op)^\op = \FunR(\cC^\op, \cD) 
\end{equation}

The symmetric monoidal structure on $\PrL$ allow us to study (commutative) algebras therein, see \cref{appendix:CAlg}. 
\begin{definition} A \emph{presentably symmetric monoidal $\infty$-category} is a commutative algebra object in $\PrL$.  
\end{definition}

More explicitly, a presentably symmetric monoidal $\infty$-category is a symmetric monoidal $\infty$-category whose underlying $\infty$-category $\cC$ is presentable and so that the tensor product functor $-\otimes-\colon \cC \times \cC \to \cC$ preserves small colimits separately in both variables. 

If $A$ is a commutative algebra object in a symmetric monoidal $\infty$-category $\cC$, consider the $\infty$-category $\Mod_A(\cC)$ of left $A$-modules in $\cC$ (see \cref{subsec:module-cats}). 
\begin{proposition}\label{prop:algprl-new}
    Given $\cC, \cD \in \CAlg(\PrL)$, together with $A, B \in \CAlg(\cC)$.
    \begin{enumerate}
        \item 
        \label{item-prop:algprl1}
        The relative tensor product $-\otimes_A - \colon \Mod_A(\cC) \times \Mod_A(\cC) \to \Mod_A(\cC)$ defines a presentably symmetric monoidal structure on $\Mod_A(\cC)$. Moreover, there is an equivalence $\CAlg(\Mod_A(\cC)) \simeq  \CAlg(\cC)_{A/}$.
        \item 
        \label{item-prop:algprl2}
        Any algebra homomorphism $f\colon A\to B$ in $\CAlg(\cC)$ induces a symmetric monoidal induction functor $-\otimes_A B\colon \Mod_A(\cC) \to \Mod_B(\cC)$ that is left adjoint to the restriction functor along $f$, and hence a morphism in $\CAlg(\PrL)$. 
        \item 
        \label{item-prop:algprl3}
        For any algebra homomorphism $A\to B$, it follows from (2) that we can view $B$  as an object in $\CAlg(\Mod_A(\cC))$. Forgetting the $A$-action induces a symmetric monoidal equivalence:
        \[
         \Mod_B(\Mod_A(\cC)) \xrightarrow{\simeq} \Mod_B(\cC)
         \]
        \item    
        \label{item-prop:algprl4}

        Any functor  $F \colon \cC \to \cD$ in $\CAlg(\PrL)$ induces a functor $\CAlg(\cC) \to \CAlg(\cD)$ on commutative algebra objects, which we will also simply denote by $F$. Moreover, it induces a functor 
\[
            \Mod_A(F) \colon \Mod_A(\cC) \to \Mod_{F(A)}(\cD)
\]
        in $\CAlg(\PrL)$.
        \item  
        \label{item-prop:algprl5}
        Any functor $F \colon \cC \to \cD$ in $\CAlg(\PrL)$ induces an equivalence in $\CAlg(\PrL)$
\[
            \cD \otimes_{\cC}(\Mod_A(\cC)) \simeq \Mod_{F(A)}(\cD),
\]
where $-\otimes_{\cC}-$ denotes the pushout in $\CAlg(\PrL)$, whose underlying presentable $\infty$-category is given by the relative tensor product in $\PrL$ \cite[Prop.~3.2.4.10]{HA}, hence the notation. 
    \end{enumerate}
\end{proposition}

\begin{proof}
    The first two statements follow from~\cite[Prop.~3.4.1.3, Cor. 4.2.3.7, Thm. 4.5.3.1]{HA}, the third statement follows from~\cite[Cor. 3.4.1.9]{HA}. The existence of the symmetric monoidal functor $\Mod_A(F)$ in part \eqref{item-prop:algprl4} follows from the functoriality of the $\Mod$ construction in ~\cite[\S~3.3.3]{HA}. Furthermore, $\Mod_A(F)$ preserves colimits by \cite[Cor. 4.2.3.5]{HA}. 
Functoriality of the construction of modules induces a commuting square in $\CAlg(\PrL)$
\[
        \begin{tikzcd}
            \cC \simeq \Mod_{1_{\cC}}(\cC) \ar[rr, "F \simeq \Mod_{1_\cC}(F)"] \ar[d] && \cD \simeq \Mod_{1_{\cD}}(\cD) \ar[d]    \\         \Mod_A(\cC) \ar[rr, "\Mod_A(F)"] &&  \Mod_{F(A)}(\cD)
        \end{tikzcd}
\]
and hence a morphism in $\CAlg(\PrL)$ from the pushout $\cD \otimes_{\cC} \Mod_{A}(\cC) \to \Mod_{F(A)}(\cD)$. This is an equivalence in $\CAlg(\PrL)$ because its underlying functor is one by \cite[Thm. 4.8.4.6]{HA}. 
\end{proof}
Since $\Spaces$ is the tensor unit of the symmetric monoidal structure on $\PrL$, it is initial in $\CAlg(\PrL)$ and any $\cC \in \CAlg(\PrL)$ comes equipped with a unique  symmetric monoidal left adjoint functor $ \iota_{\cC} \colon \Spaces \to \cC$.  
\begin{cor}\label{obs:algspaces}
For $\Monoid \in \CAlg(\Spaces)$, \cref{prop:algprl-new}.\eqref{item-prop:algprl5} implies that there is an equivalence in $\CAlg(\PrL)$: \[\cC \otimes \Mod_\Monoid(\Spaces)
\simeq
\cC \otimes_{\Spaces}\Mod_{\Monoid}(\Spaces)
\simeq
\Mod_{\iota_{\cC}(\Monoid)}(\cC \otimes_\Spaces \Spaces)
\simeq 
\Mod_{\iota_{\cC}(\Monoid)}(\cC).
\]
\end{cor}

\subsubsection{Adjoining colimits}

\begin{notation}
\label{nota:catk}
For a small set $\cK$ of simplicial sets, let $\cat^{\cK}$ denote the subcategory of $\cat$ on those small $\infty$-categories which admit colimits of diagrams indexed by elements of $\cK$, and those functors which preserve such colimits. 
\end{notation}
The $\infty$-categories $\catidem, \catprod$ and $\catrex$ from  \cref{not:smallcats} are instances of $\cat^\cK$ for $\cK$ consisting of the `walking idempotent' of \cite[\S~4.4.5]{HTT}, or the walking idempotent together with the set of finite (discrete) sets or the set of finite simplicial sets, respectively.

Just like the presheaf $\infty$-category $\Pres(\cC)$ is the free completion of a small $\infty$-category $\cC$ under small colimits, we may complete under other classes of colimits. The following combines {\cite[Lem.~4.8.4.2, Rem.~4.8.1.8]{HA} and \cite[Cor. 5.3.6.10]{HTT}}:

\begin{prop}
\label{prop:adjoiningcolims} Let $\cK$ be a small set of simplicial sets.
\begin{enumerate}
\item \label{item-prop:adjoiningcolims-1}
The $\infty$-category $\cat^{\cK}$ is presentable and admits a presentably symmetric monoidal structure, which can be characterized as follows: If $\cC, \cD \in \cat^{\cK}$, the tensor product $\cC \otimes \cD$ is equipped with a functor $\cC \times \cD \to \cC \otimes \cD$ which preserves $\cK$-colimits separately in both variables and which induces for all $\cE \in \cat^{\cK}$ an equivalence
$$\Fun^{\cK}(\cC \otimes \cD, \cE) \to \mathrm{\Fun^{\cK \times \cK}}(\cC \times \cD, \cE),$$
where $\Fun^{\cK}(\cC\otimes \cD, \cE)$ denotes the full subcategory of $\Fun(\cC \otimes \cD,\cE)$ on those functors which preserve $\cK$-colimits and where $\Fun^{\cK \times \cK}(\cC \times \cD, \cE)$ denotes the full subcategory of $\Fun(\cC \times \cD, \cE)$ on those functors which preserve $\cK$-colimits separately in both variables.
\item 
\label{item-prop:K-symmetric-monoidal}
Let $\cK'$ be a small set of simplicial sets with containing $\cK$. Then the subcategory inclusion $\cat^{\cK'} \to \cat^{\cK}$ admits a
symmetric monoidal left adjoint 
$$\Pres_{\cK}^{\cK'}\colon \cat^{\cK} \to \cat^{\cK'}$$
whose unit  $\cC \to \Pres_{\cK}^{\cK'}(\cC)$ for $\cC\in \cat^{\cK'}$  is a fully faithful functor.
\end{enumerate}
\end{prop}
The second statement of \cref{prop:adjoiningcolims} implies that $ \cC \hookrightarrow \Pres_{\cK}^{\cK'}(\cC)$ may be thought of as a generalized Yoneda embedding: It is the free cocompletion of $\cC$ under $\cK'$-shaped colimits subject to the relation that $\cK$-shaped colimits in $\cC$ are preserved. 

\subsection{Compact generation and ind-completion}\label{subsec:compact-gen-and-ind}
\subsubsection{Compact and projectively generated \texorpdfstring{$\infty$}{infinity}-categories}
Many presentable $\infty$-categories are generated by $\omega$-compact objects, where $\omega$ is the cardinality of the natural numbers. We will henceforth refer to $\omega$-filtered diagrams simply as \emph{filtered diagrams} and to $\omega$-compact objects as \emph{compact objects}. Hence, an object $c$ of an $\infty$-category $\cC$ with filtered colimits is compact if $\Hom_{\cC}(c,-)\colon \cC \to \Spaces$ preserves filtered colimits.

Similarly, we recall the following definitions:
\begin{definition}\label{def:proj-and-cmpt-proj}
    \begin{enumerate}
        \item An object $c$ of an $\infty$-category $\cC$ with geometric realizations (i.e. colimits indexed by $\Delta^{\op}$) is called \emph{projective} \cite[Def. 5.5.8.18]{HTT} if $\Hom_{\cC}(c,-) \colon \cC \to \Spaces$ preserves geometric realizations.
        \item An object $c$ of an $\infty$-category $\cC$ with sifted colimits ~\cite[Def. 5.5.8.1]{HTT} is called \emph{compact-projective} if $\Hom_{\cC}(c,-) \colon \cC \to \Spaces$ preserves sifted colimits ~\cite[Rem. 5.5.8.20]{HTT}.
    \end{enumerate}
\end{definition}

\begin{observation}
    Since sifted colimits are generated by filtered colimits and geometric realizations ~\cite[Cor. 5.5.8.17]{HTT}, an object is compact-projective if and only if it is compact and projective.
\end{observation}

We refer to \cref{def:compact-1-generation} and \cref{exm:compact-1-proj-abelian} for a comparison with classical notions of compactness and projectivity in ordinary (abelian) categories.

\begin{definition}
    \label{def:cpgen} We will use the following terminology.
\begin{enumerate}
\item A \emph{compactly generated} $\infty$-category is an $\infty$-category with small colimits, for which there exists a small set of compact objects which generates $\cC$ under small colimits. 

\item A \emph{projectively generated} $\infty$-category is an $\infty$-category with small colimits, for which there exists a small set of compact-projective objects which generates $\cC$ under small colimits.
\end{enumerate}
\end{definition}
\begin{nota}
Let $\PrLc$ (resp. $\PrLcp$) denote the subcategory of $\PrL$ on the compactly (resp. projectively) generated presentable $\infty$-categories and the cocontinuous functors which preserve compact (resp. compact-projective) objects.
\end{nota}

\begin{lemma}\label{lem:PrLcpsubcat} Consider an adjunction   between $\infty$-categories
\[\begin{tikzcd}[column sep=1.5cm]
\cC
\arrow[yshift=0.9ex]{r}{L}
\arrow[leftarrow, yshift=-0.9ex]{r}[yshift=-0.2ex]{\bot}[swap]{R}
&
\cD
\end{tikzcd}.
\]
\begin{enumerate}
\item 
\label{item-lem:compact-generated-left-adjoint-and-filter-right}
If $\cC$ is compactly generated and $\cD$ has filtered colimits, then the left adjoint $L$ preserves compact objects if and only if the right adjoint $R$ preserves filtered colimits.
\item 
\label{item-lem:projective-generated-left-adjoint-and-sifted-right}
If $\cC$ is projectively generated and $\cD$ has sifted colimits, then the left adjoint $L$ preserves compact-projective objects if and only if the right adjoint $R$ preserves sifted colimits.
\end{enumerate}
\end{lemma}
\begin{proof}
We prove the first statement; the proof of the second statement is analogous. Suppose $R$ preserves filtered colimits and $c\in \cC$ is compact. Then, for every  filtered diagram $d\colon I \to \cD$ we have
\begin{gather*}
    \Hom_{\cC}(Lc, \colim_i d_i) \simeq \Hom_{\cD}(c, R\colim_i d_i) \simeq \Hom_{\cD}(c,\colim_i Rd_i)
    \\ \hspace{3cm}\simeq \colim_i \Hom_{\cD}(c, Rd_i) \simeq \colim_i \Hom_{\cD}(Lc, d_i)
\end{gather*}
and hence $Lc$ is compact. Conversely, suppose that $L$ preserves compact
objects. It follows that for a compact object $c\in \cC$ and a filtered diagram $d\colon I
\to \cD$, we have 
\begin{gather*}
\Hom_{\cC}(c, R(\colim_i d_i))\simeq \Hom_{\cD}(Lc, \colim_i d_i) \simeq \colim_i \Hom_{\cD}(Lc, d_i)
\\
 \hspace{3cm}  
 \simeq \colim_i \Hom_{\cC}(c, Rd_i)\simeq
\Hom_{\cC}(c, \colim_i Rd_i).
\end{gather*}
Since $\cC$ is compactly generated, every object in $\cC$ is a small colimit
of compact objects, and thus $\colim_i Rd_i \simeq R\colim_i d_i$.
\end{proof}

\begin{corollary}\label{cor:PrLcp-subcategory} The $\infty$-category $\PrLcp$ is a subcategory of $\PrLc$. 
\end{corollary}
\begin{proof}
A projectively generated presentable $\infty$-category is also compactly
generated since compact-projective objects are in particular compact. 
Thus, we only need to show that a left adjoint
functor $L$ between projectively generated presentable $\infty$-categories, which
preserves compact-projective objects, also preserves compact objects. Indeed, by
\cref{lem:PrLcpsubcat}.\eqref{item-lem:projective-generated-left-adjoint-and-sifted-right}, $L$ has a right adjoint which preserves sifted
colimits and hence preserves filtered colimits. Applying the reverse direction of 
\cref{lem:PrLcpsubcat}.\eqref{item-lem:compact-generated-left-adjoint-and-filter-right} now shows that $L$ preserves compact objects.
\end{proof}

\begin{obs}\label{obs:fullsub-of-compact-and-compact-projective-objects}
    Let $\cC$  be a cocomplete $\infty$-category. 
    \begin{enumerate}
        \item \label{item-obs:fullsub-of-compact}
        The full
        subcategory $\cC^{\mrc}$ of compact objects is closed under retracts and \emph{finite
        colimits}~\cite[Cor. 5.3.4.15 and Rem. 5.3.4.16]{HTT} and hence yields an
        object $\cC^\mrc \in \catrex$. This defines a functor $(-)^{\mrc} \colon \PrLc \to \catrex$.
        \item \label{item-obs:fullsub-of-compact-proj}
        The full subcategory $\cC^{\cp}$ of compact-projective
        objects is closed under retracts and \emph{finite coproducts}~\cite[Rem.
        5.5.8.19]{HTT} and hence defines an object $\cC^{\cp} \in \catprod$. This  defines a functor $(-)^{\cp} \colon \PrLcp \to \catprod$.
    \end{enumerate}
\end{obs}

In \cref{prop:PrLcp} we will show that these functors are in fact equivalences.

\subsubsection{Ind-completion}\label{subsubsec:ind-completion}
The \emph{ind-completion} $\Ind(\cC)$ of a small $\infty$-category $\cC$ is
defined to be the smallest full subcategory of $\Pres(\cC)$ which contains the
image of the Yoneda embedding and is closed under filtered colimits,
\cite[Rem.~5.3.5.2, Prop.~5.3.5.3]{HTT}. Then,
$\Ind(\cC)$ has filtered colimits and the inclusion $\cC \to \Ind(\cC)$ is
characterized by the universal property that for any $\infty$-category $\cD$
with filtered colimits, it induces an equivalence
\begin{equation}
    \label{eq:YonedaOmega}
    \Fun^{\omega}(\Ind(\cC),\cD) \to \Fun(\cC,\cD),
\end{equation}
where $\Fun^{\omega}$ denotes the full subcategory of functors which preserve filtered colimits. 

If $\cC$ moreover has finite colimits, then $\Ind(\cC)$ is equivalent to the
full subcategory of $\Pres(\cC)$ on those functors $\cC^{\op} \to \Spaces$ which
send finite colimits in $\cC$ to finite limits of spaces by \cite[Cor.  5.3.5.4]{HTT}. In this case,
$\Ind(\cC)$ is presentable, the inclusion $\cC \to \Ind(\cC)$ preserves finite colimits, and for
any presentable $\infty$-category $\cD$, the induced functor 
$$\FunL(\Ind(\cC),\cD) \to \Fun^{\mathrm{rex}}(\cC, \cD)$$
is an equivalence by \cite[Cor.  5.3.5.10]{HTT}, where
$\Fun^{\mathrm{rex}}$ denotes the full subcategory of functors which preserve
finite colimits. For more details see \cite[Section 5.3]{HTT}.

Similarly, the \emph{$\PresSigma$-completion} $\PresSigma(\cC)$ of a small $\infty$-category
$\cC$ is the smallest full subcategory of $\Pres(\cC)$ that contains the image
of the Yoneda embedding and is closed under sifted colimits. The
$\infty$-category $\PresSigma(\cC)$ has sifted colimits and the inclusion $\cC
\to \PresSigma(\cC)$ is characterized by the universal property that for any
$\infty$-category $\cD$ with sifted colimits, it induces an equivalence
\[
    \Fun^{\Sigma}(\PresSigma(\cC),\cD) \to \Fun(\cC,\cD),
\]
where $\Fun^{\Sigma}$ denotes the full subcategory of functors which preserve sifted colimits, \cite[Prop.~5.5.8.15]{HTT}.

If $\cC$ moreover has finite coproducts, then $\PresSigma(\cC)$ is equivalent to
the full subcategory of $\Pres(\cC)$ on those functors $\cC^{\op} \to \Spaces$
which send finite coproducts in $\cC$ to finite products of spaces~\cite[Def.
5.5.8.8 and Rem. 5.5.8.16.(1)]{HTT}. In this case, $\PresSigma(\cC)$ is
presentable, the inclusion $\cC \to \PresSigma(\cC)$ preserves finite coproducts and for any
presentable $\infty$-category $\cD$, the induced functor
$$\FunL(\PresSigma(\cC), \cD) \to \Fun^{\sqcup}(\cC, \cD)$$ is an equivalence,
where $\Fun^{\sqcup}$ denotes the full subcategory of functors which preserve
finite coproducts~\cite[Rem. 5.5.8.16(iii)]{HTT}, see ~\cite[\S~5.5.8]{HTT}
for details.

\begin{prop}
\label{prop:PrLcp} The following hold.
 \begin{enumerate}
\item \label{item-prop:PrLcp-1}
The $\Ind$-completion restricts to an equivalence $\Ind\colon \catrex \to
\PrLc$ inverse to the functor $(-)^{\mrc}$ from \cref{obs:fullsub-of-compact-and-compact-projective-objects}.\eqref{item-obs:fullsub-of-compact}.
\item \label{item-prop:PrLcp-2}
The $\PresSigma$-completion restricts to an equivalence $\PresSigma\colon
\catprod \to \PrLcp$ inverse to the functor $(-)^{\cp}$ from \cref{obs:fullsub-of-compact-and-compact-projective-objects}.\eqref{item-obs:fullsub-of-compact-proj}.
\item \label{item-prop:PrLcp-3}
The composite $\catprod \simeq \PrLcp \to \PrLc \simeq
\catrex$ is left adjoint to the subcategory inclusion $\catrex \to \catprod$. 
\item \label{item-prop:PrLcp-4}
The $\infty$-categories $\PrLc$ and $\PrLcp$ are presentable and the
inclusion functor $\PrLcp\to \PrLc$ is cocontinuous, i.e. a morphism in $\PrL$.
\end{enumerate}
\end{prop}
\begin{proof} The first statement is \cite[Lem.~5.3.2.9]{HA} for
$\kappa=\omega$, also see~\cite[Prop.~5.5.7.8]{HTT}.

To prove the second statement, we note that \cite[Cor. 5.3.6.10, Rem.
5.5.8.16]{HTT} implies that $\PresSigma(-)$ defines a functor from $\catprod$ to the
huge $\infty$-category $\widehat{\Cat}_{\infty}^{\mathrm{cocpl}}$ of large
$\infty$-categories which admit all small colimits and colimit preserving
functors. By~\cite[Prop.~5.5.8.10]{HTT},  $\PresSigma(\cC)$ is an accessible
localization of $\Pres(\cC)$ and hence is presentable, so that $\PresSigma$
factors through the full subcategory $\PrL$ of
$\widehat{\Cat}_{\infty}^{\mathrm{cocpl}}$. By~\cite[Prop.~5.5.8.22]{HTT}, every
object in the image of the Yoneda embedding $\cC \hookrightarrow
\PresSigma(\cC)$ is compact-projective, and since $\PresSigma(\cC)$ is a
localization of $\Pres(\cC)$, it is generated under small colimits by objects in
$\cC$; hence $\PresSigma(\cC)$ is projectively generated. Moreover,
by~\cite[Prop.~5.5.8.25]{HTT}, the compact-projective objects of
$\PresSigma(\cC)$ are precisely the objects in (the essential image of) $\cC$
(note: this uses that $\cC$ is idempotent complete). Therefore, for any morphism
$f\colon  \cC \to \cD$ in $\catprod$ the cocontinuous functor $\PresSigma(f)\colon
\PresSigma(\cC) \to \PresSigma(\cD)$ preserves compact-projectives; hence
$\PresSigma\colon \catprod \to \PrL$ factors through the subcategory $\PrLcp \to
\PrL$. To show that it factors as an equivalence, notice that it is fully
faithful since for $\cC, \cD \in \catprod$,  $$\Fun^{\sqcup}(\cC, \cD) \simeq
\Fun^{\sqcup}(\cC, \PresSigma(\cD)^{\cp}) \simeq\mathrm{\Fun^{L, cp}} (\PresSigma(\cC),
\PresSigma(\cD)) $$ where the first equivalence uses that $\cD \simeq
\PresSigma(\cD)^{\cp}$ and the second equivalence uses that for any presentable
$\cE$, the map $\FunL(\PresSigma(\cC), \cE) \to \Fun^{\sqcup}(\cC, \cE)$ is an
isomorphism (this follows e.g. from~\cite[Prop.~5.5.8.10]{HTT}) and the fact
that $\cC \simeq \PresSigma(\cC)^{\cp}$. Lastly, surjectivity on objects follows
since by~\cite[Prop.~5.5.8.25]{HTT} any projectively generated presentable
$\infty$-category $\cD$ is equivalent to $\PresSigma(\cC)$ where $\cC$ is the
smallest full subcategory of $\cD$ spanned by finite coproducts of objects in
the set $S$ of compact-projective generators.

The third statement follows since the induced composite $\catprod \to \catrex$
sends $\cC$ to $\PresSigma(\cC)^{\mathrm{c}}$, which in the notation
of \cref{prop:adjoiningcolims} is equivalent to $\cP_{\cK}^{\cK'}(\cC)$ for $\cK$ the
collection of finite sets together with the `walking idempotent' $\mathrm{Idem}$
of \cite[Sec. 4.4.5]{HTT} and $\cK'$ the collection of finite categories
together with $\mathrm{Idem}$. Hence, by \cref{prop:adjoiningcolims}.\eqref{item-prop:K-symmetric-monoidal} this functor
$\catprod \to \catrex$ is left adjoint  to the forgetful functor $\catrex \to
\catprod$.
The fourth statement then follows from the previous ones and \cref{prop:adjoiningcolims}.\eqref{item-prop:adjoiningcolims-1}, since $\catprod$ and $\catrex$ are presentable and since the functor $\PrLcp \to \PrLc$ is equivalent to $\catprod \to \catrex$ and hence a left adjoint. 
\end{proof}

Given a presentable $\infty$-category projectively/compactly generated by a small set $S$ of objects, then the objects in this set also generate the full subcategories $\cC^{\cp}$, or $\cC^\mrc$, respectively:
\begin{lemma}
\label{lem:generation} The following hold.
\begin{enumerate}
\item \label{item-lem:generation-cp}
Let $\cC \in \PrLcp$ and $S$ a small set of compact-projective generators. Then 
the set $S\subseteq \cC^{\cp}$ generates the full subcategory $\cC^{\cp}$ under retracts and finite coproducts. In particular, every compact-projective object in $\cC$ is a retract of a finite coproduct of objects in $S$. 
\item \label{item-lem:generation-c}
Let $\cC \in \PrLc$ and $S$ a small set of compact generators. 
 Then the set $S\subseteq \cC^{c}$ generates the full subcategory $\cC^{c}$ under retracts and finite colimits. In particular, every compact object in $\cC$ is a retract of an iterated finite colimit of objects in $S$. 
\end{enumerate}
\end{lemma}
\begin{proof} 
The first statement is~\cite[Prop.~5.5.8.25.(2).(iii)]{HTT}, the proof of the
second statement is analogous. 
\end{proof}

\begin{prop}\label{prop:prlcpresentable} 
The symmetric monoidal structure of
$\PrL$ restricts to \emph{presentably} symmetric monoidal structures on the
subcategories $\PrLcp$ and $\PrLc$ of $\PrL$.

Together with the symmetric monoidal left adjoint of the forgetful functor $\catrex \to \catprod$ from \cref{prop:adjoiningcolims}.\eqref{item-prop:K-symmetric-monoidal}, the symmetric monoidal subcategory inclusion $\PrLcp \to \PrLc$ assemble into a commutative diagram in $\CAlg(\PrL)$:
    \begin{equation}\label{eq:comm-square-in-CALg-PRL}
        \begin{tikzcd} \PrLcp \arrow[r]& \PrLc\\ \catprod \arrow[u, "\PresSigma",
"\simeq"'] \arrow[r] & \catrex \arrow[u, "\Ind", "\simeq"']
\end{tikzcd}
    \end{equation}

\end{prop}
\begin{proof}
Recall from \cref{prop:adjoiningcolims} the functor $\catprod \to \catrex$ in $\CAlg(\PrL)$  which is left adjoint to the forgetful functor. By \cref{prop:PrLcp}.\eqref{item-prop:PrLcp-3},  this functor is equivalent to the composite $
\catprod \simeq \PrLcp \to \PrLc \simeq \catrex$ for the subcategory inclusion $\PrLcp \to \PrLc$. Hence, this induces presentably symmetric monoidal structures on $\PrLcp$ and $\PrLc$ and on the inclusion $\PrLcp \to \PrLc$. 

It remains to show that the subcategory inclusion $\PrLc \to \PrL$ is symmetric monoidal which follows since the composite $ \catrex \xrightarrow{\Ind} \PrLc\to
\PrL$ is~\cite[Lem.~5.3.2.11]{HA}.
\end{proof}

Recall  that any presentably symmetric monoidal $\infty$-category  $\cC$ comes equipped with a unique symmetric monoidal left adjoint functor $\iota_{\cC}\colon \Spaces \to \cC$. We now show that the unique symmetric monoidal left adjoint $\iota_{\PrLcp} \colon \cS \to \PrLcp$ sends a space $X$ to the presheaf category $\Pres(X)$. 

\begin{lemma}\label{lem:unit} 
The symmetric monoidal functor $\Pres(-)\colon \cS\hookrightarrow \cat \to \PrL$ factors through the symmetric monoidal subcategory $\PrLcp \to \PrL$. The induced symmetric monoidal functor $\Spaces \to \PrLcp$ is a left adjoint, and hence is the unit $\iota_{\PrLcp}\colon \Spaces \to \PrLcp$ of the presentably symmetric monoidal $\infty$-category $\PrLcp$.
\end{lemma}
\begin{proof}
For any small $\infty$-category $\cC$, the presheaf category $\Pres(\cC)$ is generated by a small set of \emph{tiny} objects, i.e. objects $c\in \cC$ for which $\Hom_{\cC}(c,-)\colon \cC \to \Spaces$ preserves all small colimits; such a small set of tiny objects is for example provided by the objects of $\cC$ itself. Moreover, any functor $F\colon \cC \to \cD$ induces a left adjoint functor $\Pres(F)\colon \Pres(\cC) \to \Pres(\cD)$ which preserves tiny objects. By an argument entirely analogous to the proof of \cref{cor:PrLcp-subcategory}, it follows that $\Pres(\cC)$ is in particular projectively generated, and that $\Pres(F)$ preserves compact-projectives. Hence, the symmetric monoidal functor $\Pres\colon \cat \to \PrL$ factors through the symmetric monoidal subcategory $\PrLcp$ from \cref{prop:prlcpresentable}.
Moreover, the functor $\cat \to \PrLcp$ is a left adjoint since after composing with the equivalence $\PrLcp \simeq \catprod$ from \cref{prop:PrLcp} it becomes equivalent to the left adjoint $\cat \to \catprod$ of the forgetful functor. Since moreover $\Spaces \to \cat$ is a left adjoint, so is the composite $\Spaces \to \cat \to \PrLcp$. 
 \end{proof}

It  follows from \cref{lem:generation} that for $\cC, \cD \in \CAlg(\PrLc)$ (resp. in $\CAlg(\PrLcp)$), the compact (resp. compact-projective) objects of the tensor product $\cC \otimes \cD$ are retracts of iterated finite colimits (retracts of finite coproducts) of external tensor products of compact (resp. compact-projective) objects. 

Explicitly, a commutative algebra $\cC \in \CAlg(\PrLc)$ (and analogously for
$\PrLcp$) therefore amounts to  a presentably symmetric monoidal $\infty$-category $\cC$, whose
underlying presentable $\infty$-category is compactly generated, so that its unit is compact, and if $c,
d$ are compact objects in $\cC$, then  $c\otimes d$ is also compact.

\begin{lemma}\label{lem:Modpresentable} 
Given $\cC \in \CAlg(\PrLcp)$ and $A \in \Alg(\cC)$, the following hold: 
\begin{enumerate}
    \item \label{item-lem:Modpresentable-5} The $\infty$-category $\RMod_A(\cC)$ of right $A$-modules (see \cref{subsubsec:left-modules}) is in $\PrLcp$.
    \item \label{item-lem:Modpresentable-6} The  action functor $\cC \otimes \RMod_A(\cC) \to \RMod_A(\cC)$ is in $\PrLcp$, thus $\RMod_A(\cC) \in \Mod_{\cC}(\PrLcp)$.
    \item \label{item-lem:Modpresentable-7} If $S$ is a set of compact projective generators of $\cC$, then the free modules $\{c \otimes A_A\}_{c \in S}$ are compact projective generators of $\RMod_A(\cC)$.
    \item  \label{item-lem:Modpresentable-8} Furthermore, if $\cC \in \CAlg(\PrLcp)$, then $\Mod_A(\cC)$ together with its symmetric monoidal relative tensor product is in $\CAlg(\PrLcp)$. 
    \setcounter{listcounter}{\value{enumi}}
\end{enumerate}
Given $\cC \in \CAlg(\PrLc)$ and $A \in \Alg(\cC)$, the following hold: 
\begin{enumerate}
\setcounter{enumi}{\value{listcounter}}
        \item \label{item-lem:Modpresentable-1}The $\infty$-category $\RMod_A(\cC)$ of right $A$-modules (see \cref{subsubsec:left-modules}) is in $\PrLc$.
    \item \label{item-lem:Modpresentable-2}The action functor $\cC \otimes \RMod_A(\cC) \to \RMod_A(\cC)$ is in $\PrLc$, thus $\RMod_A(\cC) \in \Mod_{\cC}(\PrLc)$.
    \item  \label{item-lem:Modpresentable-3}If $S$ is a set of compact generators of $\cC$, then the free modules $\{c \otimes A_A\}_{c \in S}$ are compact generators for $\RMod_A(\cC)$.
    \item \label{item-lem:Modpresentable-4}Furthermore, if $A \in \CAlg(\cC)$, then $\Mod_A(\cC)$ together with its symmetric monoidal relative tensor product is in  $\CAlg(\PrLc)$. 
\end{enumerate}

 \end{lemma}
\begin{proof} Statements \eqref{item-lem:Modpresentable-5} to \eqref{item-lem:Modpresentable-7} are proven in~\cite[Cor. 7.1.4.14]{HA}, statement \eqref{item-lem:Modpresentable-8} is an immediate consequence.
Statements \eqref{item-lem:Modpresentable-1} to \eqref{item-lem:Modpresentable-4} can be proven analogously (and also follow from the proof of~\cite[Lem.~5.3.2.12 (3)]{HA}).

\end{proof}

\subsection{Additive and stable \texorpdfstring{$\infty$}{infinity}-categories}
\label{sec:stable}
Here, we briefly review the theory of additive and stable $\infty$-categories. 
For more details, we refer to~\cite{HA}, ~\cite{BZFN} and~\cite{MR3450758}.

\subsubsection{Definitions}

An $\infty$-category is called \emph{zero-pointed} if it has an
initial and a terminal object and if the unique morphism from the initial to the
terminal object is an isomorphism. In this
case, we call the initial/terminal object a \emph{zero object}. 
 A zero-pointed
$\infty$-category is called \emph{semi-additive} if it furthermore has finite
products and finite coproducts and if the canonical morphism $x\sqcup y \to x \times y
$ is an isomorphism. In this case, we write the product/coproduct as $x\oplus y$
and refer to it as a \emph{direct sum}. A semi-additive $\infty$-category is
called \emph{additive} if furthermore the shear map $( \pi_1, \nabla)\colon x\oplus x
\to x \oplus x$ is an isomorphism, where $\pi_1\colon x \oplus x \to x$ denotes the
projection to the first factor (using that $x \oplus x$ is a product) and $\nabla\colon x \oplus x \to x$ is the fold map (using that $x \oplus x$ is a coproduct).
A functor between
additive $\infty$-categories is called \emph{additive} if it preserves finite coproducts. We denote the $\infty$-category of additive functors
between two additive $\infty$-categories $\cA, \cB$ by $\Funadd(\cA,\cB)$.
This $\infty$-category is itself an additive $\infty$-category~\cite[Cor. 2.9]{MR3450758}. The notion of an additive $\infty$-category is a
direct generalization of the ordinary $1$-categorical notion, and indeed an ordinary $1$-category is additive in the usual sense if and only if it(s nerve) is additive in the
$\infty$-categorical sense. Conversely, a semi-additive $\infty$-category $\cC$ is
additive if and only if its homotopy category $h_1\cC$ is additive as an ordinary $1$-category \cite[Prop.~2.8]{MR3450758}.

A zero-pointed $\infty$-category is called
\emph{stable} if it admits finite colimits and if any square is a pullback
square if and only if it is a pushout square. A functor between
stable $\infty$-categories is called \emph{exact} if it preserves finite
colimits. Given two stable $\infty$-categories $\cC, \cD$, we denote the $\infty$-category of
exact functors between them by $\Funex(\cC,\cD)$. This $\infty$-category
$\Funex(\cC,\cD)$ is itself stable since it is a full subcategory of $\Fun(\cC, \cD)$ (which is stable by \cite[Prop.~1.1.3.1]{HA}) that contains the zero object and is stable under forming fibers and cofibers, as a straightforward computation shows. In what follows, we will only consider idempotent complete stable
categories.

\begin{nota} We use the following notation:
\begin{itemize}
\item $\add$ for the full $\infty$-subcategory of $\catprod$ consisting of  additive, idempotent complete,  small $\infty$-categories.
    \item 
$\st$ for the full $\infty$-subcategory of $\catrex$ consisting of  stable, idempotent
complete, small $\infty$-categories. 
\end{itemize}
\end{nota}

Since  stable $\infty$-categories are additive and exact functors preserve finite coproducts, there is a forgetful functor $\st \to \add$. 

\begin{warning}
\label{warn:idem}
 All additive and stable $\infty$-categories will be implicitly assumed to be idempotent complete. In particular, we have defined $\add$ and $\st$  as full subcategories of $\catprod$ and $\catrex$.  
\end{warning}
As in \cref{sec:presentable} (in particular \cref{prop:PrLcp}), it will be useful to characterize small additive or stable $\infty$-categories in terms of projectively resp. compactly generated presentable $\infty$-categories.

\begin{nota} Following \cref{def:cpgen}, we use the following notations:
\begin{itemize}
\item $\PrLst$ for the full subcategory of $\PrL$ on the stable, presentable $\infty$-categories and $\PrLadd$ for the full subcategory on the additive, presentable $\infty$-categories.
\item $\PrLstc$ for the full subcategory of $\PrLc$ on the stable presentable $\infty$-categories which are compactly generated as $\infty$-categories, and $\PrLaddcp$ for the full subcategory of $\PrLcp$ on the additive presentable $\infty$-categories which are projectively generated as $\infty$-categories.
\end{itemize}
\end{nota}

\begin{prop} \label{prop:indst} The following hold.
\begin{enumerate}
\item \label{item-prop:indst-1}
The equivalence $\PresSigma\colon \catprod \to \PrLcp$ restricts to an equivalence between full subcategories 
$$\PresSigma\colon\add  \xrightarrow{\simeq}  \PrLaddcp.$$
Its inverse is $(-)^{\cp}$ which takes a projectively generated additive presentable $\infty$-category $\cC$ to its full subcategory $\cC^{\cp}$ on the compact-projective objects.
\item \label{item-prop:indst-2}
The equivalence $\Ind\colon \catrex \to \PrLc$  restricts to an equivalence between  full subcategories 
$$\Ind\colon\st   \xrightarrow{\simeq}\PrLstc.$$
Its inverse is $(-)^{\mrc}$ which takes a compactly generated stable presentable $\infty$-category $\cC$ to its full subcategory $\cC^{\mrc}$ on the compact objects.
\end{enumerate}
\end{prop}
\begin{proof}
We will prove part  \eqref{item-prop:indst-1}, the proof of part \eqref{item-prop:indst-2} is entirely analogous and can
for example be found in~\cite[Lem.~2.20]{BGT}. Recall from \cref{prop:PrLcp} that $\PresSigma\colon \catprod \to \PrLcp$ is an
equivalence, whose inverse is $(-)^{\cp}$. To prove statement (1), it therefore suffices to show that
the essential image of the composite $\add \hookrightarrow \catprod \simeq
\PrLcp$ is the full subcategory $\PrLaddcp$. 

If $\cC$ is a small additive $\infty$-category, then $\PresSigma(\cC) \simeq \Funadd(\cC^{\op}, \Spaces)$ is additive by~\cite[Cor. 2.9]{MR3450758}. 
On the other hand, if $\cD$ is any projectively generated additive presentable category, then the full subcategory on its compact-projective objects is closed under finite coproducts and hence is again additive. Therefore, $\cD$ is in the image of $\add \hookrightarrow \PrLcp$. 
\end{proof}

\subsubsection{Symmetric monoidal structure}

The universal example of an stable presentable $\infty$-category is
the $\infty$-category $\Spectra$ of spectra. Likewise, the universal example of an
additive presentable $\infty$-category is the $\infty$-category
 $\Spectra_{\geq 0}$ of connective spectra, equivalent to the
$\infty$-category $\mathrm{Grp}_{\mathbb{E}_{\infty}}(\Spaces)$ of grouplike
$\mathbb{E}_{\infty}$-spaces, see~\cite{MR3450758}. Both  $\Spectra$ and
$\ConnSpectra$ are \emph{idempotent algebras} in $\PrL$, i.e. commutative algebras
$A\in \CAlg(\PrL)$ so that the multiplication $A\otimes A \to A$ is an
isomorphism. It is shown in~\cite[Prop.~4.8.2.18]{HA}
and~\cite[Cor. 4.8]{MR3450758} that the full subcategories
$\Mod_{\Spectra}(\PrL)$ and $\Mod_{\ConnSpectra}(\PrL)$ of $\PrL$ are equivalent
to $\PrLst$ and $\PrLadd$, respectively. As categories of modules of a
commutative algebra, this induces symmetric monoidal structures on $\PrLst$
and $\PrLadd$ respectively by  \cref{prop:algprl-new}.\eqref{item-prop:algprl1}.

By ~\cite[Prop.~1.4.3.7]{HA}, the $\infty$-category $\Spectra$ is compactly generated (by the single object $\mathbb{S}$, the sphere spectrum).
It follows from \cref{lem:generation} that the compact objects in $\Spectra$ are finite spectra, i.e. finite colimits of the sphere spectrum (note that a retract of a finite spectrum is again finite).
However, $\Spectra$ is not projectively generated (its only projective object is the zero spectrum, cf.~\cite[Rem.~7.2.2.5]{HA}).
On the other hand, $\ConnSpectra$ is projectively generated by the sphere spectrum~\cite[Cor. 7.1.4.13]{HA}. It therefore follows from \cref{lem:generation} that the compact-projective objects in $\ConnSpectra$ are finite sums of the sphere spectrum (note that a retract of a finite sum of sphere spectra is again a finite sum of sphere spectra). 

\begin{lemma}
\label{lem:prlstviamod} The following hold. 
\begin{enumerate}
\item The equivalence $\Mod_{\ConnSpectra}(\PrL) \xrightarrow{\simeq}   \PrLadd$ restricts to an equivalence between the subcategories $\Mod_{\ConnSpectra}(\PrLcp) \xrightarrow{\simeq}   \PrLaddcp$.
\item The equivalence $\Mod_{\Spectra}(\PrL)\xrightarrow{\simeq}   \PrLst$ restricts to an equivalence between the subcategories $\Mod_{\Spectra}(\PrLc)\xrightarrow{\simeq}   \PrLstc$.
\end{enumerate}  
\end{lemma}
\begin{proof}
We prove the first statement, the second is analogous. The $\infty$-category
$\Mod_{\ConnSpectra}(\PrLcp)$ may be understood as the subcategory of
$\Mod_{\ConnSpectra}(\PrL)$ on those presentable $\ConnSpectra$-module
$\infty$-categories $\cC$ whose underlying $\infty$-category is projectively
generated and for which the action functor $\ConnSpectra \otimes \cC \to \cC$
preserves compact-projectives, and those cocontinuous $\ConnSpectra$-module
functors $\cC \to \cD$ for which the underlying functor preserves compact
projectives. In particular, the equivalence  $\Mod_{\ConnSpectra}(\PrL) \to
\PrLadd$ restricts to a fully faithful functor $\Mod_{\Spectra_{\geq 0}}(\PrLcp)
\to \PrLaddcp$. It therefore suffices to verify that for an additive presentable
$\infty$-category $\cC$, the action $\Spectra_{\geq 0} \times \cC \to \cC$ sends
a pair of compact-projective objects $(a, b) \in \Spectra_{\geq 0}^{\cp} \times
\cC^{\cp}$ to a compact-projective of $\cC$. This follows since any compact
projective in $\ConnSpectra$ is generated under finite coproducts and retracts
by the unit object $\mathbb{S}$; see \cref{lem:generation}.\end{proof}

Using the theory of commutative algebras in presentable categories, we
immediately obtain the following stable and additive analogues of the first half of \cref{prop:prlcpresentable} concerning symmetric monoidal structures
on subcategories of $\PrL$. The passage from $\PrLaddcp \simeq \add$ to $\PrLstc
\simeq \st$ will be treated in the next section.

\begin{cor} The following hold.
\begin{enumerate}
\item  The symmetric monoidal structure of $\PrLadd$ restricts to a presentably
symmetric monoidal structure on the subcategory $\PrLaddcp$, which induces a
presentably symmetric monoidal structure on the $\infty$-category $\add$ via the
equivalence $\PresSigma\colon \add \xrightarrow{\simeq}   \PrLaddcp$.
\item The symmetric monoidal structure of $\PrLst$ restricts to a presentably
symmetric monoidal structure on the subcategory $\PrLstc \to \PrLst$, which
induces a presentably symmetric monoidal structure on the $\infty$-category
$\st$ via the equivalence $\Ind\colon \st \xrightarrow{\simeq}   \PrLstc$.
\end{enumerate}
\end{cor}
\begin{proof}
Since $\PrLaddcp$ and $\PrLstc$ are module
categories by \cref{lem:prlstviamod}, they inherit via \cref{prop:algprl-new}.\eqref{item-prop:algprl1}
presentably symmetric monoidal structures from the presentably symmetric monoidal categories
$\PrLc$ and $\PrLcp$ (see \cref{prop:prlcpresentable}), respectively. Symmetric
monoidality of the functors $\PrLstc \to \PrLst$ and $\PrLaddcp \to \PrLadd$
follows from symmetric monoidality of $\PrLc \to \PrL$ and $\PrLcp \to \PrL$.
\end{proof}

Tracing through the proof, the symmetric monoidal structures on $\add$ respectively $\st$ may be characterized as follows (c.f.~\cite[Prop.~4.4]{BZFN}): 
For $\cC, \cD\in \add$ the tensor product $\cC\otimes \cD$ is equipped with a functor $\cC \times \cD \to \cC \otimes \cD$, additive in both variables, and satisfies the universal property that for any $\cE \in \add$ the induced functor 
\[\Fun^{\mathrm{add}}(\cC \otimes \cD, \cE) \to \Fun^{\mathrm{add}\times\mathrm{add}}(\cC \times \cD, \cE)\]
is an equivalence, where $ \Fun^{\mathrm{add}\times\mathrm{add}}(\cC \times \cD, \cE)$ denotes the full subcategory of $\Fun(\cC \times \cD, \cE)$ on the functors which are additive in both variables (i.e. which preserve finite coproducts separately in either variable).

For $\cC, \cD \in \st$, the tensor product $\cC \otimes \cD$ is characterized analogously in terms of functors $\cC \times \cD \to \cE$ which are exact in both variables (i.e. which preserve finite colimits separately in both variables).

\begin{warning} As in \cref{warn:idem}, the $\infty$-categories $\add$ and $\st$ are the $\infty$-categories of additive, resp. stable, \emph{idempotent complete} $\infty$-categories. In particular, the tensor product of additive/stable  idempotent complete $\infty$-categories we consider here is automatically idempotent complete. \end{warning}

\subsection{From additive to stable \texorpdfstring{$\infty$}{infinity}-categories}\label{subsec:additive-to-stable}
Given an ordinary additive $1$-category $\cA$, one may form a stable $\infty$-category $\Kb(\cA)$ of bounded (in both directions) chain complexes, chain homomorphisms, and (higher) chain homotopies between these. In this section, we review this construction and prove that it satisfies a universal property: the stable $\infty$-category $\Kb(\cA)$ is the free stable $\infty$-category on the additive category $\cA$.

\subsubsection{The \texorpdfstring{$\infty$}{infinity}-category of chain complexes}
Given an ordinary additive $1$-category $\cA$, the $\infty$-category $\Kb(\cA)$ can  be defined, see \cite[\S~1.3.1]{HA}, using the technology of \emph{dg nerves} as follows.

\begin{definition}[{\cite[Cons. 1.3.1.6 and Rem. 1.3.2.2]{HA}}]\label{def:chaincx}
For an ordinary additive $1$-category $\cA$, we let $\Kb(\cA):=N_{\mathrm{dg}}(\mathrm{Ch}^b(\cA))$ denote the dg nerve of
the dg category of bounded chain complexes in $\cA$. 
\end{definition}

Viewing $\cA$ as an additive $\infty$-category,  there is a canonical additive functor $\cA \to \Kb(\cA)$ induced from the functor that interprets objects of $\cA$ as chain complexes concentrated in degree zero.

\begin{prop} \label{prop: Kbstable} For an ordinary additive $1$-category $\cA$, the dg nerve $\Kb(\cA)= N_{\mathrm{dg}}(\mathrm{Ch}^b(\cA))$ is a \emph{stable} $\infty$-category. 
\end{prop}
\begin{proof}
The dg nerve of the dg category of (unbounded) chain complexes is an $\infty$-category by \cite[Prop.~1.3.1.10]{HA} and stable by \cite[Prop.~1.3.2.10]{HA}. The full dg subcategory of bounded chain complexes is closed under shifts and formation of mapping cones, and thus its dg nerve $\Kb(\cA)$ is itself a stable $\infty$-category, by \cite[Lem.~1.1.3.3]{HA} and the discussion after \cite[Proof of Prop.~1.3.2.10]{HA}.
\end{proof}

\begin{remark}\label{handover3}
By \cite[Rem.~1.3.1.11]{HA}, the homotopy 1-category $h_1\Kb(\cA)$, recalled in~\cref{subsubsec:basic}, is the chain homotopy category
$\oldKb{\cA}$ in the sense of \cref{def:classicalK}. 
Note that, unlike the notion of derived category, which can only be defined for abelian categories, the stable $\infty$-category $\Kb(\cA)$ is defined for any (possibly non-abelian) additive category $\cA$. 
\end{remark}

In \cref{cor:universalch}, we will prove that $\Kb(\cA)$ is the universal stable $\infty$-category associated to $\cA$. To do so, we express $\Kb$ in terms of the functors constructed in the previous sections.

\subsubsection{The free stable category on an additive category}
\label{sec:freestable}

Recall that $\PrLcp \to \PrLc$ is a symmetric monoidal subcategory, and hence that $\ConnSpectra\in \CAlg(\PrLcp)$ may also be considered an algebra in $\CAlg(\PrLc)$. Notice also that the full subcategory inclusion $\ConnSpectra \to \Spectra$ is symmetric monoidal, has a right adjoint (namely the  $0$-th connective cover functor $\tau_{\geq 0}$) and sends compact objects in $\ConnSpectra$ to compact objects in $\Spectra$, as the sphere spectrum compactly generates $\ConnSpectra$ and $\Spectra$.

\begin{construction}\label{const:inductionSp}
Applying \cref{prop:algprl-new}.\eqref{item-prop:algprl2} and \eqref{item-prop:algprl4} to the full subcategory inclusion $\ConnSpectra \to \Spectra$ in $\CAlg(\PrLc)$ and to the subcategory inclusion $\PrLcp \to \PrLc$ in $\CAlg(\PrL)$ (see \cref{prop:prlcpresentable}), we construct the following composite morphism in $\CAlg(\PrL)$:
\[
\PrLaddcp \simeq \Mod_{\ConnSpectra}(\PrLcp)  \xrightarrow{\Mod_{\ConnSpectra}\left( \PrLcp \to \PrLc\right)} \Mod_{\ConnSpectra}(\PrLc)  \xrightarrow{-\otimes_{\ConnSpectra} \Spectra} \Mod_{\Spectra}(\PrLc) \simeq \PrLstc
\]
\end{construction}
As $\ConnSpectra$ is an idempotent algebra in $\PrL$, the second functor $-\otimes_{\ConnSpectra} \Spectra$ here 
is equivalent to the composite 
\begin{equation}
\label{eq:relativetensorSp}
\Mod_{\ConnSpectra}(\PrLc)  \xrightarrow{\mathrm{forget}} \PrLc  \xrightarrow{-\otimes \Spectra} \Mod_{\Spectra}(\PrLc).
\end{equation}

Recall now the symmetric monoidal left adjoint functor from \cref{const:inductionSp} and the equivalences $ {\PresSigma}$ and $(-)^{\mrc}$ from \cref{prop:indst}. Then the following holds.

\begin{prop}\label{prop:Ksymmetric} The composite 
\begin{equation}
\label{eq:definition-fin}
(-)^{\fin} \colon \add \xrightarrow{\PresSigma} \PrLaddcp \xrightarrow{\mathrm{Const.}~\ref{const:inductionSp}} \PrLstc  \xrightarrow{(-)^{\mrc}} \st\end{equation}
defines a morphism in $\CAlg(\PrL)$ which is the left adjoint to the forgetful functor $\st \to \add$. \\
For $\cC \in \add$, the unit $\cC \to \cC^{\fin}$ of the adjunction is a fully faithful additive functor. 
\end{prop}
\begin{proof}
We show that the composite $(-)^{\fin}$ is indeed left adjoint to the forgetful functor. 
By construction, we have a commutative diagram in $\CAlg(\PrL)$
\[\begin{tikzcd} \PrLaddcp \arrow[rr, "-\otimes_{\ConnSpectra} \Spectra"]&& \PrLstc\\ \PrLcp \arrow[u, " - \otimes \ConnSpectra"] \arrow[rr] && \PrLc \arrow[u, "- \otimes \Spectra"]
\end{tikzcd},
\]
where the top horizontal morphism is the functor from \cref{const:inductionSp}, and the bottom horizontal functor is the subcategory inclusion (which is a symmetric monoidal left adjoint by \cref{prop:prlcpresentable}). 
Taking right adjoints, the middle square of functors in the following diagram commutes:
\[\begin{tikzcd} \eqmathbox[A]{\add} \arrow[d, hook] \ar[draw=none]{r}[sloped, auto=false]{\text{\scalebox{2}[1]{$\simeq$}}} & \eqmathbox[C]{\PrLaddcp}  \arrow[d, hook] &  \eqmathbox{\PrLstc} \arrow[d, hook] 
   \arrow[l] & \ar[draw=none]{l}[sloped, auto=false]{\text{\scalebox{2}[1]{$\simeq$}}} \eqmathbox[A]{\st} \arrow[d, hook] 
   \\
  \eqmathbox[A]{ \catprod} \ar[draw=none]{r}[sloped, auto=false]{\text{\scalebox{2}[1]{$\simeq$}}} &  \eqmathbox[C]{\PrLcp}  & \eqmathbox{\PrLc} \arrow[l] &\ar[draw=none]{l}[sloped, auto=false, shorten >=1.5ex]{\text{\scalebox{2}[1]{$\simeq$}}}  \eqmathbox[A]{\catrex} \end{tikzcd},
\]
The left and right square commute by \cref{prop:indst}. 
By \cref{prop:PrLcp}.\eqref{item-prop:PrLcp-3}, the bottom horizontal composite is  the forgetful functor $\catrex \to \catprod$. The top horizontal functor is the right adjoint to $(-)^{\fin}$ and hence agrees with the forgetful functor  $\st \to \add$.

We next prove fully faithfulness of the unit: Since both are left adjoints of the forgetful functor, the functor~\eqref{eq:definition-fin}  is equivalent to the functor $(-)^{\mathrm{fin}}\colon \add \to \st$ constructed in~\cite[Def. 2.1.17]{MR4504902}
   which sends an additive, idempotent-complete $\infty$-category $\cC$ to the smallest full stable subcategory of $\Fun^{\times}(\cC^\op, \Spectra)$ (the category of functors taking finite coproducts in $\cC$ to finite products in $\Spectra$) containing the image of the Yoneda embedding.  In~\cite[Cor. 2.1.5]{MR4504902}, it is shown that the inclusion $\cC \to \Fun^{\times}(\cC^{\op}, \Spectra)$ is fully faithful, and hence so is the inclusion $\cC \to \cC^{\fin}$. 
\end{proof}

\begin{remark} As used in the proof of \cref{prop:Ksymmetric}, the functor \eqref{eq:definition-fin}  is equivalent to the functor $(-)^{\mathrm{fin}}\colon \add \to \st$ constructed in~\cite[Def. 2.1.17]{MR4504902}
    taking an additive, idempotent-complete $\infty$-category $\cC$ to the stable, idempotent-complete $\infty$-category $\cC^{\mathrm{fin}}$ of \emph{finite cell $\cC$-modules}, explicitly defined to be the smallest full stable subcategory
    of $\Fun^{\times}(\cC^\op, \Spectra)$ (the category of functors taking finite coproducts in $\cC$ to finite products in $\Spectra$) containing the image of the Yoneda embedding. The inclusion $\cC \hookrightarrow \cC^{\fin}$ is induced by the Yoneda embedding. 
\end{remark}

\subsubsection{\texorpdfstring{$\Kb$}{Kb} as a left adjoint}
\label{sec:Kbasfreestable}

We now show that for an ordinary additive idempotent-complete $1$-category $\cA$, the universal stable $\infty$-category $\cA^{\mathrm{fin}}$ from \cref{prop:Ksymmetric} is equivalent to $\Kb(\cA)$.

Consider the symmetric monoidal left adjoint $(-)^{\mathrm{fin}} \colon \add \to \st$ of the forgetful functor from \cref{prop:Ksymmetric}.  
Categories in the image of $(-)^{\mathrm{fin}}$ carry so-called \emph{weight structures}, which were originally introduced independently by Bondarko in~\cite{MR2746283} and (under the name of co-t-structures) by Pauksztello in \cite{Pauk}, and afterwards adopted to the $\infty$-categorical setting by Elmanto and Sosnilo~\cite{MR4504902}, whose exposition we closely follow. 
\begin{remark}
For a representation theoretic point of view on (classical) weight structures in the context of Soergel bimodules we refer to \cite{MR4368481}. Soergel bimodules appear in there as Springer motives attached to Bott--Samelson resolutions of Schubert varieties in the full flag variety and form an additive idempotent complete coheart, see \cite[Ex.~2.2]{MR4368481} and compare with \cref{exm:chainweight}.  
\end{remark}

By \cite[Def.~2.2.1]{MR4504902} a \emph{weight structure} on an idempotent complete stable $\infty$-category $\cD$ is a pair $(\cD_{\leq 0}, \cD_{\geq 0})$ of two full idempotent complete subcategories fulfilling the following conditions:
    \begin{enumerate}
        \item $\Sigma \cD_{\geq 0} \subset \cD_{\geq 0}, \Sigma^{-1} \cD_{\leq 0} \subset \cD_{\leq 0}$. We write $\cD_{\geq n} = \Sigma^n \cD_{\geq 0} $, $\cD_{\leq n}  = \Sigma^n \cD_{\leq 0}$.
        \item For $x \in \cD_{\leq 0}$ and $y \in \cD_{\geq 1}$, we have
$
            \pi_0(\Hom_{\cD}(x,y)) \simeq 0.
$
        \item For any object $x$, there is a  fiber sequence          $  x_{\leq 0} \to x \to x_{\geq 1}$
        with $x_{\leq 0} \in \cD_{\leq 0}, x_{\geq 1} \in \cD_{\geq 1}$. 
    \end{enumerate}
    Note that condition (3) merely requires the existence of such a fiber
    sequence, neither is it unique nor functorially associated to $x$. A weight
    structure is called \emph{bounded} if  $\cD = \bigcup_{n}(\cD_{\geq -n} \cap \cD_{\leq n})$. For any weight structure, the
    \emph{weight heart} $\cD^{\heartsuit} \coloneqq \cD_{\geq 0} \cap \cD_{\leq 0}$ is
    additive and idempotent complete. 

 Just like t-structures, weight structures only depend on
 and may be constructed in terms of the underlying (triangulated) homotopy
 category of $\cD$.

    \begin{example}\label{exm:chainweight} Given an ordinary
    additive, idempotent complete $1$-category $\cA$, the $\infty$-category $\Kb(\cA) $
    has a canonical bounded weight structure with $\Kb(\cA)_{\geq 0}$ the (dg nerve on the) full
    subcategory of chain complexes supported in non-negative homological degrees. 
    The weight heart $\Kb(\cA)^\heartsuit \simeq \cA$ recovers the original additive
    $1$-category and its inclusion is the canonical functor $\cA \to \Kb(\cA)$. \end{example}
    
   A functor $F\colon\cC \to \cD$ between stable, idempotent-complete $\infty$-categories with weight structure is
   \emph{weight exact} if it is exact and the restriction of $F$ to the full subcategory $\cC_{\geq 0} \subseteq \cC$ factors through the full subcategory $\cD_{\geq 0} \subseteq \cD$  and the restriction of $F$ to $\cC_{\leq 0} \subseteq \cC$ factors through $\cD_{\leq 0} \subseteq \cD$. Weight
   exact functors $F \colon \cC \to \cD$ restrict to additive functors $F^{\heartsuit}\colon \cC^{\heartsuit} \to \cD^{\heartsuit}$ between the weight hearts.  

   \begin{notation}
   Let $\st^{bw}$ denote the $\infty$-category of idempotent complete stable categories equipped with bounded
   weight structures and weight exact functors. 
   \end{notation}
   A key result of~\cite{MR4504902} is that $(-)^{\mathrm{fin}}\colon\add \to \st$ factors as an equivalence $\add \to \st^{bw}$ followed by the functor $\st^{bw} \to \st$
which forgets the weight structure, see \cite[Const.~2.2.7]{MR4504902}. An inverse of this equivalence $ \add \to
\st^{bw}$ is given by the functor $(-)^{\heartsuit}\colon \st^{bw} \to \add$ taking
the weight heart, see \cite[Theorem 2.2.9]{MR4504902}. The following corollary is a consequence of this theorem:
\begin{corollary}
\label{cor:universalch} 
Let $\cA$ be an ordinary additive, idempotent-complete $1$-category. We have an equivalence of $\infty$-categories $\cA^{\mathrm{fin}} \simeq \Kb(\cA)$.
In particular, for any stable, idempotent-complete $\infty$-category $\cB$, the inclusion of degree-zero chain complexes $\cA \to \Kb(\cA)$ induces an equivalence 
\[\Funex(\Kb(\cA), \cB) \to \Funadd(\cA, \cB).\]
\end{corollary}

\begin{proof} By~\cite[Thm. 2.2.9]{MR4504902}, for any additive, idempotent complete $\infty$-category $\cA$, the $\infty$-category $\cA^{\mathrm{fin}}$ is uniquely characterized by being a stable, idempotent-complete $\infty$-category with bounded weight structure with weight heart $\cA$. 
Since for an ordinary, additive, idempotent-complete $1$-category $\cA$, the $\infty$-category  $\Kb(\cA)$ is a stable, idempotent-complete $\infty$-category with weight structure and weight heart $\cA$, see \cref{prop: Kbstable},  \cref{exm:chainweight}. The result follows. 
\end{proof}

Extending \cref{cor:universalch}, we think of $(-)^{\mathrm{fin}}:\add \to \st$
 as the correct 
generalization of $\Kb(\cA)$ from ordinary additive, idempotent-complete $1$-categories to additive, idempotent-complete $\infty$-categories $\cA$.

\begin{notation}\label{nota:K-as-fin}
    Abusing notation, we will henceforth write $\Kb(-) \coloneqq (-)^{\mathrm{fin}}\colon \add \to \st$ for the left adjoint to the forgetful functor $\st \to \add$, even when applied to additive $\infty$-categories.
\end{notation}

\subsection{\texorpdfstring{$\infty$}{infinity}-categories of graded modules}
\label{subsec:Z-graded-k-linear-modules}
\renewcommand\mod{\mathrm{mod}}

The categories appearing in this paper will not just be additive or stable, but will typically be enriched in chain complexes of $k$-modules for a commutative ring $k$, equipped with an additional $\mathbb{Z}$-grading.
In this section, we recall the necessary technical machinery to address this coherently. 
This machinery will apply more generally to  $\EE_{\infty}$-ring spectra, i.e. commutative algebra objects in $\Spectra$. In \S\ref{subsec:derived-inf-cat-graded-mod}, we relate these structures with possibly more familiar variants of derived categories.
Similar definitions are discussed in~\cite{SAG}.

\subsubsection{\texorpdfstring{$\mbbK$}{K}-modules}
\label{Kmod}
\begin{notation} For $k$ an ordinary commutative ring, we let $\mod_k$ denote the ordinary symmetric monoidal $1$-category of $k$-modules.
\end{notation}

We now discuss the $\infty$-categorical analog of $\mod_k$.  It follows from \cref{lem:Modpresentable}.\eqref{item-lem:Modpresentable-4} that for an $\EE_{\infty}$-ring spectrum $\mbbK \in \CAlg(\Spectra)$, the $\infty$-category $\Mod_{\mbbK}(\Spectra)$ is a compactly generated stable, presentably symmetric monoidal category, i.e. $\Mod_{\mbbK}(\Spectra) \in \CAlg(\PrLstc)$. 

\begin{nota}
\label{nota:modk}
 For $\mbbK \in \CAlg(\Spectra)$, we write $\Mod_{\mbbK}$ for the category of \emph{$\mbbK$-modules} $\Mod_{\mbbK}(\Spectra) \in \CAlg(\PrLstc)$ and $\Perf_{\mbbK}$ for the category of \emph{perfect $\mbbK$-modules} $\Perf_{\mbbK} \coloneqq \Mod_{\mbbK}(\Spectra)^{\mrc} \in \CAlg(\st)$. 
\end{nota}
The symmetric monoidal equivalence $\Ind\colon \st \to \PrLstc$ transports $\Perf_{\mbbK}$ to $\Mod_{\mbbK}$ and vice versa.

\begin{example}
\label{exm:kmod}
 The main application of this paper will only be concerned with
the case that $\mbbK = Hk$ is an Eilenberg-MacLane spectrum of a classical commutative ring
$k$. In this case, $\Mod_{\mbbK}$ is equivalent to the
unbounded derived $\infty$-category $\Derived(\mod_k)$ of the abelian category $\mathrm{mod}_k$ of
$k$-modules~\cite[Thm 7.1.2.13]{HA} with symmetric monoidal structure given by the derived tensor product $-\otimes^L_k-$.  The $\infty$-category $\Perf_{\mbbK}$ is equivalent to its full subcategory on the perfect chain complexes, i.e. the chain complexes quasi-isomorphic to a bounded complex of finitely generated projective $k$-modules.
  \end{example}
 Since \cref{exm:kmod} is the situation relevant to our paper, the reader can safely view $\mbbK$ as a classical ring $k$ and $\Mod_{\mbbK}$ as $\Derived(\mod_k)$. The situation of \cref{exm:kmod} will be discussed in more detail in \S\ref{subsec:derived-inf-cat-graded-mod}.

\subsubsection{Compact-projective \texorpdfstring{$\mbbK$}{K}-modules}
As above, we will be concerned with the additive variants of the notions in \S\ref{Kmod}.
Let $\mbbK$ be a connective $\EE_{\infty}$-ring spectrum, i.e. a commutative algebra $\mbbK \in \CAlg(\ConnSpectra)$. Since $\ConnSpectra\in \CAlg(\PrLaddcp)$, it follows from \cref{lem:Modpresentable} that the category $\Mod_{\mbbK}(\ConnSpectra)\in \CAlg(\PrLaddcp)$.

\begin{nota} Fix $\mbbK\in \CAlg(\ConnSpectra)$, we  write $\Mod^{\geq 0}_{\mbbK}$ for the $\infty$-category of \emph{connective $\mbbK$-modules} $\Mod_{\mbbK}(\ConnSpectra) \in \CAlg(\PrLaddcp)$ and $\CProj_{\mbbK}$ for the category of \emph{compact-projective $\mbbK$-modules} $\CProj_{\mbbK} \coloneqq \Mod_{\mbbK}(\ConnSpectra)^{\cp} \in \CAlg(\add)$. 
\end{nota}
Note that $\Mod^{\geq 0}_{\mbbK}$ is a full subcategory of $\Mod_{\mbbK}$.

\begin{example}\label{exm:connected-modk} For $\mbbK = Hk$ an Eilenberg-MacLane spectrum of a classical commutative ring $k$, the $\infty$-category $\Mod^{\geq 0}_{Hk}$ is equivalent to the full subcategory $\Derived(\mod_k)_{\geq 0}$ of the unbounded derived $\infty$-category $\Derived(\mod_k)$ of the ring $k$ on those chain complexes with homology in non-negative homological degree.
It follows from \cref{lem:CProjR-finite-coproduct} below that the full subcategory $\CProj_{Hk}$ is equivalent to the $1$-category of finitely generated projective $k$-modules in the usual sense (with fully faithful inclusion into $\Derived(\mod_k)_{\geq 0}$ as complexes concentrated in degree zero), see also \S\ref{subsec:derived-inf-cat-graded-mod}.
\end{example}

\begin{observation}
The symmetric monoidal equivalence $\Ind\colon \st \to \PrLstc$ transports $\Perf_{\mbbK}$ to $\Mod_{\mbbK}$. Similarly, the symmetric monoidal equivalence $\PresSigma\colon \add \to \PrLaddcp$ transports $\CProj_{\mbbK}$ to $\Mod^{\geq 0}_{\mbbK}$.
\end{observation}

\begin{lemma}\label{lem:CProjR-finite-coproduct}
    The following hold:
    \begin{enumerate}
    \item \label{item-lem:CProjR-finite-coproduct}
    Let $\mbbK \in \CAlg(\ConnSpectra)$. The rank one free module $\mbbK_{\mbbK}$ generates $\CProj_{\mbbK}$ under retracts and finite direct sums; in particular, every object of $\CProj_{\mbbK}$ is a retract of a finite coproduct of modules isomorphic to $\mbbK_{\mbbK}$. 
    \item \label{item-lem:PerfR-finite-colim}
    Let $\mbbK \in \CAlg(\Spectra)$. The rank one free module $\mbbK_{\mbbK}$ generates $\Perf_{\mbbK}$ under retracts and finite colimits; in particular, every object of $\Perf_{\mbbK}$ is a retract of an iterated finite colimit of modules isomorphic to $\mbbK_{\mbbK}$.
    \end{enumerate}
    \end{lemma}
    \begin{proof}
    Immediate from \cref{lem:generation} and the fact that $\Mod_{\mbbK}$ and $\Mod^{\geq 0}_{\mbbK}$ are compact and compact projectively generated by $\mbbK_{\mbbK}$ respectively. 
    \end{proof}

    If $k$ is an ordinary ring, then an object of $\Perf_{Hk}$ can be represented by a bounded chain complex of finitely generated projective $k$-modules. This generalizes to any $\mbbK \in \CAlg(\ConnSpectra)$: 
    \begin{prop} \label{prop:KProj}
        For $\mbbK \in \CAlg(\ConnSpectra)$ there is a symmetric monoidal equivalence 
        $$
        \Kb(\CProj_{\mbbK}) \simeq \Perf_{\mbbK}.
        $$
        \end{prop}
        \begin{proof} Starting with the definition of $\Kb=(-)^{\fin}$ in \cref{prop:Ksymmetric}, we obtain the equivalence  
        \[
        \Kb(\CProj_{\mbbK}) := \left( \PresSigma(\CProj_{\mbbK})\otimes_{\ConnSpectra}\Spectra\right)^{c} \simeq \left( \Mod_{\mbbK}(\ConnSpectra) \otimes_{\ConnSpectra}\Spectra\right)^{c}\simeq \left( \Mod_{\mbbK}(\Spectra)\right)^{c} =: \Perf_{\mbbK}
        \]
        where the last step follows from \cref{prop:algprl-new}.\eqref{item-prop:algprl5}.
        \end{proof}

\subsubsection{$\Monoid$-graded \texorpdfstring{$\mbbK$}{K}-modules}
Given an ordinary monoid $Z$ and a commutative ring $k$, the category $\Fun(Z, \mod_k)$ of $Z$-graded $k$-modules admits a convolution monoidal structure, for which the tensor product of $Z$-graded modules $(M_z)_{z \in Z}$ and $(N_z)_{z\in Z}$ is given by the $Z$-graded module which in degree $z\in Z$ is $\oplus_{z_1z_2 = z} M_{z_1} \otimes N_{z_2}$. 
This construction is a special case of the Day convolution monoidal structure on a functor category \cite[\S~2.2.6]{HA}. Here, we focus on the symmetric monoidal case.

We briefly recall this construction of a symmetric monoidal structure on $\Fun(J, \cC)$  in the case where  $J$ is a small symmetric monoidal $\infty$-category and $\cC $ is a presentably symmetric monoidal $\infty$-category.

\begin{lemma}\label{lem:funcat} 
For $J\in \cat$ and $\cC\in \PrL$, the functor $\cC \times J  \to \Fun(J^{\op}, \cC)$, 
\begin{equation}\label{eq:skyscraper}
(c,j)\mapsto c \otimes \Hom_{J}(-, j)  \in \Fun(J^{\op}, \cC)
\end{equation}
(where $\otimes$ denotes the action of $\Spaces$ on $\cC$ inherited from the presentability of $\cC$)
induces an equivalence 
\begin{equation}\label{eq:tensorproductinPrL}
\cC\otimes \Pres(J)   \simeq \Fun(J^{\op}, \cC)
\end{equation}
in $\PrL$ (where $\otimes$ denotes the tensor product of $\PrL$).
\end{lemma}
\begin{proof}
Consider the chain of equivalences 
\[\cC \otimes \Pres(J) \simeq \FunL(\Pres(J), \cC^{\op})^{\op} \simeq \Fun(J, \cC^{\op})^{\op} \simeq \Fun(J^{\op}, \cC)\]
Here, the first equivalence follows from~\eqref{eq:prtensor}, the second equivalence is the universal property of the Yoneda embedding~\cite[Thm. 5.1.5.6]{HTT}, and the last  records the interplay between functor categories and opposites. Precomposing this equivalence with the inclusion functor $ \cC \times J \to \cC \otimes \Pres(J)$ (which is cocontinuous in its second argument) unpacks to the functor~\eqref{eq:skyscraper}.
\end{proof}

Assume $J\in \CAlg(\cat)$ and $\cC \in \CAlg(\PrL)$. Since $\Pres\colon \cat \to \PrL$ is symmetric monoidal by \cref{prop:prlsym}, it follows that for $J \in \CAlg(\cat)$, the $\infty$-category $\Pres(J)$ inherits a presentably symmetric monoidal structure, i.e. $\Pres(J) \in \CAlg(\PrL)$. Then  \eqref{eq:tensorproductinPrL} provides the following \emph{Day convolution monoidal structure} on $\Fun(J^{\op}, \cC)$.
\begin{corollary}\label{cons:convolution}
    Let $J\in \CAlg(\cat)$ and $\cC \in \CAlg(\PrL)$. Then $\Fun(J^{\op}, \cC)$ inherits 
    a presentably symmetric monoidal structure from the tensor product $\cC \otimes \Pres(J)$ of commutative algebras  in $\PrL$. 
\end{corollary}

\begin{remark}
By~\cite[Prop.~3.10]{moshe2021higher}, this construction agrees with the Day convolution structure on functor categories, as e.g. defined in~\cite[Rem. 2.2.6.8]{HA}, also see~\cite[Thm. 3.1]{moshe2021higher}. 
Explicitly, the tensor product of functors $F\colon J^{\op} \to \cC$ and $G \colon J^{\op} \to \cC$ is given by the left Kan extension of the functor $J^{\op} \times J^{\op}  \xrightarrow{F\otimes G}\cC$ along the tensor product $J^{\op} \times J^{\op} \to J^{\op}$. 
\end{remark}

\begin{lemma} If $J \in \CAlg(\cat)$ and $\cC$ is in $\CAlg(\PrLc)$ or $\CAlg(\PrLcp)$, then $\Fun(J^{\op}, \cC)$ with its Day convolution monoidal structure is also in $\CAlg(\PrLc)$ or $\CAlg(\PrLcp)$, respectively.
\end{lemma}
\begin{proof}
The Day convolution monoidal structure was defined by identifying $\Fun(J^{\op}, \cC)$ with $ \cC \otimes \Pres(J) $. The presheaf category $\Pres(J)$ is an object of $\CAlg(\PrLcp)$ (in fact, it is generated by a small set of objects which commute with all small colimits). Hence, if $\cC$ is in  $\CAlg(\PrLc)$ or in the subcategory $\CAlg(\PrLcp)$, then so is $\cC \otimes \Pres(J)$. 
\end{proof}

\begin{observation}\label{obs:adjoint-of-unit}
The monoidal unit $I \in J$ of any symmetric monoidal $\infty$-category $J\in \CAlg(\cat)$ induces a symmetric monoidal functor $\Spaces \to \Pres(J)$ left adjoint to the evaluation functor $\mathrm{ev}_{I} \colon  \Pres(J) \to \Spaces$, and explicitly given by sending a space $X$ to the functor $\Hom_{J}(-, I) \times X\colon  J^{\op} \to \Spaces$. 
It follows that for any presentably symmetric monoidal category $\cC$, there is a symmetric monoidal left adjoint 
\[\cC \simeq \cC \otimes \Spaces \to \cC \otimes\Pres(J) \simeq \Fun(J^{\op}, \cC)
\] to the evaluation functor $\ev_{I}\colon\Fun(J^{\op}, \cC) \to \cC$, explicitly given by sending $c\in \cC$ to the functor $\Hom_{J}(-, I) \otimes c \colon J^{\op} \to \cC$. 
\end{observation}

We will particularly focus on gradings by  a homotopy coherent abelian monoid, i.e. a $\Monoid\in \CAlg(\Spaces)$. 
\begin{definition} 
Let $\Monoid \in \CAlg(\Spaces)$ and recall Day convolution from  \cref{cons:convolution}.
\begin{enumerate}
\item For  $\mbbK \in \CAlg(\ConnSpectra)$, we define the $\infty$-category of \emph{$\Monoid$-graded connective $\mbbK$-modules} $\Mod_{\mbbK}^{\geq 0, \Monoid} \in \CAlg(\PrLcp)$ as the functor category $\Fun(\Monoid, \Mod^{\geq 0}_{\mbbK})$ with the Day convolution structure.
\item 
For  $\mbbK \in \CAlg(\Spectra)$, we define the $\infty$-category of \emph{$\Monoid$-graded $\mbbK$-modules} $\Mod_{\mbbK}^{\Monoid} \in \CAlg(\PrLc)$ to be the functor category $\Fun(\Monoid, \Mod_{\mbbK})$ with the Day convolution structure. 
 \end{enumerate} 
\end{definition}

\begin{example}\label{exm:gradedderived}
Following \cref{exm:kmod}, if $\Monoid$ is a discrete (i.e. ordinary) commutative monoid $Z$ and $\mbbK= Hk$ the Eilenberg-MacLane spectrum of an ordinary commutative ring $k$, the $\infty$-category $\Mod_{\mbbK}^{\Monoid}$ is the unbounded derived $\infty$-category $\Derived(\mod_k^Z)$ of the ordinary abelian $1$-category $\mathrm{mod}_k^{Z}\coloneqq \Fun(Z, \mod_k)$ of $Z$-graded $k$-modules.  This will be discussed in more detail in \cref{subsec:derived-inf-cat-graded-mod}.

Unpacking Day convolution from \cref{cons:convolution} in these terms, the tensor product of an ordinary $k$-module $M$ concentrated in degree $z\in Z$ and an ordinary $k$-module $N$ concentrated in degree $w\in Z$ is given by the derived tensor product $M\otimes_k^L N$  concentrated in degree $z+w\in Z$.
\end{example}
\begin{example}
Still in the setup of \cref{exm:gradedderived}, the $\infty$-categories  $\left(\Mod_{Hk}^{\geq 0, Z}\right)^{\cp}$ and $\left(\Mod_{Hk}^{Z}\right)^{\mathrm{c}}$ may be identified with the full subcategories $\Fun^{\mathrm{fin.supp.}}(Z, \CProj_{k})$ and $\Fun^{\mathrm{fin.supp.}}(Z, \Perf_{k})$ of the functor  $\infty$-categories $\Fun(Z, \CProj_{k})$ and $\Fun(Z, \Perf_{k})$, respectively, on the \emph{finitely supported} functors, i.e. functors that vanish on all but finitely many elements of $Z$.
\end{example}

\begin{observation}\label{obs:connMod-to-Mod-symmetric}
Assume $\Monoid \in \CAlg(\Spaces)$ and $\mbbK \in \CAlg(\ConnSpectra)$. The symmetric monoidal functor $- \otimes \Spectra \colon \PrLaddcp \to \PrLstc$ from \cref{const:inductionSp} takes $\Mod_{\mbbK}^{\geq 0, \Monoid}$ with its Day convolution monoidal structure to $\Mod_{\mbbK}^{\Monoid}$ with its Day convolution monoidal structure. Indeed, we have the following sequence\[        \Mod_{\mbbK}^{\geq 0, \Monoid} \otimes \Spectra \simeq \cP(\Monoid) \otimes \Mod_{\mbbK}^{\geq 0} \otimes \Spectra \simeq \cP(\Monoid) \otimes \Mod_{\mbbK} \simeq \Mod_{\mbbK}^{\Monoid}
\]
 of symmetric monoidal equivalences. In particular, it follows that the fully faithful inclusion $\Mod_{\mbbK}^{\geq 0, \Monoid} \hookrightarrow \Mod_{\mbbK}^{\Monoid}$ is symmetric monoidal and hence a morphism in $\CAlg(\PrL)$.
\end{observation}

\subsection{Derived \texorpdfstring{$\infty$}{infinity}-categories of graded modules}\label{subsec:derived-inf-cat-graded-mod}
Many of the constructions of \cref{sec:2} center around discrete (i.e. ordinary) graded $k$-algebras, and derived graded bimodules between them. 
In this section, we  therefore focus on the case where $\mbbK$ is a discrete commutative ring $k$ and $\Monoid$ is a discrete commutative monoid $Z$ and  unpack our constructions in terms of homological algebra, generalizing Examples~\ref{exm:kmod},~\ref{exm:connected-modk} and \ref{exm:gradedderived}.

\subsubsection{Derived \texorpdfstring{$\infty$}{infinity}-categories}

We quickly review the basics of the theory of derived $\infty$-categories; we refer the reader to ~\cite[\S~1.3]{HA} for more details.

Given an abelian $1$-category $\cA$, its \emph{(unbounded) derived $\infty$-category} $\Derived(\cA)$ is  the $\infty$-categorical localization of the $\infty$-category of unbounded chain complexes in $\cA$ (constructed as the dg nerve~\cite[\S~1.3.1]{HA} of the corresponding differential graded category) at the quasi-isomorphisms. In particular, the homotopy 1-category $h_1\Derived(\cA)$ agrees with the ordinary derived 1-category of $\cA$ in the usual sense. 

Let $\Derived(\cA)_{\geq 0}$ denote the full subcategory of $\Derived(\cA)$ on the chain complexes with vanishing homology in negative degrees. When $\cA$ is particularly well-behaved, the $\infty$-categories $\Derived(\cA)_{\geq 0}$ and $\Derived(\cA)$ can be expressed in terms of completions (of the type introduced throughout \cref{sec:stableLA}), as we discuss now.

Recall the following classical analogues of \cref{def:proj-and-cmpt-proj}:
\begin{definition}\label{def:compact-1-generation}
Let $c$ be an object in an ordinary $1$-category $\cC$ with small colimits. Then, $c$ is called
\begin{enumerate}
\item  \emph{compact}, if $\Hom_{\cC}(c, -) \colon \cC \to \Set$ preserves filtered colimits;
\item \emph{$1$-projective}, if $\Hom_{\cC}(c,-)  \colon \cC \to \Set$ preserves geometric realizations (equivalently, reflective coequalizers); 
\item \emph{compact $1$-projective} if $\Hom_{\cC}(c,-)  \colon \cC \to \Set$ preserves sifted colimits, or equivalently if $c$ is compact and $1$-projective. 
\end{enumerate}

We say that $\cC$ is \emph{compactly generated} (resp. \emph{$1$-projectively generated}) if there is a small set of compact (resp. compact 1-projective) objects which generate $\cC$ under small colimits.  We denote the full subcategory of compact, resp. compact $1$-projective, objects in $\cC$ by $\cC^{\mrc}$, resp. $\cC^{\conep}$.
\end{definition}

\begin{example}\label{exm:compact-1-proj-abelian1}
If $\cA$ is an abelian $1$-category, an object $c \in \cA$ is $1$-projective if and only if it is projective in the usual sense.
\end{example}
\begin{example}\label{exm:compact-1-proj-abelian}
A presentable abelian category $\cA$ is $1$-projectively generated if it is compactly generated and if the full subcategory of compact objects $\cA^{\mrc}$ has enough projective objects, i.e. if for every compact object $a\in \cA$ there exists a compact $1$-projective object $p$ and an epimorphism $p \twoheadrightarrow a$. In particular, this implies that also $\cA$ has enough projective objects, i.e. that for \emph{every} object $a\in \cA$ there exists a $1$-projective $p$ and an epimorphism $p \twoheadrightarrow a$.

For example, the abelian category $\mod_k$ is a $1$-projectively generated presentable $1$-category with $\mod_k^{\mrc}$ the full subcategory of finitely generated modules and  $\mod_k^{\conep}$ the full subcategory of finitely generated projective $k$-modules. 
\end{example}

\begin{remark}\label{rk:1-projective}
 Because $\Set \to \Spaces$ preserves filtered colimits, an object in an ordinary $1$-category $\cC$ is compact in the sense of \cref{def:compact-1-generation} if and only if it is compact in the sense of \S~\ref{subsec:compact-gen-and-ind}  when $\cC$ is considered as an $\infty$-category. 
 \end{remark}
 \begin{warning}\label{warn:1-projective}
 \cref{rk:1-projective}  not true projectivity:  The condition for an object $c\in \cC$ to be  $1$-projective (i.e. $\Hom_{\cC}(c,-)\colon \cC \to \mathrm{Set}$ preserving geometric realizations) is \emph{different} to the condition for it to be projective (i.e. $\Hom_{\cC}(c,-)\colon \cC \to \mathrm{Set} \to \Spaces$ preserving geometric realizations), simply because the inclusion $\mathrm{Set} \hookrightarrow \Spaces$ does not preserve geometric realizations. This difference is at the heart of the process of \emph{animation} \cite[\S~5.1.4]{cesnavicius2023purity}, which takes an ordinary cocomplete category $\cC$ to $\PresSigma(\cC^{\conep})$, i.e. freely making the compact 1-projective objects into compact-projective objects.
 \end{warning}

The following statements are well-known and can be gathered from various parts of \cite[\S~1.3]{HA}:
\begin{prop}\label{prop:derived-infty}
    Let $\cA$ be a $1$-projectively generated presentable abelian $1$-category. 
    \begin{enumerate}
        \item \label{enum-item:derived-infty-1}
        The additive presentable $\infty$-category $\Derived(\cA)_{\geq 0}$ is equivalent to $\PresSigma(\cA^{\conep})$.       
         \item \label{enum-item:derived-infty-2}
        The stable presentable $\infty$-category $\Derived(\cA)$ is equivalent to its stabilization \[\PresSigma(\cA^{\conep}) \otimes \Spectra \simeq \Ind \Kb(\cA^{\conep}).\]
    \end{enumerate} 
\end{prop}
\begin{proof} 
For the first statement, note that $\cA$ has enough projective objects (see \cref{exm:compact-1-proj-abelian}) and let  $\Derived_-(\cA)$ be the dg-nerve of the differential graded category of bounded-below chain complexes of $1$-projective objects (i.e. projective objects in the standard abelian sense). Let $\Derived_{-}(\cA)_{\geq 0}$ be the full subcategory on the chain complexes with vanishing homology in negative degrees. Entirely analogous\footnote{\cite[Prop.~1.3.3.14]{HA} does not apply directly since the full subcategory  $\cA^{\mrc}$ of compact objects in a presentable abelian category $\cA$ is not necessarily itself abelian.} to the proof of \cite[Prop.~1.3.3.14]{HA}, the Dold-Kan correspondence shows that $\Derived_{-}(\cA)_{\geq 0} \simeq \PresSigma(\cA^{\conep})$. 
Since any $1$-projectively generated presentable abelian $1$-category is Grothendieck abelian~ \cite[Def.~1.3.5.1]{HA}, it follows from~\cite[Prop.~1.3.5.24, Def.~1.3.5.8, Prop.~1.3.5.13]{HA} that there is a fully faithful embedding $\Derived_{-}(\cA) \to \Derived(\cA)$ with image  the chain complexes with bounded-below homology. In particular, this embedding identifies $\Derived_{-}(\cA)_{\geq 0}$ with $\Derived(\cA)_{\geq 0}$. 

For the second statement, since the $t$-structure $(\Derived(\cA)_{\leq 0}, \Derived(\cA)_{\geq 0})$ on $\Derived(\cA)$ is right-complete \cite[Prop.~1.3.5.21]{HA}, it follows that $\Derived(\cA)$ is the stabilization of $\Derived(\cA)_{\geq 0}$; since $\Derived(\cA)_{\geq 0} = \PresSigma(\cA^{\conep})$ is presentable this stabilization is given by tensoring with $\Spectra$ by \cite[Ex.~4.8.1.23]{HA}. The equivalence $\PresSigma(\cA^{\conep}) \otimes \Spectra \simeq \Ind \Kb(\cA^{\conep})$ follows then from the definition of  $(-)^{\mathrm{fin}}$ in \cref{prop:Ksymmetric} and its equivalence with $\Kb$ from \cref{cor:universalch}.
\end{proof}

\subsubsection{Derived \texorpdfstring{$\infty$}{infinity}-categories of graded modules}
We return to the main goal of this subsection to give a homological perspective on the constructions of the last sections. Let $\mbbK$ be a discrete commutative ring $k$ and $\Monoid$ a discrete commutative monoid $Z$. Recall the notation $\modkZ$ for the ordinary category of $Z$-graded $k$-modules. Throughout this subsection, we also fix an ordinary (not necessarily commutative) $Z$-graded $k$-algebra $A \in \Alg(\modkZ)$.

\begin{notation}We let $\grmod_A \coloneqq \RMod_A(\modkZ)$ denote the ordinary $1$-category of $Z$-graded right $A$-modules. 
\end{notation}
This category $\mathrm{grmod}_A$ is a $1$-projectively generated, in the sense of \cref{def:compact-1-generation}, presentable abelian $1$-category. A standard computation shows that its compact $1$-projective objects (i.e. its compact projective objects in the usual abelian sense) are precisely given by the graded-compact projective modules, defined as follows.

\begin{definition}\label{def:graded-compact-perfect}
     An (ordinary) $Z$-graded $A$-module $M \in \grmod_A$ is \emph{graded-compact-projective} if it is a retract of a finite direct sums of grading shifts of the free module $A$.
     Let $\grmod_A^{\mathrm{gr-cp}} \subset \grmod_A$ denote the full subcategory on the graded-compact-projective $A$-modules. 
     \end{definition}
In the notation of \cref{def:compact-1-generation}, $\grmod_A^{\mathrm{gr-cp}} = \left(\grmod_A\right)^{\conep}$.

Using \cref{prop:derived-infty}, we can identify the $\infty$-category $\RMod_{HA}(\Mod_{Hk}^{\geq 0, Z})$ as well as its various subcategories in terms of homological algebra:
\begin{prop} \label{prop:gradedmodules}
    Let $Z$ be a discrete monoid, $k$ a discrete commutative ring, and $A$ a discrete $Z$-graded (not necessarily commutative) $k$-algebra. 
    \begin{enumerate}
    \item \label{enum-prop:gradedmod-1} The $\infty$-category $\left(\RMod_{HA}(\Mod_{Hk}^{\geq 0, Z}) \right)^{cp}$ is equivalent to $\mathrm{grmod}_A^{\mathrm{gr-cp}}$. In particular, it is a $1$-category.
    \item \label{enum-prop:gradedmod-2} 
    The $\infty$-category $\left( \RMod_{HA}(\Mod_{Hk}^Z) \right)^{c}$ is equivalent to  the $\infty$-category $\Kb(\mathrm{grmod}_A^{\mathrm{gr-cp}})$.
    \item \label{enum-prop:gradedmod-3}  The $\infty$-category $\RMod_{HA}(\Mod_{Hk}^{\geq 0, Z})$ is equivalent to the $\infty$-category $\Derived(\grmod_A)_{\geq 0}$.
    \item \label{enum-prop:gradedmod-4} The $\infty$-category $\RMod_{HA}(\Mod_{Hk}^{Z})$ is equivalent to the (unbounded) derived $\infty$-category $\Derived(\mathrm{grmod}_A)$.
    \end{enumerate}
    \end{prop}
    \begin{proof} 
    The $\infty$-category $\Mod_{Hk}^{\geq 0, Z}  = \Fun(Z, \Mod_{Hk}^{\geq 0})$ is generated by the set of compact $1$-projective objects $Hk[z]$ for $z\in Z$, i.e. the ground ring $k$ in homological degree zero, and grading-degree $z \in Z$. Hence, by \cref{lem:Modpresentable}.\eqref{item-lem:Modpresentable-3}, $\RMod_{HA}(\Mod_{Hk}^{\geq 0, Z})$ is generated by shifted-free modules $HA[z] = HA \otimes_{Hk} Hk[z]$ for $z\in Z$. By \cref{lem:generation}.\eqref{item-lem:generation-cp}, the compact-projective objects of $\RMod_{HA}(\Mod_{Hk}^{\geq 0, Z})$ are retracts of finite direct sums of such modules, and hence are precisely the graded-compact-projective modules. This proves \eqref{enum-prop:gradedmod-1}.
    
    For \eqref{enum-prop:gradedmod-3}, note that  $\RMod_{HA}(\Mod_{Hk}^{\geq 0, Z})$ is projectively generated (see \cref{lem:Modpresentable}), and hence equivalent to \[\PresSigma\left(\RMod_{HA}(\Mod_{Hk}^{\geq 0, Z})^{\cp}\right) = \PresSigma(\grmod_{A}^{\mathrm{gr-cp}}).\] Since $\grmod_A^{\mathrm{gr-cp}}$ is the full subcategory on the compact 1-projectives in the $1$-projectively generated presentable abelian category  $\grmod_A$, it follows from \cref{prop:derived-infty}.\eqref{enum-item:derived-infty-1} that this is equivalent to $\Derived(\grmod_A)_{\geq 0}$.
    
    Statement~\eqref{enum-prop:gradedmod-4} follows from \cref{prop:derived-infty}.\eqref{enum-item:derived-infty-2} since by \cite[Thm. 4.8.4.6]{HA}, $\RMod_{HA}(\Mod_{Hk}^{\geq 0, Z})  \otimes \Spectra \simeq \RMod_{HA}(\Mod_{Hk}^{\geq 0, Z} \otimes \Spectra) \simeq \RMod_{HA} (\Mod_{Hk}^Z).$
        
Statement~\eqref{enum-prop:gradedmod-2} then  follows since $\Derived(\grmod_A) \simeq \Ind(\Kb(\grmod_A^{\mathrm{gr-cp}}))$ by \cref{prop:derived-infty}.\eqref{enum-item:derived-infty-2}. 
 \end{proof}

Motivated by \cref{prop:gradedmodules}, we call the objects in the full subcategory $\Derived(\grmod_A)^{\mrc} \subseteq \Derived(\grmod_A)$ \emph{graded-perfect}.
\begin{remark}\label{rk:graded-perfect}
     Since  $\Derived(\grmod_A)^{\mrc}  \simeq \Kb(\grmod_A^{\mathrm{gr-cp}})$ an object is graded-perfect if it is quasi-isomorphic to a bounded (in either direction) chain complex of graded-compact-projective $A$-modules. 
\end{remark}
\begin{nota}\label{nota:graded-perfect}
We write $\Derived(\grmod_A)^{\mathrm{gr-perf}}\coloneqq \Derived(\grmod_A)^{\mrc}$ for the full subcategory of $\Derived(\grmod_A)$ on the graded-perfect modules.  
\end{nota}

\section{Graded-linear \texorpdfstring{$\infty$}{infinity}-categories and Morita theory}\label{sec:morita-categories}

The goal of this section is twofold: in the first half of this section we  introduce the relevant notions of graded and linear $\infty$-categories, and prove an $\infty$-categorical version of the familiar equivalence between categories enriched in graded modules and categories with an action. We have seen a concrete  $1$-categorical instance already in form of the categories $\overline{\BSbimcl}^{\mathrm{gr}}_n$ and $\BSbimcl_n$ in \cref{sec:2}. 
In the second half of this section, we introduce the Morita categories relevant for the construction of our monoidal $(2,2)$-category $\SBim$.

\subsection{Presentably enriched \texorpdfstring{$\infty$}{infinity}-categories}\label{subsec:presentable-enriched-infty-cats}
\subsubsection{Closed monoidal \texorpdfstring{$\infty$}{infinity}-categories and closed module \texorpdfstring{$\infty$}{infinity}-categories}

\begin{definition}[{\cite[Def. 4.2.1.28]{HA}}]\label{def:morphism-objects}
Let $\VV$ be a (possibly large) monoidal $\infty$-category, and $\cC$ a left $\VV$-module $\infty$-category. A \emph{morphism object} between objects $x,y \in \cC$ is an object $\eHom_{\cC}(x,y) \in \VV$ representing the presheaf $\Hom_{\cC}(- \otimes x, y) \colon \VV^{\op} \to \Spaces$, i.e. equipped with isomorphisms natural in $v\in \VV$
\[
    \Hom_{\VV}(v, \eHom_{\cC}(x,y)) \simeq \Hom_{\cC}(v \otimes x, y).
\]
A $\VV$-module category $\cC$ is \emph{closed} if a morphism object exists between every pair of objects $x, y \in \cC$. A \emph{closed monoidal $\infty$-category} is a monoidal $\infty$-category whose left action on itself is closed. 
\end{definition}

\begin{observation}
\label{prop:adjunction-and-hom} 
    Let $F \colon  \VV \to \WW$ be a monoidal functor from a monoidal $\infty$-category to a closed monoidal $\infty$-category $\WW$ which is left adjoint to a functor $G$. Then, the induced $\VV$-action on $\WW$ is closed with morphism object $G \eHom_{\WW}(w, w') \in \VV$ for $w, w' \in \WW$.
    If $\VV$ is also closed monoidal, then for $v, v' \in \VV$, the map of spaces 
$\Hom_{\VV}(v,v') \to \Hom_{\WW}(Fv, Fv')
$
lifts\footnote{Explicitly,  under the equivalence $\Hom_{\VV}(\eHom_{\VV}(v,v'), G\eHom_{\WW}(Fv, Fv')) \simeq \Hom_{\WW}(F\eHom_{\VV}(v,v'), \eHom_{\WW}(Fv, Fv')) \simeq \Hom_{\WW}(F\eHom_{\VV}(v,v') \otimes Fv, Fv') \simeq\Hom_{\WW}(F(\eHom_{\VV}(v,v')  \otimes v), Fv')$, the morphism becomes the $F$- image of the counit $\eHom_{\VV}(v, v') \otimes v \to v'$.} along $\Hom_{\VV}(I, -)\colon \VV \to \Spaces$ to a $\VV$-morphism
\[
    \eHom_{\VV}(v,v') \to G \eHom_{\WW}(Fv, Fv').
\]

\end{observation}

\begin{example}\label{exm:presentable-are-closed}
    Let $\VV \in \Alg(\PrL)$ and $\cC \in \LMod_{\VV}(\PrL)$, i.e. $\cC$ is a presentable $\infty$-category with an action $- \otimes -\colon \VV\times \cC \to \cC$ by a presentable monoidal $\infty$-category $\VV$ which is cocontinuous in both variables. It follows from the adjoint functor theorem, \cref{adjfuncthm}, that $\Hom_{\cC}(- \otimes x, y)\colon \VV^{\op} \to \Spaces$ is representable for all $x, y \in \cC$, i.e. that the $\VV$-module category $\cC$ is closed. 
    In particular, any presentably monoidal $\infty$-category is closed monoidal. 
   \end{example}

\begin{example}\label{exm:internalhomcatk}
Let $\cK$ be a set of simplicial sets and recall from  $\cat^{\cK}$ the presentably symmetric monoidal $\infty$-category of $\infty$-categories with $\cK$-colimits and $\cK$-colimit preserving functors. 
For $\cC, \cD \in \cat^{\cK}$, the full subcategory $\Fun^{\cK}(\cC, \cD)$ of $\Fun(\cC, \cD)$ on the $\cK$-colimit preserving functors is closed under $\cK$-colimits~\cite[Rem.~4.8.4.14]{HA} and hence is an object of $\cat^{\cK}$. It follows directly from the characterization of the tensor product in $\cat^{\cK}$, see \cref{prop:adjoiningcolims}, that $\Fun^{\cK}(\cC, \cD) \in \cat^{\cK}$ is the morphism object between $\cC$ and $\cD$ in $\cat^{\cK}$ (cf. proof of~\cite[Lem.~4.8.4.2]{HA}). 
\end{example}

We generalize \cref{exm:internalhomcatk} to module categories using the following terminology.

\begin{notation}\label{nota:catVk}
Let $\cK$ be a small set of simplicial sets, $\VV \in \Alg(\cat^{\cK})$  and $\cC, \cD \in \LMod_{\VV}(\cat^{\cK})$. Let $\Fun_{\VV}(\cC, \cD)$ be the $\infty$-category of $\VV$-module functors~\cite[Def. 4.6.2.7]{HA} and $\Fun^{\cK}_{\VV}(\cC, \cD) \subset \Fun_{\VV}(\cC, \cD)$ the full subcategory on those module functors whose underlying functors preserve $\cK$-colimits. 
\end{notation}

By~\cite[Rem. 4.8.4.14]{HA}, $\Fun^{\cK}_{\VV}(\cC, \cD)$ is closed under $\cK$-colimits, thus an object of $\cat^{\cK}$. 

\begin{lemma} \label{lem:morphism-object-and-restriction}
Let $\cK$ be a small set of simplicial sets and let $\VV \in \Alg(\cat^{\cK})$. Cconsider the  right action of $\cat^{\cK}$ on $\LMod_{\VV}(\cat^{\cK})$. Let $\cC, \cD \in \LMod_{\VV}(\cat^{\cK})$. Then, the following hold.
\begin{enumerate}
\item\label{item-lem:morphism-object-1}
  $\Fun^{\cK}_{\VV}(\cC, \cD)$ is a morphism object in $\cat^{\cK}$ between $\cC, \cD \in \LMod_{\VV}(\cat^{\cK})$.
\item \label{item-lem:morphism-object-2} If $\VV$ is furthermore symmetric monoidal, then $\Fun^{\cK}_{\VV}(\cC, \cD)$ admits a  $\VV$-action which makes it into a  morphism object in $\Mod_{\VV}(\cat^{\cK})$. 
\end{enumerate}\end{lemma}

\begin{proof}
We first prove statement \eqref{item-lem:morphism-object-1} for $\cK = \emptyset$.
Consider the locally coCartesian fibration $\cC^{\circledast} \to \VV^{\circledast}$ from  \cite[Not. 4.2.2.17, Lem.~ 4.2.2.20]{HA}  associated to a $\VV$-module category $\cC$. It follows from \cite[Lem.~ 4.8.4.12]{HA} that  $\Fun_{\VV}(\cC, \cD) \subset \Fun_{/\VV^{\circledast}}(\cC^{\circledast}, \cD^{\circledast})$ is the full subcategory on those  functors which preserve locally coCartesian morphisms, where for given functors  $F\colon \cA \to \cB \leftarrow \cC \colon G$ of $\infty$-categories, we let $\Fun_{/ \cB}(\cA, \cC) \coloneqq \Fun(\cA, \cC) \times_{\Fun(\cA, \cB)} \{F\}$ denote the  over-functor category. If $\cC, \cD \in \LMod_{\VV}(\cat)$ and $\cA \in \cat$, the evident equivalence
\[\Fun(\cA, \Fun_{/\VV^\circledast}(\cC^{\circledast}, \cD^{\circledast})) \simeq \Fun_{/\VV^{\circledast}}(\cA \times \cC^{\circledast} , \cD^{\circledast}) \simeq \Fun_{/\VV^{\circledast}}((\cA \times \cC)^{\circledast} , \cD^{\circledast}) 
\]
restricts to an equivalence 
\begin{equation}
\label{eq:wanted-equivalence-of-infty-cats}
\Fun(\cA, \Fun_{\VV}(\cC, \cD)) \simeq \Fun_{\VV}(\cA \times \cC, \cD)
\end{equation} which upon passing to maximal $\infty$-subgroupoids shows that $\Fun_{\VV}(\cC, \cD)$ is the morphism object for the action of $\cat$ on $\LMod_{\VV}(\cat)$. 

Now let $\cK$ be general. Let  $\cA \in \cat^{\cK}$ and $\cC, \cD \in \LMod_{\VV}(\cat^{\cK})$, and let $\otimes$ denote the action of $\cat^{\cK}$ on $\LMod_{\VV}(\cat^{\cK})$. By definition of the action, it induces an equivalence 
    \begin{equation}\label{eq:equivalence-1}
        \Fun^{\cK}_{\VV}(\cA \otimes \cC, \cD) \simeq \Fun^{\cK\times\cK}_{\VV}(\cA \times \cC, \cD),
    \end{equation} 
    where $\Fun^{\cK\times \cK}_{\VV}(\cA \times \cC, \cD) \subset \Fun_{\VV}(\cA \times \cC, \cD)$ denotes the full subcategory of $\VV$-linear functors whose underlying functor $\cA \times \cC \to \cD$ preserves $\cK$-index colimits separately in each variable.
On the other hand, 
    by the description of $\cK$-indexed colimits in $\Fun^{\cK}_{\VV}(\cC, \cD)$ \cite[Lem.~4.8.4.13]{HA}, the equivalence \eqref{eq:wanted-equivalence-of-infty-cats} restricts to an equivalence of full subcategories
    \begin{equation}\label{eq:equivalence-2}
        \Fun^{\cK}(\cA, \Fun^{\cK}_{\VV}(\cC, \cD)) \simeq  \Fun^{\cK, \cK}_{\VV}(\cA \times \cC, \cD).
     \end{equation} 
     Composing \eqref{eq:equivalence-1} and \eqref{eq:equivalence-2} exhibits $\Fun^{\cK}_{\VV}(\cC, \cD)$ as the morphism object of $\cC, \cD$ in $\cat^{\cK}$. This proves part \eqref{item-lem:morphism-object-1}. Part \eqref{item-lem:morphism-object-2} follows now with \cref{prop:adjunction-and-hom} applied to the (symmetric) monoidal left adjoint $\cat^{\cK} \to \Mod_{\VV}(\cat^{\cK})$.
     \end{proof}

\subsubsection{Presentably enriched $\infty$-categories}\label{subsubsec:presentable-enrichment}
    Let $\VV \in \Alg(\PrL)$ and $\cC \in \LMod_{\VV}(\PrL)$, i.e. $\cC$ is a presentable $\infty$-category with an action $- \otimes -\colon \VV\times \cC \to \cC$ by a presentable monoidal $\infty$-category $\VV$ which is cocontinuous in both variables. As in \cref{exm:presentable-are-closed}, it follows from the adjoint functor theorem that the action is closed, i.e. for any pair of objects $x, y \in \cC$, there exists a morphism object $\eHom_{\cC}(x, y ) \in \VV$.  It is shown in~\cite[Cor. 7.4.13]{GH13} that these morphism objects assemble $\cC$ into a $\VV$-enriched $\infty$-category with space of objects $\cC^{\simeq}$, and which we will also denote by $\cC$. By {\cite[Thm. 7.21, Thm. 1.2]{heine23}}, this construction is functorial and multiplicative in the following sense:

    \begin{proposition}\label{prop:heine-presentable-enrichement}
    Let $\VV \in \CAlg(\PrL)$ and let $\hatCat[\VV]$ denote the $\infty$-category of large $\VV$-enriched $\infty$-categories equipped with the enriched tensor product. The construction of an enriched $\infty$-category from a presentable module category then assembles into a lax symmetric monoidal faithful functor
    \[
        \Mod_{\VV}(\PrL) \to \hatCat[\VV].
\]
In particular, this induces a functor \begin{equation} \label{eq:symmetricenrichedfrompresentable} \CAlg(\Mod_{\VV}(\PrL)) \simeq \CAlg(\PrL)_{\VV/} \to \CAlg(\widehat{\Cat}[\VV]).\end{equation}
\end{proposition}
The functor~\eqref{eq:symmetricenrichedfrompresentable} will be our meain tool to construct symmetric monoidal enriched $\infty$-categories and symmetric monoidal enriched functors between them. In particular, if $\VV \in \CAlg(\PrL)$, then $\VV$ itself may be considered as self-enriched, i.e. $\VV \in \CAlg(\widehat{\Cat}[\VV])$. 

\begin{notation} We follow~\cite[\S~A.3]{mazelgee2021universal} and call a $\VV$-enriched $\infty$-category $ \cC\in \hatCat[\VV]$  \emph{presentably $\VV$-enriched} if its underlying $\infty$-category is presentable, admits tensors\footnote{A $\VV$-enriched $\infty$-category $\cC$ \emph{admits tensors} if the $\VV$-enriched functor $\eHom_{\VV}(v, \eHom_{\cC}(c, -)) \colon \cC \to \VV$ is corepresentable for all $v\in \VV$ and $c \in \cC$. Denoting the corepresenting objects by $v\otimes c \in \cC$, this induces an action of $\VV$ on the underlying $\infty$-category of $\cC$.}, and if moreover for every $v\in \VV$, the induced functor $v\otimes -\colon \cC \to \cC$ between the underlying $\infty$-categories preserves small colimits. 
\end{notation}

\begin{remark}
\label{rem:presentableenriched}
 Let $\Pr^L_{\VV}$ denote the (non-full) subcategory of $\widehat{\Cat}[\VV]$ on the presentably $\VV$-enriched $\infty$-categories $\cC$  and on those $\VV$-enriched functors that are left adjoint in the $\VV$-enriched sense,\footnote{Equivalently, by~\cite[Lem.~A.2.14]{mazelgee2021universal}, $\VV$-enriched functors whose underlying functor is left adjoint and  preserves the induced $\VV$-action. }see~\cite[Def. A.2.12]{mazelgee2021universal}. Then, it is shown in~\cite[Thm. A.3.8]{mazelgee2021universal} that the functor $\Mod_{\VV}(\PrL) \to \widehat{\Cat}[\VV]$ factors as an equivalence through $\Pr^L_{\VV}$.
 In particular, $\Mod_{\VV}(\PrL) \simeq \Pr^L_{\VV}$ is a \emph{subcategory} of  $\widehat{\Cat}[\VV]$; it is merely a property of large $\VV$-enriched categories and  $\VV$-enriched functors to be in the image of $\Mod_{\VV}(\PrL) \to \hatCat[\VV]$.
\end{remark}

\subsection{Graded linear \texorpdfstring{$\infty$}{infinity}-categories}
\label{z-graded-k-linear-cats}
To incorporate $\mbbK$-linearity for $\mbbK \in \CAlg(\Spectra)$ into the setup, we could define small $\mbbK$-linear $\infty$-categories as small $\infty$-categories enriched in the presentably symmetric monoidal $\infty$-category $\Mod_{\mbbK}$ from \cref{nota:modk}.
Due to \cref{prop:heine-presentable-enrichement} and \cref{rem:presentableenriched}, it is technically easier to work with presentably enriched $\infty$-categories instead, as these can be expressed purely in the language of module categories. Our  `presentable $\mbbK$-linear' terminology is justified by \cref{rem:enrichedlinearity}.

\subsubsection{$\mbbK$-linear \texorpdfstring{$\infty$}{infinity}-categories}
We start with some definitions which are crucial throughout the paper.

\begin{definition} 
\label{def:linearcats}
\newcounter{listcounterx}
For $\mbbK \in \CAlg(\ConnSpectra)$, we define
\begin{enumerate}
\item the $\infty$-category $\PrLaddR$ of \emph{additive presentable $\mbbK$-linear $\infty$-categories} as \[\PrLaddR:= \Mod_{\Mod^{\geq 0}_{\mbbK}}(\PrL);\]
\item the  $\infty$-category $\add_{\mbbK}$ of \emph{small additive, idempotent-complete  $\mbbK$-linear $\infty$-categories} as \[\add_{\mbbK}:=\Mod_{\CProj_{\mbbK}}(\add).\]
\setcounter{listcounterx}{\value{enumi}}
\end{enumerate}
For $\mbbK\in \CAlg(\Spectra)$, we define
\begin{enumerate}
\setcounter{enumi}{\value{listcounterx}}
\item
the $\infty$-category $\PrLstR$ of \emph{stable presentable $\mbbK$-linear $\infty$-categories} as 
\[\PrLstR:= \Mod_{\Mod_{\mbbK}}(\PrL);\]

\item
the  $\infty$-category $\stR$ of \emph{small stable, idempotent-complete  $\mbbK$-linear $\infty$-categories} to be \[\stR:= \Mod_{\Perf_{\mbbK}}(\st).\]
\end{enumerate}
\end{definition}
\begin{remark}
In other words, an additive/stable presentable $\mbbK$-linear $\infty$-category is an additive/stable presentable $\infty$-category $\cC$ with an action by $\Mod_{\mbbK}^{\geq 0}$/$\Mod_{\mbbK}$, so
that the action functor $\Mod_{\mbbK}^{(\geq 0)} \times \cC \to \cC$ 
preserves small colimits in both variables. A small additive/stable idempotent-complete  $\mbbK$-linear $\infty$-category is a small, additive/stable idempotent complete $\infty$-category $\cC$ with
an action by $\CProj_{\mbbK}$ or $\Perf_\mbbK$, respectively so that the action functor $\CProj_{\mbbK}
\times \cC \to \cC$ is additive in either variable, or so that the action functor $\Perf_{\mbbK} \times \cC \to \cC$ is exact in either variable, respectively.
\end{remark}

\begin{remark} \label{rem:enrichedlinearity}
Following \cref{rem:presentableenriched}, an additive presentable $\mbbK$-linear $\infty$-category is precisely a presentably $\Mod_{\mbbK}^{\geq 0}$-enriched $\infty$-category in the sense of \cref{rem:presentableenriched}, i.e. a $\Mod_{\mbbK}^{\geq 0}$-enriched $\infty$-category fulfilling certain presentability properties. Similarly, a stable presentable $\mbbK$-linear $\infty$-category is precisely a presentably $\Mod_{\mbbK}$-enriched $\infty$-category, i.e. a $\Mod_{\mbbK}$-enriched $\infty$-category fulfilling certain presentability properties. 
\end{remark}
The following justifies the terminology `stable/additive presentable $\mbbK$-linear' in \cref{def:linearcats}.
\begin{observation}\label{obs:equivalentklinear}
Since $\Mod_{\mbbK}$ is stable and $\Mod_{\mbbK}^{\geq 0}$ is additive, 
we obtain the following equivalences from \cref{prop:algprl-new}.\eqref{item-prop:algprl3}:
\begin{align*}
\PrLaddR&:=  \Mod_{\Mod^{\geq 0}_{\mbbK}}(\PrL)
\simeq \Mod_{\Mod^{\geq 0}_{\mbbK}}(\PrLadd)\\
\PrLstR & := \Mod_{\Mod_{\mbbK}}(\PrL)\simeq \Mod_{\Mod_{\mbbK}}(\PrLst)
\end{align*}
In particular, any presentably $\Mod_{\mbbK}$-enriched $\infty$-category is automatically stable, and any presentably $\Mod_{\mbbK}^{\geq 0}$-enriched $\infty$-category is automatically additive. 
Combining \cref{prop:algprl-new}.\eqref{item-prop:algprl3} with the equivalences from \S\ref{sec:presentable} and~\cref{sec:stable}, we obtain the analogous characterizations of their small variants:
\begin{align*}
\addR & :=\Mod_{\CProj_{\mbbK}}(\add) \simeq \Mod_{\Mod_{\mbbK}^{\geq 0}}(\PrLaddcp) \simeq \Mod_{\Mod_{\mbbK}^{\geq 0}}(\PrLcp) \simeq  \Mod_{\CProj_{\mbbK}}(\catprod)
\\
\stR &:=\Mod_{\Perf_{\mbbK}}(\st)\simeq \Mod_{\Mod_{\mbbK}}(\PrLstc) \simeq \Mod_{\Mod_{\mbbK}}(\PrLc)  \simeq  \Mod_{\Perf_{\mbbK}}(\catrex)
\end{align*}
\end{observation}

The next remark covers our main case of interest and connects to the framework from \cref{sec:2}. 
\begin{remark}In case $\mbbK= Hk$ for a field $k$ of characteristic zero,
    it follows from~\cite{Cohn} that $\st_k\coloneqq \st_{Hk}$ is the localization of the ordinary $1$-category $\mathrm{dgCat}^{\mathrm{idem}, \mathrm{pretriang}}_k$ of small idempotent-complete pretriangulated dg-categories at the quasi-equivalences, i.e. those dg-functors which induces triangulated equivalences on homotopy categories. Hence, the reader may consider $\st_k$ as our $\infty$-categorical stand-in for the theory of dg-categories. In practice, the localization functor $\mathrm{dgCat}^{\mathrm{idem}, \mathrm{pretriang}}_k \to \st_k$ provides an easy way
    to construct objects and morphisms of $\st_k$. 
    \end{remark}

\begin{observation}As $\infty$-categories of modules of commutative algebras in presentably symmetric monoidal $\infty$-categories, both $\addR$ and $\stR$ are presentably symmetric monoidal. The symmetric monoidal structure on $\add_{\mbbK}$ can be characterized as follows: for $\cC, \cD \in \add_{\mbbK}$,  there is a functor $\cC \times \cD \to \cC \otimes \cD$ which is additive and $\mbbK$-linear in either variable, and which for all $\cE \in \add_{\mbbK}$ induces an equivalence between the $\infty$-category of additive $\mbbK$-linear functors $\cC \otimes \cD \to \cE$ and the $\infty$-category of functors $\cC \times \cD \to \cE$ that are additive and $\mbbK$-linear in either variable. An analogous characterization with additive replaced by exact holds for $\st_{\mbbK}$.
\end{observation}

    \begin{prop}\label{prop:Klinear} 
        Let $\mbbK \in \CAlg(\ConnSpectra)$. Recall the functor $\Kb  \colon \st \to \add$ from \cref{nota:K-as-fin}.
        \begin{enumerate}
            \item 
        This functor induces a symmetric monoidal functor $\Kb \colon st_{\mbbK} \to \add_{\mbbK}$ which is left adjoint to the forgetful functor $\add_{\mbbK} \to \st_{\mbbK}$. 
        \item For $\cC \in \add_{\mbbK}$,  the unit of the adjunction $\cC \to \Kb(\cC)$ is fully faithful. 
           \end{enumerate}
        \end{prop}
        \begin{proof}
        By \cref{prop:algprl-new}.\eqref{item-prop:algprl4}, the symmetric monoidal left adjoint $\Kb: \add\to \st$ induces a symmetric monoidal left adjoint functor $\add_{\mbbK} = \Mod_{\CProj_{\mbbK}}(\add) \to \Mod_{\Kb(\CProj_{\mbbK})} (\st)$. Composing with the equivalence $\Kb(\CProj_{\mbbK}) \simeq \Perf_{\mbbK}$ from \cref{prop:KProj} results in the desired functor proing the first part.
        Fully faithfulness of the unit of the adjunction follows from \cref{prop:Ksymmetric}.
        \end{proof}

\subsubsection{ \texorpdfstring{$\infty$}{infinity}-categories enriched in graded modules}

     \cref{rem:presentableenriched} motivates the following terminology:
    \begin{definition} 
Let $\Monoid$ be a homotopy coherent abelian monoid, i.e. $\Monoid \in \CAlg(\Spaces)$.\\
For $\mbbK\in \CAlg(\ConnSpectra)$, we define 
\begin{enumerate}
\item the $\infty$-category $\PrLaddRG$ of \emph{presentably $\Mod_{\mbbK}^{\geq 0, \Monoid}$-enriched $\infty$-categories},  
$$\PrLaddRG \coloneqq \Mod_{{\Mod_{\mbbK}^{\geq 0,\Monoid}}}(\PrL) ~;$$
\item the $\infty$-category $\PrLaddRGcp$ of \emph{projectively generated presentably $\Mod_{\mbbK}^{\geq 0, \Monoid }$-enriched $\infty$-categories},  
$$\PrLaddRGcp \coloneqq \Mod_{{\Mod_{\mbbK}^{\geq 0,\Monoid}}}(\PrLcp) ~. $$
\setcounter{listcounterx}{\value{enumi}}
\end{enumerate} 

For $\mbbK\in \CAlg(\Spectra)$, we define 
\begin{enumerate}
\setcounter{enumi}{\value{listcounterx}}
\item the $\infty$-category $\PrLstRG$ of \emph{presentably $\Mod_{\mbbK}^{ \Monoid}$-enriched $\infty$-categories}, 
$$\PrLstRG \coloneqq \Mod_{{\Mod_{\mbbK}^{\Monoid}}}(\PrL) ~;$$
\item the $\infty$-category $\PrLstRGc$ of \emph{compactly generated presentably $\Mod_{\mbbK}^{\Monoid}$-enriched $\infty$-categories}, 
$$\PrLstRGc \coloneqq \Mod_{{\Mod_{\mbbK}^{\Monoid}}}(\PrLc)~. $$
\end{enumerate} 
\end{definition}
    As $\infty$-categories of modules of commutative algebras in presentably symmetric monoidal categories, both, $\PrLaddRGcp$ and $\PrLstRGc$, are presentably symmetric monoidal $\infty$-categories.

\begin{remark}
\label{rem:RGequivalent}
As in \cref{obs:equivalentklinear}, any presentably $\Mod_{\mbbK}^{\geq 0,\Monoid}$-enriched $\infty$-category is additive and any presentably $\Mod_{\mbbK}^{\Monoid}$-enriched $\infty$-category is stable, and we have the following equivalences:
\begin{align*}
\PrLaddRG& \coloneqq  \Mod_{\Mod^{\geq 0,\Monoid}_{\mbbK}}(\PrL) \simeq \Mod_{\Mod^{\geq 0,\Monoid}_{\mbbK}}(\PrLadd)
\\
\PrLstRG & \coloneqq  \Mod_{\Mod^{\Monoid}_{\mbbK}}(\PrL)\simeq\Mod_{\Mod^{\Monoid}_{\mbbK}}(\PrLst)
\end{align*}
Similarly, we have the following equivalences for their small variants, abbreviating $M=\Mod_{\mbbK}^{\geq 0,  \Monoid}$:
\begin{align*}
\PrLaddRGcp & \coloneqq\Mod_{\Mod_{\mbbK}^{\geq 0, \Monoid}}(\PrLcp) \simeq \Mod_{M}(\PrLaddcp) \simeq \Mod_{M^{\mathrm{cp}}}(\add) \simeq  \Mod_{M^{\mathrm{cp}}}(\catprod) 
\\
\PrLstRGc &\coloneqq \Mod_{\Mod_{\mbbK}^{\Monoid}}(\PrLc) \simeq \Mod_{\Mod_{\mbbK}^{\Monoid}}(\PrLstc) \simeq \Mod_{\left(\Mod_{\mbbK}^{\Monoid}\right)^{\mathrm{c}}}(\st)\simeq \Mod_{\left(\Mod_{\mbbK}^{\Monoid}\right)^{\mathrm{c}}}(\catrex)\end{align*}
\end{remark}

\begin{example}
As discussed in \cref{exm:gradedderived}, if $\Monoid$ is a discrete (i.e. ordinary) commutative monoid $Z$, then $(\Mod_{\mbbK}^{\Monoid})^{\cp}$ and $(\Mod_{\mbbK}^{\Monoid})^c$ are the $\infty$-categories of \emph{finitely supported} functors $\Fun^{\text{fin.supp.}}(Z, \CProj_{\mbbK})$ and $\Fun^{\text{fin.supp.}}(Z, \Perf_{\mbbK})$, i.e. of functors that vanish on all but finitely many elements of $Z$. 
In particular, in the case of grading by a discrete monoid $Z$, we obtain the following equivalences:\begin{align*}
\mathrm{Pr}^{\mathrm{L},\cp}_{\mathrm{Mod}_{\mbbK}^{\geq 0, Z}} &\simeq \Mod_{\Fun^{\text{fin.supp.}}(Z, \CProj_{\mbbK})}(\add)
\\
\mathrm{Pr}^{\mathrm{L},\mathrm{c}}_{\mathrm{Mod}_{\mbbK}^{ Z}} &\simeq \Mod_{\Fun^{\text{fin.supp.}}(Z, \Perf_{\mbbK})}(\st)
\end{align*}
\end{example}

\subsection{From gradings to actions}
\label{subsubsec:gradingstoactions}

Given a ring $k$, and a $k$-linear category $\cC$ with an action by a discrete monoid $Z$, then the category $\cC$ is canonically enriched in the ordinary category $\mathrm{mod}_k^Z$ of $Z$-graded $k$-modules. Indeed, given objects $c, d\in \cC$ we define the $k$-module of degree-$z$ morphisms to be
\[
\hom_{\cC}(c,d)_z \coloneqq \hom_{\cC}(c, d[z]) 
\]
where $(-)[z] \colon \cC \to \cC$ denotes the action of $z\in Z$ on $\cC$. 
Conversely, if $\cC$ is a category enriched in $Z$-graded $k$-modules and if moreover for every $z\in Z$, the inner-hom functor $\eHom_{\cC}(c,-) \colon \cC \to \mathrm{mod}_k^Z$ is corepresentable (e.g. if $\cC$ is presentably enriched), then the enrichment arises from a $Z$-action on the underlying category.
These constructions provide an equivalence between the category of presentable $k$-linear categories with an $Z$-action and the category of categories presentably enriched in $Z$-graded $k$-modules.
In this section, we generalize these constructions to our $\infty$-categorical setting. 

 \begin{definition}\label{def:gradedcats}
     Let $J\in \CAlg(\cat)$.
     \begin{enumerate}
     \item For $\mbbK \in \CAlg(\ConnSpectra)$, we define the $\infty$-category $\add_{\mbbK}^J$ of \emph{$J$-graded additive, idempotent-complete $\mbbK$-linear categories} as \[\add_{\mbbK}^J\coloneqq \Fun(J^{\op}, \add_{\mbbK}).\]
     \item  For $\mbbK \in \CAlg(\Spectra)$, we define the $\infty$-category $\st_{\mbbK}^J$ of \emph{$J$-graded stable, idempotent-complete $\mbbK$-linear categories} as \[\st_{\mbbK}^J \coloneqq \Fun(J^{\op}, \stR).\]
     \end{enumerate} \end{definition}
     We need the following compatibilities of structures with the functor  $\Kb$ from \cref{prop:Klinear}. 
     \begin{prop}
         \label{prop:presentablestKI}
         Let $J\in \CAlg(\cat)$. 
         \begin{enumerate}
         \item Day convolution induces presentably symmetric monoidal  structures on $\add_{\mbbK}^J$ and on $\st_\mbbK^J$  for $\mbbK \in \CAlg(\ConnSpectra)$ and $\mbbK \in \CAlg(\Spectra)$, respectively.
         \item \label{item-prop:presentablestKI-1}
         Composing with the symmetric monoidal left adjoint $\Kb \colon  \add_{\mbbK} \to \st_{\mbbK}$  from \cref{prop:Klinear} induces a symmetric monoidal functor \[\Kb: \add_{\mbbK}^J = \Fun(J^{\op}, \add_{\mbbK}) \to \Fun(J^{\op}, \st_{\mbbK}) =\st_{\mbbK}^J\] left adjoint to the forgetful functor. 
         Moreover, for $\cC \in \add_{\mbbK}^J = \Fun(J^{\op}, \add_{\mbbK})$,  the unit of the adjunction $\cC\to \Kb(\cC)$ is pointwise (i.e. for every $j\in J$) fully faithful. 
         \end{enumerate}
         \end{prop}
         \begin{proof}
         Since $\st_{\mbbK}$ and $\add_{\mbbK}$ are in $\CAlg(\PrL)$ by \cref{prop:Klinear} and $J\in \CAlg(\cat)$,  \cref{cons:convolution} induces a presentably symmetric monoidal structure on $\Fun(J^{\op}, \st_{\mbbK})$ and $\Fun(J^{\op}, \add_{\mbbK})$. 
        Under the equivalence $\Fun(J^{\op}, \add_{\mbbK}) \simeq \add_{\mbbK} \otimes \Pres(J)$ of \cref{lem:funcat}, the postcomposition functor becomes the functor $\Kb \otimes \id_{\Pres(J)}$  and hence is a morphism in $\CAlg(\PrL)$.
        Given $\cC\in \add_{\mbbK}^J$, i.e. $\cC_{-}\colon J^{\op} \to \add_{\mbbK}$, the unit of the adjunction $\cC \to \Kb(\cC)$ is given by the natural transformation which at an object $j\in J$ is the unit $\cC_j \to \Kb(\cC_j)$ of the adjunction $\Kb\colon \add_{\mbbK} \to \st_{\mbbK}$. This is fully faithful by \cref{prop:Klinear}.  \end{proof}

We are interested in $\infty$-categories with an action by a commutative monoid. For $\Monoid \in \CAlg(\Spaces)$, let $B \Monoid \in \CAlg(\cat)$ denote its delooped symmetric monoidal $\infty$-category\footnote{If $\Monoid$ is a \emph{grouplike} commutative monoid, the delooping $B\Monoid$ is simply the classifying space of $\Monoid$ with its induced grouplike commutative monoid structure.}. Since $\Monoid$ is commutative, there is a symmetric monoidal equivalence $B\Monoid \simeq B\Monoid^{\op}$. Then an object of $\add_{\mbbK}^{B\Monoid} = \Fun(B\Monoid, \add_{\mbbK})$ is precisely a small additive, idempotent complete $\mbbK$-linear $\infty$-category with an action by $\Monoid$ via $\mbbK$-linear additive functors (and similarly for $\st_{\mbbK}^{B\Monoid}$).

The equivalence between gradings and actions derives from the following proposition:
\begin{prop} 
\label{prop:actiongrading}
Let $\Monoid \in \CAlg(\Spaces)$ with delooping $B\Monoid \in \CAlg(\cat)$. Then, there is a symmetric monoidal equivalence between $\Fun(B\Monoid, \Spaces)$ with its Day convolution symmetric monoidal structure and $\Mod_{\Monoid}(\Spaces)$ with symmetric monoidal structure given by relative tensor product over $\Monoid$.\end{prop}
\begin{proof}

For $\cC \in \CAlg(\PrL)$ for which $\Hom_{\cC}(I,-)\colon \cC \to \Spaces$ preserves all small colimits and is conservative, it follows from~\cite[Prop.~4.8.5.21]{HA} that $\cC$ is symmetric monoidally equivalent to $\Mod_{\End_{\cC}(I)}(\Spaces)$ where $\End_{\cC}(I) \in \CAlg(\Spaces)$ is equipped with the commutative monoid structure induced from symmetric monoidality of $\cC$. 

If $J \in \CAlg(\cat)$, then the Yoneda embedding $J \to \Fun(J^{\op}, \Spaces)$ is symmetric monoidal for the Day convolution symmetric monoidal structure. 
In particular, the monoidal unit of $\Fun(J^{\op}, \Spaces)$ is the image under the Yoneda embedding of $I \in J$, and its endomorphism algebra agrees with $\End_J(I)$. 
In particular, as a representable presheaf, $\Hom_{\Fun(J^{\op}, \Spaces)}(I,-) \colon \Fun(J^{\op}, \Spaces) \to \Spaces$ preserves all small colimits.

Let now $J=B\Monoid$ for a $\Monoid \in \CAlg(\Spaces)$.  The functor  $\Hom_{\Fun(B \Monoid^{\op}, \Spaces)}(I,-) \colon \Fun(B \Monoid^{\op}, \Spaces) \to \Spaces$ then forgets the $\Monoid$ action and is hence conservative. 
Since the unit of $B\Monoid$ is the basepoint $\pt$  with $\End_{B\Monoid}(\pt) \simeq \Monoid$ as commutative algebras in spaces, and using the symmetric monoidal equivalence  $ B\Monoid^{\op} \simeq B\Monoid$ induced by commutativity of $\Monoid$, it therefore follows ~\cite[Prop.~4.8.5.21]{HA} that we have symmetric monoidal equivalences
$\Fun(B\Monoid, \Spaces) \simeq \Fun(B\Monoid^{\op}, \Spaces)  \simeq \Mod_{\End_{B\Monoid}(\pt)}(\Spaces) = \Mod_{\Monoid}(\Spaces)$.
\end{proof}

We now prove the main proposition of this subsection: For a homotopy coherent abelian monoid $\Monoid$, the $\infty$-category of compactly generated presentably $\Mod_{\mbbK}^{\Monoid}$-enriched $\infty$-categories is equivalent to the category $\Fun(B\Monoid, \st_{\mbbK})$, i.e. to the category of $\mbbK$-linear  stable $\infty$-categories with an \emph{action} by $\Monoid$. This equivalence is symmetric monoidal for the Day convolution structure on $\Fun(B\Monoid, \st_{\mbbK})$.
\begin{prop}
    \label{prop:gradedaction}
     Fix $\Monoid\in \CAlg(\Spaces)$. 
    \begin{enumerate}
    \item \label{item-prop:gradedaction-1}
    For $\mbbK \in \CAlg(\ConnSpectra)$, the equivalence $(-)^{\cp}\colon \PrLaddcp \simeq \add_{\mbbK} \colon \PresSigma$ induces a symmetric monoidal equivalence:
    $$
        \begin{tikzcd}
      \PrLaddRGcp\arrow[r,shift left=.5ex,"(-)^{\cp}"]
        &
         \add_{\mbbK}^{B\Monoid} \arrow[l,shift left=.5ex,"\PresSigma"]
        \end{tikzcd}
    $$
    \item \label{item-prop:gradedaction-2}
    For $\mbbK \in \CAlg(\Spectra)$, the equivalence $(-)^{\mathrm{c}} \colon \PrLstc \simeq \st_{\mbbK} \colon \Ind$ induces a symmetric monoidal equivalence:
    $$    \begin{tikzcd}
     \PrLstRGc \arrow[r,shift left=.5ex,"(-)^{\mathrm{c}}"]
        &
        \st_{\mbbK}^{B\Monoid} \arrow[l,shift left=.5ex,"\Ind"]
        \end{tikzcd}
    $$    
    \end{enumerate}
    \end{prop}
 
    \begin{proof}We prove statement~\eqref{item-prop:gradedaction-2}, the proof of statement~\eqref{item-prop:gradedaction-1} is entirely analogous.
Consider the symmetric monoidal equivalences
\[\st_{\mbbK}^{B\Monoid} \coloneqq \st_{\mbbK} \otimes \Pres(B\Monoid) \simeq  \st_{\mbbK} \otimes \Mod_{\Monoid}(\Spaces)  \simeq \Mod_{\Mod_\mbbK}(\PrLc) \otimes \Mod_{\Monoid}(\Spaces),
 \]
 where the first equivalence  is given by \cref{prop:actiongrading} and the second equivalence follows from \cref{obs:equivalentklinear}.  It then follows from \cref{obs:algspaces} and \cref{lem:unit} that 
\[\Mod_{\Mod_\mbbK}(\PrLc) \otimes \Mod_{\Monoid}(\Spaces) \simeq \Mod_{\cP(\Monoid) \otimes \Mod_{\mbbK}}(\PrLc) = \Mod_{\Mod_{\mbbK}^{\Monoid}}(\PrLc).\qedhere
\] 
\end{proof}

\begin{observation}
\label{obs:praddenriched}
As presentably symmetric monoidal $\infty$-categories, $\PrLaddRGcp$ and  $\PrLstRGc$ are self-enriched. Transporting these self-enrichments along the equivalences from \cref{prop:gradedaction} provides enrichments in $\addRG$ and $\stRG$, respectively, i.e. \[\PrLaddRGcp \in \CAlg(\hatCat[\addRG])\hspace{0.25cm}\text{ and }\hspace{0.25cm} \PrLstRGc \in \CAlg(\hatCat[\stRG]).\]

It follows from \cref{lem:morphism-object-and-restriction} applied to $\Mod_{\left(\Mod_{\mbbK}^{\geq 0, \Monoid}\right)^{\cp}}(\catprod)$ that given $\cC, \cD \in  \PrLaddRGcp$, their $\addRG$-enriched hom is the small idempotent-complete additive $\infty$-category \[\Fun^{\mathrm{L}, \cp}_{\Mod_{\mbbK}^{\geq 0, B\Monoid}}\left( \cC, \cD\right)\in \addRG\]
with $\CProj_{\mbbK}$ and $\Monoid$-action induced by the $\Mod_{\mbbK}^{\geq 0, \Monoid}$-action on $\cD$. 

Similarly, it follows from \cref{lem:morphism-object-and-restriction} applied to $\Mod_{\left(\Mod_{\mbbK}^{ \Monoid}\right)^{\cp}}(\catrex)$ that given $\cC, \cD \in  \PrLstRGc$, their $\stRG$-enriched hom is the small idempotent-complete stable $\infty$-category 
\[\Fun^{\mathrm{L}, \mathrm{c}}_{\Mod_{\mbbK}^{B\Monoid}}\left( \cC, \cD\right) \in \stRG\]
with $\Perf_{\mbbK}$ and $\Monoid$-action induced by the $\Mod_{\mbbK}^{\Monoid}$-action on $\cD$.
 \end{observation}

\subsection{\texorpdfstring{$\infty$}{infty}-Morita theory}\label{subsec:morita-theory}

For any monoidal $1$-category $\VV$ with reflective coequalizers distributing over the tensor product, one may construct a \emph{Morita $2$-category} whose objects are algebras  in $\VV$, whose $1$-morphisms are bimodules and whose $2$-morphisms are bimodule maps. 
If $\VV$ is moreover presentably symmetric monoidal, and hence self-enriched, then also the categories of bimodules $_{A}\BMod_B(\VV)$ will inherit a $\VV$-enrichment and thus the Morita 2-category inherits a $\VV$-enrichment at the level of $2$-morphisms.
The goal of this section it to establish $\infty$-categorical variants of these statements to be used for our homotopy coherent construction of the Soergel $(2,2)$-category in \cref{sec:SBim}. 

Various $\infty$-categorical constructions of $(\infty,2)$-Morita categories exist in the literature, see e.g. ~\cite{HA}, \cite{HaugsengMor}, \cite{FrS} and references therein. Due to their compatibility with enrichment, we follow ideas from~\cite{HA}. Our starting point is the following:
\begin{prop}[{\cite[Thm.~4.8.5.15, Rem.~4.8.4.9]{HA}}]\label{prop:HA-morita-prop}
    Let $\VV \in \CAlg(\PrL)$ and $A, B \in \Alg(\VV)$. 
    \begin{enumerate}
        \item\label{item-prop:HA-morita-prop-1} The $\infty$-category $\RMod_A(\VV)$ carries a left action by $\VV$, and can be viewed as an object in $\Mod_{\VV}(\PrL)$. This defines a symmetric monoidal functor 
       \[
            \RMod_{-}(\VV) \colon \Alg(\VV) \to \Mod_\VV(\PrL).
\]
        \item\label{item-prop:HA-morita-prop-2} Given an $A$--$B$ bimodule $_{A}M_{B} \in {}_{A}\BMod_{B}(\VV)$,  tensoring with $M$ over $A$ 
        $$- \otimes_{A}M_{B} \colon \RMod_{A}(\VV) \to \RMod_{B}(\VV)$$ defines a cocontinuous $\VV$-linear functor, i.e. an object in $\Fun^L_{\VV}(\RMod_A(\VV), \RMod_B(\VV))$. These assemble into an equivalence: 
        \[
            {}_{A}\BMod_{B}(\VV) \xrightarrow{\simeq} \Fun^L_{\VV}(\RMod_A(\VV), \RMod_B(\VV)).
       \]
        Furthermore, composition of functors corresponds to the relative tensor product of bimodules. 
    \end{enumerate}
\end{prop}
We can therefore think of the full subcategory of $\Mod_{\VV}(\PrL)$ on those presentable $\VV$-module categories which are of the form $\RMod_A(\VV)$ for algebra objects $A$ in $\VV$, as an $\infty$-categorical Morita category  with objects algebras, morphisms given by bimodules and composition given by relative tensor product (cf.~\cite[Rem. 4.8.4.9]{HA}). In the following, we will be interested in versions of the Morita category where we further restrict our bimodules requiring certain compactness or projectivity properties: 

\begin{nota}\label{nota:rightcp}
    Given $\VV \in \Alg(\PrLcp)$  and $A, B \in \Alg(\VV)$, we denote by \[{}_{A}\BMod^{\cp}_{B}(\VV)\subseteq {}_{A}\BMod_{B}(\VV)\]the full subcategory on those $A$--$B$-bimodules which are compact-projective as right $B$-modules. 
    
    Similarly, given $\VV \in \Alg(\PrLc)$ and $A, B \in \Alg(\VV)$, we denote by \[{}_{A}\BMod^{\mrc}_{B}(\VV)\subseteq {}_{A}\BMod_{B}(\VV)\] the full subcategory on those $A$--$B$-bimodules which are compact as right $B$-modules. 
\end{nota}

For $\VV\in \CAlg(\PrLc)$ and $A\in \Alg(\VV)$, the $\infty$-category  $\RMod_{A}(\VV)$ is compactly generated and the $\VV$-action preserves compact generators, see \cref{lem:Modpresentable}.\eqref{item-lem:Modpresentable-1}-\eqref{item-lem:Modpresentable-2}; the functor $\RMod_{-}(\VV) \colon \Alg(\VV) \to \Mod_{\VV}(\PrL)$ thus factors through $\Mod_{\VV}(\PrLc)$. 
The analogous statement holds for $\VV \in \CAlg(\PrLcp)$ by \cref{lem:Modpresentable}.\eqref{item-lem:Modpresentable-5}-\eqref{item-lem:Modpresentable-6}.
\begin{corollary}\label{cor:morita-cor-for-PrLcp-and-PrLc}The following hold.
    \begin{enumerate}
        \item \label{item-cor:morita-cor-for-PrLcp}
         For $\VV \in \CAlg(\PrLcp)$, $A, B \in \Alg(\VV)$ and  ${}_{A}M_B \in{}_A\BMod_B(\VV)$, the functor 
        \[- \otimes_{A}M \colon \RMod_{A}(\VV) \to \RMod_{B}(\VV)\]
        preserves compact projective  objects if and only if  $M_B$, viewed as a right $B$-module, is a compact projective object in $\RMod_{B}(\VV)$.   
        The equivalence from \cref{prop:HA-morita-prop}.\eqref{item-prop:HA-morita-prop-2}  restricts to
     \[
            {}_{A}\BMod^{\cp}_{B}(\VV) \simeq \Fun^{L, \cp}_{\VV}(\RMod_A(\VV), \RMod_B(\VV)).
   \]
        \item \label{item-cor:morita-cor-for-PrLc}
         For $\VV \in \CAlg(\PrLc)$, $A, B \in \Alg(\VV)$ and  ${}_{A}M_B \in{}_A\BMod_B(\VV)$, the functor 
        \[- \otimes_{A}M \colon \RMod_{A}(\VV) \to \RMod_{B}(\VV)\]
        preserves compact objects if and only if  $M$, viewed as a right $B$-module, is a compact object in $\RMod_{B}(\VV)$. 
        The equivalence from \cref{prop:HA-morita-prop}.\eqref{item-prop:HA-morita-prop-2} restricts to an equivalence
      \[
            {}_{A}\BMod^{\mrc}_{B}(\VV) \simeq \Fun^{L, \mrc}_{\VV}(\RMod_A(\VV), \RMod_B(\VV)).
\]
    \end{enumerate}
\end{corollary}

\begin{proof}
    We will prove statement \eqref{item-cor:morita-cor-for-PrLcp}, the proof of statement \eqref{item-cor:morita-cor-for-PrLc} is completely analogous. Since $\RMod_A(\VV)$ is projectively generated by free modules $v\otimes A$, see \cref{lem:Modpresentable}.\eqref{item-lem:Modpresentable-7}, where $v \in \VV$ is compact projective, it suffices to show that $-\otimes_A M \colon \RMod_{A}(\VV) \to \RMod_B(\VV)$ preserves compact projective objects if and only if it sends such free modules $v\otimes A $ to compact projectives  in $\RMod_B(\VV)$ for all compacts $v\in \VV$.  Since the action functor $\VV\times \RMod_B(\VV) \to \RMod_B(\VV)$  takes pairs of compact projectives to compact projectives by  \cref{lem:Modpresentable}.\eqref{item-lem:Modpresentable-6},  this in turn is equivalent to the assertion that $A\otimes_A M_B \simeq M_B \in \RMod_B(\VV)$ is compact projective. 
\end{proof}

We are now ready to define our Morita categories of interest. 

\begin{definition}\label{def:MoritaLarge} Fix $\Monoid \in \CAlg(\Spaces)$. 

Given $\mbbK \in \CAlg(\ConnSpectra)$, we define $\Moritacp(\Mod_{\mbbK}^{\geq 0, \Monoid})$ to be the full symmetric monoidal $\addRG$-enriched subcategory of $\PrLaddRGcp$ (equipped with $\addRG$-enrichment as in \cref{obs:praddenriched})  on the objects in the image of the symmetric monoidal functor $\Alg(\Mod_{\mbbK}^{\geq 0, \Monoid}) \to \PrLaddRGcp$ from \cref{prop:HA-morita-prop}.\eqref{item-prop:HA-morita-prop-1}.

Given $\mbbK \in \CAlg(\Spectra)$, we 
define $\Moritac(\Mod_{\mbbK}^{\Monoid})$ to be the full symmetric monoidal $\stRG$-enriched subcategory of $\PrLstRGc$ (equipped with $\stRG$-enrichment as in \cref{obs:praddenriched})  on the objects in the image of the symmetric monoidal functor $\Alg(\Mod_{\mbbK}^{ \Monoid}) \to \PrLstRGc$ from \cref{prop:HA-morita-prop}.\eqref{item-prop:HA-morita-prop-1}.
\end{definition}

We unpack the relevant  properties of these Morita categories:

\begin{corollary}\label{cor:morita}
Fix $\Monoid \in \CAlg(\Spaces)$.
\begin{enumerate}
        \item 
        \label{item-cor:moritacp-k-z}
        Given $\mbbK \in \CAlg(\ConnSpectra)$, \cref{def:MoritaLarge} defines a large symmetric monoidal $\addRG$-enriched $\infty$-category \[\Moritacp(\Mod_{\mbbK}^{\geq 0, \Monoid}) \in \CAlg(\hatCat[\addRG])\] with a symmetric monoidal surjective-on-objects functor $\Alg(\Mod_{\mbbK}^{\geq 0, \Monoid}) \to \Moritacp(\Mod_{\mbbK}^{\geq 0, \Monoid})$ and such that the $\addRG$-enriched hom between $A, B$ in $\Alg(\Mod_{\mbbK}^{\geq 0, \Monoid})$ is given by \[{}_A\BMod^{\cp}_{B}(\Mod_{\mbbK}^{\geq 0, \Monoid}) \in \addRG,\]
        with the $\CProj_{\mbbK}$ and $\Monoid$-action induced by their respective actions on $\Mod_{\mbbK}^{\geq 0, \Monoid}$. 
        
        The symmetric monoidal structure is given by the tensor product in $\Mod_{\mbbK}^{\geq 0, \Monoid}$, and the composition of $1$-morphisms is given by the relative tensor product of bimodules therein.       

        \item 
        \label{item-cor:moritac-k-z}
        Given $\mbbK \in \CAlg(\Spectra)$, \cref{def:MoritaLarge}  defines a large symmetric monoidal $\stRG$-enriched $\infty$-category \[\Moritac(\Mod_{\mbbK}^{ \Monoid}) \in \CAlg(\hatCat[\stRG])\]  with a symmetric monoidal surjective-on-objects functor $\Alg(\Mod_{\mbbK}^{\Monoid}) \to \Moritacp(\Mod_{\mbbK}^{ \Monoid})$, and such that the $\stRG$-enriched hom between $A, B$ in $\Alg(\Mod_{\mbbK}^{\Monoid})$ is given by \[{}_A\BMod^{\mrc}_{B}(\Mod_{\mbbK}^{\Monoid}) \in \stRG,\]
        with the $\Perf_{\mbbK}$ and $\Monoid$-action induced by their respective actions on $\Mod_{\mbbK}^{\Monoid}$.
        
        The symmetric monoidal structure is given by the tensor product in $\Mod_{\mbbK}^{ \Monoid}$, and the composition of $1$-morphisms is given by the relative tensor product of bimodules therein.  
        \item 
        \label{item-cor:morita-add-to-st}
         Given $\mbbK \in \CAlg(\ConnSpectra)$, the symmetric monoidal inclusion $\Mod_{\mbbK}^{\geq 0, \Monoid} \hookrightarrow \Mod_{\mbbK}^{\Monoid}$ from \cref{obs:connMod-to-Mod-symmetric} induces a symmetric monoidal $\addRG$-enriched functor     \[\Moritacp(\Mod_{\mbbK}^{\geq 0, \Monoid}) \to \Moritac(\Mod_{\mbbK}^{ \Monoid}),\]
         where $\Moritac(\Mod_{\mbbK}^{ \Monoid})$ is considered $\addRG$-enriched by transporting its $\stRG$-enrichment along the forgetful functor $\stRG \to \addRG$.
         
         On objects, this functor acts via the inclusion $(\Alg(\Mod_{\mbbK}^{\geq 0, \Monoid}))^{\simeq} \hookrightarrow (\Alg(\Mod_{\mbbK}^{\Monoid}))^{\simeq}$, and  on hom-categories as the additive $\mbbK$-linear $\Monoid$-equivariant (i.e. $\addRG$-morphism) full inclusion 
         \[
            {}_A\Mod^{\cp}_B(\Mod_{\mbbK}^{\geq 0, \Monoid}) \hookrightarrow {}_A\Mod^{c}_B(\Mod_{\mbbK}^{\Monoid}).
\]
    \end{enumerate}    
\end{corollary}

\begin{proof}Statements \eqref{item-cor:moritacp-k-z} and \eqref{item-cor:moritac-k-z} follow immediately from \cref{cor:morita-cor-for-PrLcp-and-PrLc} and   \cref{obs:praddenriched}. For statement~\eqref{item-cor:morita-add-to-st}, the inclusion $\Mod_{\mbbK}^{\geq 0, \Monoid} \hookrightarrow \Mod_{\mbbK}^{\Monoid}$ induces a symmetric monoidal left adjoint functor 
\begin{align*}
\PrLaddRGcp &\simeq \Mod_{\Mod_{\mbbK}^{\geq 0, \Monoid}}(\PrLcp) \xrightarrow{\Mod_{\Mod_{\mbbK}^{\geq 0, \Monoid}}\left( \PrLcp \to \PrLc\right)} \Mod_{\Mod_{\mbbK}^{\geq 0, \Monoid}}(\PrLc)\\& \xrightarrow{-\otimes_{\Mod_{\mbbK}^{\geq 0, \Monoid}} \Mod_{\mbbK}^{\Monoid}}\Mod_{\Mod_{\mbbK}^{\Monoid}}(\PrLc) \simeq \PrLstRGc.
\end{align*}
Analogous to \cref{prop:Ksymmetric}, this functor fits into a commuting square in $\CAlg(\PrL)$ of the form
\begin{equation}\label{eq:Morita-map-square}\begin{tikzcd}
\addRG \ar[r, "\Kb"] \ar[d, "\simeq"] & \ar[d, "\simeq"]\stRG\\
\PrLaddRGcp \ar[r] & \PrLstRGc,
\end{tikzcd}
\end{equation}
where the top horizontal morphism is left adjoint to the forgetful functor. 
Using~\eqref{eq:Morita-map-square} to consider the bottom horizontal morphism as a morphism in $\CAlg(\PrL)_{\addRG/}$,  it enhances by \cref{prop:heine-presentable-enrichement} to a symmetric monoidal $\addRG$-enriched functor $\PrLaddRGcp \to \PrLstRGc$. (By commutativity of~\eqref{eq:Morita-map-square} and \cref{prop:adjunction-and-hom}, we may understand the $\addRG$-enrichment of $\PrLstRGc$  as induced by restricting its $\stRG$-enrichment from \cref{obs:praddenriched} along the forgetful functor $\stRG \to \addRG$.)
For an algebra $A\in \Alg(\Mod_{\mbbK}^{\geq 0, \Monoid})$, it follows from  \cite[Thm. 4.8.4.6]{HA} that \[\RMod_A(\Mod_{\mbbK}^{\geq 0, \Monoid}) \otimes_{\Mod_{\mbbK}^{\geq 0, \Monoid}} \Mod_{\mbbK}^{\Monoid} \simeq \RMod_{A}(\Mod_{\mbbK}^{\Monoid}).\]
Thus, the $\addRG$-enriched functor $\PrLaddRGcp \to \PrLstRGc$ restricts to an $\addRG$-enriched functor between the full subcategories  \[\Moritacp(\Mod_{\mbbK}^{\geq 0, \Monoid}) \to \Moritac(\Mod_{\mbbK}^{ \Monoid}).\] 
Its explicit description on additive hom-categories can be unpacked from \cref{prop:adjunction-and-hom}. 
 \end{proof}

\subsection{Morita categories of discrete flat algebras}\label{subsec:morita-cat-for-discrete-alg}
\newcommand{\modkZflat}{\mod_k^{Z, \mathrm{flat}}}

For the rest of this section, we focus on the case where $\mbbK = Hk$ for an ordinary commutative ring  $k$ and where $\Monoid$ is a discrete commutative monoid $Z$. Our goal of this subsection is to restrict our Morita categories from \cref{cor:morita} to certain symmetric monoidal full subcategories which only contain discrete (i.e. ordinary) $k$-algebras and whose hom-categories are given by ordinary categories of discrete graded bimodules, or their derived variants.

\begin{definition}
    A discrete $Z$-graded $k$-module $M$ is \emph{flat} if $\oplus_{z \in Z} M_z$ is flat\footnote{Recall that a module $M$ over a ring $k$ is called \emph{flat} if $- \otimes_k M$ is an exact functor.} as a $k$-module. We let $\modkZflat$ denote the full subcategory of the ordinary category of discrete $Z$-graded $k$-modules $\modkZ$ on the flat modules. A (not necessarily commutative) discrete $Z$-graded $k$-algebra $A$ is \emph{flat} if it is flat as a $Z$-graded $k$-module.
    \end{definition}
\begin{remark}
    In particular, the ordinary category of discrete $Z$-graded $k$-algebras is $\Alg(\modkZflat)$. 
\end{remark}

The following two observations and example are crucial when connecting back to \cref{sec:2}.
\begin{observation}
    An ordinary $Z$-graded $k$-module $M$ is flat if and only if it is degreewise flat, i.e.  each graded component $M_z$ is a flat $k$-module for all $z\in Z$. This follows from the fact that flatness is preserved under infinite coproducts and retracts.

\end{observation}

\begin{example}\label{exm:graded-poly-flat}
Free modules are flat. In particular, if $k$ is a field, all $Z$-graded $k$-vector spaces are flat, and if $k$ is an ordinary commutative ring and $n\geq 0$, the polynomial algebra $k[x_1,\ldots, x_n]$ is flat, and hence it is also flat if considered as a $\mbbZ$-graded $k$-algebra with generators $x_i$ in some degree $n_i \in \mbbZ$.
\end{example}

\begin{observation}\label{obs:degreewise-flat-closed-under-tensor}
    As flatness is closed under tensor products, $\modkZflat$ is a symmetric monoidal full subcategory of $\modkZ$. On the other hand, 
    $\modkZflat$ is also a symmetric monoidal full subcategory of $\Mod_{Hk}^{ \geq 0, Z} \hookrightarrow \Mod_{Hk}^Z$: under the equivalence $\Mod_{Hk}^Z \simeq \Derived(\modkZ)$ by \cref{exm:gradedderived}. The tensor product in  $\Derived(\modkZ)$ is given by Day convolution of derived tensor products, which reduces to the Day convolution of ordinary tensor products on flat modules. In particular, tensor products of discrete flat $Z$-graded $k$-algebras are also discrete and flat.
\end{observation}

\newcommand{\grbmod}{\mathrm{grbmod}}

\begin{nota}\label{nota:gr-perf-right}
For flat $Z$-graded $k$-algebras $A$ and $B$, we let ${}_A\grbmod_B\coloneqq{}_A\BMod_B(\modkZ)$ denote the abelian $1$-category of ordinary graded $A$--$B$ bimodules. 
Let ${}_A\grbmod_B^{\mathrm{gr-cp}}$ denote its full subcategory on those bimodules that are  graded-compact-projective, see \cref{def:graded-compact-perfect}, as right $B$-modules. 

Let $\Derived({}_A\grbmod_B)^{\mathrm{gr-perf}}$ denote the full subcategory of the derived $\infty$-category $\Derived({}_A\grbmod_B)$ on those objects that are graded-perfect, see \cref{rk:graded-perfect} and preceeding definition, as derived right $B$-modules. 
\end{nota}

Most of the constructions in \cref{sec:SBim} will build on the following $\infty$-categories.

\begin{definition}\label{def:moritakz-and-dmoritakz}
Let $k$ be an ordinary commutative ring and $Z$ a discrete commutative monoid. We define \[\MoritakZ \subseteq \Moritacp(\Mod_{Hk}^{\geq 0, Z})\] to be the full $\add_{Hk}^{BZ}$-enriched subcategory on the discrete flat $Z$-graded $k$-algebras. 

Similarly, we define \[\DMoritakZ \subseteq \Moritac(\Mod_{Hk}^{Z})\] to be the full $\st_{Hk}^{BZ}$-enriched subcategory on the discrete flat $Z$-graded $k$-algebras.
\end{definition}

The following justifies the terminology 'Morita categories', see also \cref{ex:Moritaunpacked}.
\begin{cor}\label{cor:finally-our-morita-categories}
Let $k$ be an ordinary commutative ring and $Z$ a discrete commutative monoid. 
    \begin{enumerate}
        \item \label{item-cor:finally-1}
        \cref{def:moritakz-and-dmoritakz} defines a  large symmetric monoidal $\add_{Hk}^{BZ}$-enriched $\infty$-category 
    \[\MoritakZ\in \CAlg(\hatCat[\add_{Hk}^{BZ}])\]
        equipped with a symmetric monoidal surjective-on-objects functor \[\Alg(\modkZflat) \to \MoritakZ.\] The additive $k$-linear hom-category between algebras $A, B \in \Alg(\modkZflat)$ is given by the ordinary category  ${}_A\grbmod_B^{\mathrm{gr-cp}}$
        with $Z$-action by grading shift. 
        
        Composition is given by the ordinary relative tensor product, and the monoidal structure by the ordinary tensor product over $k$. 
        \item \label{item-cor:finally-2}
 \cref{def:moritakz-and-dmoritakz} defines a large symmetric monoidal $\st_{Hk}^{BZ}$-enriched $\infty$-category 
        \[\DMoritakZ \in \CAlg(\hatCat[\st_{Hk}^{BZ}])\]
          equipped with a symmetric monoidal surjective-on-objects functor \[\Alg(\modkZflat) \to \DMoritakZ.\]
           The stable $k$-linear hom-category between algebras $A, B \in \Alg(\modkZflat)$ is  
           $\Derived(_A\grbmod_B)^{\mathrm{gr-perf}}$
        with $Z$-action by grading shift. 

        Composition is given by the derived relative tensor product, and the monoidal structure by the derived tensor product over $k$. 
        \item \label{item-cor:finally-3}
        The  functor from \cref{cor:morita}.\eqref{item-cor:morita-add-to-st} restricts to a symmetric monoidal $\add_{Hk}^{BZ}$-enriched functor 
       \[\MoritakZ \to\DMoritakZ.\]
 On objects this functor acts via the identity on $\Alg(\modkZflat)^{\simeq}$; on hom-categories between  $A, B \in \Alg(\modkZflat)$ it is given by the additive $k$-linear $Z$-equivariant fully faithful inclusion
 \[
            _A\grbmod_B^{\mathrm{gr-cp}} \hookrightarrow \Derived(_A\grbmod_B)^{\mathrm{gr-perf}}.
\]
    \end{enumerate}
\end{cor}
\begin{proof}
Since the derived tensor product of discrete flat $Z$-graded algebras is again a discrete and flat algebra, the functor $\Alg(\modkZflat) \to \Alg(\Mod_{Hk}^{\geq 0, \Monoid})$ is symmetric monoidal, and hence the full subcategories $\MoritakZ$ and $\DMoritakZ$ are closed under the tensor product in $\Moritacp(\Mod_{Hk}^{\geq 0, Z})$ and $\Moritac(\Mod_{Hk}^{ Z})$, respectively. Denoting the derived and underived Day convolution tensor product by $\otimes^{L, \mathrm{gr}}$ and $\otimes^{\mathrm{gr}}$, respectively, and using \cref{prop:gradedmodules} and \cref{obs:degreewise-flat-closed-under-tensor} we obtain the following equivalences for discrete flat $Z$-graded $k$-algebras $A$ and $B$
\begin{align*}
       {}_A\BMod_B(\Mod_{Hk}^{Z}) &  \simeq \RMod_{A^{\op}\otimes^{L, \mathrm{gr}} B}(\Mod_{Hk}^Z) \simeq \RMod_{A^{\op}\otimes^{\mathrm{gr}} B}(\Mod_{Hk}^Z) \\& \simeq \Derived(\grmod_{A^{\op}\otimes^{\mathrm{gr}} B}) \simeq \Derived({}_A\grbmod_B).
\end{align*}
Recalling \cref{nota:rightcp} for the full subcategories ${}_{A}\BMod^{\cp}_B(\Mod_{Hk}^{\geq 0, Z})$ and ${}_{A}\BMod^{\mrc}_B(\Mod_{Hk}^{Z})$  on those  bimodules which are compact-projective, resp. compact as \emph{right $B$-modules}, the above equivalence restricts to an equivalence between subcategories (see \cref{nota:gr-perf-right})
\[
        {}_A\BMod^{\cp}_{B}(\Mod_{Hk}^{\geq 0, Z})  \simeq {}_A\grbmod_B^{\mathrm{gr-cp}} \quad  \mathrm{and}  \quad {}_A\BMod^{\mrc}_{B}(\Mod_{Hk}^{Z}) \simeq \Derived(_A\grbmod_B)^{\mathrm{gr-perf}}.
\]
Using these observations, \cref{cor:finally-our-morita-categories} follow directly from \cref{cor:morita}.
\end{proof}
\begin{remark}
The hom-categories $ {}_A\grbmod_B^{\mathrm{gr-cp}}$ of $\MoritakZ$ are ordinary $1$-categories, hence $\MoritakZ$ is a (2,2)-category. On the other hand, $\DMoritakZ$ is a genuine $(\infty,2)$-category with non-trivial higher morphisms.
\end{remark}

\begin{observation}
\label{obs:dmor-to-st}
By definition, $\DMoritakZ$ is a full symmetric monoidal subcategory of $\PrLc_{\Mod_{Hk}^{Z}}$. Thus, it comes equipped with a symmetric monoidal 
fully faithful $\st_{Hk}^{BZ}$-enriched functor 
\[\DMoritakZ \hookrightarrow \PrLc_{\Mod_{Hk}^{Z}} ~\stackrel{(-)^{\mrc}}{\simeq} ~\st_{Hk}^{BZ}.
\]
\end{observation}
Explicitly, the functor in \cref{obs:dmor-to-st} sends a flat $Z$-graded  $k$-algebra $A$ to the stable $\infty$-category 
\[\RMod_{HA}\left(\Mod_{Hk}^{BZ}\right)^{\mathrm{c}}~~ \stackrel{\mathrm{Prop.}~\ref{prop:gradedmodules}}{\simeq} ~~\Derived(\grmod_A)^{\mathrm{gr-perf}}\]
of graded-perfect right $A$-modules, with $Z$-action given by grading shift.

\section{\texorpdfstring{$(\infty,k)$}{(infinity,k)}-categories and their factorization systems}
\label{sec:inf-n-cats}

In this section, we introduce our $(\infty,k)$-categorical framework, and establish the existence of various factorization systems generalizing the familiar (surjective-on-objects, fully faithful)-factorization system on $\cat$. We refer the reader to \cref{sec:appendix-recollections}, especially~\ref{subsec:appendix-infty-n-cats},  for motivation and a leisurely introduction to $(\infty,k)$-categories.

\subsection{Basic notions in \texorpdfstring{$(\infty,k)$}{(infinity,k)}-category theory}\label{subsec:basics-of-inf-k-cat}
Throughout, we will use the theory of \emph{enriched $\infty$-categories} developed in~\cite{GH13}, see also  \cref{subsec:enriched-infty-cats}.

\begin{definition}\label{def:CatInfty-inductive-def}
We set  $\CatInfty{0} \coloneqq \Spaces$ to be the $\infty$-category of small spaces, and equip it with its Cartesian presentably symmetric monoidal structure. We inductively define the Cartesian\footnote{
For $\VV$ a Cartesian symmetric monoidal $\infty$-category, the induced symmetric monoidal structure on $\Cat[\VV]$ is also Cartesian: 
by construction, the symmetric monoidal structure on $\AlgCat[\VV]$ is Cartesian; since the inclusion $\Cat[\VV] \hookrightarrow \AlgCat[\VV]$ is a right adjoint and preserves products, it follows that the symmetric monoidal structure on $\Cat[\VV]$ is also Cartesian.}  presentably  symmetric monoidal $\infty$-category of $(\infty, k)$-categories $\CatInfty{k}\coloneqq\Cat[\CatInfty{k-1}]$.
\end{definition}
\begin{remark}
In~\cite[Thm. 1.2]{Haugseng_2015}, Haugseng showed that the $\infty$-category $\CatInfty{k}$ from \cref{def:CatInfty-inductive-def} satisifes the axioms of Barwick and Schommer-Pries~\cite{barwick-schommer} and hence is equivalent to most other known models of the $\infty$-category of $(\infty,k)$-categories. 
\end{remark}

\begin{observation}
Consider the diagram
\[ \begin{tikzcd}[column sep=1.5cm]
\Cat_\infty
\arrow[bend left=45]{r}{|-|}
\arrow[hookleftarrow]{r}[xshift=-0.15cm, yshift=0.2cm]{\bot}[swap, xshift=-0.15cm, yshift=-0.2cm]{\bot}[description, xshift=-0.15cm]{i}
\arrow[bend right=45]{r}[swap]{\iota_0}
&
\Spaces
\end{tikzcd} \]
of adjunctions, where $i$ denotes the fully faithful inclusion of spaces as $\infty$-groupoids. Since all these functors preserve finite products\footnote{One can see that the functor $\Cat_\infty \xra{|-|} \Spaces$ preserves finite products e.g.\! by observing that it can be computed as the geometric realization of complete Segal spaces and $\Deltaop$ is sifted \cite[Cor. 4.2.3.5]{HTT}.} , they are all symmetric monoidal.  By applying $\Cat[-]$ iteratively, we obtain an analogous diagram
\[ \begin{tikzcd}[column sep=1.5cm]
\CatInfty{k+1}
\arrow[bend left=30]{r}[xshift=0.1cm]{|-|_k}
\arrow[hookleftarrow]{r}[xshift=-0.05cm, yshift=0.2cm]{\bot}[swap, xshift=-0.05cm, yshift=-0.2cm]{\bot}[description, xshift=-0.05cm]{i_{k+1}}
\arrow[bend right=30]{r}[swap, xshift=0.1cm]{\iota_k}
&
\CatInfty{k}
\end{tikzcd} \]
of symmetric monoidal adjoint functors for any $k \geq 0$ (with $i_{k+1}$ fully faithful). Thereafter, for any $j \geq k \geq 0$ we obtain an analogous diagram
\begin{equation}
\label{eq:adjns-betw-kcats-and-ncats}
\begin{tikzcd}[column sep=1.5cm]
\CatInfty{j}
\arrow[bend left=30]{r}[xshift=0.05cm]{|-|_k}
\arrow[hookleftarrow]{r}[yshift=0.2cm]{\bot}[swap, yshift=-0.2cm]{\bot}[description]{i_j}
\arrow[bend right=30]{r}[swap]{\iota_k}
&
\CatInfty{k}
\end{tikzcd}
\end{equation}
of symmetric monoidal adjoint functors by composition. 
\end{observation}

\begin{definition}
In diagram \eqref{eq:adjns-betw-kcats-and-ncats}, we refer to $|-|_k$ as the \bit{$(\infty,k)$-category completion} functor and to $\iota_k$ as the \bit{maximal sub-$(\infty,k)$-category} functor.\footnote{In the case that $k=0$, we may instead respectively refer to these as the \textit{$\infty$-groupoid completion} and \textit{maximal sub-$\infty$-groupoid} (or even simply \textit{maximal subgroupoid}) functors.} For brevity, we may omit the fully faithful inclusion functor $i_j$ from our notation, implicitly considering an $(\infty,k)$-category as an $(\infty,j)$-category with no noninvertible $i$-morphisms for any $i > k$.
\end{definition}
\begin{observation}
    For any $j \geq k \geq 0$, the inclusion $\CatInfty{k} \xhookra{i_j} \CatInfty{j}$ identifies $\CatInfty{k}$ as the full subcategory of $\CatInfty{j}$ on those $(\infty,j)$-categories whose $i$-morphisms are all invertible for all $i > k$ \cite[Prop. 6.1.7(iv)]{GH13}.\footnote{Note that the functor $[1] = c_1 \ra c_0 = \pt$ is a localization at the universal $1$-morphism: it is merely a condition that a functor $c_1 \xra{\alpha} \cC$ admit an extension along it, namely that $\alpha$ selects an equivalence. By induction (using the universal property of $\Sigma[-]$), it follows that the functor $c_{n+1} \xra{\Sigma^n[c_1 \ra c_0]} c_n$ is a localization at the universal $(n+1)$-morphism: it is merely a condition that a functor $c_{n+1} \xra{\alpha} \cC$ admit an extension along it, namely that $\alpha$ selects an invertible $(n+1)$-morphism.} We use this fact without further comment.
    \end{observation}

\begin{definition}\label{def:c_n-and-partial-c_n}
The \bit{$n$-cell} (or \bit{walking $n$-morphism}) is the $(\infty,n)$-category 
$c_n \coloneqq \Sigma^n[\pt] \in \CatInfty{n}$.\footnote{Here, $\VV \xra{\Sigma[-]} \Cat[\VV]$ denotes the ``categorical suspension'' functor (see \Cref{subsec:enriched-infty-cats}). As the name suggests, $c_n$ is the free $(\infty,n)$-category on an $n$-morphism.} Its \bit{boundary} (or the \bit{walking pair of parallel $(n-1)$-morphisms}\footnote{ Note that $\partial c_1 = S^0 = c_{0} \sqcup_{\partial c_{0}} c_{0}$. Since $\Sigma$ commutes with colimits, we see that $\partial c_n = c_{n-1} \sqcup_{c_{n-1}} c_{n-1}$ for $n \geq 1$, i.e., $\partial c_n$ indeed corepresents pairs of parallel $(n-1)$-morphisms.}) is the $(\infty,n)$-category 
$\partial c_n \coloneqq \partial \Sigma^n[\pt] \coloneqq \Sigma^n[\emptyset]$ (which is in fact an $(\infty,n-1)$-category). 
We use both notations interchangeably, depending on our desired emphasis. We also introduce the notation
\[
j_n
\colon
\partial c_n
\coloneqq
\Sigma^n[\emptyset]
\xra{\Sigma[\emptyset \longra \pt]}
\Sigma^n[\pt]
\eqqcolon
c_n
\]
for the inclusion, which corepresents the functor taking an $n$-morphism to its source and target (which are parallel $(n-1)$-morphisms). 
\end{definition}

\begin{observation}\label{obs:truncation-of-cells}
    For $n \geq 0$, it follows from \cite[Lem. 6.1.9]{GH13} that the map  $|j_n|_0 \colon |\partial c_n|_0 \to |c_n|_0$ is equivalent to the map $ S^{n-1} \to \pt.$ By induction, it follows that for $k < n$, $|j_n|_k \colon |\partial c_n|_k \to |c_n|_k$ is equivalent to the map $\Sigma^k[S^{n-k-1}] \to \Sigma^k[\pt] = c_k$ induced by $S^{n-k-1} \to \pt$. 
\end{observation}

\begin{nota}\label{nota:n-hom-notation}
Let $\alpha \colon \partial c_k \to \cC$ be a pair of parallel $(k-1)$-morphisms in an $(\infty,k)$-category $\cC$. The \emph{space of $k$-morphisms filling $\alpha$} is \[\kHom_{\cC}(\alpha) \coloneqq \Hom_{\CatInfty{k}}(c_k, \cC) \times_{\Hom_{\CatInfty{k}}(\partial c_k, \cC)} \{\alpha\}.\]
\end{nota}

\begin{observation}\label{obs:space-of-lift-of-Sigma-k}
    Given a space $X \in \Spaces$ and $k \geq 0$, the map $\emptyset \to X$ induces a functor of $(\infty,k)$-categories 
    $\partial c_k = \Sigma^k[\emptyset] \to \Sigma^k [X]$.
    It then follows from the universal property of $\Sigma$ that for any $\cC \in \CatInfty{k}$ and any pair of parallel $(k-1)$-morphisms $\alpha\colon \partial c_k \to \cC$, we obtain an equivalence of spaces 
\[
        \Hom_{\CatInfty{k}}(\Sigma^k[X], \cC) \times_{\Hom_{\CatInfty{k}}(\partial c_k, \cC)} \{\alpha\} \simeq \Hom_{\Spaces}(X,\kHom_{\cC}(\alpha)).
\]
\end{observation}

\subsection{Truncatedness and connectedness}\label{subsec:trun-and-conn}
Here we recall the standard definition of  the ($n$-connected, $n$-truncated) factorization system on the $\infty$-category $\Spaces$ of spaces. This will be the base case for our factorization systems on $(\infty, k)$-categories. 
\begin{definition}
\label{defn:nconn-and-ntrunc}
For any $n \geq 0$, a space $X \in \Spaces$ is called
\begin{itemize}
\item \bit{$n$-connected} if $X$ is connected and if $\pi_i(X,x) = 0$ for all $i \leq n$ and all $x \in X$ and
\item \bit{$n$-truncated} if $\pi_i(X,x) = 0$ for all $i > n$ and all $x \in X$.
\end{itemize}
We extend this to the case that $n=-1$ by declaring that $X$ is
\begin{itemize}
\item \bit{$(-1)$-connected} if it is nonempty and
\item \bit{$(-1)$-truncated} if it is either empty or contractible,
\end{itemize}
and to the case that $n=-2$ by declaring that $X$ is
\begin{itemize}
\item always \bit{$(-2)$-connected} and
\item \bit{$(-2)$-truncated} if it is contractible.
\end{itemize}
For any $n \geq -2$, we declare that a map of spaces is \bit{$n$-connected}\footnote{Readers be aware that it is also common to refer to $n$-connected spaces and maps as ``$(n+1)$-connective'', see for example \cite[Terminology]{HTT}.} (resp.\! \bit{$n$-truncated}) if its fibers are all such. By \cite[Ex. 5.2.8.16]{HTT} the classes of ($n$-connected, $n$-truncated) maps form a factorization system of small generation on $\Spaces$, generated by the single morphism $S^{n+1} \to \pt$. See also \cref{ex:nconn-ntrunc-fs-on-Spaces}.
\end{definition}

\begin{example}
\label{ex:nconn-and-ntrunc}
To obtain examples, the following explicit alternative descriptions of $n$-connectedness and $n$-truncatedness for low values of $n$ are useful.
\begin{enumerate}

\item

A space is $0$-connected if and only if it is connected (and in particular nonempty), and it is $1$-connected if and only if it is simply connected (and in particular connected).

\item

A map of spaces is always $(-2)$-connected, and it is $(-1)$-connected if and only if it is surjective.

\item

A space is $n$-truncated if and only if it is an $n$-type, e.g.\! it is $0$-truncated if and only if it is discrete.

\item\label{item-ex:nconn-and-ntrunc-truncmaps}

A map of spaces is $(-2)$-truncated if and only if it is an equivalence, it is $(-1)$-truncated if and only if it is a monomorphism, and it is $0$-truncated if and only if it is a covering map (in the classical sense).

\end{enumerate}
\end{example}

\begin{observation}
\label{obs:basic-facts-regarding-nconnected-and-ntruncated-spaces-and-maps}
We note the following basic facts, which we use without further comment.
\begin{enumerate}

\item

For any $n \geq -2$, a space $X$ is $n$-connected (resp.\! $n$-truncated) if and only if the map $X \ra \pt$ is such.

\item

For any $n \geq -2$, a space is both $n$-connected and $n$-truncated if and only if it is contractible, and hence a map is both $n$-connected and $n$-truncated if and only if it is an equivalence.

\item

For any $n \geq -2$, we have the implications
\[
\text{$n$-connected} \Longleftarrow \text{$(n+1)$-connected}
\qquad
\text{and}
\qquad
\text{$n$-truncated} \Longrightarrow \text{$(n+1)$-truncated}
\]
for spaces and hence also for maps of spaces.

\item

For any $n \geq -2$, both $n$-connected and $n$-truncated maps are stable under base change.

\item\label{item-obs:basic-facts-regarding-nconnected-and-ntruncated-spaces-and-maps-use-LES}

By the long exact sequence in homotopy groups, for any $n \geq -1$, a map $X \xra{f} Y$ of spaces is
\begin{itemize}
\item $n$-connected if and only if for every $x \in X$ the map $\pi_i(X,x) \xra{\pi_i(f)} \pi_i(Y,f(x))$ is
\begin{itemize}
\item an isomorphism for all $0 \leq i < n+1$ and
\item surjective for $i = n+1$,
\end{itemize}
and
\item $n$-truncated if and only if for every $x \in X$ the map $\pi_i(X,x) \xra{\pi_i(f)} \pi_i(Y,f(x))$ is
\begin{itemize}
\item an isomorphism for all $i > n+1$ and
\item injective for $i = n+1$.
\end{itemize}
\end{itemize}
\end{enumerate}
\end{observation}

Throughout, we will use the following cancellation properties generalizing well-known facts about surjections and injections of sets.

\begin{lemma}
\label{lem:cancellation-for-connectedness-and-truncatedness}
Suppose that $A \xra{f} B \xra{g} C$ are composable maps of spaces, and let $n \geq -2$.
\begin{enumerate}
\item\label{item-lem:cancellation-for-connectedness}
If $g$ is $(n+1)$-connected and $gf$ is $n$-connected, then $f$ is $n$-connected. 
\item\label{item-lem:cancellation-for-truncatedness}
If $g$ is $(n+1)$-truncated and $gf$ is $n$-truncated, then $f$ is $n$-truncated.
\end{enumerate}
\end{lemma}
\begin{proof}
Since connectivity and truncatedness of maps of spaces are defined fiberwise, we may henceforth assume that $C=*$. In this case, (1) and (2) become: \begin{enumerate}
\item $f \colon A \to B$ is a map from an $n$-connected space to an $(n+1)$-connected space, then the fibers of $f$ are $n$-connected. 
\item  If $f\colon A \to B$ is a map from an $n$-truncated space to an $(n+1)$-truncated space, then the fibers of $f$ are $n$-truncated. 
\end{enumerate}
For $n=-2$, these statements are obvious, for higher $n$ they can be easily verified from the long exact sequence of homotopy groups associated to $f$. 
\end{proof}

\begin{proposition}
\label{prop:pullback-square-via-conn-and-trunc}
Fix any $n \geq -2$. A commutative square of spaces
\[ \begin{tikzcd}[column sep=2cm]
A
\arrow{r}{n\trunc}
\arrow{d}[swap]{n\conn}
&
B
\arrow{d}{(n+1)\conn}
\\
C
\arrow{r}[swap]{(n+1)\trunc}
&
D
\end{tikzcd} \]
 in which the maps are truncated and connected as indicated is necessarily a pullback square.
\end{proposition}

\begin{proof}
Consider the commuting diagram
\[ \begin{tikzcd}[column sep=2cm, row sep=1cm]
A
\arrow[bend left=5]{rrd}[sloped]{n\trunc}
\arrow[bend right=5]{rdd}[sloped, swap]{n\conn}
\arrow{rd}
\\
&
B \times_D C
\arrow{r}[swap]{(n+1)\trunc}
\arrow{d}{(n+1)\conn}
&
B
\arrow{d}{(n+1)\conn}
\\
&
C
\arrow{r}[swap]{(n+1)\trunc}
&
D.
\end{tikzcd}
~ \]
where we have used that truncated and connected maps are stable under pullback. 
By \Cref{lem:cancellation-for-connectedness-and-truncatedness}, the map $A \ra B \times_D C$ is both $n$-connected and $n$-truncated, and so is an equivalence.
\end{proof}

\subsection{Factorization systems for \texorpdfstring{$(\infty,k)$}{(infinity,k)}-categories}
\label{subsec:fact-for-inf-k-cat}

In this subsection we define $n$-surjective and $n$-faithful morphisms in $\CatInfty{k}$ and prove in \cref{thm:nsurj-nfaithful-fs-for-infty-k-cats} that they form factorization systems, using the main result \cref{thm:fs_and_enriched_cat} of \cref{app:f-s-for-enriched-cats}. Furthermore, we establish various properties of these factorization systems.

\begin{definition}\label{def:nsurj-and-nfaith}
Consider a morphism $\cC \xra{F} \cD$ in $\CatInfty{k}$ for some $k \geq 0$ and let $n\geq -2$.
\begin{enumerate}

\item
We declare that any $F$ is \bit{$(-2)$-surjective} and that $F$ is \bit{$(-2)$-faithful} if it is an equivalence.

\item
If $k=0$, we say that $F$ is \bit{$n$-surjective} if it is $n$-connected and \bit{$n$-faithful} if it is $n$-truncated.\footnote{We nevertheless continue to use the terms ``$n$-connected'' and ``$n$-truncated'' (referring to maps of spaces) since they play a fundamental role as the base case of our inductive definitions.}

\item
For $n > -2$ and $k > 0$, we inductively define $F$ to be
\begin{enumerate}

\item
\bit{$n$-surjective} if it is surjective on objects and for every $c,c' \in \cC$ the morphism $\eHom_\cC(c,c') \ra \eHom_\cD(Fc,Fc')$ in $\CatInfty{k-1}$ is $(n-1)$-surjective, and

\item
\bit{$n$-faithful} if for every $c,c' \in \cC$ the morphism $\eHom_\cC(c,c') \ra \eHom_\cD(Fc,Fc')$ in $\CatInfty{k-1}$ is $(n-1)$-faithful.

\end{enumerate}
\end{enumerate}
\end{definition}

\begin{example}
    \label{exa:surj}
We give a few explicit alternative descriptions of $n$-surjectivity and $n$-faithfulness for low values of $n$ (and any $k \geq 0$).
\begin{enumerate}

\item A functor is $(-1)$-surjective if and only if it is surjective on objects.

\item A functor is $(-1)$-faithful if and only if it is fully faithful.

\item A functor is $0$-surjective if and only if it is surjective on objects and on $1$-morphisms.

\item A functor is $0$-faithful if and only if  the induced functors on hom-categories are fully faithful. 
\end{enumerate}
\end{example}

\begin{notation}
    \label{def:faithful}
To simplify our terminology, we refer to a $0$-faithful functor simply as \bit{faithful}.
\end{notation}

\begin{remark}
A functor of $(\infty,0)$-categories is faithful if and only if it is a covering map (recall \Cref{ex:nconn-and-ntrunc}.\eqref{item-ex:nconn-and-ntrunc-truncmaps}). Hence, in general one may think of a faithful functor as a sort of ``directed covering map''. 
\end{remark}

 Recall the cancellation property \cref{lem:cancellation-for-connectedness-and-truncatedness} of truncated and connected maps of spaces. The second part of \cref{lem:cancellation-for-connectedness-and-truncatedness}  generalizes to functors of $(\infty,k)$-categories:

\begin{lemma}\label{lem:cancellation-for-surj-and-faithful}
Let $k \geq 0$ and $n\geq -2$, and consider composable functors $\cA \xra{F} \cB \xra{G} \cC$ of $(\infty,k)$-categories.
Then, if $G$ is $(n+1)$-faithful and $GF$ is $n$-faithful, then $F$ is $n$-faithful.
\end{lemma}
\begin{proof}
The case $n=-2$ is the straight-forward statement that a section of a fully faithful functor is an equivalence. 
For $n\geq -1$, we induct on $k \geq 0$. The base case  $k=0$ is Lemma~\ref{lem:cancellation-for-connectedness-and-truncatedness}. For $k \geq 1$,  $F$ being $n$-faithful is equivalent to proving that for $a, a' \in \cA$ the induced functor of $(\infty,k-1)$-categories $\eHom_{\cA}(a,a') \to \eHom_{\cB}(Fa, Fa')$ is $(n-1)$-faithful. 
Since the composable sequence of functors of $(\infty,k-1)$-categories $\eHom_{\cA}(a,a') \to \eHom_{\cB}(Fa, Fa') \to \eHom_{\cC}(GFa, GFa')$ the last functor is $n$-faithful, and the composite is $(n-1)$-faithful by assumption, the first functor is $(n-1)$-faithful by induction. 
\end{proof}

 \begin{remark} 
 The first part of \cref{lem:cancellation-for-connectedness-and-truncatedness} does \emph{not} generalize to higher categories: 
 If $G$ is $(n+1)$-surjective and $GF$ is $n$-surjective, then it does not necessarily follow that $F$ is $n$-surjective. For example, let $\cA$ be the category freely generated by two objects $a$ and $b$ and two morphisms $f\colon a \to b$ and $g \colon b \to a$. Then, in the composite $\pt \to \cA \to \pt$, the second functor is $0$-surjective (i.e. surjective on objects and hom-spaces), and the composite is an equivalence (in particular $(-1)$-surjective), but the first functor is not surjective on objects and hence not $(-1)$-surjective. 
 \end{remark}

Just as with the ($n$-connected, $n$-truncated) factorization system on the $\infty$-category of spaces, the $n$-surjective and $n$-faithful functors form a factorization system on the $\infty$-category of $(\infty,k)$-categories:
\begin{theorem}
\label{thm:nsurj-nfaithful-fs-for-infty-k-cats}
Let $k \geq 0$ and  $n \geq -2$.
\begin{enumerate}

\item

The pair ($n$-surjective functors, $n$-faithful functors) defines a factorization system on the $\infty$-category $\CatInfty{k}$ of $(\infty,k)$-categories.

\item

This factorization system is compatible with the Cartesian symmetric monoidal structure on $\CatInfty{k}$.

\item \label{item-thm:cat-f-s-generators}

This factorization system is of small generation. More specifically,
\begin{enumerate}

\item if $n+2 \leq k$ then it is generated by the set $\{ \partial c_i \ra c_i \}_{n+2 \leq i \leq k}$, and

\item if $n+2 \geq k$ then it is generated by the single morphism $\{ \Sigma^k[S^{n-k+1}] \ra \Sigma^k[\pt] \eqqcolon c_k \}$.

\end{enumerate}
\end{enumerate}
\end{theorem}

\begin{proof}
In the base case that $k = 0$, this factorization system is recorded as \cite[Ex.~5.2.8.16]{HTT}, which is easy to check is compatible with the Cartesian symmetric monoidal structure  and generated by the single morphism $\{ S^{n+1} \ra \pt \}$. So, let us assume that $k > 0$.

If $n = -2$, then this is the trivial factorization system $(\CatInfty{k},\CatInfty{k}^\simeq)$ (as in \Cref{ex:trivial-fs}). Moreover, it is trivially compatible with the Cartesian symmetric monoidal structure, and it is also clearly generated by the set $\{ \partial c_i \ra c_i \}_{0 \leq i \leq k}$: by induction (and the universal property of the categorical suspension functor $\Sigma[-]$), a morphism $\cC \ra \cD$ is right orthogonal to this set if and only if it is an equivalence on maximal subgroupoids and on hom-$(\infty,k-1)$-categories.

From here, in the case that $n > -2$ (and $k > 0$) the claim follows by applying \Cref{thm:fs_and_enriched_cat} inductively (varying both $k$ and $n$ simultaneously).
\end{proof}

\begin{definition}\label{def:factn-on-inf-k-cats}
Given a morphism $\cC \xra{F} \cD$ in $\CatInfty{k}$, we refer to its unique factorization\footnote{This is the unique $(\infty,k)$-category equipped with a factorization $\cC \ra \Fact_n(F) \ra \cD$ of $F$ via an $n$-surjective functor followed by an $n$-faithful functor (see \Cref{obs:factorizations-from-fs-are-unique} and \Cref{notn:Fact-for-factorizn-from-fs}).} guaranteed by the ($n$-surjective, $n$-faithful) factorization system as its \bit{$n$-factorization}, and denote it by
\[
\Fact_n(F)
\coloneqq
\Fact_{\textup{($n$-surjective, $n$-faithful)}}(F)
\in \CatInfty{k}.
\]
\end{definition}

\begin{warning}
    \label{warn:nfaithful-neq-ntruncated}
    
    Recall from~\cite[Prop. 4.6]{gep17}  that any presentable $\infty$-category admits a factorization system of ($n$-connected, $n$-truncated) morphisms. We warn the reader that for any $k \geq 1$ and any $n > -2$, the ($n$-surjective, $n$-faithful) factorization system on $\CatInfty{k}$ of \Cref{thm:nsurj-nfaithful-fs-for-infty-k-cats} does \textit{not} coincide with this ($n$-connected, $n$-truncated) factorization system induced from presentability of $\CatInfty{k}$. This can already be seen in the case that $k=1$: 
    A functor $\cC \ra \cD$ of $(\infty,1)$-categories is $n$-truncated if and only if it is so on spaces of objects and morphisms.\footnote{For general $k$, a morphism is $n$-truncated if and only if it is so on spaces of $i$-morphisms for all $0 \leq i \leq k$.} Indeed, we have the diagram of irreversible implications
    \[ \begin{tikzcd}
    \textup{equivalence}
    \arrow[equals]{d}
    &
    \textup{monomorphism}
    \arrow[equals]{d}
    \\[-0.25cm]
    \textup{$(-2)$-truncated}
    \arrow[equals]{d}
    \arrow[Rightarrow]{r}
    \arrow[Rightarrow]{rd}
    &
    \textup{$(-1)$-truncated}
    \arrow[Rightarrow]{r}
    \arrow[Rightarrow]{rd}
    &
    \textup{$0$-truncated}
    \arrow[Rightarrow]{r}
    \arrow[Rightarrow]{rd}
    &
    \textup{$1$-truncated}
    \arrow[Rightarrow]{r}
    \arrow[Rightarrow]{rd}
    &
    \cdots
    \\
    \textup{$(-2)$-faithful}
    \arrow[Rightarrow]{r}
    &
    \textup{$(-1)$-faithful}
    \arrow[Rightarrow]{r}
    \arrow[Rightarrow]{u}
    &
    \textup{$0$-faithful}
    \arrow[Rightarrow]{r}
    \arrow[Rightarrow]{u}
    &
    \textup{$1$-faithful}
    \arrow[Rightarrow]{r}
    \arrow[Rightarrow]{u}
    &
    \cdots
    \\[-0.25cm]
    \textup{equivalence}
    \arrow[equals]{u}
    &
    \textup{fully faithful}
    \arrow[equals]{u}
    &
    \textup{faithful}
    \arrow[equals]{u}
    \end{tikzcd} \]
    for morphisms in $\CatInfty{1}$. For example, a $(-1)$-truncated functor of $(\infty,1)$-categories, i.e. a monomorphism in $\CatInfty{1}$, is a functor $F\colon  \cC \to \cD$ which for any two objects $c, c' \in \cC$ induces a $(-1)$-truncated map $\Hom_{\cC}(c,c') \hookrightarrow \Hom_{\cD}(Fc, Fc')$ which restricts to an equivalence between the full subspaces of isomorphisms $\Hom_{\iota_0\cC}(c,c') \to \Hom_{\iota_0\cD}(Fc,Fc')$. In particular, any $(-1)$-faithful, i.e. fully faithful, functor is $(-1)$-truncated, and any $(-1)$-truncated functor is $0$-faithful, but neither of these implications is reversible. In particular, a $0$-faithful functor  $F \colon \cC \to \cD$ does not necessarily exhibit $\cC$ as a \textit{subcategory} of $\cD$ in the sense of \cref{subsubsection:monos-and-subcats}.

    Homwise iterating these observations, a similar diagram applies for $(\infty,k)$-categories with $k+1$ rows corresponding to the enrichment-depth at which functors between higher hom-categories are required to be truncated rather than faithful. 
    \end{warning}

\begin{observation}
\label{obs:nsurj-and-nfaithful-indep-of-k}
Fix any $j \geq k \geq 0$ and $n \geq -2$.
By the description of the generators in \cref{thm:nsurj-nfaithful-fs-for-infty-k-cats}.\eqref{item-thm:cat-f-s-generators}, and their truncations in \cref{obs:truncation-of-cells}, we see that
the inclusion $\CatInfty{k} \xhookra{i_j} \CatInfty{j}$ 
preserves and detects the ($n$-surjective, $n$-faithful) factorization system.\footnote{Using the notation of \Cref{defn-fs}, this is to say that the diagram $\CatInfty{0} \xra{i_1} \CatInfty{1} \xra{i_2} \cdots$ lies in $\hat{\Cat}_\infty^{\fs,\cL,\cR}$ when we equip all of these $\infty$-categories with their ($n$-surjective, $n$-faithful) factorization systems.} In particular, a map of spaces $f \colon X \to Y$ is $n$-connected (or $n$-truncated) if and only if it is $n$-surjective (or $n$-faithful) as a map of $(\infty, k)$-categories for any $k \geq 0$. 

It follows that $n$-factorizations in $\CatInfty{k}$ remain so in $\CatInfty{j}$. It also follows that the left adjoint $\CatInfty{j} \xra{|-|_k} \CatInfty{k}$ preserves the notion of $n$-surjectivity and that the right adjoint $\CatInfty{j} \xra{\iota_k} \CatInfty{k}$ preserves the notion of $n$-faithfulness. 
\end{observation}

In fact, the maximal sub-$(\infty,k)$-category functor $\CatInfty{j} \xra{\iota_k} \CatInfty{k}$ also preserves $n$-surjectivity provided $n$ lies outside of the interval $[k, j)$:

\begin{lemma}
\label{lem:iota-k-preserves-n-surjectivity-for-n-not-equal-k}
For $j \geq k \geq 0$ and either $n\geq j$ or $k>n \geq -2$,  the maximal sub-$(\infty,k)$-category functor  $\iota_k \colon \CatInfty{j} \xra{\iota_k} \CatInfty{k}$ preserves $n$-surjective functors. 
\end{lemma}
\begin{proof}

The case $n=-2$ is trivial; we henceforth assume $n \geq -1$. Similarly, the case $j=k$ is trivial. It suffices to prove the statement for $j=k+1$, the general case follows from the observation that $\iota_k = \iota_{k} \iota_{k+1} \cdots \iota_{j-1}$. Thus we need to prove that if $k\geq 0$ and $k \neq n \geq -1,$ then $\iota_k \colon \CatInfty{k+1} \to \CatInfty{k}$ preserves $n$-surjective functors. We prove this statement by induction on $k \geq 0$:

For the basecase $k =0$, and hence $0 \neq n\geq -1$, we show that $\iota_0 \colon \CatInfty{1} \to \CatInfty{0}= \Spaces$ preserves $n$-surjective functors. 
Consider an $n$-surjective functor $F \colon \cC \to \cD$ between $(\infty,1)$-categories. We claim that $\iota_0 F \colon \iota_0 \cC \to \iota_0 \cD$ is $n$-connected. For $n=-1$ this follows since if $F$ is surjective on objects, then $\iota_0 F$ is surjective on $\pi_0$ and hence $(-1)$-connected. For the remaining cases $n\geq 1$, it  suffices to show that $\iota_0 F$ induces $(n-1)$-connected maps on hom-spaces. Let $c,d\in \cC$ and consider the commuting diagram of spaces
\[
        \begin{tikzcd}
            \Hom_{\iota_0 \cC}(c, d) \ar[r, "\iota_0 F"] \ar[d, hook] & \Hom_{\iota_0 \cD}(Fc, Fd) \ar[d, hook] \\ 
            \Hom_{\cC}(c,d) \ar[r, "F"] & \Hom_{\cD}(Fc, Fd).
        \end{tikzcd}
\]
By assumption, the bottom horizontal map is $(n-1) \geq 0$-connected. Since the vertical maps are inclusions of components, to show that the top horizontal map is $(n-1) \geq 0$-connected, it suffices to show that the top horizontal map is surjective on $\pi_0$. 
Given any $\beta \in \Hom_{\iota_0 \cD}(Fc, Fd)$, i.e. an isomorphism between $Fc$ and $Fd$ in $\cD$, let $\alpha \in \Hom_{\cC}(c,d)$ be a lift of $\beta$ in $\cC$. We claim that $\alpha$ is an isomorphism: let $\beta^{-1} \in \Hom_{\iota_0 \cD}(Fd, Fc)$ be an inverse of $\beta$ and $\overline{\alpha} \in \Hom_{\cC}(d,c)$ a lift of $\beta^{-1}$. Then $\alpha \circ \overline{\alpha}$ and $\overline{\alpha} \circ \alpha$ are in the same component of $\Hom_{\cC}(c,c)$ and $\Hom_{\cC}(d,d)$ as the respective  identities $\id_c$ and $\id_d$ as $F$ is $\geq 1$-connected and hence induces bijections on the sets of components of all hom-spaces. Therefore, $\overline{\alpha}$ is an inverse of $\alpha$ and thus $\alpha$ lifts to $\Hom_{\iota_0 \cC}(c,d)$. 

For the induction step, let $k \geq 1$ and hence $k \neq n \geq -1$. Given an $n$-surjective functor $F \colon \cC \to \cD$ in $\CatInfty{k+1}$, the functor $\iota_k F$ is surjective on objects since $F$ is. For objects $c,d\in \cC$, note that the component $(\iota_k F)_{c,d} \colon \eHom_{\iota_k \cC}(c, d) \to \eHom_{\iota_k \cD}(\iota_k F c, \iota_k Fd)$ agrees with the functor $\iota_{k-1} (F_{c,d})$ which is $(n-1)$-surjective by induction.
\end{proof}

\begin{remark}
The statement of \cref{lem:iota-k-preserves-n-surjectivity-for-n-not-equal-k} is false when $j >n \geq k$. For example, let $\cA$ be the free $(\infty,1)$-category generated by two objects $a$ and $b$, a morphism $f\colon a\to b$ and a morphism $g \colon b \to a.$ Then, the unique functor $F\colon \cA \to \pt$ is $0$-surjective, but $\iota_0 F  \colon \iota_0 \cA = S^0 \to \pt$ is not $0$-connected.
\end{remark}

\begin{corollary}\label{cor:iota-k-commutes-with-fact-n-for-n-not-k}
    Given $j \geq k \geq 0$ and either $n \geq j$ or $k>n \geq -2$, and $F \colon \cA \to \cB$ in $\CatInfty{j}$. Then, the maximal sub-$(\infty,k)$-category functor $\iota_k \colon \CatInfty{j} \xra{\iota_k} \CatInfty{k}$ preserves $n$-factorizations. That is, if $\Fact_n (F)$ is the factorization of $F$ with respect to the ($n$-surjective, $n$-faithful) factorization system, then the induced factorization $
        \iota_k \cA \to \iota_k \Fact_n(F) \to \iota_k \cB
$
    realizes $\iota_k \Fact_n(F)$ as the factorization $\Fact_n (\iota_k F)$ of $\iota_k F$ with respect to the ($n$-surjective, $n$-faithful) factorization system on $\CatInfty{k}$.
\end{corollary}

\begin{proof}
    The statement follows from the fact that $\iota_k$ preserves both $n$-surjectivity (\cref{lem:iota-k-preserves-n-surjectivity-for-n-not-equal-k}) and $n$-faithfulness (\cref{obs:nsurj-and-nfaithful-indep-of-k}).
\end{proof}

Throughout, we will also repeatedly use the following simple observation:

\begin{lemma}\label{lem:counit-for-iota-k-is-k-1-surjective}
For $j > k \geq 0$ and $\cC \in \CatInfty{j}$, the inclusion of the maximal sub-$(\infty,k)$-category $\iota_k \cC \to \cC$ is $(k-1)$-surjective. 
\end{lemma}
\begin{proof}
Factoring $\iota_k \cC \to \iota_{k+1} \cC \to \ldots \to \iota_{j-1} \cC \to \cC$, it suffices to prove the case $j=k+1$. 
For $k=0$ and $\cC \in \CatInfty{1}$,  the functor $\iota_0 \cC \to \cC$ is surjective on objects, i.e. $(-1)$-surjective. 
For $k\geq 1$ and $\cC \in \CatInfty{k+1}$, the functor $\iota_k \cC \to \cC$ is surjective on objects and by induction homwise $(k-2)$-surjective, hence $\iota_k \cC \to \cC$ is $(k-1)$-surjective. 
\end{proof}

We end this subsection with the following useful proposition, generalizing~\cite[Prop. 4.2.8]{SY19}.

\begin{prop}\label{prop:space-of-lift-of-n-surj-m-faithful}
Let $k \geq 0$ and $m \geq n \geq -2$. Given a commuting (solid) square in $\CatInfty{k}$ 
\begin{equation}\label{eq:another-space-of-lifts}
        \begin{tikzcd}
            \cA \ar[d, "F"] \ar[r] & \cC \ar[d, "G"] \\ 
            \cB \ar[r] \ar[ru, dashed] & \cD
        \end{tikzcd}
\end{equation}
    where $F$ is $n$-surjective and $G$ is $m$-faithful. Then the space of (dashed) lifts is $(m-n-2)$-truncated.
\end{prop}
\begin{proof}
        We induct on $k \geq 0$. The base case  $k = 0$ is proven in \cite[Prop. 4.2.8]{SY19}. 
        For $k>0$, fix a functor $G\colon \cC \to \cD$ which is $m$-faithful and let $S$ be the class of morphisms $F$ in $\CatInfty{k}$ for which the space of lifts~\eqref{eq:another-space-of-lifts} against $G$ is $(m-n-2)$-truncated. We will now prove that $S$ contains the $n$-surjective functors. 
        Since $S$ contains equivalences, and is closed under composition, small colimits and cobase change, it forms a saturated class of morphisms (\cref{def:saturated}).  By  \cref{prop:fact-system-of-small-generation-on-presentable}, to show that $S$ contains all $n$-surjective morphisms, it suffices to show that it contains the generators of the left class; i.e. by \cref{thm:nsurj-nfaithful-fs-for-infty-k-cats}.\eqref{item-thm:cat-f-s-generators} the functors $\{\partial c_{i} \to c_{i}\}_{n+2 \leq i \leq k}$ for $n +2 <  k$ and the functor $\Sigma^k[S^{n-k+1}] \to c_k$  for $n+2 \geq k$.

We first consider the case that $n + 2 \geq k$, where we need to show that a commuting diagram of $(\infty,k)$-categories
 \begin{equation}\label{eq:diagram-huh}
        \begin{tikzcd}
            \Sigma^k[S^{n-k+1}]  \arrow[r] \arrow[d] & \cC \ar[d, "G"] \\
            \Sigma^k[\pt] \arrow[ur, dashed]  \ar[r] &  \cD.
            \end{tikzcd}
    \end{equation}
    has $(m-n-2)$-truncated space of lifts. 
    Let $\alpha$ denote the composite $\partial c_k = \Sigma^k[\emptyset] \to \Sigma^k[S^{n-k+1}] \to \cC$, picking out a pair of parallel $(k-1)$-morphisms. 
    By \cref{obs:space-of-lift-of-Sigma-k}, the space of lifts of \eqref{eq:diagram-huh} is equivalent to the space of lift of the following diagram in spaces
\begin{equation}\label{eq:diagram-huh-2}
        \begin{tikzcd}
            S^{n-k+1}  \arrow[r] \arrow[d, "(n-k)\conn"'] & \kHom_{\cC}(\alpha)\ar[d, "(m-k)\trunc"] \\
            \pt \arrow[ur, dashed]  \ar[r] & \kHom_{\cD}(G\alpha).
            \end{tikzcd} 
\end{equation}
By definition of faithfulness, if $G$ is $m$-faithful, then the right vertical map is $(m-k)$-truncated for any $\alpha  \colon  \partial c_k \to \cC$. Hence, it follows from~\cite[Prop. 4.2.8]{SY19} that the space of lifts of~\eqref{eq:diagram-huh-2} is $(m-n-2)$-truncated. 

Now we consider the case $n+2 < k$, where we need to show that $\partial c_i \to c_i$ is in $S$ for $n+2 \leq i \leq k$. For $i=k$, since $\partial c_k \to c_k$ is $(k-2)$-surjective and $(k-2) +2 \geq k$, it follows from the previous case that the space of lifts of $\partial c_k \to c_k$ against $G$ is $(m-(k-2)-2)$-truncated, and since $(m-(k-2)-2) \leq (m-n-2)$ also $(m-n-2)$-truncated. It remains to show that 
$\partial c_i \to c_i$ is in $S$ for $n+2 \leq i < k$.
As both $\partial c_i$ and $c_i$ are $(\infty, k-1)$-categories, by the $(i_{k},\iota_{k-1})$ adjunction, the two space of lifts are equivalent:
    \begin{equation}
        \left\{
         \begin{tikzcd}
        \partial c_{n+2} \arrow[r] \arrow[d] & \cC \ar[d, "G"] \\
        c_{n+2} \arrow[ur, dashed]  \ar[r] &  \cD
        \end{tikzcd}
        \right\}
        \simeq 
        \left\{
         \begin{tikzcd}
        \partial c_{n+2} \arrow[r] \arrow[d] & \iota_{k-1} \cC \ar[d, "\iota_{k-1} G"] \\
        c_{n+2} \arrow[ur, dashed]  \ar[r] &  \iota_{k-1} \cD
        \end{tikzcd}
        \right\}.
    \end{equation}
    Since $\iota_{k-1} G$ is a $n$-faithful functor between $(\infty, k-1)$-categories by \cref{obs:nsurj-and-nfaithful-indep-of-k}, by induction the space of lifts is $(m-n-2)$-truncated. 
\end{proof}

\subsection{Homotopy \texorpdfstring{$(n, k)$}{(n,k)}-categories of \texorpdfstring{$(\infty, k)$}{(infinity,k)}-categories}
\label{subsec:hom-cat-of-inf-k-cat}
In this subsection, we define the homotopy $(n,k)$-category of an $(\infty,k)$-category.

\begin{definition}
\label{defn:n-k-cat}
For any $k \geq 0$ and any $n \geq -2$, an \bit{$(n,k)$-category} is an $(\infty,k)$-category $\cC$ such that the functor $\cC \ra \pt$ is $n$-faithful. We write $\Cat_{(n,k)} \subseteq \CatInfty{k}$ for the full subcategory on the $(n,k)$-categories. 
\end{definition}

\begin{remark}
Unwinding \Cref{defn:n-k-cat} gives an alternative inductive description: For $k > 0$ and $n > -2$, an $(\infty,k)$-category $\cC$ is an $(n,k)$-category if and only if its hom-$(\infty,k-1)$-categories are in fact $(n-1,k-1)$-categories. In particular, an $(n,n)$-category is indeed an $\infty$-category that is weakly enriched in $(n-1,n-1)$-categories (as proposed in \Cref{subsec:appendix-infty-n-cats}), with a $(0,0)$-category being a set (i.e.\! a $0$-truncated space).\footnote{At first glance, it may be confusing that e.g.\! for an $(\infty,2)$-category $\cC$ to be a $(2,2)$-category we only need to impose a condition on its spaces of $2$-morphisms: one might wonder whether e.g.\! its space of $1$-morphisms $\hom_{\CatInfty{2}}(c_1,\cC)$ might have higher homotopy groups. But actually, path spaces of the latter contribute path components of the former. (Indeed, this same reasoning shows that the space of objects of a $(1,1)$-category must in fact be a $1$-groupoid.)}
\end{remark}

\begin{remark}
For $n \geq k$, it is immediate from \cite[Thm. 6.1.8]{GH13} that $\Catnk{n}{k}$ coincides with \cite[Def. 6.1.1]{GH13}. In particular, $\Catnk{n}{k}$ is the $\infty$-category obtained from applying $\Cat[-]$ $(n-k)$ times to the $\infty$-category of $(n-k)$-truncated spaces $\Spaces_{\leq n-k}$, with its Cartesian presentably symmetric monoidal structure.
\end{remark}

\begin{example}
We list a few edge cases of Definitions~\ref{defn:n-k-cat}.
\begin{enumerate}

\item For any $k \geq 0$, there is only one $(-2,k)$-category, namely $\pt$. 

\item For any $k \geq 0$, there are only two $(-1,k)$-categories, namely $\emptyset$ and $\pt$. 

\item For any $k > 0$, a $(0,k)$-category is precisely a partially ordered set (i.e.\! an $\infty$-category enriched in $(-1)$-truncated spaces). In particular, the inclusions $\Cat_{(0,1)} \hookra \Cat_{(0,2)} \hookra \cdots$ are all equivalences. 

\item More generally, for any $k > n \geq 0$, an $(n,k)$-category is precisely an $(n+1,n+1)$-category whose spaces of $(n+1)$-morphisms are all either empty or contractible. In particular, the inclusions $\Cat_{(n,n+1)} \hookra \Cat_{(n,n+2)} \hookra \cdots$ are all equivalences. 

\item Taking $k = 1$,  for any $n \geq 1$, an $(n,1)$-category is precisely an $(\infty,1)$-category whose hom-spaces are $(n-1)$-truncated.
\end{enumerate}
\end{example}

\begin{definition}\label{def:taun}
By \Cref{obs:reflective-localizn-from-fs}.\eqref{item-obs:reflective-localizn-from-fs}, the fully faithful inclusion $\Cat_{(n,k)} \hookrightarrow \CatInfty{k}$ admits a left adjoint
\begin{equation}\label{eq:taun}
 \begin{tikzcd}[column sep=1.5cm]
\CatInfty{k}
\arrow[yshift=0.9ex, dashed]{r}{\tau_n}
\arrow[hookleftarrow, yshift=-0.9ex]{r}[yshift=-0.2ex]{\bot}
&
\Cat_{(n,k)}
\end{tikzcd}
\end{equation}
given by the formula $\tau_n(\cC) \coloneqq \Fact_n(\cC \ra \pt)$.\footnote{Beware that (despite our usage of the letter $\tau$) this is \textit{not} the abstractly-defined ``$n$-truncation'' functor in $\CatInfty{k}$ (see \Cref{warn:nfaithful-neq-ntruncated})} This left adjoint $\tau_n$ is symmetric monoidal as it preserves products.
\end{definition}

\begin{example}\label{exm:how-tau-works}
Unpacked, the functor $\tau_n$ can be described as follows, depending on $n\geq -2$ and $k \geq 0$. 
\begin{enumerate}
\item  For any $(\infty,k)$-category $\cC$, we have $\tau_{-2} \cC = \pt$.

\item  For any $(\infty,k)$-category $\cC$, we have $\tau_{-1} \cC = \emptyset$ if $\cC$ is empty and $\tau_{-1} \cC = \pt$ otherwise.

\item\label{item:how-tau-works-ngeqk} For $n \geq k \geq 0$ and any $(\infty,k)$-category $\cC$, $\tau_n \cC \in \Catnk{n}{k}$ is obtained by $(n-k)$-truncating its $k$-morphism spaces (\cref{nota:n-hom-notation}),  and univalently completing the result. In particular, for an $(\infty,1)$-category $\cC$, $\tau_1\cC$ is its ordinary homotopy category. 
 
\item For $k \geq 1$ and any $(\infty,k)$-category $\cC$, $\tau_0 \cC \in \Catnk{0}{k} \simeq \Catnk{0}{1}$ is the posetification of $\cC$, obtained by applying $\tau_{-1}$ to its hom-$(\infty,k-1)$-categories.

\item For $k>n \geq 0$ and any $(\infty,k)$-category $\cC$, $\tau_n \cC \in \Catnk{n}{k} \simeq \Catnk{n}{n+1}$ is obtained by $k$-homwise  applying $\tau_{-1}$ (and univalently completing the result).
\end{enumerate}
\end{example}

In general, the hom-categories in $\tau_n \cC$ can be computed by applying $\tau_{n-1}$ to hom-categories of $\cC$:
\begin{lemma}\label{lem:homwise-tau}
For any $n\geq -2, k\geq 0$ and any $(\infty,k)$-category $\cC$ and objects $c,c' \in\cC$, the adjunction~\eqref{eq:taun} induces an equivalence 
\[ \tau_{n-1} \eHom_{\cC}(c,c') \xrightarrow{\simeq} \eHom_{\tau_n \cC}(c,c').
\]
\end{lemma}
\begin{proof}
Consider the factorization $\cC \to \tau_n \cC \to \pt$ into an $n$-surjective followed by an $n$-faithful functor. Since $n$-faithfulness/surjectivity implies homwise $(n-1)$-faithfulness/surjectivity, it follows that for any $c,c' \in \cC$, in the induced factorization on hom-$(\infty,k-1)$-categories
\[\eHom_{\cC}(c,c') \to \eHom_{\tau_n \cC}(c,c') \to \pt 
\]
the first functor is $(n-1)$-surjective and the second functor is $(n-1)$-faithful, hence exhibiting $\eHom_{\tau_n \cC}(c,c')$ as the unique factorization $\tau_{n-1} \eHom_{\cC}(c,c')$. 
\end{proof}

\begin{observation}\label{obs:tau-n-preserves-m-faithful}
Using~\cref{obs:adjunctions-left-right}, for any $k\geq 0$ and $n,m \geq -2$, one can deduce that the functor $\tau_n\colon \CatInfty{k} \to \Catnk{n}{k}$ preserves $m$-faithful functors since its right adjoint preserves $m$-surjective functors.
\end{observation}

For any $j> k \geq 0$ and $n\geq -2$, since the inclusion $i_j  \colon  \CatInfty{k} \hookrightarrow \CatInfty{j}$ preserves $n$-factorizations (see \cref{obs:nsurj-and-nfaithful-indep-of-k}),  the diagram 
\[\begin{tikzcd}
\CatInfty{j}
\arrow{r}{\tau_n}
&
\Cat_{(n,j)}
\\
\CatInfty{k}
\arrow{r}[swap]{\tau_n}
\ar[u, hook, "i_j"]
&
\Cat_{(n,k)}
\ar[u, hook, "i_j"']
\end{tikzcd}
\]
commutes. By adjunction, this induces for any $(\infty,j)$-category $\cC$ a canonical functor of $(n,k)$-categories 
\begin{equation}\label{eq:canonical-functor}
\tau_n \iota_k \cC \to \iota_k \tau_n \cC.
\end{equation}
\begin{lemma}
\label{lem:tau-n-commutes-with-iota-k}
For any $j > k \geq 0$ and any $n \geq j$  or $k>n \geq -2$, and an $(\infty,j)$-category $\cC$, the canonical functor~\eqref{eq:canonical-functor} is an equivalence. 
\end{lemma}
\begin{proof} This follows immediately from applying \cref{cor:iota-k-commutes-with-fact-n-for-n-not-k}  to the factorization $\cC \to \tau_n \cC \to \pt$. 
\end{proof}

\begin{remark}
\cref{lem:tau-n-commutes-with-iota-k}  does not hold for $j > n \geq k$: For instance, for $j=1$ and $n=k=0$ and an $(\infty,1)$-category $\cC$, the space $\tau_0 \iota_0 \cC$ is the set of isomorphism classes of objects of $\cC$. On the other hand, $ \iota_0 \tau_0 \cC$ is the set of connected components of $\cC$, i.e the quotient of the set of isomorphism classes of object by the equivalence relation that $c\sim d$ if there exists a zigzag of morphisms between $c$ and $d$. 
\end{remark}

Now we are ready to define the $n$-homotopy category functor:
    \begin{definition}\label{def:homotopy-cat-definition}
    For $n, k \geq 0$, define the  \bit{homotopy $n$-category functor} \[h_n\colon \CatInfty{k} \to \Catnk{n}{n}\]
to be $\CatInfty{k} \xrightarrow{\tau_n} \Catnk{n}{k} \hookrightarrow \Catnk{n}{n}$ when $n \geq k$ and to be $\CatInfty{k} \xrightarrow{\iota_n} \CatInfty{n} \xrightarrow{\tau_n} \Catnk{n}{n}$ when $n < k$. Note that $h_n$ is symmetric monoidal as $\iota_n$ and $\tau_n$ are symmetric monoidal.
         For $\cC \in \CatInfty{k}$, we call $h_n \cC$ the \bit{homotopy $n$-category} of $\cC$.     \end{definition}

\begin{remark}
Since $\iota$ is a right adjoint and $\tau$ is a left adjoint, there are no natural maps in either direction between $\cC$ and $h_n\cC$.    
\end{remark}

    \begin{example}
        The homotopy $n$-category functor $h_n$ takes a space $X$ to its $n$-truncation $\tau_{n} X$. Given a $(\infty, 1)$-category $\cC$, $h_0 \cC = \tau_0 \iota_0 \cC$ is the set of isomorphism classes of objects, and $h_1 \cC$ is its homotopy $1$-category \cite[Def. 1.1.3.2]{HTT}. For $n \geq 1$, $h_n = \tau_n$ is equivalent to the `$n$-homotopy category' of \cite[Def. 2.9]{SY19cat}.   Of particular relevance to this paper will be the case of $(\infty,2)$-categories $\cC$ where $h_1\cC = \tau_1 \iota_1 \cC$. 
    \end{example}

\begin{observation}\label{obs:h_n-preserves-m-faithful}
    Fix $n, m \geq 0$, since both $\iota_n$ and $\tau_n$ preserves $m$-faithfulness by \cref{obs:nsurj-and-nfaithful-indep-of-k} and \cref{obs:tau-n-preserves-m-faithful}, $h_n$ also preserves $m$-faithfulness.
\end{observation}

\subsection{Faithful functors and homotopy categories}
\label{subsec:n-1-faithful-and-hom-n-cat}
For an $(\infty,1)$-category $\cC$, the $\infty$-category of full subcategories of $\cC$ is equivalent to the poset of subsets of the set $h_0(\cC)$ of isomorphism classes of objects in $\cC$, or equivalently to the $\infty$-category of `full subcategories' of the set $h_0(\cC)$.
In the next two subsections, we generalize this to arbitrary $n, k\geq 0$ and characterizes $(n-1)$-faithful functors of $(\infty,k)$-categories into some $(\infty,k)$-category $\cC$ in terms of $(n-1)$-faithful functors of $(n,n)$-categories into the homotopy $n$-category $h_{n}\cC$.

\begin{notation}\label{nota:ar-cR-full-sub}
For $n, k \geq 0$, we let $\mathrm{Ar}^{n}(\CatInfty{k}) \subseteq \mathrm{Ar}(\CatInfty{k})$ denote the full subcategory of the arrow category of $\CatInfty{k}$ on the $n$-faithful functors. Moreover, for $\cD \in \CatInfty{k}$, we write  $\CRoverd{(\CatInfty{k})}{n}{\cD} \subseteq \CRoverd{(\CatInfty{k})}{\phantom{}}{\cD}$ for the full subcategory of the over-category on the $n$-faithful functors $\cC \to\cD$.  In particular, $\CRoverd{(\CatInfty{k})}{n}{\pt} = \Catnk{n}{k}$.
\end{notation}

The goal of the next subsection \S\ref{subsec:proof-factorization} will be to prove the following theorem. 
\begin{theorem}
\label{thm:inf-n-fancy-pullback}
Fix $n,k \geq 0$.
The commuting square of $\infty$-categories
        \begin{equation}\label{eq:ar-d-homotopy-pullback}
        \begin{tikzcd}
           \mathrm{Ar}^{(n-1)}(\CatInfty{k}) \ar[r, "t"] \ar[d, "h_n"] & \CatInfty{k} \ar[d, "h_n"] \\  
           \mathrm{Ar}^{(n-1)}(\Catnk{n}{n}) \ar[r, "t"] & \Catnk{n}{n} 
        \end{tikzcd}
    \end{equation} 
    is a pullback square. Note the left map exists by \cref{obs:h_n-preserves-m-faithful}.
\end{theorem}

Before proving \cref{thm:inf-n-fancy-pullback} in \S\ref{subsec:proof-factorization}, we record a few corollaries.
First, taking fibers at a $\cD \in \CatInfty{k}$ immediately leads to the following corollary:

\begin{corollary}\label{cor:equivalence-of-over-categories}
       Let $n, k \geq 0$ and $\cD$ an $(\infty, k)$-category.
        Then, the $n$-homotopy category functor $h_n$ induces an equivalence of $\infty$-categories:        
  \[
        h_n \colon \CRoverd{(\CatInfty{k})}{(n-1)}{\cD} \to 
        \CRoverd{(\Catnk{n}{n})}{(n-1)}{h_n\cD}.\]
    \end{corollary}
    Hence, \cref{cor:equivalence-of-over-categories} is indeed a generalization of the statement at the beginning of this subsection: The $\infty$-category of $(n-1)$-faithful functors into $\cD$ is equivalent to the $\infty$-category of $(n-1)$-faithful functors into $h_n\cD$. 
    
    For later use, we need an analogous statement for $\cO$-monoidal $\infty$-categories for a given $\infty$-operad $\cO$. 
    For any small $\infty$-operad $\cO$ (see \cref{appendix:operads} and \S\ref{sec:prebraidings-to-E2-str} for definitions and notation), we can extend this to a statement about categories of $\cO$-algebras:
    \begin{definition}\label{def:faithful-O-monoidal}
    Let $\cO$ be an $\infty$-operad and $k,n\geq -2$. An $\cO$-monoidal functor $F\colon \cC \to \cD$ of $\cO$-monoidal $(\infty,k)$-categories (i.e. a morphism of $\cO$-algebras in the Cartesian symmetric monoidal category $\CatInfty{k}$) is called $n$-surjective/$n$-faithful if for every color $X\in \underline{\cO}$, the underlying functor $F_X  \colon  \cC_X \to \cD_X$ is $n$-surjective/$n$-faithful. 
    \end{definition}
    \begin{corollary}\label{cor:alg-equivalence-of-over-categories}
  Let $n,k \geq 0$, $\cO$ an $\infty$-operad, and $\cD$ an $\cO$-monoidal $(\infty, k)$-category.
    Then  the $n$-homotopy category functor $h_n$ induces an equivalence of $\infty$-categories:     
\[
    \CRoverd{(\Alg_{\cO}(\CatInfty{k}))}{(n-1)}{\cD} \xrightarrow{h_n} 
    \CRoverd{(\Alg_{\cO}(\Catnk{n}{n}))}{(n-1)}{h_n\cD}.
\]
\end{corollary}
    \begin{proof}
    The full subcategory $\mathrm{Ar}^{(n-1)}(\CatInfty{k}) \subseteq \mathrm{Ar}(\CatInfty{k})$ is closed under products and hence defines a Cartesian symmetric monoidal subcategory. Since all functors in ~\eqref{eq:ar-d-homotopy-pullback} preserve products, the pullback square is a pullback square of Cartesian symmetric monoidal $\infty$-categories and hence induces a pullback square of $\infty$-categories:
        \begin{equation}\label{eq:ar-Alg-d-homotopy-pullback}
        \begin{tikzcd}
          \Alg_{\cO}\left( \mathrm{Ar}^{(n-1)}(\CatInfty{k})\right) \ar[r, "t"] \ar[d, "h_n"] & \Alg_{\cO} \left(\CatInfty{k}\right) \ar[d, "h_n"] \\  
          \Alg_{\cO} \left( \mathrm{Ar}^{(n-1)}(\Catnk{n}{n})  \right) \ar[r, "t"] & \Alg_{\cO}\left(\Catnk{n}{n} \right)
        \end{tikzcd}
    \end{equation} 
    Under the equivalence of $\infty$-categories $\mathrm{Ar}(\Alg_{\cO}(\CatInfty{k})) \simeq \Alg_{\cO}(\mathrm{Ar}(\CatInfty{k}))$, the full subcategory $\mathrm{Ar}^{(n-1)}(\Alg_{\cO}(\CatInfty{k}))$ on the $\cO$-monoidal functors $F$ whose underlying functors $F_X$ are $(n-1)$-faithful becomes identified with $\Alg_{\cO}\left(\mathrm{Ar}^{(n-1)}(\CatInfty{k})\right)$. Hence, the pullback square~\eqref{eq:ar-Alg-d-homotopy-pullback} is equivalent to the square 
    \[
            \begin{tikzcd}
         \mathrm{Ar}^{(n-1)}\left( \Alg_{\cO}\left( \CatInfty{k}\right)\right) \ar[r, "t"] \ar[d, "h_n"] & \Alg_{\cO} \left(\CatInfty{k}\right) \ar[d, "h_n"] \\  
         \mathrm{Ar}^{(n-1)}\left( \Alg_{\cO} \left( \Catnk{n}{n}\right)  \right) \ar[r, "t"] & \Alg_{\cO}\left(\Catnk{n}{n} \right)
        \end{tikzcd}
    \]
    and taking fibers at the $\cO$-algebra $\cD \in\Alg_{\cO} \left(\CatInfty{k}\right)$  induces the desired equivalence. 
    \end{proof}

    We record a further straight-forward consequence:
\begin{cor}\label{cor:coCart-all-k} Let $k, n \geq 0$ and let $F  \colon  \cC \to \cD$ be an $(n-1)$-faithful functor between $(\infty,k)$-categories. Then, for every $\cX \in \CatInfty{k}$, the square 
\begin{equation}
\label{eq:faithful-Cartesian-square}
\begin{tikzcd}
\Hom_{\CatInfty{k}}(\cX, \cC) \ar[r, "F\circ-"] \ar[d, "h_n"] & \Hom_{\CatInfty{k}}(\cX, \cD)\ar[d, "h_n"]\\
\Hom_{\Catnk{n}{n}}( h_n \cX, h_n \cC) \ar[r, "h_n F\circ-"] & \Hom_{\Catnk{n}{n}}(h_n \cX, h_n \cD)
\end{tikzcd}
\end{equation}
is a pullback square of spaces. 

Equivalently, $F$ is a Cartesian morphism for the functor $h_n  \colon  \CatInfty{k} \to \Catnk{n}{n}$. 
\end{cor}
\begin{proof}
For any pair of objects in  $\mathrm{Ar}^{(n-1)}(\CatInfty{k}) $, the pullback square~\eqref{eq:ar-d-homotopy-pullback} of $\infty$-categories induces a pullback square between the respective hom-spaces. In particular, for the pair $(\id_{\cX} \colon  \cX \to\cX)$ and ($F \colon  \cC \to \cD$) of objects in $\mathrm{Ar}^{(n-1)}(\CatInfty{k}) $, we note that \[\Hom_{\mathrm{Ar}(\CatInfty{k})}(\id_{\cX}, F) \simeq \Hom_{\CatInfty{k}}(\cX, \cC) \quad \Hom_{ \mathrm{Ar}(\Catnk{n}{n})}(\id_{h_k \cX}, h_k F) \simeq \Hom_{\Catnk{n}{n}}(h_k\cX, h_k\cC),\]
and hence that the resulting pullback square of hom-spaces precisely results in the square~\eqref{eq:faithful-Cartesian-square}.
\end{proof}

Taking fibers at some $G\in  \Hom_{\CatInfty{k}}(\cX, \cD)$, \cref{cor:coCart-all-k} immediately implies the following corollary which will play a key role in the proof of our main theorem:
\begin{corollary}\label{cor:inf-n-pullbackfullyfaithful}Let $k, n \geq 0$, let $F\colon \cC \to \cD$ and $G\colon \cX \to \cD$ be functors between $(\infty,k)$-categories and assume that $F$ is $(n-1)$-faithful. Then, the map of spaces
\[
            \Hom_{(\CatInfty{k})_{/\cD}}(\cX, \cC) \to \Hom_{(\Catnk{n}{n})_{/h_n\cD}}(h_n\cX, h_n\cC)
\]
        is an equivalence.        \end{corollary}

\subsection{Proof of \texorpdfstring{\cref{thm:inf-n-fancy-pullback}}{Theorem 5.45}}\label{subsec:proof-factorization}
This subsection is devoted to the proof of  \cref{thm:inf-n-fancy-pullback}. 

We first consider the special case that $k=n$. This will be a consequence of the  following observation relating faithfulness and homotopy categories, well known in the case $k=n=1$ of $(\infty,1)$-categories. 

\begin{prop}\label{lem:infty-n-pull-back}
Let $n\geq k \geq  0$ and $F \colon \cC \to \cD$ be a functor between $(\infty,k)$-categories which is $(n-1)$-faithful.  Then, the following commutative diagram is a pullback square in  $\CatInfty{k}$:
    \begin{equation}\label{eq:magic-pullback}
        \begin{tikzcd}
            \cC \ar[r] \ar[d] & \cD \ar[d] \\ 
            h_n \cC = \tau_{n} \cC \ar[r] & h_n \cD = \tau_n \cD
        \end{tikzcd}
    \end{equation}
\end{prop}

\begin{proof}
Since the inclusion $\CatInfty{k} \hookrightarrow \CatInfty{n}$ preserves pullbacks, and preserves $n$-factorizations and hence commutes with $\tau_n$, it suffices to prove the statement for $n=k$. 
We induct on $n\geq 0$. The base case $n=0$ is immediate. For general $n$, we prove that the underlying diagram of spaces 
\begin{equation}\label{eq:want-to-show-pullback}
        \begin{tikzcd}
            \iota_0 \cC \ar[r] \ar[d] & \iota_0 \cD \ar[d] \\ 
           \iota_0  \tau_{n} \cC \ar[r] & \iota_0 \tau_n \cD
        \end{tikzcd}
        \end{equation}
is a pullback square and that for each $c, c' \in \cC$ the induced square of $(\infty,n-1)$-categories
\[
        \begin{tikzcd}
            \eHom_{\cC}(c,c') \ar[r] \ar[d] & \eHom_{\cD}(Fc, Fc') \ar[d] \\ 
            \eHom_{\tau_n \cC}(c,c')  \ar[r] & \eHom_{\tau_n \cD}(Fc, Fc')
        \end{tikzcd}
\]
is a pullback square. Using \cref{lem:homwise-tau}, the latter square is a pullback square by induction. 

For the former square~\eqref{eq:want-to-show-pullback}, the top horizontal map is $(n-1)$-truncated by \cref{obs:nsurj-and-nfaithful-indep-of-k}. By \cref{lem:tau-n-commutes-with-iota-k}, the bottom horizontal map is equivalent to $\tau_n \iota_0 \cC \to \tau_n \iota_0 \cD$, i.e. to the map between the $n$-truncations of the spaces $\iota_0 \cC$ and $\iota_0 \cD$. Since $\tau_n$ preserves truncatedness, the bottom horizontal map is also $(n-1)$-truncated. On the other hand, for any space $X$, the truncation map $X \to \tau_n X$ is $n$-connected, and hence so are the vertical maps.  Now \cref{prop:pullback-square-via-conn-and-trunc} implies that \eqref{eq:want-to-show-pullback} is a pullback square.
\end{proof}

We use \cref{lem:infty-n-pull-back} to prove the $n\geq k \geq 0$ case of \cref{thm:inf-n-fancy-pullback}. 
\begin{lemma}\label{lem:bottom-square-pullback}
For $n\geq k \geq  0$, the following commuting square of $\infty$-categories  is a pullback square
    \begin{equation}\label{eq:inf-n-tau-n-pullback}
    \begin{tikzcd}
        \mathrm{Ar}^{(n-1)}(\CatInfty{k}) \ar[r, "t"] \ar[d, "\tau_n"] & \CatInfty{k} \ar[d, "\tau_n"] \\
        \mathrm{Ar}^{(n-1)}(\Catnk{n}{k}) \ar[r, "t"] & \Catnk{n}{k}.
        \end{tikzcd}
\end{equation}
    \end{lemma}

\begin{proof}
We show that the functor 
\begin{equation}\label{eq:inf-n-fancy-pullback-helper3}
  \textrm{Ar}^{(n-1)}(\CatInfty{k}) \to \CatInfty{k} \times_{\Catnk{n}{k}} \textrm{Ar}^{(n-1)}(\Catnk{n}{k})  
\end{equation}
 is surjective and fully faithful. 

Surjectivity amounts to the following: For any $(\infty,k)$-category $\cD$ equipped with a $(n-1)$-faithful functor $\cC' \to \tau_n \cD$ from an $(n,k)$-category $\cC'$, there exists an $(\infty,k)$-category $\cC$ and a $(n-1)$-faithful functor $\cC \to \cD$ which is sent to $\cC' \to\tau_n \cD$ under $\tau_n$. 

Define $\cC$ to be the pullback in $\CatInfty{k}$
\[
     \begin{tikzcd}
         \cC \ar[rr, "(n-1)\faith"] \ar[d] && \cD \ar[d] \\ 
         \cC' \ar[rr, "(n-1)\faith"] &\ar[ul, phantom, "\lrcorner", very near end]& \tau_n \cD.
     \end{tikzcd}
\]
Since the right class of a factorization system is preserved under pullback, and since  $\cC' \to \tau_n \cD$ is $(n-1)$-faithful, so is its pullback $\cC \to\cD$. We will now prove by induction on $n\geq 0$ that the functor $\tau_n \cC \to \cC'$ adjunct to $\cC \to \cC'$ is an equivalence, proving surjectivity of \eqref{eq:inf-n-fancy-pullback-helper3}. The base case $n=0$ is immediate. In general, we will show that $\tau_n \cC \to \cC'$ is surjective on objects and fully faithful. Since $\cD \to \tau_{n}\cD$ is surjective on object (in fact $(n-1)$-surjective), the pullback $\cC \to \cC'$ is surjective on objects. Since $\cC \to \cC'$ factors as $\cC \to \tau_n \cC \to \cC'$, it follows that also  $\tau_n \cC \to \cC'$ is surjective on objects. Fully faithfulness of $\tau_n \cC \to \cC'$ follows by induction using \cref{lem:homwise-tau}.

We now prove that~\eqref{eq:inf-n-fancy-pullback-helper3} induces an equivalence on the hom-space between any pair of objects $\{\cC_1 \to \cD_1\}, \{\cC_2 \to \cD_2\} \in\ \textrm{Ar}^{(n-1)}(\CatInfty{k})$, and hence that \eqref{eq:inf-n-fancy-pullback-helper3} is fully faithful. Unwinding the hom-spaces in the relevant arrow categories, this is equivalent to the statement that for any fixed functor $G\colon \cD_1 \to \cD_2$ of $(\infty, k)$-categories, the map of spaces of dashed lifts 
\[
\left\{
            \begin{tikzcd}
                \cC_1 \ar[r, dashed] \ar[d] & \cC_2 \ar[d] \\ 
                \cD_1 \ar[r, "G"] & \cD_2
            \end{tikzcd}
\right\} \xrightarrow{\tau_n} 
\left\{
            \begin{tikzcd}
                \tau_{n}\cC_1 \ar[r, dashed] \ar[d] & \tau_n \cC_2 \ar[d] \\ 
                \tau_n \cD_1 \ar[r, "\tau_n G"] & \tau_n \cD_2.
            \end{tikzcd}
\right\}
\]
is an equivalence.  This follows immediately from \cref{lem:infty-n-pull-back}.
 \end{proof}

To generalize \cref{lem:bottom-square-pullback} to also allow for the case $k>n$, we will use the following lemma. 

\begin{lemma}\label{lem:top-square-pullback}
For all $k>n \geq 0$, the commutative square of $\infty$-categories 
    \begin{equation}\label{eq:inf-n-iota-n-pullback}
    \begin{tikzcd}
        \textrm{Ar}^{(n-1)}(\CatInfty{k}) \ar[r, "t"] \ar[d, "\iota_n"] & \CatInfty{k} \ar[d, "\iota_n"] \\ 
        \textrm{Ar}^{(n-1)}(\CatInfty{n}) \ar[r, "t"]  & \CatInfty{n} 
        \end{tikzcd}
\end{equation}
is a pullback square.
\end{lemma}

\begin{proof}
We show that the functor 
\begin{equation}\label{eq:inf-n-fancy-pullback-helper2}
  \textrm{Ar}^{(n-1)}(\CatInfty{k}) \to \CatInfty{k} \times_{\CatInfty{n}} \textrm{Ar}^{(n-1)}(\CatInfty{n})  
\end{equation}
 is surjective and fully faithful. 

Surjectivity amounts to the following: For any $(\infty,k)$-category $\cD$ equipped with an $(n-1)$-faithful  functor $\cC' \to  \iota_n \cD$ from an $(\infty,n)$-category $\cC'$, there exits an $(\infty,k)$-category  $\cC$ with an $(n-1)$-faithful functor $\cC \to \cD$ which under $\iota_n$ gets mapped to the original functor $\cC' \to \cD$. 

Define $\cC \coloneqq \Fact_{n-1}(\cC' \to \iota_n \cD \to \cD)$  as the factorization with respect to the ($(n-1)$-surjective, $(n-1)$-faithful) factorization system in $\CatInfty{k}$, and hence equipped with morphisms $\cC' \to \cC \to \cD$ where the former is $(n-1)$-surjective and the latter is $(n-1)$-faithful. To conclude, we show that the map $ \cC' \simeq \iota_n \cC'  \to \iota_n \cC$ is an equivalence, and hence that $\iota_n(\cC \to\cD)$ is equivalent to $\cC' \to \iota_n \cD$. Consider the following commutative diagram:
\[ \begin{tikzcd}[column sep=2cm, row sep=1cm]
\cC'
\arrow[bend left=5]{rrd}[sloped]{(n-1)\faith}
\arrow[bend right=5]{rdd}[sloped, swap]{(n-1)\surj}
\arrow{rd}
\\
&
\iota_n \cC
\arrow{r}[swap]{(n-1)\faith}
\arrow{d}{}
&
\iota_n \cD
\arrow{d}{}
\\
&
\cC
\arrow{r}[swap]{(n-1)\faith}
&
\cD.
\end{tikzcd}
~ \]
The bottom horizontal and leftmost diagonal functor are surjective/faithful as indicate by the definition of $\cC$. The top-most diagonal functor is $(n-1)$-faithful by assumption. The top horizontal functor is $(n-1)$-faithful since $\iota_n$ preserves faithfulness by \cref{obs:nsurj-and-nfaithful-indep-of-k}. 
It then follows from \cref{lem:cancellation-for-surj-and-faithful} that the functor $\cC'\to \iota_n \cC$ is $(n-1)$-faithful. Since $\cC' \to \cC$ is $(n-1)$-surjective and since $\iota_n \colon  \CatInfty{k} \to \CatInfty{n}$ preserves $(n-1)$-surjective functors by \cref{lem:iota-k-preserves-n-surjectivity-for-n-not-equal-k} it follows that $\cC' \simeq \iota_n \cC' \to \iota_n \cC$ is $(n-1)$-surjective. Hence, $\cC' \to \iota_n \cC$ is $(n-1)$-faithful and $(n-1)$-surjective and hence an equivalence.

We now prove that~\eqref{eq:inf-n-fancy-pullback-helper2} induces an equivalence on the hom-space between any pair of objects $\{\cC_1 \to \cD_1\}, \{\cC_2 \to \cD_2\} \in\ \textrm{Ar}^{(n-1)}(\CatInfty{k})$, and hence that \eqref{eq:inf-n-fancy-pullback-helper3} is fully faithful. Unwinding the hom-spaces in the relevant arrow $\infty$-categories, this is equivalent to the statement that for any fixed $\cD_1 \to \cD_2$ and any 
fixed dashed lift as shown in the first diagram in \eqref{eq:dashedlifts}, the space of dashed lifts as shown in the commuting square in the second diagram in \eqref{eq:dashedlifts} is contractible.
\begin{equation}\label{eq:dashedlifts}
\begin{tikzcd}
\iota_n \cC_1  \ar[d] \ar[r, dashed] & \iota_n \cC_2 \ar[d] \\
\iota_n \cD_1 \ar[r] & \iota_n \cD_2
\end{tikzcd}, 
\quad\quad
            \begin{tikzcd}
                \iota_n \cC_1 \ar[r] \ar[d] & \iota_n \cC_2 \ar[d] \\ 
                \cC_1 \ar[r, dashed] \ar[d] & \cC_2 \ar[d] \\ 
                \cD_1 \ar[r] & \cD_2.
            \end{tikzcd}
        \end{equation}
By \cref{lem:counit-for-iota-k-is-k-1-surjective}, $\iota_n \cC_1 \to \cC_1$ is $(n-1)$-surjective, and $\cC_2 \to \cD_2$ is $(n-1)$-faithful by assumption, hence the space of lift is contractible since $(n-1)$-surjective/$(n-1)$-faithful functors form a factorization system on $\CatInfty{k}$. 
\end{proof}

We can combine \cref{lem:bottom-square-pullback} and \cref{lem:top-square-pullback} into a proof of \cref{thm:inf-n-fancy-pullback}.
\begin{proof}[Proof of \cref{thm:inf-n-fancy-pullback}]
The case $n\geq k$ is \cref{lem:bottom-square-pullback}. For $k> n$, decompose the square as 
\[
    \begin{tikzcd}
        \mathrm{Ar}^{(n-1)}(\CatInfty{k}) \ar[r, "t"] \ar[d, "\iota_n"] & \CatInfty{k} \ar[d, "\iota_n"] \\ 
        \mathrm{Ar}^{(n-1)}(\CatInfty{n}) \ar[r, "t"] \ar[d, "\tau_n"] & \CatInfty{n} \ar[d, "\tau_n"] \\
        \mathrm{Ar}^{(n-1)}(\Catnk{n}{n}) \ar[r, "t"] & \Catnk{n}{n}.
    \end{tikzcd}
\]
By Lemmas~\ref{lem:bottom-square-pullback} and~\ref{lem:top-square-pullback} the bottom and top squares are pullbacks, respectively. 
\end{proof}

\section{The monoidal \texorpdfstring{$(\infty,2)$}{(infinity, 2)}-category of chain complexes of Soergel bimodules}
\label{sec:SBim}

Throughout this section, we let $k$ be a $\mathbb{Q}$-algebra.\footnote{All of the results in this section which do not specifically refer to the categories of Bott-Samelson and Soergel bimodules apply generally to arbitrary connective ring spectra $k \in \CAlg(\ConnSpectra)$, and to arbitrary commutative monoids $Z \in \CAlg(\Spaces)$ in place of $\mbbZ$.}
Recall from \cref{def:gradedcats} the presentably symmetric
monoidal $\infty$-categories $\addkBZ \coloneqq \Fun(B\mbbZ, \add_k)$ and $\stkBZ\coloneqq \Fun(B \mbbZ, \st_k)$ of small $k$-linear\footnote{In this section, and the rest of the paper, we abuse notation and let $k$ denote both the commutative ring, and the induced commutative ring spectrum $Hk$ and simply write $\add_k$ and $\st_k$ instead of $\add_{Hk}$ and $\st_{Hk}$ as in \S\ref{sec:morita-categories}.} additive,
resp. stable, idempotent complete $\infty$-categories equipped with a
$\mathbb{Z}$-action. These categories  will always be understood as equipped with the Day convolution symmetric monoidal structure.

\begin{notation} Throughout this section, we will use the following terminology:
\begin{enumerate}
\item We refer to objects and morphisms of $\Cat[\SetZ]$ as ordinary $1$-categories \emph{with local shifts} and \emph{shift-preserving functors}. 
\item We refer to objects and morphisms of $\Cat[\addkBZ]$ as $k$-linear $(\infty,2)$-categories \emph{with local shifts} and \emph{shift-preserving $k$-linear functors}. 
\item We refer to objects and morphisms of $\Cat[\stkBZ]$ as $k$-linear stable $(\infty,2)$-categories \emph{with local shifts} and \emph{shift-preserving $k$-linear exact functors}. 
\item Similarly, we refer to objects and morphisms in $\Alg_{\EE_1}$ of the $\infty$-categories in (1)--(3) as monoidal ordinary categories \emph{with local shifts}, etc. 
\end{enumerate}
\end{notation}
Unpacked, an ordinary category with local shifts is a $1$-category $\cC$ equipped with, for every $c, c' \in \cC$ a $\mathbb{Z}$-action on the hom-set  $\Hom_{\cC}(c,c')$, denoted by $[n]\colon \Hom_{\cC}(c,c') \to \Hom_{\cC}(c,c')$ for $n \in \mbbZ$, which is compatible with composition in the sense that the following diagram commutes:\[ \begin{tikzcd}
\Hom_{\cC}(c,c') \times \Hom_{\cC}(c', c'') \arrow[d, "{[n] \times [m]}"'] \arrow[r, "\circ"]& \Hom_{\cC}(c,c'') \arrow[d, "{[n+m]}"]\\ 
\Hom_{\cC}(c,c') \times \Hom_{\cC}(c', c'') \arrow[r, "\circ" ]& \Hom_{\cC}(c,c'')
\end{tikzcd}
\]
Similarly, a $k$-linear (stable) $(\infty,2)$-category with local shifts is an $(\infty,2)$-category whose hom-categories are  additive (stable), $k$-linear, and have a homotopy coherent $\mbbZ$-action which is compatible with composition. 

\begin{example}\label{ex:Moritaunpacked}
Most of our constructions in this section are built on the symmetric monoidal $k$-linear $(2,2)$-category with local shifts $$\MoritaS \in \CAlg(\widehat{\Cat}[\addkBZ])$$
from \cref{def:moritakz-and-dmoritakz}. 
By \cref{cor:finally-our-morita-categories}.\eqref{item-cor:finally-1}, its objects are $\mathbb{Z}$-graded flat $k$-algebras, and its $k$-linear additive, idempotent-complete hom-categories between two such algebras $A$ and $B$ is given by the full subcategory of the $1$-category ${}_{A}\mathrm{grbmod}_{B}$ of (ordinary) $\mathbb{Z}$-graded $A$--$B$-bimodules on those bimodules which are graded-compact-projective (\cref{def:graded-compact-perfect}), as a right $B$-module. The homwise $\mbbZ$-action is given by shifting the grading degree of these bimodules. 
\end{example}

We warn the reader that $\MoritaS$ is a locally small, but not a small  (the flat $\mathbb{Z}$-graded $k$-algebras do not form a small set) $(\infty,2)$-category. Since all algebras of relevance to this paper are graded polynomial algebras, it will be convenient to restrict to the \emph{small} full subcategory on those:
\begin{notation}
Let $\MorPoly \subseteq \MoritaS$ denote the small full subcategory on the graded polynomial algebras $k[x_1, \ldots, x_n]$ for $n \geq 0$, with all $x_i$ in degree $2$. (Graded polynomial algebras are flat, see
\cref{exm:graded-poly-flat}, this hence indeed defines a full subcategory.)  Since tensor products of polynomial algebras are polynomial algebras, this is in fact a symmetric monoidal subcategory and hence defines an object 
\[ \MorPoly \in  \CAlg(\Cat[\addkBZ]).
\]\end{notation}

The goal of this section is to first define $\BSbimp$ as a monoidal $(2,2)$-category\footnote{As a subcategory of $\SBim$, $\BSbimp$ then inherits its local $\mathbb{Z}$-action from $\SBim$.}  and then $\SBim$ as $k$-linear monoidal $(2,2)$-category with local shifts, both equipped with \emph{faithful} monoidal functors to $\MorPoly$, i.e. functors which induces fully-faithful inclusions on hom-categories (see \cref{def:faithful}).

The following warning makes our construction of a monoidal structure on $\BSbimp$ and $\SBim$ from the monoidal structure of $\MorPoly$ somewhat more subtle than one might think. 
\begin{warning} 
    \label{warn:nosub}
    We warn the reader that one should \emph{not} think of $\BSbimp$ and $\SBim$ as \emph{sub}-$(2,2)$-categories of $\MorPoly$. Indeed, $\BSbimp$ and $\SBim$ do \emph{not} contain all $1$-equivalences  between their objects that exists in $\MorPoly$. Hence, the faithful functors $\BSbimp\to\MorPoly$ and $\SBim \to \MorPoly$ are not monomorphisms in $\CatInfty{2}$; see also \cref{warn:nfaithful-neq-ntruncated}.  In particular, for an $(\infty,2)$-functor $\cX \to \MorPoly$, it is not merely a \emph{property} to factor through $\BSbimp$ or $\SBim$ but \emph{additional data}. 
    \end{warning}

\subsection{The monoidal \texorpdfstring{$(2,2)$}{(2,2)}-category of Bott-Samelson bimodules}
\label{subsec:definining-BSBim}

Recall that to define a full subcategory of an $(\infty,1)$-category $\cA$, it suffices to specify a subset of the set $h_0 \cA$ of isomorphism classes of objects  of $\cA$. 
Similarly, it follows from  \cref{cor:equivalence-of-over-categories} that to define  an $(\infty,2)$-category $\cC$ together with a faithful functor $\cC \to \cA$ into a fixed $(\infty,2)$-category $\cA$, it suffices to define an ordinary $1$-category $h_1\cC$ together with an ordinary faithful functor $h_1\cC \to h_1\cA$ into the homotopy $1$-category of $\cA$, see \cref{def:homotopy-cat-definition}. We will use this to define our monoidal $(2,2)$-category $\BSbimp$ together with its faithful functor to $\MorPoly$.

\begin{observation}\label{obs:h1Morita}
Following \cref{cor:finally-our-morita-categories}.\eqref{item-cor:finally-1}, $h_1 \MorPoly$  is the (small) ordinary symmetric-monoidal $1$-category whose objects are given by graded polynomial algberas, and whose morphisms are given by \emph{isomorphism classes} of ordinary graded bimodules which are graded-compact-projective as right modules. 
Furthermore, composition is given by relative tensor product and the symmetric monoidal structure is given by tensoring over $k$.
\end{observation}

\begin{prop}
The ordinary monoidal $1$-category $h_1\BSbimp$ from \cref{def:oldhoneBSbim} admits a faithful monoidal functor 
\[ h_1\BSbimp \to h_1 \MorPoly. 
\]
\end{prop}
\begin{proof}
Recall from \cref{def:oldhoneBSbim} that $h_1\BSbimp$ has objects $n \in \N_0$ and endo-hom-sets defined as the subset $h_0 \BSbimp_n \subseteq h_0({}_{R_n} \mathrm{grbmod}_{R_n})$ of isomorphism classes of graded bimodules for the graded polynomial algebra $R_n = k[x_1, \ldots, x_n]$ with $x_i$ in degree $2$ on the Bott-Samelson bimodules. The desired functor to $h_1\MorPoly$ follows immediately from this description: 
It sends an object $n \in \N_0$ of $h_1\BSbimp$ to the graded polynomial algebra $R_n= k[x_1,\ldots, x_n]$ and is defined on hom-sets as the full inclusion $h_0 \BSbimp_n \hookrightarrow h_0({}_{R_n}\mathrm{grbmod}^{\mathrm{gr-cp}}_{R_n}).$ Since  Bott-Samelson bimodules are
graded-compact-projective as right (and left) modules by \cref{rem:projective}, and since the composition and monoidal structure in $h_1 \MorPoly$ are defined by the relative (underived) tensor product and the tensor product $\otimes_k$, this indeed defines a faithful monoidal functor. \end{proof}

The following corollary justifies the notation $h_1 \BSbimp$ from \S\ref{sec:2}. 
\begin{corollary}
\label{cor:uniqueBSBim}
There exists a unique monoidal $(2,2)$-category  \[\BSbimp \in \Alg_{\EE_1}(\CatInfty{2})\] equipped with a faithful monoidal functor $\BSbimp \to \MorPoly$ which agrees on homotopy categories with the monoidal functor $h_1 \BSbimp \to h_1\MorPoly$ from \cref{obs:h1Morita}.
\end{corollary}
\begin{proof} 
Recall \cref{cor:alg-equivalence-of-over-categories}, that for any small monoidal $(\infty,2)$-category $\cD \in \Alg_{\EE_1}(\CatInfty{2})$, taking the homotopy 1-category induces an equivalence \[h_1 \colon \CRoverd{\Alg_{\EE_1}(\CatInfty{2})}{\faithful}{\cD} \to \CRoverd{\Alg_{\EE_1}(\Catnk{1}{1})}{\faithful}{h_1\cD}\] between the full subcategories of $\Alg_{\EE_1}(\CatInfty{2})_{/\cD}$ and $\Alg_{\EE_1}(\Catnk{1}{1})_{/h_1\cD}$ on the faithful functors. 
Applying this to $\cD = \MorPoly$, the monoidal $(2,2)$-category $\BSbimp$ is the unique pre-image of $h_1 \BSbimp \to h_1 \MorPoly$.  Since any $(\infty,2)$-category with a faithful functor to a $(2,2)$-category is again a $(2,2)$-category, it follows that $\BSbimp$ is a $(2,2)$-category. 
\end{proof}

\begin{definition}\label{def:BSBimp} We refer to the monoidal $(2,2)$-category $\BSbimp$ from \cref{cor:uniqueBSBim} as the \emph{Bott-Samelson $(2,2)$-category}.
\end{definition}

\begin{observation} By construction, the isomorphism classes of objects of $\BSbimp$ are in bijection with the natural numbers $n\in \mathbb{N}_0$. For every $n \in \mathbb{N}_0$ we now fix a representing object in $\BSbimp$ in that isomorphism class and simply denote it by $n$.
For $n,m \in \BSbimp$, the hom-category in $\BSbimp$ is given by  $$\eHom_{\BSbimp}(n,m)= \left\{ \begin{array}{lr} 0 & n \neq m\\ \mathrm{BSBim}_n & n=m \end{array} \right. ,$$ where $\BSbimp_n$ is the category from \cref{def:BSbimn}. The monoidal functor $ \BSbimp \to \MorPoly$ sends objects $n$ to the polynomial algebra $R_n$ and is given on hom-categories by the evident full inclusion of $\mathrm{BSBim}_n$ into the category of graded $R_n$--$R_n$-bimodules that are graded-compact-projective as right $R_n$-modules. 
\end{observation}

\subsection{The monoidal \texorpdfstring{$(2,2)$}{(2,2)}-category of Soergel bimodules}
\label{subsec:definining-SBim}

Having defined $\BSbimp$ together with its inclusion $\BSbimp \to \MorPoly$, we will now define $\SBim$ as the homwise completion of $\BSbimp$ under $\mbbZ$-shifts, direct sums and splitting of idempotents inside of $\MorPoly$.

To formally implement this, we need to construct a factorization system on the presentably symmetric monoidal $\infty$-category $\addkBZ$ of additive $k$-linear $\infty$-categories with a $\mbbZ$-action, equipped with the Day convolution monoidal structure, constructed in \cref{subsubsec:gradingstoactions}.

\begin{definition} We will use the following terminology:
\begin{enumerate}
\item A functor $F\colon \cC \to \cD$ between idempotent-complete $\infty$-categories is called \emph{dominant}  if every object in $\cD$ is a retract of an object in the image of $F$. 
\item A morphism in $\addkBZ$ is \emph{dominant}, resp. \emph{fully faithful}, if its underlying functor is. 
\end{enumerate}
\end{definition}

\begin{prop}
\label{prop:fs-addkbz} 
The (dominant, fully faithful)-functors define a factorization system on $\addkBZ$ which is of small generation and compatible with the symmetric monoidal structure.
\end{prop}
\begin{proof}
Consider the sequence of morphisms in $\CAlg(\PrL)$ 
$$\cat \to \cat^{\sqcup} \to \catprod \to \Mod_{\CProj_k}(\catprod) \simeq  \add_k  $$
left adjoint to the respective forgetful functors (see \cref{prop:adjoiningcolims} for the first two functors, and \cref{obs:equivalentklinear} for the latter one).
We will successively lift the (surjective-on-objects, fully faithful)-factorization system on $\cat$ to $\add_k$. 

Step 1: We first show that the (surjective-on-objects, fully faithful)-functors define a factorization system on $\cat^{\sqcup}$ which is of small generation and compatible with the symmetric monoidal structure.
Given a morphism $F\colon \cC \to \cD$ in $\cat^{\sqcup}$, i.e. a functor between $\infty$-categories with finite coproducts that preserves finite coproducts, we consider its (surjective-on-objects, fully faithful) factorization in $\cat$
$$
        \begin{tikzcd}
            \cC \ar[rr, "F"] \ar[rd, "\widetilde{F}"] & & \cD\\ 
            & \mathrm{Im}(F) \ar[ur, hook, "\iota"] & 
        \end{tikzcd},
$$    
where $\widetilde{F}$ is surjective on objects and $\iota$ is fully faithful.
Using \cref{obs:restricted-fs-on-subcat} and \cref{lem:restricted-fs-on-subcat} we can deduce that the factorization
system on $\cat$ restricts to one on $\cat^{\sqcup}$ which is of small generation and compatible with the symmetric monoidal structure, provided we can show that
$\mathrm{Im}(F)$ admits finite coproducts and that $\widetilde{F}$ and $\iota$
preserve them. In fact, since $\iota$ is fully faithful, it is enough to prove
that $\mathrm{Im}(F)$ is closed under finite coproducts in $\cD$. To this end, let $S$
be a finite set and $I\colon S \to \mathrm{Im}(F)$ a diagram. Since
$\widetilde{F}\colon \cC \to \mathrm{Im}(F)$ is surjective on objects, we can
lift $I$ to a functor $\widetilde{I} \colon S \to \cC$ which has a colimit
$\colim~\widetilde{I}\in \cC$ by assumption. Since $F$ preserves coproducts and
using the factorization, the coproduct $\colim~\iota \circ I$ agrees with
$F(\colim~\widetilde{I})$ and hence is in the image of $F$, as required.

Step 2: To lift from $\cat^{\sqcup}$ to $\catprod$, observe that the symmetric monoidal left adjoint $(-)^{\mathrm{idem}}:\cat^{\sqcup} \to \catprod$ is a reflective localization, i.e. that the right adjoint is fully faithful. We observe the following: \begin{enumerate}
\item A morphism in $\catprod$ is of the form $(F)^{\mathrm{idem}}$ for a fully faithful morphism $F$ in $\cat^{\sqcup}$ if and only if it is fully faithful. 
\item A morphism in $\catprod$ is of the form $(F)^{\mathrm{idem}}$ for a surjective-on-objects morphism $F$ in $\cat^{\sqcup}$ if and only if it is dominant. 
\end{enumerate}
Since the class of dominant functors is stable under retracts in $\catprod$, it therefore follows from \cref{lem:fs-on-reflective-loc} that the (surjective-on-objects, fully faithful)-factorization system on $\cat^{\sqcup}$ induces the (dominant, fully faithful)-factorization system on $\catprod$, and that this factorization system is of small generation and compatible with the monoidal structure on $\catprod$.

Step 3:
Since 
\[
\begin{tikzcd}[column sep=1.5cm]
\catprod
\arrow[yshift=0.9ex]{r}{}
\arrow[leftarrow, yshift=-0.9ex]{r}[yshift=-0.2ex]{\bot}[swap]{}
&
\Mod_{\CProj_k}(\catprod) \simeq \add_k
\end{tikzcd}
\]
is a monadic adjunction, whose underlying monad $\CProj_k \otimes - \colon \catprod \to \catprod$  preserves colimits (and in particular geometric realizations), and preserves dominant functors (since $\otimes$ is compatible with the (dominant, fully faithful)-factorization system on $\catprod$, as follows from Step 2), it follows from\cref{lem:fs-and-monadic} that the (dominant, fully faithful) functors form a factorization system on $\add_k$ which is of small generation and compatible with its symmetric monoidal structure.

Step 4: Lastly, by \cref{lem:fs-and-functor-cat}, the (dominant, fully faithful)-factorization system induces one on $\Fun(B\mbbZ, \add_k)$ which is of small generation and compatible with the Day convolution symmetric monoidal structure.
\end{proof}

\begin{notation}\label{nota:surjective-and-dominant} We will use the following terminology:
\begin{enumerate}
\item A morphism $F\colon \cC \to \cD$ in $\Cat[\addkBZ]$ is called \emph{faithful} if its underlying $(\infty,2)$-functor is (see \cref{def:faithful}). It is called \emph{surjective-on-objects-and-dominant-on-1-morphisms} if it is surjective on objects and if for each $c, c'\in \cC$, the induced additive functor $\eHom_{\cC}(c,c') \to \eHom_{\cC}(Fc,Fc')$ is dominant. 
\item  A morphism in $\Alg_{\EE_1}(\Cat[\addkBZ])$ is called \emph{faithful} or \emph{surjective-on-objects-and-dominant-on-1-morphisms} if the underlying morphism in $\Cat[\addkBZ]$ is.
\end{enumerate}
\end{notation}

Applying the factorization system from \cref{prop:fs-addkbz} homwise leads to the following corollary.

\begin{cor}
    \label{cor:surj-dominant-faithful-fs-on-Cat-add}
    The (surjective-on-objects-and-dominant-on-1-morphisms, faithful)-functors define factorization systems on $\Cat[\addkBZ]$ and $\Alg_{\EE_1}(\Cat[\addkBZ])$, respectively, which are of small generation and compatible with the symmetric monoidal structure.
\end{cor}
\begin{proof} This follows immediately from applying \cref{thm:fs_and_enriched_cat} and then \cref{thm:fs-and-alg}  to the factorization system of \cref{prop:fs-addkbz}.
\end{proof}

The forgetful functor $\add_k \to \cat$ has a symmetric monoidal left adjoint `linearization functor' 
\[\mathrm{Lin}_k\colon \cat \to \add_k,\]
which is the composite of symmetric monoidal left adjoints \begin{equation}
\label{eq:freeadd}
\cat  \xrightarrow{(-)^{\sqcup, \idem}}\catprod  \xrightarrow{\CProj_k \otimes -} \Mod_{\CProj_k}(\catprod) \simeq \add_k,
\end{equation}
where $(-)^{\sqcup, \idem}$ freely adjoints finite coproducts and splittings of idempotents (see \cref{prop:adjoiningcolims}) and $\CProj_k\otimes -$ constructs free $\CProj_k$-modules (see \cref{prop:algprl-new}). 
This induces a symmetric monoidal left adjoint of the forgetful functor $\addkBZ \to \catZ$
\[\Fun(B\mbbZ, \mathrm{Lin}_k(-))\colon \catZ \to \addkBZ,\]
which we will also denote by $\mathrm{Lin}_k(-) \colon  \catZ \to \addkBZ$.
Unpacking \cref{obs:adjoint-of-unit}, the forgetful functor $\cat^{B \mbbZ} \to \cat$ (i.e. the functor $\ev_*\colon \Fun(B\mbbZ, \cat) \to \cat$) has a  left adjoint 
\[ - \times \mbbZ: \cat \to \cat^{B\mbbZ}
\]
which sends an $\infty$-category $\cC$ to the $\infty$-category $\cC \times \mbbZ$ with free $\mbbZ$-action, and which is symmetric monoidal with respect to the Day convolution structure on $\cat^{B\mbbZ}$.

Combining these left adjoints, we obtain a symmetric monoidal left adjoint 
\begin{equation}
    \label{eqn:linkz}
\mathrm{Lin}_k(- \times \mbbZ): \cat \to \addkBZ.
\end{equation}

Applying $\Lin_k(-\times \mbbZ)$ homwise, this induces by \cref{subsec:enriched-infty-cats} and \cref{subsubsec:adjunctions-of-cO-algebras} symmetric monoidal left adjoints 
\begin{equation}
\label{eq:LinkZloc}
\Lin_{k, \mathrm{loc}}^{\mbbZ}\coloneqq \Cat[\mathrm{Lin}_k(-\times \mbbZ)]\colon \CatInfty{2} = \Cat[\cat] \to \Cat[\addkBZ]\end{equation} and (abusing notation)
\begin{equation}\label{eq:LinkZlocAlg}\Lin_{k, \mathrm{loc}}^{\mbbZ}   \coloneqq \Alg_{\EE_1}(\Cat[\mathrm{Lin}_k(-\times \mbbZ)])\colon \Alg_{\EE_1}(\CatInfty{2}) \to \Alg_{\EE_1}(\Cat[\addkBZ])\end{equation}
of the respective forgetful functors. 

\begin{prop}The functor $\BSbimp \to \MorPoly$ of monoidal $(\infty,2)$-categories from \cref{cor:uniqueBSBim} factors through a monoidal $\addkBZ$-enriched functor \[\addkZloc{\BSbimp} \to \MorPoly.\]
\end{prop}
\begin{proof}Since $\MorPoly$ is an object in  $\Alg_{\EE_1} (\Cat[\addkBZ])$, the statement immediately follows from the adjunction~\eqref{eq:LinkZlocAlg}.\end{proof}

Explicitly, $\Lin_{k, \mathrm{loc}}^{\mbbZ} (\BSbim)$ has objects natural numbers and additive $k$-linear hom-categories with $\mathbb{Z}$-action given by the $k$-linearization $\Lin_k(\BSBim_n \times \mathbb{Z})$ of the 1-category $\BSBim_n \times \mathbb{Z}$ with free $\mathbb{Z}$-action. 

\begin{definition} 
\label{def:SBim}
Define the monoidal $k$-linear (2,2)-category with local shifts $\SBim \in \Alg_{\EE_1}(\Cat[\addkBZ])$, the \emph{Soergel $(2,2)$-category}, as the unique factorization
\[\begin{tikzcd}
\addkZloc{\BSbimp} \arrow[dr, "\substack{\mathrm{surjective-on-objects}\\\mathrm{-and-dominant-on-1-morphisms}}"'] \arrow[rr] && \MorPoly\\
&\SBim \arrow[ur, "\mathrm{faithful}"']
\end{tikzcd}
\] of $\addkZloc{\BSbimp} \to \MorPoly$ with respect to the (surjective-on-objects-and-dominant-on-morphisms, faithful)-factorization system on $\Alg_{\EE_1}(\Cat[\addkBZ])$. 

We denote the composite monoidal shift-preserving functor $\BSbimp \to \addkZloc{\BSbimp} \to \SBim$ by 
\begin{equation}
\label{eq:fromBSBimtoSBim}
\iota\colon \BSbimp \to \SBim.
\end{equation}
\end{definition} 

Composing with the monoidal $(\infty,2)$-functor $\BSBim \to\addkZloc{\BSbimp}$ (i.e. the unit of the adjunction~\eqref{eq:LinkZlocAlg}), we can summarize the categories and functors defined so far thus:
\begin{cor}\label{cor:SBim-to-Mor-faithful}
The above defined functors assemble into a commuting diagram of monoidal $(2,2)$-functors
\[
\begin{tikzcd}
\BSbimp \arrow[dr, "\iota"'] \arrow[rr] && \MorPoly\\
&\SBim \arrow[ur]
\end{tikzcd},
\]
where the top and right diagonal functors are faithful and the right diagonal functor is shift-preserving and $k$-linear.\end{cor}

\begin{remark}
Intuitively, $\SBim$ is the smallest locally $k$-linear additive and idempotent-complete `sub'-$(\infty,2)$-category (mind \cref{warn:nosub}) of $\MorPoly$ that is closed under the homwise $\mbbZ$-action and contains $\BSbimp \to \MorPoly$. 
Indeed, it follows from the definition that  $\SBim$ together with its faithful monoidal $\addkBZ$-enriched functor $\SBim \to \MorPoly$ is  initial amongst factorizations of the monoidal $(\infty,2)$-functor $\BSbimp \to \MorPoly$ through faithful monoidal $\addkBZ$-enriched functors.
More formally, it is the initial object of the pullback of the following span of $\infty$-categories:
\[
\begin{tikzcd}[column sep=-2em]
&  \left( \Alg_{\EE_1} \left( \CatInfty{2}\right)_{/\MorPoly} \right)_{\left( \BSbimp \to \MorPoly\right)/} \arrow[d]\\
\CRoverd{\Alg_{\EE_1} \left(\Cat[\addkBZ]\right)}{f}{\MorPoly}\arrow[r] 
&  \Alg_{\EE_1}\left( \CatInfty{2}\right)_{/\MorPoly}
\end{tikzcd}
\]
\end{remark}

\subsection{The Soergel \texorpdfstring{$(2,2)$}{(2,2)}-category agrees with its classical variant}\label{subsec:checking-SBim}
We now explain how the monoidal $(2,2)$-category $\SBim$ is indeed just a homwise additive and idempotent-completion of $\BSbim$. We first record the following useful observation about the adjunction 
\[
\begin{tikzcd}[column sep=1.5cm]
\cat
\arrow[yshift=0.9ex]{r}{\Lin_k(- \times \mbbZ)}
\arrow[leftarrow, yshift=-0.9ex]{r}[yshift=-0.2ex]{\bot}[swap]{\mathrm{forget}}
&
\addkBZ
\end{tikzcd}.
\]
\begin{lemma}\label{lem:concrete-dominance}
Let $F\colon \cC \to \cD$ be a functor of $(\infty,1)$-categories, where $\cC\in \cat $ and $\cD \in \addkBZ$. Then the following  properties of $F$ are equivalent:
\begin{enumerate}
\item Its adjunct $ \Lin_k(\cC \times \mbbZ) \to \cD$ is dominant.
\item Every object of $\cD$ is a retract of a finite coproduct of shifts (under the $\mbbZ$-action) of objects in the image of $F$.
\end{enumerate}
\end{lemma}
\begin{proof}
The tensor unit of $\catprod$
is the category $\mathrm{Set}^{\mathrm{fin}}$ of finite sets, and since
$\CProj_k \in \CAlg(\catprod)$, the unit induces a finite coproduct preserving
functor $\mathrm{Set}^{\mathrm{fin}} \to \CProj_k$ which sends a finite set $X$
to the coproduct $\sqcup_X k$ and is therefore dominant by
\cref{lem:CProjR-finite-coproduct}.\eqref{item-lem:CProjR-finite-coproduct}. 
Hence, since the tensor product in $\catprod$ of dominant functors is again dominant (since its (dominant, fully faithful)-factorization system is compatible with its monoidal structure, as follows from the proof of \cref{prop:fs-addkbz}), it follows that for any $\cA \in \catprod$, the unit $\cA \simeq \mathrm{Set}^{\mathrm{fin}} \otimes \cA \to \CProj_k \otimes \cA$  of the adjunction between $\catprod$ and $\add_k$  is dominant. 
In particular, it immediately follows that for any $\cB \in \add_k$, a functor $\cA \to \cB$ in $\catprod$ is dominant if and only if its adjunct (drawn horizontally) is: \[\begin{tikzcd} \cA \ar[d, two heads]\ar[dr]& \\
\CProj_k \otimes \cA \ar[r] & \cB
\end{tikzcd}
\]

Hence, a functor $F$ as in the statement of the lemma is dominant, if and only if the functor $(\cC \times \mbbZ)^{\sqcup, \mathrm{idem}} \to \cD$ is, equivalently if every object of $\cD$ is a retract of a finite coproduct of objects in the image of $\cC \times \mbbZ$, i.e. of objects which are $\mbbZ$-shifts of objects in the image of $\cC$.
\end{proof}

For the following we recall from \cref{sec:diagrammatics} the notation
$R_n \coloneqq k[x_1,\ldots, x_n]$ for the graded polynomial algebra over $k$ in $n\in
\N_0$ variables, each of degree two.

\begin{prop} 
\label{prop:sbimiscorrect}
The functor $\iota\colon \BSbimp \to \SBim$ is surjective on objects.  For
objects $n \neq m\in \BSbimp$, the hom-category $\eHom_{\SBim}(\iota n, \iota m)$
is the zero category, and for $n=m$ it is the smallest additive and idempotent-complete full
subcategory of the ordinary additive category $${}_{R_n}\mathrm{grbmod}_{R_n}$$
of graded $R_n$-bimodules, which contains the full subcategory $\BSbimp_n$ and is closed under the grading-shift $\mbbZ$-action.
\end{prop}
\begin{proof}
Since $\addkZloc{-}$ is defined by applying $\Lin_k(-\times \mbbZ)$ homwise,  the functor \[\alpha\colon \BSbimp \to \addkZloc{\BSbimp}\]
is the identity on objects and hence the composite $\iota\colon \BSbimp \to
\addkZloc{\BSbimp} \to \SBim$ is surjective on objects. The
induced functor on hom-categories between $n, m \in \BSbimp$  therefore factors
as follows in $\addkBZ$:
\[
\Lin_k(\BSbimp(n,m) \times \mbbZ) =:\addkZloc{\BSbimp} (\alpha n, \alpha m) \to  \Sbim(\iota n, \iota m) \to {}_{R_n}\mathrm{grbmod}_{R_m}
\]
The first functor is dominant (since $\addkZloc{\BSbimp} \to
\SBim$ is dominant on morphisms) and the second functor is fully faithful (since
$\SBim \to \MorPoly$ is faithful). 

It follows from fully faithfulness of the second functor (and the fact that it is a morphism in $\addkBZ$), that $\SBim(\iota n, \iota m)$ is a full additive and
idempotent-complete subcategory of ${}_{R_n}\mathrm{grbmod}_{R_m}$ which is closed under grading shifts.
By \cref{lem:concrete-dominance}, 
dominance of the first functor implies that every object of this full subcategory is a retract of a finite coproduct of shifts of objects in the image of $\BSbimp(n,m) \hookrightarrow {}_{R_n}\mathrm{grbmod}_{R_m}$.
\end{proof}

\begin{observation}
As there are no non-zero morphisms between $n\neq m$, it follows immediately from \cref{prop:sbimiscorrect} that the monoidal $(2,2)$-functor $\iota\colon \BSbimp \to \SBim$ induces a bijection on the set
of isomorphism classes of objects. For objects $n, m \in \SBim$, the additive
$k$-linear hom-category is given by 
\[\eHom_{\SBim}(n,m) \simeq \left\{\begin{array}{lr} 0 & n \neq m\\
\SBim_n & n = m \end{array}\right. ,\] where $\SBim_n$ is the ordinary
$k$-linear additive category from \cref{def:Sbimcl}
with $\mbbZ$-action by grading shift. The monoidal shift-preserving $k$-linear
functor $ \SBim \to \MorPoly$ sends objects $n$ to the polynomial algebra $R_n$
and is given on hom-categories by the evident full inclusion of
$\mathrm{Sbim}_n$ into the category of all graded bimodules which are
graded-compact-projective as right modules. 

It follows from functoriality that the composition of $1$-morphisms in $\eHom_{\SBim}(n,n)$
is given by the relative tensor product $-\otimes_{R_n}-$ (and hence, that the endomorphism $1$-category $\eHom_{\SBim}(n,n)$ is equivalent to the \emph{monoidal} category $\SBim_n$ as in
\cref{def:Sbimcl}), and that the monoidal structure of $\SBim$ is given by $-\otimes_k-$ (and hence acts
by parabolic induction on the hom-categories $\SBim_n \times \SBim_m \to \SBim_{n+m}$ as in \cref{def:parind}, see \cref{rem:monbicat}).
\end{observation}

\subsection{The monoidal \texorpdfstring{$(\infty,2)$}{(infty,2)}-category of chain complexes of Soergel bimodules}
\label{sec:defKSBim}

Recall from Sections~\ref{sec:freestable} and~\ref{sec:Kbasfreestable} that the forgetful functor $\st \to \add$ from the presentably symmetric monoidal $\infty$-category of small stable $\infty$-categories to that of small additive $\infty$-categories has a symmetric monoidal left adjoint 
$$\Kb \colon \add \to \st ,$$
which sends an ordinary additive $1$-category $\cA$ to the $\infty$-category of bounded chain complexes in $\cA$ with chain maps and (higher) chain homotopies between them. 
By \cref{prop:presentablestKI},  this is compatible with linearity and $\mbbZ$-action and induces a symmetric monoidal left adjoint 
$$\Kb \colon \addkBZ \to \stkBZ$$
of the forgetful functor $\stkBZ \to \addkBZ$. 
Applying $\Kb$ homwise, it follows from \cref{subsec:enriched-infty-cats} and \cref{subsubsec:adjunctions-of-cO-algebras}
that we obtain symmetric monoidal left-adjoints of the respective forgetful functors:
\[        \Kbloc := \Cat[\Kb] \colon \Cat[\addkBZ] \to \Cat[\stkBZ]
\]
\begin{equation}\label{eq:Kbloc-adjunction}
    \Kbloc  \colon \Alg_{\EE_1}\left(\Cat[\addkBZ]\right) \to \Alg_{\EE_1}\left(\Cat[\stkBZ] \right).
\end{equation}

\begin{definition}\label{def:KblocSBim}
     We call the monoidal $k$-linear stable $(\infty,2)$-category with local shifts  \[\Kbloc\left(\Sbim\right) \in \Alg_{\EE_1} \left( \Cat[\stkBZ]\right)\] the \emph{chain complex Soergel $(\infty, 2)$-category}. \end{definition}

The unit of the adjunction~\eqref{eq:Kbloc-adjunction} is a monoidal $\addkBZ$-enriched $(\infty,2)$-functor 
\begin{equation}\label{eq:SBim-to-KblocSBim}
\SBim \to \Kbloc(\SBim).
\end{equation}

Since the functor $\Kbloc$ is defined by applying $\Kb$ homwise, we immediately obtain the following explicit description of $\Kbloc(\SBim)$, which completes the construction of $\Kbloc$ as described in \cref{thm:main-M}:

\begin{prop}\label{prop:Kbloc-Sbim-explicit} The objects of $\Kbloc(\SBim)$ agree with those of $\SBim$, while the stable $k$-linear hom-categories are given by 
\[\eHom_{\Kbloc(\SBim)}(n,m) = \left\{ \begin{array}{lr} 0 & n \neq m \\ \Kb(\SBim_n) & n = m \end{array}\right. ,
\]
with $\mbbZ$-action given by internal (i.e. non-homological!) grading shift. 

The functor $\SBim \to \Kbloc(\SBim)$ from~\eqref{eq:SBim-to-KblocSBim} sends objects to themselves and is on hom-categories given by the additive $k$-linear $\mbbZ$-equivariant functor $\SBim_n \hookrightarrow \Kb(\SBim_n)$ including Soergel bimodules as chain complexes concentrated in degree zero. In particular, it is faithful as an $(\infty,2)$-functor. 
\end{prop}
Note that while $\Sbim$ is a $(2,2)$-category, the category  $\Kbloc(\SBim)$ is a true $(\infty,2)$-category with non-trivial higher cells.

The following makes the connection to \cref{sec:2} and justifies the notation $h_1\oldKbloc{\SBim}$ from there.
\begin{corollary} \label{cor:h1KSBim}
The homotopy $1$-category $h_1\Kbloc(\SBim)$ of the monoidal $(\infty,2)$-category $\Kbloc(\SBim)$ agrees with the monoidal $1$-category $h_1\oldKbloc{\SBim}$ from \cref{def:oldhoneKbloc}.

Moreover, after taking homotopy $1$-categories, the monoidal  $(\infty,2)$-functor $\BSBim \to \SBim \to \Kbloc(\SBim)$ becomes the ordinary monoidal $1$-functor \[h_1 \BSBim \to h_1 \oldKbloc{\SBim}\] from~\eqref{eq:defh1K}. 
\end{corollary}

\subsection{The fiber functor on \texorpdfstring{$\Kbloc(\SBim)$}{chain complex Soergel category}}\label{subsec:fiber-functor-on-KbSBim}

We now construct the monoidal shift-preserving $k$-linear exact functor $\Kbloc(\SBim) \to \stkBZ$.

Recall from \cref{cor:finally-our-morita-categories}.\eqref{item-cor:finally-2} the (large) derived Morita category \[\DMoritaS \in \CAlg(\widehat{\Cat}[\st_k^{B\mbbZ}]).\] Its objects can be understood as ordinary $\mbbZ$-graded flat $k$-algebras, and its $k$-linear stable, idempotent-complete hom-category between two objects $A$ and $B$ is given by the full subcategory of the derived $\infty$-category $\Derived( {}_{A}\mathrm{grbmod}_{B})$ 
of the abelian category ${}_{A}\mathrm{grbmod}_{B}$ of graded $A$--$B$ bimodules on those objects which are graded-perfect as right $B$-modules, i.e. are quasi-isomorphic to bounded chain complexes of graded-compact-projective $B$-modules.

As before, it suffices to consider the small full subcategory of polynomial algebras:

\begin{notation}
Let $\DMorPoly \subseteq \DMoritaS$ denote the small full subcategory on the graded polynomial algebras $k[x_1, \ldots, x_n]$ for $n \geq 0$, with all $x_i$ in degree $2$. (Graded polynomial algebras are flat, see
\cref{exm:graded-poly-flat}, this hence indeed defines a full subcategory.)  Since tensor products of polynomial algebras are polynomial algebras, this is in fact a symmetric monoidal subcategory and hence defines an object 
\[ \DMorPoly \in  \CAlg(\Cat[\stkBZ]).
\]\end{notation}

In \cref{cor:finally-our-morita-categories}.\eqref{item-cor:finally-3}  we constructed a symmetric monoidal $\addkBZ$-enriched functor \[\MoritaS \to \DMoritaS\] which sends flat graded algebras to themselves and includes graded bimodules as the discrete objects into the derived $\infty$-category of graded bimodules, and hence restricts to a symmetric monoidal $\addkBZ$-enriched functor 
\[\MorPoly \to \DMorPoly.
\]

By \cref{def:SBim}, there is also a faithful monoidal shift-preserving $k$-linear functor $\Sbim \to \MorPoly$. 

\begin{prop}\label{prop:Hloc}
The monoidal $\addkBZ$-enriched functor \[\Sbim \to \MorPoly \to \DMorPoly\]  factors
through a monoidal $\stkBZ$-enriched
functor
\begin{equation}
\label{eq:fiberfunctorMorita}
\Hloc\colon \Kbloc(\SBim) \to \DMorPoly.
\end{equation}
\end{prop}
\begin{proof} This follows immediately from the adjunction 
\[\begin{tikzcd}[column sep=1.5cm]
\Alg_{\EE_1}(\Cat[\addkBZ])
\arrow[yshift=0.9ex]{r}{\Kbloc({-})}
\arrow[leftarrow, yshift=-0.9ex]{r}[yshift=-0.2ex]{\bot}[swap]{\mathrm{forget}}
&
\Alg_{\EE_1}(\Cat[\stkBZ]).
\end{tikzcd}
\]
\end{proof}

\begin{observation} \label{obs:Hloc-unpacked}Unpacked, the functor $\Hloc\colon \Kbloc(\SBim) \to
\DMorPoly$ sends an object $n$ to the polynomial algebra $R_n:=k[x_1,
\ldots, x_n]$, and a bounded chain complex of Soergel bimodules to the induced
object in $\Derived({}_{R_n} \mathrm{grbmod}_{R_n})^{\mathrm{gr-perf}}$, i.e. the chain complex
considered up to quasi-isomorphism.
Thinking of this as a homotopy coherent version of `taking homology' motivates the notation $\Hloc$. 
\end{observation}

The unpacking of the functor $\Hloc$ from \cref{prop:Hloc} in \cref{obs:Hloc-unpacked} immediately implies the following:
\begin{corollary} \label{handover2} The induced functor on homotopy categories \[h_1\Hloc \colon h_1\Kbloc(\SBim) \to h_1\DMorPoly\] agrees with the functor $h_1{\Hloc}$ from \eqref{eq:h1Hloc}. 
\end{corollary}
    
\begin{nota} \label{not:fiberfunctorst}
Recall from \cref{obs:dmor-to-st} that $\DMoritaS$ has a symmetric monoidal fully faithful $\stkBZ$-enriched functor 
\[\DMoritaS \hookrightarrow \stkBZ.
\] 
We will abuse notation
and also denote by $\Hloc$ the monoidal $\stkBZ$-enriched composite
\begin{equation}\label{eq:fiber-functor}
\Hloc \colon \Kbloc(\Sbim) \to \DMorPoly \hookrightarrow \DMoritaS \hookrightarrow \stkBZ.
\end{equation}
The composite sends a graded $k$-algebra $A$ to the
stable $\infty$-category of bounded chain complexes of graded-compact-projective
right $A$-modules, with $\mbbZ$-action given by the internal (i.e.
non-homological) grading shift. 
\end{nota}

\begin{lemma}\label{lem:SBim-to-st-faithful}
The composite $\addkBZ$-enriched functor $\SBim \to \Kbloc(\SBim) \to \DMorPoly$ and its further composite 
$\SBim \to \Kbloc(\SBim) \to\DMorPoly \hookrightarrow \stkBZ$ are faithful. 
\end{lemma}
\begin{proof}
We show that the former functor is faithful, the latter functor is a composite with a fully faithful functor and hence also faithful. By construction, the former functor factors as $\SBim \to \MorPoly\to \DMorPoly$,  the first of which is faithful by \cref{cor:SBim-to-Mor-faithful}, and  the second is faithful since the induced map on hom-categories between two polynomial algebras $A$ and $B$ is given by the full inclusion. 
\[{}_{A}\mathrm{grbmod}^{\mathrm{gr-cp}}_{B} \hookrightarrow \cD({}_{A} \mathrm{grbmod}_{B})^{\mathrm{gr-perf}}
\]
by \cref{cor:finally-our-morita-categories}.\eqref{item-cor:finally-3}.
\end{proof}

\section{Prebraidings and braidings via \texorpdfstring{$\infty$}{infinity}-operads}

\label{sec:prebraidings-to-E2-str}
\newcommand{\Opun}{\mathrm{Op}^{\mathrm{un}}}

In this section, we introduce the operadic machinery we use in our construction of the braiding on $\Kbloc(\SBim)$. Throughout, we use Lurie's theory of $\infty$-operads developed in~\cite[\S~2]{HA}. \cref{appendix:operads} contains an introduction to our terminology and notation. 
We remind the reader that for an $\infty$-operad $\cO$ we denote its $\infty$-category of operators by $\cO^{\otimes} \to \Fin_*$, and its underlying $\infty$-category by $\underline{\cO}$ and if it is clear from context sometimes also just by $\cO$. Given another $\infty$-operad $\cP$, we denote the $\infty$-operad of $\cO$-algebras in $\cP$ by $\Alg_{\cO}(\cP)$ with underlying $\infty$-category $\underline{\Alg}_{\cO}(\cP)$ --- if clear from context sometimes also just denoted by $\Alg_{\cO}(\cP)$ --- and space of objects $\underline{\Alg}_{\cO}(\cP)^{\simeq} = \hom_{\Op}(\cO, \cP)$. 
Using the left adjoint $\cat \to \Op$ of the underlying-category functor $\underline{(-)}  \colon  \Op \to \cat$, any $\infty$-category $\cC$ can be considered as an $\infty$-operad with empty non-1-ary mapping spaces. Abusing notation, we will also denote this free $\infty$-operad on $\cC$ by $\cC$.

\subsection{Recollections on unital \texorpdfstring{$\infty$}{infty}-operads}\label{subsec:operad-recollection}
We recall some facts about $\infty$-operads from~\cite{HA} and~\cite{SY19} which we will use throughout.

\begin{definition}
    An $\infty$-operad $\cO$ is \emph{unital} if for every color $X \in \underline{\cO}$, the $0$-ary mapping space $\Mul_{\cO}(\emptyset, X)$ is contractible. We let $\Opun$ denote the full subcategory of the $\infty$-category of $\infty$-operads $\Op$ on the unital $\infty$-operads.
\end{definition}

\begin{example}
\label{exm:Enunital}
For all $n \geq 0$, the $\infty$-operads $\EE_n$ are unital.
\end{example}

For $\infty$-operads $\cO$ and $\cP$, recall that the $\infty$-category $\underline{\Alg}_{\cO}(\cP)$ of $\cO$-algebras in $\cP$ admits a \emph{pointwise} operad structure constructed in~\cite[Ex.~3.2.4.4]{HA} and henceforth denoted $\Alg_{\cO}(\cP)$. 

\begin{lemma}
\label{lem:unitalcharacterization} Let $\cO$ be an $\infty$-operad and consider the operad maps $\cO \to \cO \otimes \EE_0$ and $\Alg_{\EE_0}(\cO) \to \cO$ induced from the operad map $\Triv \to \EE_0$. Then the following hold:
\begin{enumerate}
\item \label{itm:unitalchar-1} The $\infty$-operad $\cO \otimes \EE_0$ is unital. Moreover, $\cO$ is unital if and only if the operad map $\cO \to \cO \otimes \EE_0$   is an isomorphism.
\item \label{itm:unitalchar-2} The $\infty$-operad $\Alg_{\EE_0}(\cO)$ is unital. Moreover, $\cO$ is unital if and only if the operad map $\Alg_{\EE_0}(\cO) \to \cO$    is an isomorphism.
\end{enumerate}
\end{lemma}
\begin{proof}
Part \eqref{itm:unitalchar-1} follows directly from~\cite[Prop.~2.3.1.9]{HA}. For part \eqref{itm:unitalchar-2}, note that for an $\infty$-operad $\cP$, the fiber of $\Hom_{\Op}(\EE_0, \cP) \to \Hom_{\Op}(\Triv, \cP) = \underline{ \cP}^{\simeq}$ at a color $X\in \underline{\cP}$ is the $0$-ary mapping space $\Mul_{\cP}(\emptyset, X)$ and hence that $\cP$ is unital if and only if $\Hom_{\Op}(\EE_0, \cP) \to \Hom_{\Op}(\Triv, \cP)$ is an isomorphism. In particular, evaluating at $\cP = \Alg_{\EE_0}(\cO)$ for an $\infty$-operad $\cO$, and using that $\Hom_{\Op}(-, \Alg_{\EE_0}(\cO)) \simeq \Hom_{\Op}(- \otimes \EE_0, \cO)$  and part \eqref{itm:unitalchar-1}, it follows that $\Alg_{\EE_0}(\cO)$ is unital. If $\cO$ is moreover unital, let $\cQ$ be a unital $\infty$-operad and consider the map 
$\Hom_{\Opun}(\cQ, \Alg_{\EE_0}(\cO))\to \Hom_{\Opun}(\cQ, \cO)$ which is equivalent to $\Hom_{\Opun}(\cQ \otimes \EE_0, \cO) \to \Hom_{\Opun}(\cQ, \cO)$. Since $\cQ$ is unital, this is an isomorphism by part \eqref{itm:unitalchar-1}. This completes the proof of part \eqref{itm:unitalchar-2}. 
 \end{proof}
 Evaluating \cref{lem:unitalcharacterization}.\eqref{itm:unitalchar-1} at $\cO= \EE_0$ shows that  $\EE_0$ is a (unital) idempotent in $\Op$ (as also follows from Dunn additivity) with image $\Opun$.
 
 The following is an immediate consequence of \cref{lem:unitalcharacterization}, also see~\cite[Prop.~2.3.1.9]{HA}.
 \begin{corollary} The full inclusion $\Opun \hookrightarrow \Op$ has left and right adjoints
 \[
 \begin{tikzcd}
	\Opun && \Op.
	\arrow[""{name=0, anchor=center, inner sep=0}, hook, from=1-1, to=1-3]
	\arrow["- \otimes \EE_0"'{name=1}, bend right=20, from=1-3, to=1-1]
	\arrow["\Alg_{\EE_0}(-)"{name=2}, bend left =20, from=1-3, to=1-1]
	\arrow["\dashv"{anchor=center, rotate=-90}, draw=none, from=2, to=0]
	\arrow["\dashv"{anchor=center, rotate=-90}, draw=none, from=0, to=1]
 \end{tikzcd}
\]
\end{corollary}

Most unital $\infty$-operads appearing in this paper arise from the following observation:
\begin{example} It follows from \cref{lem:unitalcharacterization} that for any unital $\infty$-operad $\cO$ and $\infty$-operad $\cP$ the operad $\Alg_{\cO}(\cP) \simeq \Alg_{\cO \otimes \EE_0}(\cP) \simeq \Alg_{\EE_0}(\Alg_{\cO}(\cP))$ is unital. 
In particular, since $\EE_n$ is a unital $\infty$-operad for any $n \geq 0$, it follows that for any $\infty$-operad $\cP$, the $\infty$-operads $\Alg_{\EE_n}(\cP)$ are unital.\end{example}

\begin{example} The underlying $\infty$-operad of a  symmetric monoidal $\infty$-category $\cC$  is unital if and only if the tensor unit $I$ of $\cC$ is an initial object. An example of such a symmetric monoidal $\infty$-category is given by the coCartesian tensor product on an $\infty$-category with finite coproducts. \end{example}

The coCartesian monoidal structure on an $\infty$-category with finite coproducts can be generalized to a certain \emph{coCartesian} unital operad structure on any $\infty$-category:

\begin{example}[{\cite[\S~2.4.3]{HA}}]
For any $\infty$-category $\cC$, there is a unital $\infty$-operad $\cC_{\sqcup}$ with underlying $\infty$-category $\cC$ and multi-ary mapping spaces
\[\Mul_{\cC_{\sqcup}}(X_1, \ldots, X_n; Y) \simeq  \Hom_{\cC}(X_1, Y) \times \cdots \Hom_{\cC}(X_n, Y)
\]
\end{example}
These coCartesian operads have the following universal characterization:
\begin{lemma}[{\cite[Prop. 2.4.3.9, Cor. 2.4.3.11]{HA}, \cite[Lem. 2.2.3]{SY19}}]The assignment $\cC \mapsto \cC_{\sqcup}$ induces a fully faithful functor $\cat \hookrightarrow \Opun$ which is right adjoint to the underlying-category functor $\Opun \to \cat$.
\end{lemma}
In particular, it follows that the unit of the adjunction is an operad map $\cO \to \underline{\cO}_{\sqcup}$ which induces the identity on underlying $\infty$-categories $\underline{\cO} \to \underline{\underline{\cO}_{\sqcup}} \simeq \underline{\cO}$. 
Fixing colors $X_1, \ldots, X_n, Y \in \underline{\cO}$, this induces a map \begin{equation}
\label{eq:sigma}
\sigma \colon \Mul_{\cO}(X_1, \ldots, X_n; Y) \to \Hom_{\underline{\cO}}(X_1, Y) \times \cdots \times \Hom_{\underline{\cO}}(X_n, Y).
\end{equation}
The components $\Mul_{\cO}(X_1, \ldots, X_n; Y) \to \Hom_{\underline{\cO}}(X_i, Y)$ of this map may be thought of as inserting units in all but the $i$-th slot.

\subsection{The operads \texorpdfstring{$\AA_2$}{A2} and \texorpdfstring{$\T_2$}{T2}}
\label{subsec:A2-T2}
For a unital $\infty$-operad $\cO$, we can use the unit-inserting map $$\sigma \colon \Mul_{\cO}(X_1, \ldots, X_n; Y) \to \Hom_{\underline{\cO}}(X_1, Y) \times \cdots \times \Hom_{\underline{\cO}}(X_n, Y)$$ from~\eqref{eq:sigma} to give a quick definition of the well-known notion of a unital $\AA_2$-algebra, which encodes a left and right-unital binary multiplication without any associativity requirements:
\begin{definition}
\label{def:A2algebrabyhand}
A \emph{unital $\AA_2$-algebra} in a unital $\infty$-operad $\cO$ is given by a color $X \in \underline{\cO}$ equipped with an element of the pullback $$\Mul_{\cO}(X,X;X)\times_{\Hom_{\underline{\cO}}(X,X)^{\times 2}} \{(\id_X, \id_X)\},$$ i.e. a  $2$-ary map $\mu \in \Mul_{\cO}(X, X; X)$ and an identification of its image $(\mu(1,-), \mu(-1))$ under $\sigma \colon \Mul_{\cO}(X, X; X) \to \Hom_{\underline{\cO}}(X, X)^{\times 2}$ with $(\id_X, \id_X)$.

A \emph{unital $\AA_2$-algebra} in a (not necessarily unital) $\infty$-operad $\cO$ is a unital $\AA_2$-algebra in the unital $\infty$-operad $\Alg_{\EE_0}(\cO)$.
\end{definition}

\cref{lem:unitalcharacterization}.\eqref{itm:unitalchar-2} ensures that \cref{def:A2algebrabyhand} is well-defined for non-unital operads and unambigous for unital $\infty$-operads.

We will also be concerned with the following relative version:
\begin{definition}
\label{def:T2algebrabyhand}
A \emph{$\T_2$-algebra} in a unital operad $\cO$ is given by a pair of colors $X,Y \in \underline{\cO}$ equipped with an element of the pullback $$\Mul_{\cO}(X,X;Y) \times_{\Hom_{\underline{\cO}}(X,Y)^{\times 2}} \Hom_{\underline{\cO}}(X,Y),$$ i.e. a $2$-ary operation $\mu \in \Mul_{\cO}(X,X;Y)$ and an identification of the $1$-ary operations $\mu(-, 1)$ and $\mu(1,-)$ in $\Hom_{\underline{\cO}}(X,Y)$. 

A \emph{$\TT_2$-algebra} in a (not necessarily unital) $\infty$-operad $\cO$ is a $\TT_2$-algebra in the unital $\infty$-operad $\Alg_{\EE_0}(\cO)$. 
\end{definition}

We now construct $\infty$-operads which corepresent $\T_2$ and $\AA_2$-algebras.  
For this purpose, for any $n \geq 0$,  consider the functor
\[ \nabla_n: \Op \to \Spaces, \hspace{0.5cm} \cO \mapsto \Hom_{(\cat)_{/\Fin_*}}
\left([1] \xrightarrow{\{\underline{n}_+ \to \underline{1}_+\}} \Fin_*,
~~\cO^{\otimes} \to\Fin_* \right), \] which may intuitively be thought of as sending an $\infty$-operad
$\cO$ to the space of $(n+1)$-tuples of colors $X_1,\ldots, X_n,Y \in \underline{\cO}$ equipped
with an $n$-ary map $\mu \in \Mul_{\cO}(X_1, \ldots, X_n; Y)$.

\begin{lemma}
\label{lem:nabla}
 The functor $\nabla_n \colon  \Op \to \Spaces$ is corepresented by an $\infty$-operad, also denoted $\nabla_n$.
\end{lemma}
\begin{proof} Using Lurie's combinatorial simplicial model category of $\infty$-preoperads~\cite[\S~2.1.4]{HA}, one can define the $\infty$-operad $\nabla_n$ as a fibrant resolution of the $\infty$-preoperad $[1]  \xrightarrow{\underline{n}_+ \to \underline{1}_+} \Fin_*$. 
\end{proof}

For a unital $\infty$-operad $\cO$ and a $\nabla_2$-algebra $(X,Y,Z, \mu \in \Mul_{\cO}(X,Y;Z))$, we may insert units into $\mu$ to extract morphisms $X\to Z$ and $Y \to Z$ in $\underline{\cO}$, This leads to the following:

\newcommand{\wspan}{\raisebox{0.2ex}{\scalebox{0.7}{$ \{\bullet \to \bullet \leftarrow \bullet\}$}}} 
\newcommand{\warrow}{\raisebox{0.2ex}{\scalebox{0.7}{$ \{\bullet \to \bullet\}$}}}

\begin{lemma} 
\label{lem:underlying}
The $\infty$-category $\underline{\nabla_2 \otimes \EE_0}$ underlying the unital $\infty$-operad $\nabla_2 \otimes \EE_0$ is equivalent to the `walking span' $\wspan$, i.e. the pushout $[1] \sqcup_{[0]} [1]$.  
\end{lemma}
\begin{proof}
By adjunction, for any $\infty$-category $\cC$, we have 
\[ \Hom_{\cat}(\underline{\nabla_2 \otimes \EE_0}, \cC) \simeq \Hom_{\Opun}( \nabla_2 \otimes \EE_0, \cC_{\sqcup}) \simeq \Hom_{\Op} (\nabla_2, \cC_{\sqcup}).
\]
The latter space explicitly unpacks to the space of triples $X, Y, Z\in \cC$ with maps $X\to Y \leftarrow Z$ and hence $\Hom_{\Op} (\nabla_2, \cC_{\sqcup}) \simeq \Hom_{\cat}(\wspan, \cC)$.
\end{proof}

Consider the codiagonal functor $\wspan \to \warrow= [1]$ which identifies the two morphisms in the span. 

\begin{definition} We define the $\TT_2$-operad and the $\AA_2$-operad as the following pushouts of unital $\infty$-operads:
\begin{equation}\label{eq:defT2}
\begin{tikzcd}
\underline{\nabla_2 \otimes \EE_0} \arrow[r] \arrow[d]\arrow[dr, phantom, "\ulcorner", very near end]&{[1]} \otimes \EE_0  \arrow[r] \arrow[d]\arrow[dr, phantom, "\ulcorner", very near end]& \arrow[d] \EE_0 \\
\nabla_2 \otimes \EE_0  \arrow[r]& \TT_2  \arrow[r]& \AA_2 
\end{tikzcd} 
\end{equation}
\end{definition}

The definitions together with \cref{lem:underlying} immediately imply that algebras of the $\T_2$- and $\AA_2$-operad are $\T_2$- and $\AA_2$-algebras in the sense of Definitions~\ref{def:T2algebrabyhand} and~\ref{def:A2algebrabyhand}. 
\begin{corollary}\label{cor:A2T2unpacked}
Let $\cO$ be a (not necessarily unital) $\infty$-operad. 
\begin{enumerate}
\item Given a map of operads $[1] \otimes \EE_0 \to \cO$, equivalently a map of operads $[1] \to \Alg_{\EE_0}(\cO)$ specified by a morphism $f \colon A \to B$ in $\underline{\Alg}_{\EE_0}(\cO)$, the above pushout induces an isomorphism of spaces:
\[ \Hom_{{\Op_{[1] \otimes \EE_0/}}} \left( \TT_2, \cO \right) \simeq \Mul_{\Alg_{\EE_0}(\cO)}(A,A;B) \times_{\Mul_{\Alg_{\EE_0}(\cO)}(A;B)^2} \{(f,f)\}.
\]
\item Given a map of operads $\EE_0 \to \cO$, equivalently an $A\in \Alg_{\EE_0}(\cO)$, the above pushout induces an isomorphism of spaces    \begin{equation}
        \label{eq:A2operadfiberprod}
    \Hom_{\Op_{\EE_0/}}(\AA_2,\cO) \simeq
    \Mul_{\Alg_{\EE_0}(\cO)}(A,A;A) \times_{\Mul_{\Alg_{\EE_0}(\cO)}(A;A)^2} \{(\id_A,\id_A)\}
\end{equation}
\end{enumerate}
\end{corollary}

\begin{example}\label{exm:fromE1toA2} Using \cref{cor:A2T2unpacked}, we may factor the canonical map $\EE_0 \to \EE_1$ through an operad map $\AA_2 \to \EE_1$ which remembers of an $\EE_1$-algebra only the binary multiplication and its unitality structure. This is the first step in the well-known filtration $\EE_0 = \AA_1 \to \AA_2 \to \ldots \to \AA_{\infty} = \EE_{1}$ of the $\EE_1$-operad by the unital $\AA_n$ operads, encoding higher coherent associativity (see~\cite[\S~4.1.4]{HA} for a non-unital version of this filtration in the setting of $\infty$-operads). In this paper, we will only need the first stage $\EE_0 \to \AA_2\to \EE_1$.  
\end{example}

\begin{notation}
\label{not:T2}
 For an $\infty$-operad $\cO$ and an operad map $[1] \otimes \EE_0 \to \cO$ corepresenting a morphism $f \colon A \to B \in \Alg_{\EE_0}(\cO)$, we write $$\T^{\cO}_2(f)  \coloneqq \Hom_{\Op_{[1] \otimes \EE_0/}}(\T_2, \cO) \simeq \Mul_{\Alg_{\EE_0}(\cO)}(A,A;B) \times_{\Mul_{\Alg_{\EE_0}(\cO)}(A;B)^2} \{(f,f)\}.$$
 for the space of $\TT_2$-structures on $f$.
 
  For an $\infty$-operad $\cO$ and an operad map $\EE_0 \to \cO$ corepresenting  an $\EE_0$-algebra $A$ in $\cO$, we write $$\AA^{\cO}_2(A) \coloneqq \Hom_{\Op_{\EE_0/}}(\AA_2, \cO) \simeq \Mul_{\Alg_{\EE_0}(\cO)}(A,A;A) \times_{\Mul_{\Alg_{\EE_0}(\cO)}(A;A)^2} \{(\id_A,\id_A)\}.$$
 for the space of $\AA_2$-structures on $A$.
 
 If clear from context, we will often drop the superscript $\cO$ indicating the ambient $\infty$-operad and simply write $\TT_2(f)$ and $\AA_2(A)$. 
\end{notation}

A straight-forward, but very useful consequence of the definition is that $\TT_2$-structures transport along adjunctions: In an adjunction, $\TT_2$-structures on a morphism $f  \colon LA \to B$ are canonically identified with $\TT_2$-structures on its adjunct $A\to RB$:

\begin{lemma}\label{prop:T2adjunction} 
Let $\cC$ and $\cD$ be symmetric monoidal $\infty$-categories whose monoidal units are initial. Let 
\[
        \begin{tikzcd}
        \cC \arrow[r,shift left=.5ex,"L"]
        &
        \cD \arrow[l,shift left=.5ex,"R"].
        \end{tikzcd}
\]
be an adjunction with (strongly) symmetric monoidal left adjoint $L$ and unit denoted by $\eta  \colon  \id_{\cC} \To RL$. Then, the adjunction isomorphism 
\[\psi_{A,B}: \Hom_{\cD}(LA, B)  \xrightarrow{R(-)} \Hom_{\cC}(RLA, RB)  \xrightarrow{- \circ \eta_A } \Hom_{\cC}(A, RB)
\]
induces via the maps from \cref{obs:functorialityT2} for every $f\in \Hom_{\cD}(LA, B)$ an isomorphism
\[\TT_2(f) \xrightarrow{R(-)} \TT_2(Rf)  \xrightarrow{-\circ \eta_A} \TT_2( Rf \circ \eta_A) = \TT_2(\psi_{A,B}(f)).
\]
\end{lemma}
\begin{proof}
By adjunction and monoidality of $L$, the horizontal maps in the commuting diagram 
\[\begin{tikzcd}
\Hom_{\cD}(LA^{\otimes 2}, B) \arrow[r] \arrow[d] & \arrow[d] \Hom_{\cC}(A^{\otimes 2}, RB) \\
\Hom_{\cD}(LA, B)^{\times 2} \arrow[r] & \Hom_{\cD}(A, RB)^{\times 2}
\end{tikzcd}
\]
are isomorphisms and hence so is the induced map between the fibers at $\{(f,f)\} \to \Hom_{\cD}(LA, B)^{\times 2}$.
\end{proof}

\begin{observation}\label{obs:functorialityT2}
The space $\TT_2(f)$ of $\TT_2$-structures on a given $\EE_0$-morphism $f  \colon A \to B$ in $\Alg_{\EE_0}(\cO)$ is compatible with composition and functorial: 
Given an operad map $F \colon \cO \to \cP$, applying $F$ induces a map of spaces 
\[ \TT_2(f)  \xrightarrow{F(-)} \TT_2(F(f)).
\]

Similarly, any $\EE_0$-morphism $g \colon  B \to C$ in $\Alg_{\EE_0}(\cO)$ induces evident maps of spaces
\begin{equation}\label{eq:functorial-T2}
\TT_2(f)  \xrightarrow{g \circ -} \TT_2(g \circ f) 
\hspace{1cm}
\TT_2(g)  \xrightarrow{- \circ (f,f)} \TT_2(g \circ f) .
\end{equation}

It will be useful to express this operation in terms of $\infty$-operads. Let $[2]\coloneqq \{0 <1<2\}$ denote the $\infty$-category (and the free $\infty$-operad on that $\infty$-category) corepresenting a pair of composable morphisms.
Consider the following pushouts of $\infty$-operads 
\[\begin{tikzcd}
{[1]} \otimes \EE_0 \ar[r] \ar[d, "\{0<1\}"']\arrow[dr, phantom, "\ulcorner", very near end]& \ar[d]\TT_2 \\
{[2]} \otimes \EE_0 \ar[r]& \TT_2 \sqcup_{\{0< 1\}} [2] 
\end{tikzcd}
\hspace{0.8cm}
\begin{tikzcd}
{[1]} \otimes \EE_0 \ar[r] \ar[d, "\{1<2\}"']\arrow[dr, phantom, "\ulcorner", very near end]& \ar[d]\TT_2 \\
{[2]} \otimes \EE_0 \ar[r]& \TT_2 \sqcup_{\{1< 2\}} [2] 
\end{tikzcd}
\hspace{0.8cm}
\begin{tikzcd}
{[1]} \otimes \EE_0 \ar[r] \ar[d, "\{0<2\}"']\arrow[dr, phantom, "\ulcorner", very near end]& \ar[d]\TT_2 \\
{[2]} \otimes \EE_0 \ar[r]& \TT_2 \sqcup_{\{0< 2\}} [2] 
\end{tikzcd}
\]
corepresenting a pair of $\EE_0$-morphisms $A \xrightarrow{f} B  \xrightarrow{g} C$ with a $\TT_2$-structure on $f$, $g$ or $g\circ f$, respectively. 
By Yoneda, the maps of spaces constructed above induce operad maps \[\TT_2\sqcup_{\{1< 2\}} [2] \leftarrow \TT_2\sqcup_{\{0<2\}}[2] \to \TT_2 \sqcup_{\{0<1\}}[2].\]  
\end{observation}

\subsection{Relative  \texorpdfstring{$\TT_2$}{T2}-structures}\label{subsec:relative-T2}
Analogous to (and as we will see later --- generalizing) \cref{def:relative-prebraiding}, we introduce $\TT_2$-structures relative to a given $\AA_2$-structure.

As a consequence of \cref{obs:functorialityT2}, given an $\EE_0$-morphism $f  \colon A \to B$ in an $\infty$-operad $\cO$, applying~\eqref{eq:functorial-T2} in the case $g=\id_B$,  we obtain a map of spaces
\[  \AA_2(B) = \TT_2(\id_B)  \xrightarrow{ -\circ (f,f)} \TT_2(f).
\]
In particular, any $\AA_2$-structure on $B$ (e.g induced by a genuine $\EE_1$- or even $\EE_{\infty}$-structure) induces a $\TT_2$-structure on $f$. 

This allows us to introduce the following notion,  analogous  to \cref{def:relative-prebraiding}.
\begin{definition} \label{def:relative-T2}Let $\cO$ be an $\infty$-operad, $C$ an $\AA_2$-algebra (with $\AA_2$-structure denoted by $\alpha \in \AA_2^{\cO}(C)$)  and let $A \xrightarrow{f} B \xrightarrow{g}C$  be morphisms of  $\EE_0$-algebras. 
We define the \emph{space $\TT_2^{\cO}(f)_{/C}$ of $\TT_2$-structures on $f$ relative to $C$} as the pullback of the span
\[ \{ \alpha\} \to \AA_2^{\cO}(C) = \TT^{\cO}_2(\id_C)  \xrightarrow{- \circ (g\circ f, g\circ g)} \TT_2(g\circ f) \xleftarrow{g \circ -}\TT_2(f).
\]
\end{definition}
In words, a $\TT_2$-structure on $f$ relative to $C$ is a $\TT_2$-structure on $f$ with an identification of the induced $\TT_2$-structure on $g\circ f$ with the $\TT_2$-structure on $g\circ f$ induced by the $\AA_2$-structure $\alpha$ on $C$.

Recall from \S\ref{subsubsec:overcat-sym-mon} that 
for an $\EE_{\infty}$-algebra $C$ in a symmetric monoidal $\infty$-category $\cV$, the over-$\infty$-category $\cV_{/C}$ inherits a symmetric monoidal structure so that for any $\infty$-operad $\cO$, there is an equivalence of $\infty$-categories $\underline{\Alg}_{\cO}(\cV_{/C}) \simeq \underline{\Alg}_{\cO}(\cV)_{/C}$, where for the latter category we consider $C$ as equipped with the $\cO$ algebra structure induced by restricting its $\EE_{\infty}$-structure along the terminal operad map $\cO \to \EE_{\infty}$. 

\begin{prop}\label{prop:two-perspectives-on-relative-T2} Let  $C$ be an $\EE_{\infty}$-algebra in a symmetric monoidal $\infty$-category $\cV$, and consider an $\EE_0$-algebra morphism in $\cV_{/C}$, i.e. equivalently a commuting diagram of $\EE_0$-algebra morphisms in $\cV$ as follows
\[\begin{tikzcd}
A\ar[rr, "f"]\ar[dr] &&B \ar[dl]\\
& C
\end{tikzcd}.
\]
Then, we have an equivalence of spaces
\[
\TT_2^{\cV_{/C}} (f) \simeq \TT^{\cV}_2(f)_{/C},
\]
where for the latter space we consider $C$ as equipped with the $\AA_2$-structure induced by its $\EE_{\infty}$-structure.
\end{prop}
In words: `Absolute' $\TT_2$-structures on $f$ in the sense of \cref{not:T2} seen as a morphism in the symmetric $\infty$-category $\cV_{/C}$ coincide with relative $\TT_2$-structures on $f$ in $\cV$ in the sense of \cref{def:relative-T2}.
\begin{proof}
By definition, the space $\TT_2^{\cV_{/C}}(f)$ is the fiber of the functor $\underline{\Alg}_{\TT_2}(\cV_{/C}) \to \underline{\Alg}_{[1] \otimes \EE_0}(\cV_{/D})$ at $f$. By the universal property of the symmetric monoidal structure on the over-category $\cV_{/C}$, this is equivalent to the functor $\underline{\Alg}_{\TT_2}(\cV)_{/C} \to \underline{\Alg}_{[1] \otimes \EE_0}(\cV)_{/C}$. Unwinding the definition of morphisms in over-categories, this results in the desired equivalence. 
\end{proof}

\subsection{Centralizers and centers}
\label{subsec:centralizers-and-centers}
\newcommand{\cent}{\mathfrak{Z}}
In light of \cref{thm:classbraidings}, we recall the $\infty$-categorical theory of centers and centralizers developed in~\cite[\S~5.3.1]{HA} and relate them to $\T_2$- and $\AA_2$-algebras.

\begin{definition}[{\cite[Def. 5.3.1.2]{HA}}] Let $f  \colon  A \to B$ be a morphism in a monoidal $\infty$-category $\cC$. A \emph{centralizer} $\cent(f)$ of $f$ is a final object in the $\infty$-category 
\[\cC_{I/} \times_{\cC_{A/}} \cC_{A// B},\] 
where the functor $\cC_{I/} \to \cC_{A/}$ sends objects $(\alpha \colon I\to X)$ of $\cC_{I/}$ to $(A \simeq I \otimes A \xrightarrow{\alpha \otimes \id_A} X\otimes A) \in \cC_{A/}$.
\end{definition}

Unpacked, a centralizer is  an object $\cent(f) \in \cC$ equipped with morphisms $u  \colon  I \to \cent(f)$ and  $\ev  \colon  \cent(f) \otimes A \to B$, such that the following diagram commutes
    \begin{equation}
        \label{eq:centralizer}
        \begin{tikzcd}
            & \cent(f) \otimes A  \ar[dr, "\ev"]&  \\ 
            I \otimes A \simeq A \ar[rr, "f"]  \ar[ru, "u \otimes \id_A"]   &&B,
        \end{tikzcd}
    \end{equation}
     and which is final among such pairs: for any  pair of morphisms $u_X\colon I \to X$ in $\cC$ and $\ev_X\colon X\otimes A \to B$ making the analog of \eqref{eq:centralizer} commute, there is a unique morphism $\phi\colon X \to \cent(f)$ such that
    \begin{equation}
        \label{eq:centralizerup}
        \begin{tikzcd}
            & X \otimes A \ar[d, dashed, "\phi\otimes \id_A"]  \ar[ddr,bend left=30,"\ev_X"]&  \\
            & \cent(f) \otimes A  \ar[dr,"\ev"]&  \\ 
            I \otimes A\ar[ru, "u \otimes \id_A"] \ar[ruu,bend left=30,"u_X \otimes \id_A"]\simeq A \ar[rr, "f"]   &&B.
        \end{tikzcd}
    \end{equation}
    commutes.
    
    \begin{example}Given a functor $F  \colon  \cC \to \cD$ of $\infty$-categories, by \cite[Rmk. 5.3.1.4]{HA} the centralizer  in $\cat$ exists  and is given by the functor category $\Fun(\cC, \cD)$ with pointing $u  \colon  \{F\} \to \Fun(\cC, \cD)$ and $\ev  \colon  \Fun(\cC, \cD) \times \cC \to \cD$ given by the evaluation functor. 
    \end{example}
    
    \begin{notation}\label{not:kcentralizer} Let $\cC$ be a presentably symmetric monoidal $\infty$-category. For every morphism $f  \colon A \to B$ in the $\infty$-category $\Alg_{\EE_k}(\cC)$, the centralizer in $\Alg_{\EE_k}(\cC)$ exists~\cite[Cor. 5.3.1.15]{HA} and will henceforth be denoted by $Z_k(f)\in\Alg_{\EE_k}(\cC)$. 
    \end{notation}

\begin{example}\label{exm:universal-property-centralizer}
It is straight-forward to verify that for a monoidal functor $F  \colon \cC \to \cD$ between ordinary monoidal $1$-categories, the centralizer $Z_1(F) \in \Alg_{\EE_1}(\Cat_1)$  in the $(2,1)$-category $\Alg_{\EE_1}(\Cat_1)$ of monoidal $1$-categories is given by the category from \cref{def:centralizer} with unit $1_{Z(f)}\colon \pt \to Z_1(F)$ and $\ev  \colon  Z_1(F) \times \cC \to \cD$ described in \cref{def:centralizer} and \cref{def:evaluation}.\end{example}

As with ordinary monoidal categories, the case $f= \id_A \colon A \to A$ is of special interest. For $\cC$ a monoidal $\infty$-category, recall from~\cite[Def. 4.2.1.13]{HA} the \emph{$\infty$-category $\LMod(\cC)$ of left module objects}, whose objects are pairs of an algebra $A\in \Alg_{\EE_1}(\cC)$ and an $A$-module $M\in \LMod_A(\cC)$, and the functor $\LMod(\cC) \to \cC$ which sends a pair of an algebra $A$ and a module ${}_{A}M$ to the underlying object $M$.

\begin{definition}[{\cite[Def. 5.3.1.6]{HA}}] Let $\cC$ be a  monoidal $\infty$-category and $M$ an object of $\cC$. A \emph{center $\cent(M)$ of
    $M$} is a final object of the $\infty$-category $\LMod(\cC) \times_{\cC} \{M\}$.
\end{definition}
Unpacked, a center $\cent(M)$ is an $\EE_1$-algebra $\cent(M)$ in $\cC$ with a left action on $M$ so that all other left actions of $\EE_1$-algebras $A$ on $M$ factor through $\cent(M)$.

\begin{proposition}[{\cite[Prop. 5.3.1.8]{HA}}]\label{prop:center} Let  $M$ be an object in a monoidal $\infty$-category $\cC$. If  $\id_M\colon M
    \to M$ has a centralizer in $\cC$, then $M$ has a center. Furthermore, an
    $\EE_1$-algebra $A$ with a left-action on $M$ is a center of $M$ if and only
    if the underlying maps $I \to A$ and $A\otimes M \to M$
    exhibit $A$ as a centralizer of $\id_M$.
\end{proposition}

    \begin{example}Given a small $\infty$-category $\cC$, the center in $\cat$ exists  and is given by the monoidal $\infty$-category $\Fun(\cC, \cC) \in \Alg_{\EE_1}(\cat)$.
    \end{example}
    
    \begin{notation} Let $\cC$ be a presentably symmetric monoidal $\infty$-category. For every $A\in \Alg_{\EE_k}(\cC)$, the center in $\Alg_{\EE_k}(\cC)$ exists~\cite[Cor. 5.3.1.15]{HA} and will henceforth be denoted by \[Z_k(A)\in\Alg_{\EE_1}(\Alg_{\EE_k}(\cC)) = \Alg_{\EE_{k+1}}(\cC).\] 
    By \cref{prop:center}, the underlying $\EE_k$-algebra of $Z_k(A)$ agrees with the centralizer $Z_k(\id_A)$ from \cref{not:kcentralizer}.
    \end{notation}

\begin{example}
For $A \in \Alg_{\EE_1}(\Cat_1)$ an ordinary monoidal $1$-category, the center is given by the Drinfeld center $Z_1(A)\in \Alg_{\EE_2}(\Cat_1)$, see \cite[Exm. 5.3.1.18]{HA}.
\end{example}

\begin{notation} 
\label{not:forgetfulcenter}
Let $\cC$ be a symmetric monoidal $\infty$-category whose tensor unit $I$ is initial and let $f  \colon A \to B$ be a morphism in $\cC$ whose centralizer $\cent(f) \in \cC$ exists. Then, we define $\ev_{1_A}  \colon  \cent(f) \to B$ to be the composite $$\cent(f) \simeq \cent(f) \otimes I  \xrightarrow{\id_{\cent(f)} \otimes 1_A} \cent(f) \otimes A  \xrightarrow{\ev} B,$$
where $1_A\in \Hom_{\cC}(I, A) \simeq *.$ Below, we will mostly use this morphism in the case that $\cC =  \Alg_{\EE_k}(\cA)$ for a presentably symmetric monoidal $\infty$-category $\cA$, in which case this defines an $\EE_k$-algebra map $Z_k(f) \to B$ for any $\EE_k$-algebra map $f  \colon A \to B$ in $\Alg_{\EE_k}(\cA)$. 
\end{notation}

The relevance of centralizers to this paper arises from the following proposition:
\begin{proposition}\label{prop:TT2-and-centralizer} 
Let $\cC$ be a symmetric
    monoidal $\infty$-category and $f\colon A\to B$ be a morphism in
    $\Alg_{\EE_0}(\cC)$. If the centralizer $Z_0(f)$ in $\Alg_{\EE_0}(\cC)$ exists, then the space
\[ \T_2(f)\coloneqq \Hom_{\Op_{[1] \otimes \EE_0/}}(\T_2, \cC) \simeq \hom_{\Alg_{\EE_0}(\cC)}(A\otimes A, A) \times_{\hom_{\Alg_{\EE_0}(\cC)}(A,B)^2} \{(f,f)\}\]
    of $\T_2$-structures  on $f$ 
    is equivalent to the following space of (dashed) lifts in $\Alg_{\EE_0}(\cC)$:
    \begin{equation}\label{eq:TT2-and-centralizer}
     \begin{tikzcd}
         & Z_0(f) \ar[dr, "\ev_{1_A}"] & \\ 
         A\ar[rr, "f"] \ar[ru, dashed] & &B
     \end{tikzcd}
    \end{equation}
\end{proposition}
\begin{proof}
The space of lifts~\eqref{eq:TT2-and-centralizer} is by definition the fiber of \[\hom_{\Alg_{\EE_0}(\cC)}(A, Z_0(f))  \xrightarrow{\ev_{1_A} \circ -} \hom_{\Alg_{\EE_0}(\cC)}(A, B)\] at $f$. By the universal property of the centralizer and since $I$ is initial in $\Alg_{\EE_0}(\cC)$, the space $\hom_{\Alg_{\EE_0}(\cC)}(A, Z_0(f))$ is equivalent to the space $\hom_{\Alg_{\EE_0}(\cC)}(A\otimes A, B)\times_{\hom_{\Alg_{\EE_0}(\cC)}(A, B)}\{f\}$ of lifts 
\[ 
     \begin{tikzcd}
     A\otimes A \arrow[r, dashed] & B \\
     A\arrow[u, "\id_A \otimes 1_A"] \arrow[ru, "f"] & \end{tikzcd}
\]
in $\Alg_{\EE_0}(\cC)$, where $1_A  \colon  I \to A$ denotes the unique morphism in $\Alg_{\EE_0}(\cC)$. 
Under this identification,  the map $\hom_{\Alg_{\EE_0}(\cC)}(A, Z_0(f))   \xrightarrow{\ev_{1_A} \circ -}\hom_{\Alg_{\EE_0}(\cC)}(A, B)$ unpacks to the composite \[\hom_{\Alg_{\EE_0}(\cC)}(A\otimes A, B)\times_{\hom_{\Alg_{\EE_0}(\cC)}(A, B)}\{f\} \to \hom_{\Alg_{\EE_0}(\cC)}(A\otimes A, B) \to \hom_{\Alg_{\EE_0}(\cC)}(A, B)\] of the projection and the  precomposition with $A\simeq I \otimes A  \xrightarrow{1_A \otimes \id_A}A\otimes A$. 
Hence, the fiber  of $\hom_{\Alg_{\EE_0}(\cC)}(A, Z_0(f))  \xrightarrow{\ev_{1_A} \circ -}\hom_{\Alg_{\EE_0}(\cC)}(A, B)$ at $f$ is equivalent to the space 
\[ 
\hom_{\Alg_{\EE_0}(\cC)}(A\otimes A, B)\times_{\hom_{\Alg_{\EE_0}(\cC)}(A, B)^{\times 2}} \{(f,f)\} =: \TT_2(f). \qedhere 
\]
\end{proof}

\begin{corollary}\label{prop:A2-and-center} 
Let $\cC$ be a symmetric
    monoidal $\infty$-category and $A \in \Alg_{\EE_0}(\cC)$. If the center $Z_0(A) \in \Alg_{\EE_1}(\cC)$ exists, then the space
\[ \AA_2(A) \coloneqq \Hom_{\Op_{ \EE_0/}}(\AA_2, \cC) \simeq \hom_{\Alg_{\EE_0}(\cC)}(A\otimes A, A) \times_{\hom_{\Alg_{\EE_0}(\cC)}(A,B)^2} \{(\id_A,\id_A)\}\]
    of $\AA_2$-structures  on $A$ 
    is equivalent to the following space of (dashed) lifts in $\Alg_{\EE_0}(\cC)$:
    \begin{equation}\label{eq:A2-and-center}
     \begin{tikzcd}
         & Z_0(A) \ar[dr, "\ev_{1_A}"] & \\ 
         A\ar[rr, "\id_A"] \ar[ru, dashed] & &A
     \end{tikzcd}
    \end{equation}
\end{corollary}
\begin{proof}
Apply \cref{prop:TT2-and-centralizer} for $f=\id_A$ using that $Z_0(A) = Z_0(\id_A)$. 
\end{proof}

\begin{observation}\label{obs:twoA2structures}
Let $\cC$ be a symmetric monoidal $\infty$-category, and $A\in \Alg_{\EE_1}(\cC)$. Assume the center $Z_0(A)\in \Alg_{\EE_1}(\cC)$ of the underlying $\EE_0$-algebra $A \in \Alg_{\EE_0}(\cC)$ exists. 
Since $A$ has a left action on itself, the universal property of $Z_0(A) \in \Alg_{\EE_1}(\cC)$ induces an $\EE_1$-homomorphism $A\to Z_0(A)$. Moreover, since the multiplication of $A$ is right unital, it follows that the composite $A\to Z_0(A)\to A$ in $\cC$  is isomorphic to $\id_A$. Hence, by \cref{prop:A2-and-center}, this induces an $\AA_2$-structure on $A$.
Unpacked, this $\AA_2$-structure coincides with the one induced from the operad map $\AA_2 \to \EE_1$ from \cref{exm:fromE1toA2}. 
\end{observation}

 By \cref{obs:twoA2structures}, any $\EE_1$-structure on an $\EE_0$-algebra $A$ induces an $\EE_1$-monoidal section $A\to \cent(A)$ of the $\EE_0$-monoidal $\cent(A)\to A$. We now show that this section and its monoidality can be uniquely recovered 
 from the $\EE_1$-structure on $A$.

 \begin{lemma} 
\label{lem:A2-to-E1}

Let $\cC$ be a symmetric monoidal $\infty$-category with initial tensor unit and let $X$ be an object in $\cC$ whose center $\cent(X) \in \Alg_{\EE_1}(\cC)$ exists. Then, the forgetful functor
\begin{equation}
\label{eq:E1vsE1lift}
\Alg_{\EE_1}(\cC)_{/\cent(X)} \times_{\cC_{/X}} \{ \id_X\} \to \Alg_{\EE_1}(\cC) \times_{\cC} \{X\}
\end{equation}
is an equivalence of $\infty$-categories.
\end{lemma}
\begin{proof}
Since $\Alg_{\EE_1}(\cC) \to \cC$ is conservative, the $\infty$-categories in~\eqref{eq:E1vsE1lift} are $\infty$-groupoids, and equivalently given by the fibers of the respective maps of maximal subgroupoids. 

For an $\EE_1$-algebra $A$ in $\cC$, it follows from the free-forgetful adjunction $  \begin{tikzcd}\cC
         \arrow[r,shift left=.5ex]
        &
         \arrow[l,shift left=.5ex] \LMod_A(\cC) 
        \end{tikzcd} $ and initiality of the unit of $\cC$ that the free $A$-module $_{A}A$ is initial in $\LMod_A(\cC)$. 
This induces a functor $\LMod_A(\cC) \simeq \LMod_{A}(\cC)_{{}_{A}A/} \to \cC_{A/}$ which sends a left $A$-module $M$ to the morphism \[\mathrm{act}_1  \colon  A\simeq A \otimes I  \xrightarrow{\id_A \otimes !} A \otimes M  \xrightarrow{\mathrm{act}} M.\]

Applying this functor  $\mathrm{act}_1  \colon  \LMod_A(\cC) \to \cC_{A/}$ fiberwise induces a commuting diagram of spaces

\begin{equation}
\label{eq:A2ToE1}
\begin{tikzcd}
\LMod(\cC)^{\simeq} \arrow[r, " \mathrm{act}_1"]  \arrow[d]& \hom_{\cat}([1], \cC) \arrow[d, "s"]\\
\Alg_{\EE_1}(\cC)^{\simeq} \arrow[r] & \cC^{\simeq}
\end{tikzcd}.
\end{equation}

By the universal property of the center,  the $\infty$-category $\Alg_{\EE_1}(\cC)_{/\cent(X)}$ is equivalent to the $\infty$-category $\LMod(\cC) \times_{\cC} \{X\}$, and hence the space $\underline{\Alg}_{\EE_1}(\cC)_{/\cent(X)} \times_{\cC_{/X}} \{ \id_X\}$ is equivalent to the fiber of the top horizontal map in~\eqref{eq:A2ToE1} at $\{ \id_X\} \in \hom_{\cat}([1], \cC)$. The functor~\eqref{eq:E1vsE1lift} is the induced map from the fiber of the top horizontal map of~\eqref{eq:A2ToE1} at $\id_X$ to the fiber of the bottom horizontal map of~\eqref{eq:A2ToE1} at $X$. 

Let $\hom_{\cat}([1], \cC)^{\mathrm{iso}}$ denote the full subspace of $\hom_{\cat}([1], \cC)$ on the invertible arrows in $\cC$.  In particular, the composite $\hom_{\cat}([1], \cC)^{\mathrm{iso}}\to\hom_{\cat}([1], \cC)   \xrightarrow{s}\cC^{\simeq}$ is an equivalence. We now show that the composite map of spaces
\begin{equation}\label{eq:A2ToE1fibers} \LMod(\cC)^{\simeq} \times_{\hom_{\cat}([1], \cC)} \hom_{\cat}([1], \cC)^{\mathrm{iso}} \to \LMod(\cC)^{\simeq} \to \Alg_{\EE_1}(\cC)^{\simeq}
\end{equation}
is an equivalence,  which concludes the proof as $\{ \id_X\} \in \hom_{\cat}([1], \cC)$ is an object of the full subspace $\hom_{\cat}([1], \cC)^{\mathrm{iso}}$.

It suffices to verify that all fibers of~\eqref{eq:A2ToE1fibers} are contractible. By definition, for $A\in \Alg_{\EE_1}(\cC)$ the fiber is the full subspace of $\LMod_A(\cC)^{\simeq} \simeq \left(\LMod_A(\cC)_{{}_{A}A/}\right)^{\simeq}$ on those modules $M$ for which the induced map $\mathrm{act}_1  \colon  A \to M$ is an equivalence. But since $\LMod_A(\cC) \to \cC$ is conservative, this is the full subcategory $\LMod_A(\cC)_{{}_{A}A/^{\mathrm{iso}}}$ on the invertible module functors ${}_{A}A \to _{A}M$ and hence contractible. 
\end{proof}

The following corollary finally justifies the presence of $\TT_2$-structures in this paper:

\begin{corollary}
\label{cor:classicalbraidings} Let $F: \cA \to \cB$ be an ordinary monoidal functor between ordinary monoidal $1$-categories.
\begin{enumerate}
\item \label{item-cor:classicalbraidings-T2}
A $\TT_2\otimes \EE_1$-structure on $F$ is a prebraiding on $F$ in the sense of \cref{def:prebraiding}. More precisely, the \emph{space} $\TT_2(F)$ of $\TT_2$-structures on $F$ is discrete and equivalent to the \emph{set} $\PreBraid(F)$ of prebraidings on $F$ defined in \cref{def:notationprebraiding}.
\item \label{item-cor:classicalbraidings-A2}
An $\AA_2 \otimes \EE_1$-structure on $\cA$ is a braiding on the monoidal category $\cA$, in the usual $1$-categorical sense.  More precisely, the \emph{space} $\AA_2(\cA)$ of $\AA_2$-structures on $\cA$ is discrete and equivalent to the \emph{set} $\Braid(\cA)$ of braidings on $\cA$ defined in \cref{def:notationprebraiding}.\end{enumerate}
\end{corollary}
\begin{proof}
By \cref{prop:TT2-and-centralizer}, a $\TT_2 \otimes \EE_1$-structure on $F$ is a lift in $\Alg_{\EE_1}(\Catnk{1}{1})$:
\[\begin{tikzcd}
&Z_1(F) \ar[dr]&\\
A\ar[rr, "F"] \ar[ur, dashed] && B
\end{tikzcd}.
\]
Explicitly, the space of $\TT_2\otimes \EE_1$-structures on $F$ is therefore the $1$-groupoid of weak lifts from \cref{thm:classbraidings}.\eqref{itm:weaklifts}. By \cref{thm:classbraidings}, this is equivalent to the set of prebraidings on $F$ in the sense of \cref{def:notationprebraiding}. This completes the proof of part \eqref{item-cor:classicalbraidings-T2}.
Part \eqref{item-cor:classicalbraidings-A2} follows by applying statement \eqref{item-cor:classicalbraidings-T2} to the case $F=\id_{\cA}$. 
\end{proof}

The operad map $\AA_2 \to \EE_1$ induces an operad map $\AA_2\otimes \EE_1 \to \EE_1 \otimes \EE_1 \simeq \EE_2$. Hence, any $\EE_2$-structure gives rise to an $\AA_2 \otimes \EE_1$-structure, but not necessarily vice versa. 
However, \cref{cor:classicalbraidings} shows that $\AA_2 \otimes \EE_1$-structures on \emph{$1$-categories} agree with braided monoidal structures, which are well-known to coincide with $\EE_2$-structures on $1$-categories. This hints at a certain connectivity of the operad map $\AA_2 \otimes \EE_1 \to \EE_2$ which we will study in the next sections.

\subsection{A factorization system on the \texorpdfstring{$\infty$}{infinity}-category of operads}\label{subsec:fact-sys-on-operads}
We introduce the obstruction theoretic machinery at the heart of our main theorem.

\begin{definition} \label{defn:n-surj-and-n-trunc-for-infty-opds} For $n \geq -1$, we say that a morphism of $\infty$-operads is \emph{$n$-surjective} if it is surjective-on-objects on $\infty$-categories of colors and multi-homwise $(n-1)$-connected and we say that it is \emph{$n$-faithful} if it is multi-homwise $(n-1)$-truncated. We extend this to the case that $n = -2$ by declaring that every morphism of $\infty$-operads is $(-2)$-surjective, and that a morphism of $\infty$-operads is $(-2)$-faithful if and only if it is an equivalence.
\end{definition}

In this section, we will show that the $n$-faithful and $n$-surjective operad maps form a factorization system on the $\infty$-category $\Op$ of $\infty$-operads.
Recall that by definition,  $\Op$ is a subcategory of ${\Cat_{\infty}}_{/\Fin_*}$.

\begin{lemma}\label{obs:n-surj-and-n-trunc-for-infty-opds-is-unambiguous}
A morphism of $\infty$-operads is $n$-surjective (resp.\! $n$-faithful) if and only if its image under the composite functor $\Op \hookra (\Cat_\infty)_{/\Fin_*} \xra{\fgt} \Cat_\infty$ is so (in the sense of \cref{def:nsurj-and-nfaith}).
\end{lemma}

\begin{proof} The Segal conditions for $\infty$-operads imply that surjectivity on underlying functors is equivalent to surjectivity on $\infty$-categories of colors. Moreover, the hom-spaces in an $\infty$-operad are disjoint unions of (finite) products of multi-hom spaces, and these operations both preserve the class of $(n-1)$-connected (resp.\! $(n-1)$-truncated) morphisms of spaces. (To see that products preserve $(n-1)$-connectedness (resp.\! $(n-1)$-truncatedness), note that this notion is determined fiberwise, that fibers of a product of morphisms of spaces are computed factorwise since limits commute with limits, and that $(n-1)$-truncated (resp.\! $(n-1)$-connected) spaces are stable under products.)
\end{proof}

\begin{proposition}
\label{prop:fs-for-infty-opds}
For any $n \geq -2$, the classes of ($n$-surjective, $n$-faithful) operad maps defines a factorization system on the $\infty$-category $\Op$ of $\infty$-operads.
\end{proposition}

\begin{proof}
By \Cref{obs:fs-on-undercat-and-overcat}.\eqref{item-obs:fs-on-undercat-and-overcat-general-case}, the ($n$-surjective, $n$-faithful) factorization system on $\Cat_\infty$ pulls back to a factorization system on $(\Cat_\infty)_{/\Fin_*}$. By \Cref{obs:n-surj-and-n-trunc-for-infty-opds-is-unambiguous}, the classes of our asserted factorization system are restricted along the inclusion $\Op \hookra (\Cat_\infty)_{/\Fin_*}$. So, in order to verify that they indeed define a factorization system on $\Op$, we verify the equivalent conditions of \Cref{obs:restricted-fs-on-subcat}.

In order to proceed, we recall that given two $\infty$-operads $\cO,\cO' \in \Op$, a morphism $\cO^{\otimes} \ra \cO'^{\otimes}$ in $(\Cat_\infty)_{/\Fin_*}$ lies in $\Op$ if and only if it is inert-coCartesian (i.e\! it preserves coCartesian lifts over inert morphisms in $\Fin_*$). Moreover, we make the following observation for repeated future use.
\begin{itemize}

\item[$(*)$] Assuming that $n \geq 0$, if a morphism in $\Op$ is $n$-surjective then it is surjective on inert-coCartesian morphisms.

\end{itemize}

We now turn to condition \eqref{item-obs:restricted-fs-on-subcat-lift-lies-in-subcat} of \Cref{obs:restricted-fs-on-subcat}: given a solid commutative diagram
\[ \begin{tikzcd}
\cO_1
\arrow{r}
\arrow{d}[swap]{f}
&
\cO_3
\arrow{d}{g}
\\
\cO_2
\arrow{r}
\arrow[dashed]{ru}
&
\cO_4
\end{tikzcd} \]
in $\Op$ in which $f$ is $n$-surjective and $g$ is $n$-faithful, we must show that the dashed lift in $(\Cat_\infty)_{/\Fin_*}$ (which exists and is unique due to its factorization system) also lies in $\Op$. This is trivial in the case that $n < 0$, and in the case that $n \geq 0$ this follows immediately from $(*)$.

We now turn to condition \eqref{item-obs:restricted-fs-on-subcat-factorization} of \Cref{obs:restricted-fs-on-subcat}: given any morphism $\cO \xra{h} \cO'$ in $\Op$, we must show that the factorization
\begin{equation}
\label{eq:factorizn-in-Cat-over-Finstar-will-be-in-Op-too}
\begin{tikzcd}
\cO^{\otimes}
\arrow{rr}{h}
\arrow[dashed]{rd}[sloped, swap]{l}
&
&
\cO'^{\otimes}
\\
&
\Fact(h)
\arrow[dashed]{ru}[sloped, swap]{r}
\end{tikzcd}
\end{equation}
in $(\Cat_\infty)_{/\Fin_*}$ determined by its ($n$-surjective, $n$-faithful) factorization system in fact lies in $\Op$. To simplify our notation, we write $\cF^{\otimes} \coloneqq \Fact(h)$. Additionally, we write $\cO^{\otimes} \xra{p} \Fin_*$, $\cO'^{\otimes} \xra{p'} \Fin_*$, and $\cF^{\otimes} \xra{q} \Fin_*$ for the indicated functors. We note immediately that the claim is trivial both when $n = -2$ (since then $\cF^{\otimes} \xra{r} \cO'^{\otimes}$ is an equivalence) and when $n = -1$ (since then $\cF^{\otimes} \xra{r} \cO'^{\otimes}$ is the inclusion of the full suboperad on the colors in the image of $\underline{\cO} \xra{\underline{h}} \underline{\cO'}$). So, we henceforth assume that $n \geq 0$.

\newcommand{\ov}{\raisebox{-0.1cm}{/}}

We first show that the functor $\cF^{\otimes} \xra{q} \Fin_*$ admits coCartesian lifts of inert morphisms. For this, fix an object $X \in \cF^{\otimes}_{\underline{m}_+}$ as well as an inert morphism $\underline{m}_+ \xra{\alpha} \underline{n}_+$ in $\Fin_*$. Because the functor $\cO^{\otimes} \xra{l} \cF^{\otimes}$ is surjective, we may choose a lift $\tilde{X} \in \cO_{\underline{m}_+}^{\otimes}$ of $X$. Let $\tilde{X} \xra{\tilde{\alpha}} Y$ be a $p$-coCartesian lift of $\alpha$. We claim that $X \simeq l(\tilde{X}) \xra{l(\tilde{\alpha})} l(Y)$ is a $q$-coCartesian lift of $\alpha$. To see this, observe first that $r(l(\tilde{\alpha})) \simeq h(\tilde{\alpha})$ is $p'$-coCartesian (since $h$ is a morphism in $\Op$). Now, to check that $l(\tilde{\alpha})$ is $q$-coCartesian, we must check that the canonical functor $\cF^{\otimes}_{l(Y)/} \ra \cF^{\otimes}_{/l(\tilde{X})} \times_{(\Fin_*)_{\underline{m}_+\ov}} (\Fin_*)_{\underline{n}_+ \ov}$ is an equivalence. This fits into a commutative diagram
\[ \begin{tikzcd}
\cO^{\otimes}_{Y/}
\arrow{r}
\arrow{d}
&
\cF^{\otimes}_{l(Y)/}
\arrow{r}
\arrow{d}
&
\cO'^{\otimes}_{rl(Y)/}
\arrow{d}
\\
\cO^{\otimes}_{\tilde{X}/} \times_{(\Fin_*)_{\underline{m}_+\ov}} (\Fin_*)_{\underline{n}_+\ov}
\arrow{r}
&
\cF^{\otimes}_{l(\tilde{X})/} \times_{(\Fin_*)_{\underline{m}_+\ov}} (\Fin_*)_{\underline{n}_+\ov}
\arrow{r}
&
\cO'^{\otimes}_{rl(\tilde{X})/} \times_{(\Fin_*)_{\underline{m}_+\ov}} (\Fin_*)_{\underline{n}_+\ov}
\end{tikzcd} \]
in which the two outer vertical functors are equivalences. We do so by showing that it is fully faithful and surjective. Since we have assumed that $n \geq 0$ (so that $n-1 \geq -1$), the lower left horizontal functor is surjective, which implies that the middle vertical functor is surjective. To show that it is fully faithful, given any pair of objects in $\cF^{\otimes}_{l(Y)/}$, we may lift them to $\cO^{\otimes}_{Y/}$ (again using that $n \geq 0$), and then examine the induced commutative diagram (of the same shape) on hom-spaces. Because its left and right vertical maps are equivalences, both of its left horizontal maps are $(n-1)$-connected, and both of its right horizontal maps are $(n-1)$-truncated, its middle vertical map is also an equivalence since factorizations for the ($(n-1)$-connected, $(n-1)$-truncated) factorization system on $\Spaces$ are unique. So indeed, the middle vertical functor in the above diagram is an equivalence. This proves that $\cF^{\otimes} \xra{q} \Fin_*$ admits coCartesian lifts of inert morphisms, as desired.

The same argument proves the Segal conditions for $\cF^{\otimes}$, which establishes that $\cF^{\otimes}$ is indeed an $\infty$-operad. Moreover, it also proves that $l$ preserves inert-coCartesian morphisms, and in combination with $(*)$ we find that $r$ preserves inert-coCartesian morphisms as well. So all in all, the factorization \eqref{eq:factorizn-in-Cat-over-Finstar-will-be-in-Op-too} lies in the subcategory $\Op \subset (\Cat_\infty)_{/\Fin_*}$, which proves condition \eqref{item-obs:restricted-fs-on-subcat-factorization} of \Cref{obs:restricted-fs-on-subcat}. So indeed, the ($n$-surjective, $n$-faithful) factorization system on $(\Cat_\infty)_{/\Fin_*}$ restricts to a factorization system on this subcategory.
\end{proof}

\begin{warning}\label{warn:operad-nfaith-not-ntrunc}
    Similar to the situation with $\infty$-categories outlined in \cref{warn:nfaithful-neq-ntruncated}, we caution the reader that the ($n$-surjective, $n$-faithful) factorization systems on $\Op$ differs from the ($n$-truncated, $n$-connected) factorization systems \cite[Prop. 4.6]{gep17} derived from the presentability of $\Op$. We refer the reader to 
    \cref{warn:nfaithful-neq-ntruncated} for an in-depth comparison which also applies here.
\end{warning}

Similar to \cref{prop:space-of-lift-of-n-surj-m-faithful},
the orthogonality of $n$-surjective and $n$-faithful operad maps may be generalized as follows:
\begin{corollary}\label{cor:operadic-version-of-space-of-lifts}
 Given a (solid) commuting square in $\Op$
  \[
        \begin{tikzcd}
            \cO \ar[d, "F"] \ar[r] & \cA \ar[d, "G"] \\ 
            \cO' \ar[r] \ar[ru, dashed] & \cB
        \end{tikzcd}
\]
    where $F$ is $n$-surjective and $G$ is $m$-faithful for   $m \geq n \geq -2$,   the space of (dashed) lifts is $(m-n-2)$-truncated.
\end{corollary}
\begin{proof}
    When $m = n$, the statement follows from \cref{prop:fs-for-infty-opds}. For $m>n$, consider the commuting square of spaces:
\[
        \begin{tikzcd}[column sep=2em]
            \Hom_{\Op}(\cO', \cA) \ar[rr] \ar[d, "(-1)\trunc"'] && \Hom_{\Op}(\cO, \cA)\times_{\Hom_{\Op}(\cO, \cB)} \Hom_{\Op}(\cO', \cB) \ar[d, "(-1)\trunc"] \\ 
            \Hom_{/\Fin_*}(\cO'^{\otimes}, \cA^{\otimes}) \ar[rr] \ar[d] \ar[drr, phantom, "\lrcorner", very near start]    &&      \Hom_{/\Fin_*}(\cO^{\otimes}, \cA^{\otimes}) \times_{  \Hom_{/\Fin_*}(\cO^{\otimes}, \cB^{\otimes})}   \Hom_{/\Fin^{\otimes}}(\cO'^{\otimes}, \cB^{\otimes}) \ar[d] \\
            \Hom(\cO'^{\otimes}, \cA^{\otimes}) \ar[rr, "(m-n-2)\trunc"] &&      \Hom(\cO^{\otimes}, \cA^{\otimes}) \times_{  \Hom(\cO^{\otimes}, \cB^{\otimes})}   \Hom(\cO'^{\otimes}, \cB^{\otimes})
        \end{tikzcd}
\]
    The space of lifts of our original square is by definition a fiber of the top horizontal map; it hence suffices to prove that this top horizontal map is $(m-n-2)$-truncated. 
    
    By \cref{obs:n-surj-and-n-trunc-for-infty-opds-is-unambiguous}, the $(\infty,1)$-functor $F^{\otimes} \colon \cO^{\otimes} \to \cO'^{\otimes}$ is $n$-surjective  and $G^{\otimes} \colon \cA^{\otimes} \to \cB^{\otimes}$ is $m$-faithful. Hence, it follows from \cref{prop:space-of-lift-of-n-surj-m-faithful} that the bottom horizontal map is $(m-n-2)$-truncated. By definition of the over-category, the bottom square is a pullback square; hence, the middle horizontal map is $(m-n-2)$-truncated.
    Since $\Op$ is a subcategory of ${\Cat_{\infty}}_{/\Fin_*}$, the top vertical maps are $(-1)$-truncated.     
    . Since $m-n-2\geq -1$,  the composite of the left vertical and the middle horizontal map is $(m-n-2)$-truncated. It follows from \Cref{lem:cancellation-for-connectedness-and-truncatedness}.\eqref{item-lem:cancellation-for-truncatedness} that the top horizontal map is $(m-n-2)$-truncated.
\end{proof}

\subsection{Lifting operadic structure}
\label{subsec:lifting-operadic-structure}

In \cref{defn:n-surj-and-n-trunc-for-infty-opds} we introduced the notion of $n$-faithful and $n$-surjective morphisms of operads and showed in~\cref{prop:fs-for-infty-opds} that they form a factorization system on the $\infty$-category $\Op$. In this section, we show how this can be used to lift $\TT_2 \otimes \EE_1$-structures, i.e. prebraidings (see \cref{cor:classicalbraidings}), along certain maps of operads. 

\begin{prop}
\label{prop:operadsurjective} The following operad maps are $0$-surjective: 
\begin{enumerate}
\item \label{item-prop:operadsurjective-1}
 $\underline{\nabla_2 \otimes \EE_0} \otimes \EE_1 \to \nabla_2 \otimes \EE_0 \otimes  \EE_1 \simeq \nabla_2 \otimes \EE_1$, where $\underline{\nabla_2 \otimes\EE_0}$ denotes the (free $\infty$-operad on) the underlying $\infty$-category of $\nabla_2 \otimes \EE_0$; 
\item \label{item-prop:operadsurjective-2}
$[1] \otimes \EE_1\to \TT_2 \otimes \EE_1$;
\item \label{item-prop:operadsurjective-3}
$\EE_1 \to \AA_2 \otimes \EE_1$.
\end{enumerate}
\end{prop}
\begin{proof}
We first prove part \eqref{item-prop:operadsurjective-1}.
By \cref{lem:underlying}, the $\infty$-category $\underline{\nabla_2 \otimes \EE_0}$ is equivalent to the walking span $\wspan$ and hence the $\infty$-operad $\underline{\nabla_2 \otimes \EE_0} \otimes \EE_1 \simeq \wspan \otimes \EE_1$ corepresents spans of $\EE_1$-morphisms.

On underlying categories, the operad map  $\underline{\nabla_2 \otimes \EE_0} \otimes \EE_1 \to \nabla_2 \otimes \EE_1 $ is the identity. It therefore suffices to show that  the induced maps on multi-hom spaces are $(-1)$-connected. Denote the colors of $\wspan \otimes \EE_1$ by $A, C$ and $B$, the morphisms by $f\in \Mul_{\wspan \otimes \EE_1}(A, C), g\in \Mul_{\wspan \otimes \EE_1} (B,C)$, the multiplication cells by $\mu_A\in \Mul_{\wspan\otimes \EE_1}(A, A; A)$,  $\mu_B \in \Mul_{\wspan \otimes \EE_1}(B, B;B)$ and $\mu_C \in \Mul_{\wspan \otimes \EE_1}(C,C;C)$, the unit cells by $1_A \in \Mul_{\wspan \otimes \EE_1}(\emptyset, A), 1_B \in \Mul_{\wspan \otimes \EE_1}(*, B)$ and $1_C \in \Mul_{\wspan \otimes \EE_1}(*, C)$  and use the same notation for their respective images in $\nabla_2 \otimes \EE_1$. The only generating cell of $\nabla_2 \otimes \EE_1$ that is not evidently in the image of $\wspan \otimes \EE_1$ is the binary multiplication stemming from the $\nabla_2$-operad, which we denote by  $\mu \in \Mul_{\nabla_2 \otimes \EE_1}(A,B;C)$. 
We will now show that this additional generator  $\mu$ is also in the image of $ \pi_0 \Mul_{\wspan \otimes \EE_1} (A, B;C) \to \pi_0 \Mul_{\nabla_2\otimes \EE_1}(A, B;C)$ which concludes the proof that $\wspan \otimes \EE_1 \to \nabla_2 \otimes \EE_1$ induces (-1)-connected maps on all multi-hom spaces.

Since $\mu$ is a map of $\EE_1$-algebras, we have a path in $\Mul_{\nabla_2 \otimes \EE_1}(A,A, B, B;C)$ (where we abuse notation and write $- \circ (-\otimes-)$ to denote the evident operadic compositions): 
 \[\mu\circ(\mu_A\otimes \mu_B)\simeq \mu_C\circ (\mu\otimes \mu).\]
 On the other hand, left and right unitality produce paths in $\Mul_{\nabla_2\otimes \EE_1}(A;C)$ and $\Mul_{\nabla_2\otimes \EE_1}(B;C)$, respectively : 
 \[\mu\circ (\id_A\otimes 1_B)\simeq f \hspace{1cm} \mu\circ (1_A \otimes \id_B)\simeq g .\]
 Composing these, we conclude:
 \begin{equation*}
 \mu\simeq \mu\circ(\mu_A\otimes\mu_B)\circ(\id_A \otimes 1_A\otimes 1_B\otimes \id_B)\simeq\mu_C\circ(\mu\otimes \mu)\circ (\id_A \otimes 1_A\otimes 1_B\otimes \id_B)\simeq\mu_C\circ(f\otimes g)
\end{equation*}
Hence, $\mu$ is in the image of \[\Mul_{\wspan \otimes \EE_1}(C,C;C) \times \Mul_{\wspan \otimes \EE_1}(A;C) \times \Mul_{\wspan \otimes \EE_1}(B, C) \]\[\to \Mul_{\wspan \otimes \EE_1} (A, B ;C) \to  \Mul_{\nabla_2\otimes \EE_1}(A, B;C).\] 

Since  $\Op$ is a presentably monoidal category, the pushout squares~\eqref{eq:defT2} induce pushout squares
\[
\begin{tikzcd}
\underline{\nabla_2 \otimes \EE_0} \otimes \EE_1 \arrow[r] \arrow[d]\arrow[dr, phantom, "\ulcorner", very near end]&{[1]} \otimes \EE_1  \arrow[r] \arrow[d]\arrow[dr, phantom, "\ulcorner", very near end]& \arrow[d] \EE_1 \\
\nabla_2 \otimes \EE_1  \arrow[r]& \TT_2 \otimes \EE_1  \arrow[r]& \AA_2  \otimes \EE_1.
\end{tikzcd} 
\]
Since left class in a factorization system is preserved under pushouts, so $[1] \otimes \EE_1 \to \TT_2 \otimes \EE_1$ and $\EE_1 \to \AA_2 \otimes \EE_1$ are also $0$-surjective. 
\end{proof}

\subsection{From \texorpdfstring{$\AA_2 \otimes \EE_1$}{A2 tensor E1}- to \texorpdfstring{$\EE_2$}{E2}-algebras}
\label{subsec:from-A2-E1-to-E-2}
Considering the filtration $\AA_1 \otimes \EE_1 \to \AA_2 \otimes \EE_2 \to \ldots \AA_{\infty} \otimes \EE_1 = \EE_1 \otimes \EE_1 \simeq \EE_2$, an $\AA_2\otimes \EE_1$-structure is less data than a fully coherent $\EE_2$-structure. 
However, as suggested by \cref{cor:classicalbraidings}, $\AA_2\otimes \EE_1$-structures on $1$-categories already agree with $\EE_2$-structures.
In this section, we prove a generization that holds for any $2$-categorical operad.

\begin{definition}
\label{def:2operad}
 For $n\geq -1$, an \emph{$n$-operad} is an $\infty$-operad all of whose multi-hom spaces, i.e. the $\Mul_{\cO}(X_1,\ldots, X_k; Y)$, 
 are $(n-1)$-truncated. We extend this to the case $n=-2$ by declaring the terminal operad to be a $(-2)$-operad.
We denote the full subcategory of $\Op$ on the $n$-operads by $\Op_n$.
\end{definition}
Equivalently, an $\infty$-operad is an $n$-operad if and only if the terminal operad map $\cO \to \EE_{\infty}$ is $n$-faithful.

\begin{example}
A $1$-operad is precisely one that is equivalent to (the nerve of) an ordinary operad.
\end{example}

\begin{example} A symmetric monoidal $(\infty, 1)$-category $\cC$, considered as an $\infty$-operad, is an $n$-operad if and only if its underlying category is an $(n,1)$-category (also see \cref{defn:n-k-cat}). \end{example}

The notion of $n$-truncated morphism defined in \cref{defn:nconn-and-ntrunc} generalizes to any $\infty$-category:

\begin{definition} For $n\geq -2$, a morphism $f  \colon A \to B$ in an $\infty$-category $\cC$ is called $n$-truncated if  the induced map of spaces $\Hom_{\cC}(X, A) \to \Hom_{\cC}(X,B)$ is $n$-truncated, see \cref{defn:nconn-and-ntrunc}, for every object $X\in \cC$. 
\end{definition}
In particular, a morphism $f  \colon A \to B$ in an $\infty$-category $\cC$ is $0$-truncated if for every object $X\in \cC$, the map $\Hom_{\cC}(X,A) \to \Hom_{\cC}(X,B)$ is $0$-truncated, i.e. all its fibers are discrete sets. 
For an ordinary monoidal $1$-category $A$, the monoidal functor $Z_1(A) \to A$ is faithful, and in particular $0$-truncated as a morphism in the $\infty$-category $\Alg_{\EE_1}(\Cat_1)$\ (see \cref{warn:nfaithful-neq-ntruncated}).  The following is a generalization of this statement: 

\begin{prop}
\label{prop:centralizertruncated}
 Let $\cC$ be a symmetric monoidal $(2,1)$-category and $A\in \Alg_{\EE_1}(\cC)$ whose center $Z_1(A) \in \Alg_{\EE_2}(\cC)$ exists. Then, the morphism $Z_1(A) \to A$ is a  $0$-truncated morphism in $\Alg_{\EE_1}(\cC)$. 
\end{prop}
\begin{proof}
Let $X\in \Alg_{\EE_1}(\cC)$. Since $Z_1(A)$ is the centralizer $\cent(\id_A)$ of the morphism $\id_A$ in $\Alg_{\EE_1}(\cC)$ and the unit of $\Alg_{\EE_1}(\cC)$ is initial, the universal property of the centralizer implies that the map  $\Hom_{\Alg_{\EE_1}(\cC) }(X, Z_1(A)) \to \Hom_{\Alg_{\EE_1}(\cC)}(X, A)$ is equivalent to the composite  
\[ \Hom_{\Alg_{\EE_1}(\cC)}(X \otimes A, A) \times_{\Hom_{\Alg_{\EE_1}(\cC)}(A,A)} \{ \id_A\} \to \Hom_{\Alg_{\EE_1}(\cC)}(X \otimes A, A) \to \Hom_{\Alg_{\EE_1}(\cC)}(X,A).\] 
By definition of the $\infty$-operad $\nabla_2$ in \cref{lem:nabla}, the fiber of this map at an $f\in \Hom_{\Alg_{\EE_1}(\cC)}(X,A)$ is precisely the space of lifts of the operad map $\wspan = \underline{\nabla_2 \otimes \EE_0} \to \Alg_{\EE_1}(\cC)$ classified by the span of $\EE_1$-morphisms $X  \xrightarrow{f}A  \xleftarrow{\id_A}A$, to an operad map $\nabla_2 \otimes \EE_0 \to \Alg_{\EE_1}(\cC)$. Equivalently, this is the space of lift
\[
\begin{tikzcd}
\underline{\nabla_2 \otimes \EE_0} \otimes \EE_1  \arrow[d]  \arrow[r] & \cC\\
\nabla_2 \otimes \EE_1 \arrow[ur, dashed] 
\end{tikzcd}.
\]
By assumption, $\cC$ is a symmetric monoidal $(2,1)$-category, hence a $2$-operad and hence the operad map $\cC \to *$ is $2$-faithful. Since the left vertical operad map is $0$-surjective by \cref{prop:operadsurjective}, it follows from \cref{cor:operadic-version-of-space-of-lifts} that this space of lifts is $0$-truncated. 
\end{proof}

\begin{corollary} \label{cor:A2E1-on-2cats} Let $\cC$ be a presentably symmetric monoidal $(2,1)$-category. Then, the map of spaces \[ \Hom_{\Op}(\EE_2, \cC) \to\Hom_{\Op}(\AA_2 \otimes \EE_1, \cC)\]
is an equivalence. 
\end{corollary}
\begin{proof}
Consider the diagram of spaces
\[
\begin{tikzcd}
\Hom_{\Op}(\EE_2, \cC) \arrow[dr]\arrow[rr] & & \arrow[dl] \Hom_{\Op}(\AA_2\otimes \EE_1, \cC) \\
& \Hom_{\Op}(\EE_1,\cC)
\end{tikzcd}.
\]
To prove that the horizontal map is an equivalence, it suffices to show that for every $A\in \Alg_{\EE_1}(\cC)$, the induced map between fibers 
\[
\Hom_{\Op}(\EE_2,\cC) \times_{\Hom_{\Op}(\EE_1,\cC)} \{A\} \to \Hom_{\Op}(\AA_2 \otimes \EE_1,\cC) \times_{\Hom_{\Op}(\EE_1,\cC)} \{A\}
\]
is an equivalence. 

Since $\cC$ is presentably symmetric monoidal, it follows that centralizers and centers exist~\cite[Cor. 5.3.1.15]{HA} and hence, by applying \cref{prop:A2-and-center}, that the latter space is equivalent to \[ 
\left(\Alg_{\EE_1}(\cC)_{/Z_1(A)}\right)^{\simeq} \times_{\left(\Alg_{\EE_1}(\cC)_{/A}\right)^{\simeq} } \{\id_A\}.\]
Applying \cref{lem:A2-to-E1} to the $\infty$-category $\Alg_{\EE_1}(\cC)$, we find that the functor 
\[
\Alg_{\EE_2}(\cC)_{/Z_1(A)} \times_{\Alg_{\EE_1}(\cC)/A} \{\id_A\}\to \Alg_{\EE_2}(\cC) \times_{\Alg_{\EE_1}(\cC)} \{A\}  
\]
is an equivalence. It therefore suffices to show that the forgetful functor 
\begin{equation}
\label{eq:E2groupoid}
\Alg_{\EE_2}(\cC)_{/Z_1(A)} \times_{\Alg_{\EE_1}(\cC)/A} \{\id_A\} \to \Alg_{\EE_1}(\cC)_{/Z_1(A)} \times_{\Alg_{\EE_1}(\cC)_{/A}} \{\id_A\}
\end{equation}
is an equivalence. 

\cref{prop:centralizertruncated} implies that the map $Z_1(A) \to A$ is $0$-truncated as a morphism in $\Alg_{\EE_1}(\cC)$. Hence, any section $A\to Z_1(A)$ is $(-1)$-truncated. Therefore, the map ~\eqref{eq:E2groupoid} is equivalent to the map 
\[
\Alg_{\EE_2}(\cC)_{/^{-1}Z_1(A)} \times_{\Alg_{\EE_1}(\cC)/A} \{\id_A\} \to \Alg_{\EE_1}(\cC)_{/^{-1}Z_1(A)} \times_{\Alg_{\EE_1}(\cC)_{/A}} \{\id_A\}
\]
where $- /^{-1} Z_1(A)$ denote full subcategories of $(-1)$-truncated $\EE_1$-maps (see \cref{nota:ar-cR-full-sub}). To prove this is an equivalence, it suffices to show that \begin{equation}
\label{eq:E2braidedlaststep}\Alg_{\EE_2}(\cC)_{/^{-1}Z_1(A)} \to \Alg_{\EE_1}(\cC)_{/^{-1} Z_1(A)}\end{equation} is an equivalence. But given any $(X\hookrightarrow Z_1(A))$
in $\Alg_{\EE_1}(\cC)_{/^{-1}Z_1(A)}$, the fiber of~\eqref{eq:E2braidedlaststep} is equivalent to the space of dashed lifts in $\Op$
\begin{equation}
\label{eq:E2braidedlift}
\begin{tikzcd}
\EE_1 \arrow[d] \arrow[rr, "\{X\hookrightarrow Z_1(A)\}"] &&  \arrow[d, "t"] \mathrm{Ar}^{-1}(\cC) \\
\EE_2 \arrow[urr, dashed] \arrow[rr, " \{Z_1(A)\}"'] && \cC
\end{tikzcd}
\end{equation}
where $\mathrm{Ar}^{-1}(\cC)$ denotes the full symmetric monoidal subcategory of the arrow category $\mathrm{Ar}(\cC) \coloneqq \Fun([1], \cC)$ on the $(-1)$-truncated morphisms.

But since $\EE_1 \to \EE_2$ is $0$-surjective, i.e. essentially surjective on objects and $(-1)$-connected on multi-hom spaces $\EE_1(n)\simeq S_n \to \EE_2(n) = \mathrm{Conf}(n, \mathbb{R}^2)$, and since $\mathrm{Ar}^{-1}(\cC) \to \cC$ is $0$-faithful\footnote{Given $f,g \in \mathrm{Ar}^{-1}(\cC)$, the fiber of the induced map on hom-spaces $\Hom_{\mathrm{Ar}(\cC)} (f, g) \to \Hom_{\cC}(tf, tg)$ at an $h \in \Hom_{\cC}(tf, tg)$ is the space of lifts $\Hom_{\cC_{/tg}}(h \circ f, g)$. Since $g$ is $(-1)$-truncated, this space is $(-1)$-truncated.}, it follows from \cref{prop:fs-for-infty-opds} that the space of lifts~\eqref{eq:E2braidedlift} is contractible. 
\end{proof}

\begin{remark}
\label{rem:classicalbraidingisEtwo}
Since an $\AA_2\otimes \EE_1$-structure on a given monoidal $1$-category is by \cref{cor:classicalbraidings}.\eqref{item-cor:classicalbraidings-A2}  precisely the data of a braiding, \cref{cor:A2E1-on-2cats} in particular implies the well-known observation (see \cite[Ex. ~5.1.2.4]{HA}) that braided monoidal structures and $\EE_2$-structures coinicide on ordinary $1$-categories. \end{remark}

\begin{corollary}
\label{cor:2operadA2}
 For any $2$-operad $\cO$, the map of spaces $\Hom_{\Op}( \AA_2 \otimes \EE_1, \cO) \to \Hom_{\Op}(\EE_2, \cO)$ is an equivalence. 
\end{corollary}
\begin{proof} 
The full inclusion $\Op_2 \hookrightarrow \Op$ admits a left adjoint $h_2 \colon  \Op \to \Op_2$ (constructed in~\cite[Thm. 3.12]{SY19cat}). Moreover, it follows from~\cite[Prop. 3.2.6(4)]{SY19} applied to \cref{cor:A2E1-on-2cats} that the operad map $h_2(\AA_2\otimes \EE_1) \to h_2(\EE_2)$ is an equivalence (i.e. that $\AA_2 \otimes \EE_1 \to \EE_2$ is a \emph{$1$-equivalence} in the terminology of~\cite{SY19}). By adjunction, it follows that for any $2$-operad $\cO$, the map $\Hom_{\Op}(\AA_2\otimes \EE_1, \cO) \to \Hom_{\Op}(\EE_2, \cO)$ is an equivalence. 
\end{proof}

\begin{remark}
In other words, \cref{cor:2operadA2} shows that $\AA_2 \otimes \EE_1 \to \EE_2$ is a \emph{$1$-equivalence} in the sense of~\cite{SY19}, i.e.  it is essentially surjective on the underlying categories and induces an equivalence on the $0$-truncations of all the multimapping spaces.
\end{remark}

\subsection{Lifting maps of algebras}
\label{subsec:lifting-maps-of-algebras}
We end this section with an elementary, but very useful observation about $\infty$-operads.

We recall the following easy fact: Given functors $F, G: \cA \to \cB$ and $H  \colon \cB \to \cC$ of $\infty$-categories and assume that for all $a,a' \in \cA$ the map \begin{equation}
\label{eq:assumption-operad-lemma}\Hom_{\cB}(Fa, Ga')  \xrightarrow{H(-)} \Hom_{\cC}(HFa, HGa')\end{equation} is an equivalence of spaces.

We will now prove that this implies that also the map between spaces of natural transformations 
\[\mathrm{Nat}(F, G) \to \mathrm{Nat}(HF, HG)\]
is an equivalence. Formally, this can be expressed as follows:
\begin{lemma}\label{lem:naturality}
Given a functor $H  \colon  \cB \to \cC$ of $\infty$-categories and a commuting square of $\infty$-categories 
\begin{equation}\label{eq:square-naturality}
\begin{tikzcd}
S^0 ={\{0, 1\}} \ar[r] \ar[d]& \Fun(\cA, \cB) \ar[d] \\
{[1] =\{0<1\}} \ar[r] & \Fun(\cA, \cC).
\end{tikzcd}
\end{equation}
Assume that for any $b_0, b_1 \in \cB$ in the image of $\{0\} \times \cA \to S^0 \times \cA \to \cB$ and $\{1\} \times \cA \to S^0 \times \cA \to \cB$, respectively, and any commuting square 
\[
\begin{tikzcd}
S^0  \ar[r, "{\{b_0, b_1\}}"] \ar[d]&  \cB \ar[d] \\
{[1]} \ar[r] \ar[ur, dashed] &  \cC,
\end{tikzcd}
\]
the space of dashed lifts is contractible. Then, the space of lifts of the square~\eqref{eq:square-naturality} is contractible.
\end{lemma}
\begin{proof}Since $\pt$ and $[1]$ generate $\cat$ under colimits, it suffices to show that for every $a\in \cA$ and every arrow $ [1]  \xrightarrow{\{f\}} \cA$, the induced total squares 
\[\begin{tikzcd}
S^0 \ar[r] \ar[d] & \Fun(\cA, \cB) \ar[d] \ar[r, "\ev_a"] &\ar[d] \cB\\
{[1]} \ar[r] & \Fun(\cA, \cC) \ar[r, "\ev_a"] & \cC
\end{tikzcd}
\hspace{0.2cm}\text{ and }\hspace{0.2cm}
\begin{tikzcd}
S^0 \ar[r] \ar[d] & \Fun(\cA, \cB) \ar[d] \ar[r, "\ev_f"] &\ar[d] \Fun([1], \cB)\\
{[1]} \ar[r] & \Fun(\cA, \cC) \ar[r, "\ev_f"] & \Fun([1], \cC)
\end{tikzcd}
\] 
have contractible spaces of lifts. 
Contractibility of the spaces of lifts of the former square follows immediately from assumption, and for the latter square is a straight-forward computation assuming~\eqref{eq:assumption-operad-lemma} is an equivalence.
\end{proof}

The goal of this subsection is to prove a generalization of this statement for $\infty$-operads.

\begin{prop}\label{prop:magic-prop}
 Let $\cO, \cP$ be $\infty$-operads and let $b$ and $c$ be $\cO$-algebras in $\cP$.
   \begin{enumerate}
            \item \label{enum-prop:magicprop-1}
            Let $F  \colon \cP \to \cQ $ be an operad map such that for all $n\geq 0$ and colors $X_1,\ldots, X_n, Y$ in $\cO$, the map of spaces
        \[
                \Mul_{\cP}(b_{X_1}, \ldots, b_{X_{n}}; c_{Y}) \xrightarrow{F(-)} \Mul_{\cQ}(Fb_{X_1}, \ldots, Fb_{X_{n}}; Fc_{Y})  
\]
            is an equivalence. 
            Then, for any $m \geq 0$, the map of multi-hom spaces
\[
                \Mul_{\Alg_{\cO}(\cP)}(\underbrace{b,\ldots, b}_{m}; c)  \xrightarrow{F(-)}  \Mul_{\Alg_{\cO}(\cQ)}(\underbrace{Fb,\ldots, Fb}_{m}; Fc)     
\]
            is an equivalence.
            \item \label{enum-prop:magicprop-2}
            Let $f  \colon a \to b$ be a morphism of  $\cO$-algebras in $\cP$. Assume that for all $n \geq 0$ and colors $X_1, \ldots, X_n, Y $ in $\cO$,
            the map of spaces
            \[\Mul_{\cP}(b_{X_1}, \ldots, b_{X_{n}}; c_{Y}) \xrightarrow{-\circ(f,\ldots, f)}  \Mul_{\cP}(a_{X_1}, \ldots, a_{X_{n}}; c_{Y})
            \]
            is an equivalence. 
            Then, for any $m \geq 0$, the map of multi-hom spaces
            \[\Mul_{\Alg_{\cO}(\cP)}(\underbrace{b,\ldots, b}_m; c) \xrightarrow{- \circ (f,\ldots, f)}  \Mul_{\Alg_{\cO}(\cP)}(\underbrace{a, \ldots, a}_m; c)
            \]
            is an equivalence for all $n$. 
            \end{enumerate}

\end{prop}

To prove \cref{prop:magic-prop}, we recall the following formula for mapping spaces in $\Alg_{\cO}(\cP)$:

\begin{observation}\label{obs:mapping-space-in-AlgOP}
Given a sequence of objects $(b_1, \ldots, b_n)$ and another object $c$ in $\underline{\Alg}_{\cO}(\cP)$, the multi-hom space $\Mul_{\Alg_{\cO}(\cP)}(b_1, \ldots, b_n; c)$ is explicitly defined, as for any $\infty$-operad, as the space of lifts of the square 
\begin{equation}
\label{eq:operad-magic1}
\begin{tikzcd}
S^0 \ar[d]\ar[rr,"{(b_1, \ldots, b_n); c} "]&& \Alg_{\cO}(\cP)^{\otimes} \ar[d] \\
{[1]} \ar[rr, "\underline{n}_+ \to \underline{1}_+"]  && \Fin_* 
\end{tikzcd}
\end{equation}
Let $\Fin_* \times\Fin_*  \xrightarrow{\wedge} \Fin_*$ denote the smash product symmetric monoidal structure of $\Fin_*$ (see~\cite[Not.~2.2.5.1]{HA}). 
Unwinding the definition of $\Alg_{\cO}(\cP)^{\otimes}$ from \cite[Cons. 3.2.4.1]{HA}, this space of lifts is equivalent to the full subspace of the space of lifts 
\begin{equation}
\label{eq:operad-magic2}
\begin{tikzcd}
S^0 \times \cO^{\otimes} \ar[d] \ar[rr] && \ar[d] \cP^{\otimes}\\
{[1]} \times \cO^{\otimes} \ar[r] \ar[urr, dashed] & \Fin_* \times \Fin_* \ar[r, "\wedge"] \ar[r] & \Fin_*
\end{tikzcd}
\end{equation}
on those lifts with the property that for every vertex $v\in [1]$,  the map $\cO^{\otimes} \simeq \{v\} \times \cO^{\otimes} \to  [1] \times \cO^{\otimes} \to \cP^{\otimes}$ sends inert coCartesian morphisms to inert coCartesian morphisms. However, since $S^0 \to [1]$ is surjective on objects this condition is automatically satisfied since it is satisfied by the top horizontal map. Thus, the space of lifts of~\eqref{eq:operad-magic1}, and hence the multi-hom space $\Mul_{\Alg_{\cO}(\cP)}(b_1, \ldots, b_n; c)$ is equivalent to the space of lifts of~\eqref{eq:operad-magic2}. Hence, after adjunction, it is equivalent to the space of lifts 
\[\begin{tikzcd}
S^0 \ar[r] \ar[d]& \ar[d]\Fun(\cO^{\otimes}, \cP^{\otimes}) \\
{[1]} \ar[r]  \ar[ur, dashed]& \Fun(\cO^{\otimes}, \Fin_*)
\end{tikzcd}
\]
More generally, given any functor of $\infty$-categories $X\to Y$ which is surjective on objects,  an $\infty$-operad map $\cP\to \cQ$ and a commuting square of $\infty$-categories 
\[
\begin{tikzcd}
X \ar[r] \ar[d]& \Alg_{\cO}(\cP)^{\otimes} \ar[d] \\ 
Y \ar[r]  \ar[ur, dashed]& \Alg_{\cO}(\cQ)^{\otimes},
\end{tikzcd}
\]
the same argument shows that the space of (dashed) lifts of this square is equivalent to the space of lifts 
\[\begin{tikzcd}
X \ar[r] \ar[d]& \ar[d]\Fun(\cO^{\otimes}, \cP^{\otimes}) \\
Y \ar[r]  \ar[ur, dashed]& \Fun(\cO^{\otimes}, \cQ^{\otimes}).
\end{tikzcd}
\]
\end{observation}

\begin{proof}[Proof of \cref{prop:magic-prop}]
To prove part \eqref{enum-prop:magicprop-1}, fix an $n\geq 0$ and consider the functor $S^0 = \{0, 1\} \to \Alg_{\cO}(\cP)^{\otimes}$, sending $0$ to $(b, \ldots, b)$ and $1$ to $(c)$. 
          Fix a $\mu \in \Mul_{\Alg_{\cO}(\cQ)}(Fb, \ldots, Fb; Fc)$. This determines a commuting square of $\infty$-categories
          \[
          \begin{tikzcd}
          S^0 \ar[r] \ar[d] & \Alg_{\cO}(\cP)^{\otimes} \ar[d] \\
         { [1]} \ar[r] & \Alg_{\cO}(\cQ)^{\otimes}.
          \end{tikzcd}
          \]
        The fiber of  $\Mul_{\Alg_{\cO}(\cP)}(b,\ldots, b; c) \to \Mul_{\Alg_{\cO}(\cQ)}(Fb,\ldots, Fb; Fc)$ at $\mu$ is precisely the space of lifts of this square. By \cref{obs:mapping-space-in-AlgOP}, this space of lifts is equivalent to the space of lifts 
        \[
                  \begin{tikzcd}
          S^0 \ar[r] \ar[d] & \Fun(\cO^{\otimes}, \cP^{\otimes} ) \ar[d] \\
          {[1]}\ar[r] \ar[ur, dashed] & \Fun(\cO^{\otimes}, \cQ^{\otimes})
          \end{tikzcd}
          \]
      By \cref{lem:naturality}, to prove contractibility of this space of lifts, it suffices to verify that for each $p_0, p_1\in \cP^{\otimes}$ in the image of $\{0\} \times \cO^{\otimes} \to S^0 \times \cO^{\otimes} \to \cP^{\otimes}$ and $\{1\} \times \cO^{\otimes} \to S^0 \times \cO^{\otimes} \to \cP^{\otimes}$, respectively, any square 
        \[
                  \begin{tikzcd}
          S^0 \ar[r, "{\{p_0, p_1\}}"] \ar[d] & \cP^{\otimes}  \ar[d] \\
         { [1]} \ar[r] \ar[ur, dashed] &  \cQ^{\otimes}
          \end{tikzcd}
          \]
          has a contractible space of lifts. Using the Segal condition on $\infty$-operads, this precisely unpacks to the condition in the statement of the proposition.

 To prove part \eqref{enum-prop:magicprop-2}, fix an $n\geq 0$ and a point $h\in \Map_{\Alg_{\cO}(\cP)}(a, \ldots, a; c)$. Let $\Lambda_0^2= \{0<1\}\sqcup_{\{0\}} \{0<2\} =
 \left\{\begin{tikzpicture}[scale=0.5]
 \draw[-{Stealth[scale=0.8, length=1.5mm]}] (0,0) -- (1,0);
 \draw[-{Stealth[scale=0.8, length=1.5mm]}] (0,0) -- (0.5,0.5);
 \end{tikzpicture}\right\}$ denote the outer horn. Then, the multi-ary operation $h$ together with our original operation $f\in \Mul_{\Alg_{\cO}(\cP)}(a;b)$ assembles into a commutative diagram as on the right:
 \[\begin{tikzcd}
 S^0 \times {[1]}\arrow[dr, phantom, "\ulcorner", very near end] \ar[d]\ar[r]&
 \Lambda_0^2
 \arrow{rr}{\left\{
 \begin{tikzpicture}[scale=0.5]
 \draw[-{Stealth[scale=0.8, length=1.5mm]}] (0,0) -- (1,0);
 \draw[-{Stealth[scale=0.8, length=1.5mm]}] (0,0) -- (0.5,0.5);
 \node[below, scale=0.5] at (0.5,0) {h};
 \node[above left, scale=0.5] at (0.25,0.25) {f};
 \end{tikzpicture}\right\}
}\ar[d]
 && \Alg_{\cO}(\cP)^{\otimes} \ar[d]\\
 {[1] \times[1]} \ar[r] & {[2]} \ar[rr, "\{\underline{n}_+ \xrightarrow{\id}\underline{n}_+ \to \underline{1}_+\}"] && \Fin_*
 \end{tikzcd}
 \] 
 The fiber of $\Mul_{\Alg_{\cO}(\cP)}(b,\ldots, b; c) \to \Mul_{\Alg_{\cO}(\cP)}(a, \ldots, a; c)$ at $h$ is precisely the space of lifts of the right square. Since the left square is a pushout, this space is equivalent to the space of lifts of the total square. 
 By \cref{obs:mapping-space-in-AlgOP}, this space of lifts is equivalent to the space of lifts of the square 
 \[\begin{tikzcd}
 S^0 \times {[1]} \ar[r] \ar[d]& \Fun(\cO^{\otimes}, \cP^{\otimes})\ar[d] \\
 {[1] \times [1]} \ar[r] & \Fun(\cO^{\otimes}, \Fin_*)
 \end{tikzcd}
 \]
 and hence to the space of lifts of the square 
 \[\begin{tikzcd}
 S^0 \ar[r] \ar[d]& \Fun(\cO^{\otimes}, \Fun({[1]}, \cP^{\otimes}))\ar[d] \\
 {[1]} \ar[r] & \Fun(\cO^{\otimes}, \Fun({[1]}, \Fin_*)).
 \end{tikzcd}
 \]
 By \cref{lem:naturality}, a sufficient condition for contractibility of this space is that for all pair of objects $c_0, c_1 \in \Fun([1], \cP^{\otimes})$ in the image of $\{0\} \times \cO^{\otimes} \to S^0 \times \cO^{\otimes} \to \Fun([1], \cP^{\otimes})$ and $\{1\} \times \cO^{\otimes} \to S^0 \times \cO^{\otimes} \to \Fun([1], \cP^{\otimes})$, respectively, the space of lifts of all commuting squares of the form
  \[\begin{tikzcd}
 S^0 \ar[r, "{\{c_0, c_1\}}"] \ar[d]& \Fun({[1]}, \cP^{\otimes})\ar[d] \\
 {[1]} \ar[r] &  \Fun({[1]}, \Fin_*)
 \end{tikzcd}
 \]
is contractible. Using the Segal condition on $\cP^{\otimes}$, this is satisfied provided the conditions in the statement of the proposition hold. 

\end{proof}

\section{The main theorem}\label{sec:main-theorem}

Throughout this section, we fix a $\mathbb{Q}$-algebra\footnote{As in \cref{sec:SBim}, all results in this section which do not specifically refer to the categories of Bott-Samelson and Soergel bimodules apply more generally to arbitrary connective ring spectra $k \in \CAlg(\ConnSpectra)$, and to arbitrary commutative monoids $\Monoid \in \CAlg(\Spaces)$ in place of $\mbbZ$.} $k$. 
The goal of this section is to prove our main \cref{intro.maintheorem} and hence to construct a fully coherent $\EE_2$-structure on the monoidal $(\infty,2)$-category $\Kbloc(\SBim)$ introduced in \cref{def:KblocSBim}, compatible with its structure as an object of $ \Alg_{\EE_1}(\Cat[\stkBZ])$ (i.e. compatible with local $k$-linearity and $\mbbZ$-action) together with an $\EE_2$-structure on its monoidal functor $\Hloc \colon \Kbloc(\SBim) \to \stkBZ$ from \cref{not:fiberfunctorst}.

Our proof proceeds by successively rewriting the space of such $\EE_2$-structures into spaces of simpler categorical structures; namely $\infty$-categorical variants of the prebraidings encountered in \S\ref{sec:prebraidCat}.

\subsection{Spaces of braidings and prebraidings}
\label{subsec:space-of-braidings}
\begin{notation} Let $\cV$ be a symmetric monoidal $\infty$-category, or more generally an $\infty$-operad.
\begin{enumerate}
\item For an $\EE_1$-algebra $A$ in $\cV$, we write 
\[\Braid_{\cV}(A):= \Hom_{\Op}(\EE_2, \cV) \times_{\Hom_{\Op}(\EE_1, \cV)} \{A\}
\]
for the space of $\EE_2$-algebra structures on $A$ compatible with the given $\EE_1$-structure and refer to this space as the \emph{space of braidings on $A$}.
\item  For an $\EE_1$-algebra $A$ in $\cV$, we write 
\[\PreBraid_{\cV}(A):= \Hom_{\Alg_{\EE_1}(\cV)}(A\otimes A, A) \times_{\Hom_{\Alg_{\EE_1}(\cV)}(A, A)^{\times 2}}\{\id_A,\id_A\}
\]
 and refer to this space as the \emph{space of prebraidings on $A$}. 
\item For $f  \colon A \to B$ a morphism of $\EE_1$-algebras in $\cV$, we write 
\[
\PreBraid_{\cV}(f) := \Hom_{\Alg_{\EE_1}(\cV)}(A \otimes A, B) \times_{\Hom_{\Alg_{\EE_1}(\cV)}(A, B)^{\times 2}} \{ f, f\}
\]
 and refer to this space as the \emph{space of prebraidings on $f$}. 
\end{enumerate}
\end{notation}
For an $\EE_1$-algebra $A$, it follows by definition that $\PreBraid_{\cV}(A) = \PreBraid_{\cV}(\id_A)$. Recall from \cref{cor:A2T2unpacked} that analogous to $\Braid_{\cV}$,  the spaces of prebraidings are also corepresented by certain $\infty$-operads:
\begin{align*}
\PreBraid_{\cV}(A) & = \Hom_{\Op}( \AA_2 \otimes \EE_1, \cV) \times_{\Hom_{\Op}(\EE_1, \cV)} \{A\}\\
\PreBraid_{\cV}(f:A \to B) &=\Hom_{\Op}(\TT_2 \otimes \EE_1, \cV) \times_{\Hom_{\Op}([1] \otimes \EE_1, \cV)} \{f\}
\end{align*}
Moreover, for any $\EE_1$-algebra $A$, composing with the operad map $\AA_2 \otimes \EE_1 \to \EE_1 \otimes \EE_1 \simeq \EE_2$ from \cref{exm:fromE1toA2} defines a `forgetful' map of spaces
\begin{equation}
\label{eq:forget-braiding} 
\Braid_{\cV}(A) \to \PreBraid_{\cV}(A).
\end{equation}

\begin{example}\label{ex:prebraidon1cat}
For an ordinary monoidal $1$-category $\cA \in \Alg_{\EE_1}(\Catnk{1}{1})$ it follows from \cref{cor:2operadA2}  that the map of spaces 
\[\Braid_{\Catnk{1}{1}}(\cA) \to \PreBraid_{\Catnk{1}{1}}(\cA).
\]
is an equivalence. By~\cref{cor:classicalbraidings}\eqref{item-cor:classicalbraidings-A2}, these spaces are equivalent to the (discrete) \emph{set} of classical braidings on $\cA$. 

Similarly, it follows from \cref{cor:classicalbraidings}\eqref{item-cor:classicalbraidings-T2} that for monoidal functors $F \colon \cA \to \cB$ between ordinary monoidal $1$-categories, the spaces $\PreBraid_{\Catnk{1}{1}}(F)$ are equivalent to the (discrete) \emph{set} of classical prebraidings on $F$ in the sense of \cref{def:notationprebraiding}. 
\end{example}

\begin{warning}
\cref{ex:prebraidon1cat} is key to our paper, and is at the heart of  an observation already encountered in \cref{rem:onecatprebraid}: While braidings and prebraidings on ordinary monoidal $1$-categories coincide, the notions already diverge for monoidal $2$-categories; the map~\eqref{eq:forget-braiding} is in general far from an equivalence. 
\end{warning}

As in \S\ref{sec:prebraidCat}, we would also like to consider spaces of prebraidings over a given fixed prebraiding. 

The following generalizes \cref{def:relative-prebraiding} to the $\infty$-categorical setting. 
\begin{definition}\label{def:relative-prebraiding-inf}
Let $\cV$ be a symmetric monoidal $\infty$-category, $C$ an $\EE_2$-algebra (or merely an $\AA_2 \otimes \EE_1$-algebra) and $A\xrightarrow{f} B \xrightarrow{g}C$ be maps of $\EE_1$-algebras.  We define the \emph{space $\PreBraid_{\cV}(f)_{/C}$ of prebraidings on $f$ over $C$} to be the space $\TT^{\Alg_{\EE_1}(\cV)}_2(f)_{/C}$ from \cref{def:relative-T2}.  

Unpacked, a prebraiding on $f$ over $C$ is therefore a prebraiding on $f$ together with an identification of the induced prebraiding on $g\circ f$ with the one induced by the $\EE_2$-structure of $C$. 
\end{definition}

\begin{example}\label{exm:relative-prebraid-1}
It follows from \cref{ex:prebraidon1cat}  that for an ordinary braided monoidal $1$-category $\cC$, and monoidal $1$-functors $\cA \xrightarrow{F} \cB \xrightarrow{g} \cC$ between ordinary monoidal $1$-category, the space $\PreBraid_{\Catnk{1}{1}}(F)_{/C}$ of prebraidings on $F$ over $C$ in the sense of \cref{def:relative-prebraiding-inf} agrees with the set $\PreBraid_{/C}(F)$ of prebraidings over $C$ from \cref{def:notationprebraiding}. 
\end{example}

Recall from \S\ref{subsubsec:overcat-sym-mon} that 
for an $\EE_{\infty}$-algebra $C$ in a symmetric monoidal $\infty$-category $\cV$, the over-$\infty$-category $\cV_{/C}$ inherits a symmetric monoidal structure so that for any $\infty$-operad $\cO$, there is an equivalence of $\infty$-categories $\underline{\Alg}_{\cO}(\cV_{/C}) \simeq \underline{\Alg}_{\cO}(\cV)_{/C}$.

\begin{example}\label{exm:relative-braidings}
It immediately follows from the defining property of $\cV_{/C}$, that for a 
given $\EE_1$-algebra in $\cV_{/C}$, i.e. an $\EE_1$-algebra $A$ equipped with an $\EE_1$-algebra map $A\to C$,  the space $\Braid_{\cV_{/C}}(A)$ encodes a compatible $\EE_2$-structure on $A$ together with  $\EE_2$-structure on the $\EE_1$-morphism $A\to C$. \end{example}

Following \cref{exm:relative-braidings},  to prove \cref{intro.maintheorem} and to study $\EE_2$-structures on $\Kbloc(\SBim)$ together with $\EE_2$-structures on its fiber functor $\Kbloc(\SBim) \to \stkBZ$, we will therefore need to study the space $\Braid_{\Cat[\stkBZ]_{/\stkBZ}}\left(\Kbloc(\SBim) \right)$ and relate it to certain spaces of prebraidings in certain over-categories. These spaces can be understood in terms of prebraidings over given prebraidings in the sense of \cref{def:relative-prebraiding-inf}:
\begin{corollary}\label{cor:relative-prebraidings-are-relative}
Let $C$ be an $\EE_{\infty}$-algebra in a symmetric monoidal $\infty$-category $\cV$, and let $F$ be a morphism of $\EE_1$-algebras in $\cV_{/C}$, i.e. equivalently a commuting diagram 
\[\begin{tikzcd}
A\ar[rr, "f"]\ar[dr] &&B \ar[dl]\\
& C
\end{tikzcd}
\]
of $\EE_1$-algebras in $\cV$. Then, the spaces 
\[\PreBraid_{\cV_{/C}}(F) \simeq \PreBraid_{\cV}(F)_{/C}
\]
are equivalent. 
\end{corollary}
\begin{proof}
Apply \cref{prop:two-perspectives-on-relative-T2} to the symmetric monoidal $\infty$-category $\Alg_{\EE_1}(\cV_{/C}) \simeq \Alg_{\EE_1}(\cV)_{/C}$. 
\end{proof}

\begin{example}\label{exm:relative-prebraid-2}
Combining \cref{cor:relative-prebraidings-are-relative} with \cref{exm:relative-prebraid-1} we find that in the setup of \cref{exm:relative-prebraid-1}, the space $\PreBraid_{(\Cat_{(1,1)})_{/C}}(F)$ agrees with the  set $\PreBraid_{/C}(F)$ of  prebraidings  over $C$ from \cref{def:notationprebraiding}. \end{example}

\subsection{Statement of the main theorem}
\label{subsec:statement-of-main-thm}
By \cref{exm:relative-braidings},  the space of compatible $\EE_2$-structures on $\Kbloc(\SBim)$ together with compatible $\EE_2$-structures on its fiber functor $\Hloc \colon \Kbloc(\SBim) \to \stkBZ$ is precisely given by the space 
\[
\Braid_{\Cat[\stkBZ]_{/\stkBZ}}(\Kbloc(\SBim)),
\]
where $\Kbloc(\SBim)$ is seen as an $\EE_1$-algebra in $\Cat[\stkBZ]_{/\stkBZ}$ via its $\EE_1$-functor $\Hloc \colon \Kbloc(\SBim) \to \stkBZ$. 

Our main theorem will prove that this space is in fact a \emph{set}, namely the set of prebraidings in the sense of \S\ref{sec:prebraidCat} on a functor between certain ordinary monoidal $1$-categories. The key fact we use is that while the $(\infty,2)$-`fiber'-functor $\Kbloc(\SBim) \to \stkBZ$ is \emph{not} faithful, the composite $\SBim \to \Kbloc(\SBim) \to \stkBZ$ \emph{is} faithful by \cref{lem:SBim-to-st-faithful}, i.e. induces fully faithful functors on all hom-categories. 
Furthermore, $\SBim$ is generated by the $(2,2)$-category $\BSbimp$  from \cref{def:BSBimp}, 
in the sense that the functor $\BSbimp \to \SBim$ is surjective on objects and that any $1$-morphism in $\SBim$ is a retract of a finite coproduct of grading shifts of $1$-morphisms in the image of $\BSbimp$ as proven in \cref{prop:sbimiscorrect}. 

For future applications, we abstract this situation as follows: 

\begin{theorem}\label{thm:main-theorem-last-sec}
Let $\cC \in \Alg_{\EE_1}(\Cat[\addkBZ])$, $\cD\in \Alg_{\EE_{\infty}}(\Cat[\stkBZ])$ and let $H  \colon  \cC \to \cD$ be a morphism in $\Alg_{\EE_1}(\Cat[\addkBZ])$ whose underlying $(\infty,2)$-functor is faithful, i.e. induces fully faithful functors on hom-categories. Consider the monoidal functor $\Kbloc(\cC) \to \cD$, induced by the adjunction~\eqref{eq:Kbloc-adjunction}, as an object of $\Alg_{\EE_1}\left(\Cat[\stkBZ]_{/\cD}\right)$. 
\begin{enumerate}
\item \label{item-num:main-theorem-last-sec-1}
Then, the map of spaces (constructed more formally in the proof below)
\begin{equation}
\label{eq:main-theorem}
\Braid_{\Cat[\stkBZ]_{/\cD}}(\Kbloc(\cC)) \to \PreBraid_{{\Catnk{1}{1}}_{/h_1\cD}}(h_1 \cC \to h_1\Kbloc(\cC)),
\end{equation}
which restricts a braiding on $\Kbloc(\cC)$ to a prebraiding on the subcategory inclusion $\cC \to \Kbloc(\cC)$ and then passes to the homotopy $1$-category $h_1$, is an equivalence. (Here, we leave the evident maps to $\cD$ and $h_1\cD$ implicit.)
\item \label{item-num:main-theorem-last-sec-2}
Further, assume there is a functor $\iota  \colon  \cB\to \cC$ in $\Alg_{\EE_1}(\Cat_{(\infty,2)})$ which is surjective on objects and such that for every two objects $b,b' \in \cB$, any object in $\eHom_{\cC}(\iota b, \iota b') \in \addkBZ$ is a retract of a finite coproduct of $\mbbZ$-shifts of objects in the image of $\eHom_{\cB}(b, b') \in \Cat_{(\infty,1)}$. Then, the pre-composition map
\[ \PreBraid_{{\Catnk{1}{1}}_{/h_1\cD}}(h_1 \cC \to h_1\Kbloc(\cC)) \to  \PreBraid_{{\Catnk{1}{1}}_{/h_1\cD}}(h_1 \cB \to h_1\Kbloc(\cC)) 
\]
is an equivalence of spaces.
\end{enumerate}
\end{theorem}

In other words, \cref{thm:main-theorem-last-sec} asserts that the space of pairs of an $\EE_2$-structure on $\Kbloc(\cC)$ and an $\EE_2$-structure on the functor $\Kbloc(\cC) \to \cD$, compatible with their given $\EE_1$-structures, is equivalent to the set of prebraidings on $h_1\cB \to h_1 \Kbloc(\cC)$ over $h_1\cD$. 

In the remainder of \cref{sec:main-theorem}, we prove \cref{thm:main-theorem-last-sec} by factoring~\eqref{eq:main-theorem} through various other spaces of prebraiding. 
Before discussing the proof, we immediately record how \cref{thm:main-theorem-last-sec} implies \cref{intro.maintheorem} from the introduction:

\begin{corollary}\label{cor:main-corollary}The \emph{space} of braidings
\begin{equation}
\label{eq:main-corollary-space-of-braidings} 
\Braid_{\Cat[\stkBZ]_{/\stkBZ}}\left(\Kbloc(\SBim)\right)
\end{equation} is equivalent to the
\emph{set} of prebraidings 
\begin{equation}\label{eq:main-corollary-set-of-prebraidings}
\PreBraid_{/h_1\DMorPoly}(h_1\BSBim \to h_1 \oldKbloc{\SBim})
\end{equation}
over $h_1\Hloc \colon h_1 \oldKbloc{\SBim} \to h_1 \DMorPoly$ as defined in \cref{def:notationprebraiding}.

In particular, the  space of pairs of an $\EE_2$-algebra structure on $\Kbloc(\SBim) \in \Cat[\stkBZ]$ together with an $\EE_2$-algebra structure on the functor $\Hloc \colon\Kbloc(\SBim) \to \stkBZ$, which enhance their monoidal structures, and satisfy the condition that the positive braiding 
\[\sigma_{1,1} \colon 1 \otimes 1\to 1 \otimes 1 \in \eHom_{\Kbloc(\SBim)}(1 \otimes 1,1\otimes 1) = \Kb(\SBim_2)\] agrees \emph{up to chain homotopy} with the shifted Rouquier complex $X_{1,1}= F(\sigma_{1,1})\langle -1\rangle$ from \cref{def:cabledcross} and \eqref{eqn:Rouq-alg}, is contractible.

\end{corollary}
\begin{proof}
Recall that the functor $\Kbloc(\SBim) \to \stkBZ$ factors by definition through the small full subcategory $\DMorPoly \hookrightarrow \DMoritaS \hookrightarrow \stkBZ$. Hence, the space~\eqref{eq:main-corollary-space-of-braidings} is equivalent to the space
\begin{equation}\label{eq:main-cor-step0}
\Braid_{\Cat[\stkBZ]_{/\DMorPoly}}\left(\Kbloc(\SBim)\right).
\end{equation}
We now invoke \cref{thm:main-theorem-last-sec} for $\cC = \SBim$, $\cD = \DMorPoly$, and $\cB =\BSBim$, the functor $\Hloc\colon \SBim \to \DMorPoly$ from \cref{prop:Hloc} which is faitfhul by \cref{lem:SBim-to-st-faithful}, and the functor $\iota\colon \BSBim \to \SBim$ from~\eqref{eq:fromBSBimtoSBim} which satisfies the relevant conditions of  \cref{thm:main-theorem-last-sec} by \cref{prop:sbimiscorrect}.
Thus, the space of braidings~\eqref{eq:main-cor-step0} is equivalent to the space of prebraidings \begin{equation}\label{eq:main-cor-step1}\PreBraid_{(\Cat_{(1,1)})_{/h_1\DMorPoly}}(h_1 \BSBim \to h_1 \Kbloc(\SBim)).\end{equation} 
Both, $h_1 \BSBim \to h_1 \Kbloc(\SBim)$ and $h_1\Kbloc(\SBim) \to h_1\DMorPoly$, agree with the respective functors from \cref{sec:2}, namely with~\eqref{eq:defh1K} by \cref{cor:h1KSBim} and with ~\eqref{eq:h1Hloc} by ~\cref{handover2} respectively. Hence,  it follows from \cref{exm:relative-prebraid-2} that this space~\eqref{eq:main-cor-step1} is equivalent to the set $\PreBraid_{/h_1\DMorPoly}\left(h_1 \BSBim \to h_1 \oldKbloc{\SBim}\right)$ from \cref{thm:classbraidings}.

The second half of \cref{cor:main-corollary} follows directly from the first: By~\cref{cor:Rouquier-correct-prebraiding}, the condition on the positive braiding $\sigma_{1,1}$ fixes an element of the set~\eqref{eq:main-corollary-set-of-prebraidings} and hence a point in the space \eqref{eq:main-corollary-space-of-braidings}. Thus, there is a contractible space of braidings on $\Kbloc(\SBim)$ over $\stkBZ$ compatible with the given prebraiding on $h_1\BSbim \to h_1\DMorPoly$. \end{proof}

\begin{remark}\label{rmk:finally-cH}
Consider the functor
\[
h_2 \colon \Alg_{\EE_2}(\Cat[\stkBZ]) \to \Alg_{\EE_2}(\Catnk{2}{2}) 
\]
induced by the lax symmetric monoidal composite 
\[
\Cat[\stkBZ] \xrightarrow{\Cat[\mathrm{forget}]} \Cat[\cat] = \CatInfty{2} \xrightarrow{h_2} \Catnk{2}{2}
\]
which first forgets along $\stkBZ \to \cat$\footnote{The forgetful functor $\stkBZ \to \cat$ is right adjoint to the composite $\cat \xrightarrow{\Lin_k(- \times \mathbb{Z})} \addkBZ \xrightarrow{\Kb}\stkBZ$ 
of the symmetric monoidal functors from \eqref{eqn:linkz} and \cref{prop:presentablestKI}\eqref{item-prop:presentablestKI-1} and is hence laxly symmetric monoidal.}
the homwise structure and then takes the homotopy $2$-category (\cref{def:homotopy-cat-definition}). Applying this to $\Kbloc(\SBim) \in \Alg_{\EE_2}(\Cat[\stkBZ])$ results in the braided monoidal $(2,2)$-category $\cH \coloneqq h_2\Kbloc(\SBim) \in \Alg_{\EE_2}(\Catnk{2}{2})$ described in \cref{subsection.intro.overview}. 
\end{remark}

\subsection{From prebraidings to braidings}
\label{subsec:from-prebraiding-to-braiding}
Our proof will proceed by successively simplifying the space of prebraidings and braidings on $\Kbloc(\cC)$. This  subsection contains the operadic heart of our proof, captured by four corollaries of results in \cref{sec:prebraidings-to-E2-str}.

\begin{definition}
Given a map of spaces $f \colon A \to B$, we let $\Im(f) \subseteq B$ denote the  \emph{full image} of $f$, i.e. the subspace of $B$ given by the union of those connected components of $B$ in the image of $\pi_0 f$. \end{definition}
In other words, $A \to \Im(f) \hookrightarrow B$ is the factorization of $f$ with respect to the ($(-1)$-connected, $(-1)$-truncated)-factorization system on spaces.

\begin{observation}
\label{obs:restricted} 
 Recall the ($0$-surjective, $0$-faithful) factorization system on the $\infty$-category $\Op$ from  \cref{defn:n-surj-and-n-trunc-for-infty-opds} and \cref{prop:fs-for-infty-opds}.
\begin{enumerate}
\item 
Given a map of operads $\EE_1 \to \cO$, corepresenting an $\EE_1$-algebra $A$ in $\cO$, we write $\cO|_{\EE_1}$ for the factorization $ \EE_1 \to \cO|_{ \EE_1} \to \cO$ into a $0$-surjective followed by a $0$-faithful map of operads. 

Explicitly, $\cO|_{\EE_1}$ has one color $A$ and the only non-empty multi-hom spaces are given by the full images 
\[ \Mul_{\cO|_{\EE_1}}(A, \ldots, A; A) = \Im\left( S_n \to \Mul_{\cO}(A, \ldots, A; A) \right)
\]
of the map $\EE_1(n) = S_n \to  \Mul_{\cO}(A, \ldots, A ; A) $ induced by the $\EE_1$-structure on $A$.

(In other words, $\Im\left( S_n \to \Mul_{\cO}(A, \ldots, A; A) \right)$ is precisely the union of those components of $ \Mul_{\cO}(A, \ldots, A ; A) $ which contain the orbit of the $n$-ary multiplication of $A$ under the $S_n$-action permuting its inputs.)

\item 
Given a map of operads $[1] \otimes \EE_1 \to \cO$, corepresenting a morphism $f \colon A \to B$ of $\EE_1$-algebras in $\cO$, we write $\cO|_{[1]\otimes \EE_1}$ for the factorization $[1] \otimes \EE_1 \to \cO|_{[1] \otimes \EE_1} \to \cO$ into a $0$-surjective followed by a $0$-faithful map of operads. 

Explicitly, $\cO|_{[1]\otimes \EE_1}$ has (at most) two colors $A, B$ and multi-hom spaces connecting them, one of them being 
\[\Mul_{\cO|_{[1] \otimes \EE_1}}(A, \ldots, A; B) = \Im \left(S_n \to \Mul_{\cO}(A, \ldots, A; B)\right),
\]where the map from $S_n = \EE_1(n) $ is induced by the $\EE_1$-structures on  $f$. (In other words, $\Im\left( S_n \to \Mul_{\cO}(A, \ldots, A; B) \right)$ is the union of those components of $ \Mul_{\cO}(A, \ldots, A ; B) $ which contain the orbit of $f \circ \mu_A \simeq \mu_B \circ f$ under the $S_n$-action permuting its inputs.)
\item Given a map of operads $[2] \otimes \EE_1 \to \cO$, corepresenting a composable pair of $\EE_1$-algebra morphisms $A \to B \to C$, we can similarly consider $\cO|_{[2] \otimes \EE_1}$, which has (at most) three objects $A, B, C$, and multi-hom spaces connecting them, such as 
\begin{align*}\Mul_{\cO|_{[2] \otimes \EE_1}}(A, \ldots, A; B)& = \Im \left(S_n \to \Mul_{\cO}(A, \ldots, A; B)\right)\\
\Mul_{\cO|_{[2] \otimes \EE_1}}(B, \ldots, B; C) &= \Im \left(S_n \to \Mul_{\cO}(B, \ldots, B; C)\right),
\end{align*}
where the maps from $S_n = \EE_1(n)$ are induced by the $\EE_1$-structure on $f$ and $g$, respectively.
\end{enumerate}
\end{observation}

Throughout we will repeatedly use the following simple observation, applied to the $\infty$-category $\cV= \Op$ with its ($0$-surjective, $0$-faithful) factorization system. 

\begin{lemma} 
\label{lem:fact-lift}
In an $\infty$-category $\cV$ with a factorization system $(\cL, \cR)$, consider a commuting square
\[ 
\begin{tikzcd}
A\ar[r] \ar[d] & \ar[d]B \\
C\ar[r] & D
\end{tikzcd}
\]
and a further morphism $Q\to A$ in $\cL$ so that also the composite $Q\to C$ is in $\cL$.  Let $Q\to B|_{Q} \to B$ and $Q \to D|_Q \to D$ denote the factorizations of the induced morphisms from $Q$. 
Then, the map between spaces of (dashed) lifts 
\[\left\{
\begin{tikzcd}
A\ar[r] \ar[d] & \ar[d]B|_Q \\
C\ar[r] \ar[ur, dashed]& D|_Q
\end{tikzcd}
\right\}
\longrightarrow
\left\{
\begin{tikzcd}
A\ar[r] \ar[d] & \ar[d]B \\
C\ar[r] \ar[ur, dashed] & D
\end{tikzcd}
\right\}.
\]
is an equivalence. \footnote{The left diagram exists since $Q \to A$ and $Q \to C$ are in $\cL$, while $B|_{Q} \to B$ and $D|_{Q} \to D$ are in $\cR$.}
\end{lemma}
\begin{proof}
Let $\cV_{Q/^{\cL}}$ denote the full subcategory of $\cV_{Q/}$ on the morphisms $Q\to X$ which are in $\cL$. The factorization system induces a right adjoint of the inclusion $\cV_{Q/^{\cL}} \hookrightarrow \cV_{Q/}$ which sends  $Q\to X$ to its factorization $Q\to X|_Q$. The statement then follows immediately from adjunction. 
\end{proof}

The fact that braidings and prebraidings agree on ordinary $1$-categories, see \cref{ex:prebraidon1cat}, generalizes to the following observation:

\begin{corollary}
\label{cor:step1}
Let $\cO$ be an $\infty$-operad, $A$ an $\EE_1$-algebra in $\cO$, and assume that the spaces 
\[
\Im\left( S_n \to \Mul_{\cO}(\underbrace{A, \ldots, A}_{n}; A) \right)
\]
 are $1$-truncated (i.e. $1$-groupoids) for all $n\geq 0$, where the map from $S_n = \EE_1(n)$ is induced by the $\EE_1$-structure on $A$. 
  Then, the map of spaces
\[ \Braid_{\cO}(A) \to \PreBraid_{\cO}(A) 
\]
is an equivalence. 
\end{corollary}
\begin{proof}
Fix a  prebraiding on $A$, represented by a lift of the map of $\infty$-operads $\EE_1 \to \cO$ representing $A$, to a map of $\infty$-operads $\AA_2 \otimes \EE_1 \to \cO$. The fiber of $\Braid_{\cO}(A) \to \PreBraid_{\cO}(A)$ at this prebraiding is precisely the space of further lifts 
\[ \begin{tikzcd}
\AA_2 \otimes \EE_1 \arrow[r] \arrow[d] & \cO\\
  \EE_2 \arrow[ur, dashed] & 
\end{tikzcd}.
\]
The operad map  $\EE_1 \to \AA_2 \otimes \EE_1$ is $0$-surjective by \cref{prop:operadsurjective}, and the composite $\EE_1 \to \AA_2 \otimes \EE_1 \to \EE_2$ is $0$-surjective since all mapping spaces of $\EE_2$ are connected and all mapping spaces of $\EE_1$ are non-empty. Hence, it follows from  \cref{lem:fact-lift} applied to the $\infty$-category $\Op$ (with its ($0$-surjective, $0$-faithful) factorization system) that this space of lifts is equivalent to the space of lifts
\[
 \begin{tikzcd}
\AA_2 \otimes \EE_1 \arrow[r] \arrow[d] & \cO|_{\EE_1}\\
  \EE_2 \arrow[ur, dashed] & 
\end{tikzcd}.
\]
Since all multi-hom spaces of  $\cO|_{\EE_1}$ are by assumption $1$-truncated, and hence $\cO|_{\EE_1}$ is a $2$-operad (see \cref{def:2operad}), it follows from \cref{cor:2operadA2} that this space of lifts is contractible.
\end{proof}

\begin{corollary}
\label{cor:step2}
Let $F \colon  \cO \to \cP$ be a map of $\infty$-operads and let  $g \colon b \to c$ be a morphism of $\EE_1$-algebras in $\cO$. Assume that for all $n \geq 0$  the map of spaces
\[\Im\left(S_n \to \Mul_{\cO}(\underbrace{b, \ldots, b}_{n};  c) \right) \xrightarrow{F(-)} \Im\left(S_n \to \Mul_{\cP}(\underbrace{F(b), \ldots, F(b)}_{n} ; F(c)) \right)
\]
is an equivalence, where the maps from $S_n$ are induced by the $\EE_1$-structure on $g$ and $F(g)$. 
Then, the map of spaces
\[\PreBraid_{\cO}(g) \to \PreBraid_{\cP}(F(g))
\]
is an equivalence.
\end{corollary}
\begin{proof}
Consider the map of operads $[1] \otimes \EE_1 \to \cO$ representing the $\EE_1$-algebra map $f \colon A \to B$. The fiber of $\PreBraid_{\cO}(f) \to \PreBraid_{\cP}(F(f))$ at a prebraiding represented by an operad map $\TT_2 \otimes \EE_1 \to \cP$ is precisely the space of (dashed) lifts of the following commuting square of operads:
\[
\begin{tikzcd}
{[1]} \otimes \EE_1 \arrow[r, "\{f:a \to b\}"] \arrow[d] & \cO \arrow[d, "F"]\\
\TT_2 \otimes \EE_1 \arrow[r] \arrow[ur, dashed] & \cP.
\end{tikzcd}
\]
Since $[1] \otimes \EE_1 \to \TT_2 \otimes \EE_1$ is $0$-surjective, it follows from \cref{lem:fact-lift} that this space of lifts is equivalent to the space of lifts 
\[
\begin{tikzcd}
{[1]} \otimes \EE_1 \arrow[r, "\{f:a \to b\}"] \arrow[d] & \cO|_{[1] \otimes \EE_1} \arrow[d]\\
\TT_2 \otimes \EE_1 \arrow[r] \arrow[ur, dashed] & \cP|_{[1] \otimes \EE_1} .
\end{tikzcd}.
\]
Hence, replacing $\cO$ by $\cO|_{[1] \otimes \EE_1}$ and $\cP$ by $\cP|_{[1] \otimes \EE_1}$ with multimapping spaces as in \cref{obs:restricted}, we may without loss of generality assume that the maps
\[\Mul_{\cO}(a, \ldots, a; b) \to \Mul_{\cP}(Fa, \ldots, Fa; Fb)
\]
are equivalences. It then follows from \cref{prop:magic-prop}.\eqref{enum-prop:magicprop-1} that the maps
\[\Mul_{\Alg_{\EE_1}(\cO)}(a, \ldots, a; b) \to \Mul_{\Alg_{\EE_1}(\cP)}(Fa, \ldots, Fa; Fb)
\]
are equivalences. 
Hence, 
\begin{align*} \PreBraid_{\cO}(g) = &\Mul_{\Alg_{\EE_1}(\cO)}(a,a; b) \times_{\Mul_{\Alg_{\EE_1}(\cO)}(a, b)^{2}} \{g\} \\
&\longrightarrow \Mul_{\Alg_{\EE_1}(\cP)}(Fa,Fa; Fb) \times_{\Mul_{\Alg_{\EE_1}(\cP)}(Fa, Fb)^{2}} \{Fg\} = \PreBraid_{\cO}(Fg)
\end{align*}
is an equivalence.
\end{proof}

\begin{corollary}
\label{cor:step3}
Let $\cO$ be an $\infty$-operad and let 
\[a  \xrightarrow{f} b  \xrightarrow{g} c
\]
be morphisms of $\EE_1$-algebras in $\cO$. Assume that for all $n\geq 0$, the map of spaces 
\[
\Im\left( S_n \to \Mul_{\cO}(\underbrace{b, \ldots, b}_{n} ; c) \right)  \xrightarrow{-\circ (f, \ldots, f)} \Im\left( S_n \to \Mul_{\cO}(\underbrace{a, \ldots, a}_{n}; c) \right)
\]
are equivalences, where the maps from $S_n$ are induced by the $\EE_1$-structures on $g$ and on $g \circ f$. 
Then, the map of spaces
\[\PreBraid_{\cO}(g) \to \PreBraid_{\cO}(g \circ f)
\]
is an equivalence.
\end{corollary}
\begin{proof}
The pair of composable morphisms of $\EE_1$-algebras $\{f,g\}$ may be corepresented by an operad map $[2] \otimes \EE_1 \to \cO$. 
Recall the operad maps $[2] \otimes \EE_0 \to \TT_2 \sqcup_{\{0<2\}} [2] \to \TT_2 \sqcup_{\{1<2\}}$ from \cref{obs:functorialityT2} corepresenting a pair of composable $\EE_0$-morphisms with a $\TT_2$-structure on $g$ and on $g\circ f$, respectively, and the construction of a $\TT_2$-structure on $g$ from a $\TT_2$-structure on $g\circ f$.

Fix a prebraiding on $g\circ f$, corepresented by a lift of the operad map $[2] \otimes \EE_1 \to \cO$ to an operad map $\TT_2 \sqcup_{\{0<2\}}[2] \otimes \EE_1 \to \cO$. The fiber of $\PreBraid_{\cO}(g) \to \PreBraid_{\cO}(g\circ f)$ is given by the space of lifts 
\[ \begin{tikzcd}
\left(\TT_2 \sqcup_{\{0<2\}}{[2]} \right) \otimes \EE_1  \arrow[r] \arrow[d] & \cO\\
 \left( \TT_2 \sqcup_{\{1<2\}}{[2]} \right) \otimes \EE_1 \arrow[ur, dashed] & 
\end{tikzcd}.
\]
We now claim that both operad maps \begin{equation}\label{eq:two-operad-maps}
[2] \otimes \EE_1 \to \left(\TT_2 \sqcup_{\{0<2\}}{[2]} \right) \otimes \EE_1 \qquad [2] \otimes \EE_1 \to \left( \TT_2 \sqcup_{\{1<2\}}{[2]} \right) \otimes \EE_1\end{equation}
are $0$-surjective. Indeed, this follows since $[1] \otimes \EE_1 \to \TT_2 \otimes \EE_1$ is $0$-surjective by \cref{prop:operadsurjective} and since both operad maps~\eqref{eq:two-operad-maps} are by definition given by a pushout of this $0$-surjective operad map  against the operad maps $[1] \otimes \EE_1 \to [2] \otimes \EE_1$ induced by the inclusions $\{0<2\} \to [2]$ and $\{1<2\} \to [2]$, respectively (see \cref{obs:functorialityT2}). 

Hence, it follows from  \cref{lem:fact-lift} applied to the $\infty$-category $\Op$ that our space of lifts is equivalent to the space of lifts
\[ \begin{tikzcd}
\left(\TT_2 \sqcup_{\{0<2\}}{[2]} \right) \otimes \EE_1  \arrow[r] \arrow[d] & \cO|_{{[2]} \otimes \EE_1}\\
 \left( \TT_2 \sqcup_{\{1<2\}}{[2]} \right) \otimes \EE_1 \arrow[ur, dashed] & 
\end{tikzcd}.
\]
Therefore, without loss of generality, we may replace $\cO$ by $\cO|_{[2] \otimes \EE_1}$ and hence with \cref{obs:restricted}, we may assume that 
\[\Mul_{\cO}(b, \ldots, b; c) \to \Mul_{\cO}(a, \ldots, a; c)
\]
are equivalences. It therefore follows from \cref{prop:magic-prop}.\eqref{enum-prop:magicprop-2} that 
\[\Mul_{\Alg_{\EE_1}(\cO)}(b, \ldots, b; c) \to \Mul_{\Alg_{\EE_1}{\cO}}(a, \ldots, a; c)
\]
are equivalences, and hence as in the proof of \cref{cor:step2} that 
\[\PreBraid_{\cO}(g) \to \PreBraid_{\cO}(g\circ f)
\]
are equivalences.
\end{proof}

Lastly, we record the following special case of \cref{prop:T2adjunction} that prebraidings transport along adjunctions:

\begin{corollary}
\label{cor:step0}
Consider an adjunction between symmetric monoidal $\infty$-categories with  (strongly) symmetric monoidal left adjoint $L$ 
\[
\begin{tikzcd}[column sep=1.5cm]
\cV
\arrow[yshift=0.9ex]{r}{L}
\arrow[leftarrow, yshift=-0.9ex]{r}[yshift=-0.2ex]{\bot}[swap]{R}
&
\cW
\end{tikzcd}
\]
and denote the induced adjunction between $\infty$-categories of $\EE_1$-algebras by 
\begin{equation}\label{eq:E1adjunction}
\begin{tikzcd}[column sep=1.5cm]
\Alg_{\EE_1}(\cV)
\arrow[yshift=0.9ex]{r}{L_{\EE_1}}
\arrow[leftarrow,yshift=-0.9ex]{r}[yshift=-0.2ex]{\bot}[swap]{R_{\EE_1}}
&
\Alg_{\EE_1}(\cW).
\end{tikzcd}
\end{equation}
Then, for any morphism of $\EE_1$-algebras $f \colon L_{\EE_1} a \to b$ in $\cW$, the induced map of spaces 
\[\PreBraid_{\cW}(f) \to \PreBraid_{\cV}(R_{\EE_1}f) \to \PreBraid_{\cV}( R_{\EE_1}f \circ \eta_a) \]
(constructed as in \cref{prop:T2adjunction}) is an equivalence, where $\eta$ denotes the unit of the adjunction. 

\end{corollary}
\begin{proof}
This is an immediate corollary of \cref{prop:T2adjunction} applied to the adjunction~\eqref{eq:E1adjunction}.\end{proof}

\subsection{Proof of the main theorem}
\label{subsec:proof-of-main-thm}
We now prove the various truncatedness conditions appearing in Corollaries~\ref{cor:step1} -- \ref{cor:step2} in the setting of \cref{thm:main-theorem-last-sec}, and assemble the obtained equivalences between spaces of (pre-)braidings into a proof of our main theorem.

We first recall from \cref{sec:inf-n-cats} that faithful $(\infty,2)$-functors are completely determined by their induced ordinary functors between homotopy $1$-categories, with the following straight-forward corollary:
\begin{corollary}\label{prop:faithful-obstruction2}
Let $n \geq 0$ and consider categories and functors in $\Cat_{(\infty,2)}$
\[ \begin{tikzcd}
\cY^{\times n} \arrow[r]& \cW&&&  \arrow[lll, "\mathrm{faithful}"']\cZ
 \end{tikzcd},
\]
with faithfulness properties as indicated.
Then, the map induced by applying $h_1$  \[\Hom_{{\CatInfty{2}}_{/\cW}}\left(\cY^{\times n}, \cZ\right) \to \Hom_{{\Catnk{1}{1}}_{/h_1\cW}}\left(h_1 \cY ^{\times n} , h_1 \cZ\right)
\]
is an equivalence of spaces.
\end{corollary}
\begin{proof}
Since $h_1 \colon  \CatInfty{2}\to \Catnk{1}{1}$ is (strongly) symmetric monoidal, this follows directly from \cref{cor:inf-n-pullbackfullyfaithful}.
\end{proof}

Moving to locally additive $(\infty,2)$-categories, recall from \cref{nota:surjective-and-dominant} that we call a morphism $F \colon  \cX \to\cY$ in $\Cat[\addkBZ]$  \emph{surjective-on-objects-and-dominant-on-1-morphisms} if its underlying $(\infty,2)$-functor is surjective on objects and if for each pair of objects $x,x' \in \cX$, every object of $\eHom_{\cY}(Fx,Fx')$ is a retract of an object in the image of $\eHom_{\cX}(x,x') \to \eHom_{\cY}(Fx,Fx')$.
The key technical statement of this section is the following straight-forward application of the (surjective-on-objects-and-dominant-on-1-morphisms, faithful)-factorization system on $\Cat[\addkBZ]$ constructed in \cref{cor:surj-dominant-faithful-fs-on-Cat-add}.
\begin{prop}
\label{prop:faithful-obstruction1}
Let $n \geq 0$ and consider categories and functors in $\Cat[\addkBZ]$
\[ \begin{tikzcd}
\cX \arrow[rrrrr, "\substack{\mathrm{surjective-on-objects}\\\mathrm{-and-dominant-on-1-morphisms}}"] &&&&&\cY,  &\cY^{\otimes n} \arrow[r]& \cW&&&  \arrow[lll, "\mathrm{faithful}"']\cZ
 \end{tikzcd},
\]
with surjectivity and faithfulness properties as indicated (and where $\otimes$ denotes the tensor product in $\Cat[\addkBZ]$).
Then, the map induced by precomposition with the tensor power $\cX^{\otimes n} \to\cY^{\otimes n}$ 
\[\Hom_{\Cat[\addkBZ]_{/\cW}}\left(\cY^{\otimes n}, \cZ\right) \to \Hom_{\Cat[\addkBZ]_{/\cW}}\left( \cX^{\otimes n} , \cZ\right)
\]
is an equivalence of spaces.
\end{prop}
\begin{proof}
 Since the (surjective on objects and dominant on 1-morphisms, faithful) factorization system is compatible with the monoidal structure on $\Cat[\addkBZ]$, as an $n$-fold tensor power of a functor in the left class, the functor $\cX^{\otimes n} \to \cY^{\otimes n}$ remains surjective on objects and dominant on $1$-morphisms, and hence the first map is an equivalence as a direct consequence of the orthogonality of surjective-on-objects-and-dominant-on-1-morphisms and faithful functors. 
\end{proof}

Below, we will repeatedly use the following simple lemma:
\begin{lemma}
\label{lem:fullimage}
Suppose we are given a commuting square of spaces
\begin{equation}
\label{eq:squareimage}
\begin{tikzcd}
A \arrow[r,"\simeq"] \arrow[d, "f"'] & B \arrow[d, "g"] \\
C  \arrow[r, "h"']& D
\end{tikzcd}
\end{equation}
where the top horizontal map is an equivalence and where for every point $c \in
C$, the induced map of fibers $\mathrm{fib}_c(f) \to \mathrm{fib}_{h(c)}(g)$ is
an equivalence. Then, the induced map 
\[\Im(f) \to \Im(g)
\] 
is an equivalence. 
\end{lemma}
\begin{proof}
The square~\eqref{eq:squareimage} factors as \[\begin{tikzcd}
A \arrow[r] \arrow[d, two heads] & B \arrow[d, two heads] \\
\Im(f) \arrow[r] \arrow[d, hook] & \Im(g) \arrow[d, hook]\\
C  \arrow[r]& D
\end{tikzcd}.
\]
Hence, we may without loss of generality assume that $\Im(f) = C$ and $\Im(g) =  D$, i.e. that $f$ and $g$ are surjective on $\pi_0$ and prove that in this case $h$ is an isomorphism.
Working fiberwise (and identifying $A$ with $B$ via the given equivalence), it suffices to consider the case $D=\pt$; in other words, given a $(-1)$-connected map $f\colon A\twoheadrightarrow C$ so that for all $c\in C$ the map  $\mathrm{fib}_c(f) \to A (=\mathrm{fib}_\pt(A\to \pt))$ is an isomorphism, we need to show that $C$ is contractible. This follows directly from the long exact sequence of homotopy groups associated to the fiber sequence. 
\end{proof}

Recall from~\eqref{eq:LinkZloc} the adjunction 
\begin{equation}
\label{eq:adjunction-fromCattoAdd}
\begin{tikzcd}[column sep=1.5cm]
  \CatInfty{2}
\arrow[yshift=0.9ex]{r}{\addkZloc{-}}
\arrow[leftarrow, yshift=-0.9ex]{r}[yshift=-0.2ex]{\bot}[swap]{\mathrm{forget}}
&
\addkBZ
\end{tikzcd}
\end{equation}
where the right adjoint forgets additivity, $k$-linearity and the $\mbbZ$-action and the (strongly) symmetric monoidal left adjoint $\addkZloc{-}$ sends an $(\infty,2)$-category $\cC$ to the locally additive $k$-linear $(\infty,2)$-category with local shifts with the same objects as $\cC$ and hom-categories given by the linearization $\Lin_k(\eHom_{\cC}(a,b) \times \mbbZ)$  with free $\mbbZ$-action.

\begin{corollary}
\label{cor:collection}
Given categories and functors as in the assumptions of \cref{thm:main-theorem-last-sec}.
Then, for each $n \geq 0$,  the following hold:
\begin{enumerate}
\item \label{itm:step1}
The space
\[ \Im \left( S_n \to \Hom_{\Cat[\stkBZ]_{/\cD}}\left( \Kbloc(\cC)^{\otimes n}, \Kbloc(\cC) \right) \right) 
\]
is $1$-truncated, i.e. a $1$-groupoid.
\item \label{itm:step2}
The map of spaces 
\begin{align*}
\Im \left(S_n \to \vphantom{\Hom_{\Cat[\addkBZ]_{/\cD}}}\right. & \left. \Hom_{\Cat[\addkBZ]_{/\cD}}  \left( \cC^{\otimes n}, \Kbloc(\cC)\right)\right) \\
&\longrightarrow \Im \left(S_n \to \Hom_{\Cat[\addkBZ]_{/\cD}}\left(\addkZloc{\cB}^{\otimes n}, \Kbloc(\cC) \right) \right)
\end{align*}
is an equivalence. 
\item \label{itm:step3} The map of spaces
\begin{align*}
\Im \left( S_n \to \vphantom{\Hom_{{\CatInfty{2}}_{/\cD}}}  \right.  & \left. \Hom_{{\CatInfty{2}}_{/\cD}} \left( \cB^{\times n}, \Kbloc(\cC)\right) \right) \\
&\longrightarrow \Im \left( S_n \to \Hom_{{\Catnk{1}{1}}_{/h_1\cD}}\left(h_1\cB^{\times n}, h_1 \Kbloc(\cC) \right) \right)
\end{align*}
is an equivalence  
\end{enumerate}
\end{corollary}
We warn the reader that the functor $\Kbloc(\cC) \to \cD$ is \emph{not} faithful, and hence Propositions~\ref{prop:faithful-obstruction2} and \ref{prop:faithful-obstruction1}  do not apply directly.

\begin{proof}
We first prove the part \eqref{itm:step2}. Since the monoidal structure on $\cC \to \Kbloc(\cC)$ arises from adjunction, the relevant maps from $S_n$ all factor as: 
\[ \begin{tikzcd}
         S_n  \ar[dr]\\[-2em]
             &[-2em]   \Hom_{\Cat[\add^{B \mbbZ}_k]_{/\cD}}(\cC^{\otimes n}, \cC) \ar[r] \ar[d]&  \Hom_{\Cat[\addkBZ]_{/\cD}}\left(\addkZloc{\cB}^{\otimes n}, \cC\right) \ar[d]\\
             &[-2em]  \Hom_{\Cat[\add^{B \mbbZ}_k]_{/\cD}}(\cC^{\otimes n}, \Kbloc(\cC)) \ar[r] &  \Hom_{\Cat[\addkBZ]_{/\cD}}\left(\addkZloc{\cB}^{\otimes n}, \Kbloc(\cC)\right)
            \end{tikzcd}
            \]
Since $\cC \to \cD$ is faithful, it follows from \cref{prop:faithful-obstruction1} that the top horizontal map is an equivalence. Moreover, for every $f\in  \Hom_{\Cat[\add^{B \mbbZ}_k]_{/\cD}}(\cC^{\otimes n}, \Kbloc(\cC))$, the induced map between the fibers of the vertical maps is
\[ \Hom_{\Cat[\add^{B \mbbZ}_k]_{/\Kbloc(\cC)}}(\cC^{\otimes n}, \cC) \to  \Hom_{\Cat[\addkBZ]_{/\Kbloc(\cC)}}\left(\addkZloc{\cB}^{\otimes n}, \cC\right).
\]
It follows from \cref{prop:presentablestKI} that $\cC \to \Kbloc(\cC)$ is faithful. Hence, this map between fibers is also an equivalence by \cref{prop:faithful-obstruction1}. Therefore, it follows from \cref{lem:fullimage} that the induced map between the full images of the vertical maps is an equivalence, and hence so is also the induced map between the full images of $S_n$. 

The proof of part \eqref{itm:step2} is entirely analogous: The relevant maps from $S_n$ all factor as 
\[ \begin{tikzcd}
         S_n  \ar[dr]\\[-2em]
          &[-2em]      \Hom_{{\CatInfty{2}}_{/\cD}}(\cB^{\otimes n}, \cC) \ar[r] \ar[d]&  \Hom_{{\Catnk{1}{1}}_{/h_1\cD}}\left(h_1 \cB^{\times n}, h_1\cC\right) \ar[d]\\
           &[-2em]    \Hom_{{\CatInfty{2}}_{/\cD}}(\cB^{\otimes n}, \Kbloc(\cC)) \ar[r] &  \Hom_{{\Catnk{1}{1}}_{/h_1\cD}}\left(h_1 \cB^{\times n}, h_1\Kbloc(\cC)\right)
            \end{tikzcd}.
            \]
            Since $\cC \to \cD$ is faithful by assumption, it follows from \cref{prop:faithful-obstruction2} that the top horizontal map is an equivalence. The induced map between the fibers of the vertical maps at an $f\in \Hom_{{\CatInfty{2}}_{/\cD}}(\cB^{\otimes n}, \Kbloc(\cC))$ is given by
           \[\Hom_{{\CatInfty{2}}_{/\Kbloc(\cC)}}(\cB^{\otimes n}, \cC) \to \Hom_{{\Catnk{1}{1}}_{/h_1\Kbloc(\cC)}}\left(h_1 \cB^{\times n}, h_1\cC\right)
           \]
           and hence is also an equivalence by \cref{prop:faithful-obstruction2} since also $\cC \to \Kbloc(\cC)$ is faithful. Therefore, it follows from \cref{lem:fullimage} that the induced map between the full images of the vertical maps is an equivalence, and hence so is also the induced map between the full images of $S_n$. 

           Part \eqref{itm:step1} follows from combining parts \eqref{itm:step2} and \eqref{itm:step3}:
           It follows from adjunction and monoidality of $\Kbloc(-)$ that  \[\Hom_{\Cat[\stkBZ]_{/\cD}}\left( \Kbloc(\cC)^{\otimes n}, \Kbloc(\cC) \right) \simeq \Hom_{\Cat[\addkBZ]_{/\cD}}\left( \cC^{\otimes n}, \Kbloc(\cC)\right).\] Similarly, it follows from adjunction and monoidality of $\addkZloc{-}$ that \[\Hom_{\Cat[\addkBZ]_{/\cD}}\left( \addkZloc{\cB}^{\otimes n}, \Kbloc(\cC) \right) \simeq \Hom_{{\CatInfty{2}}_{/\cD}} \left( \cB^{\times n}, \Kbloc(\cC)\right).\]
           Hence, combining the second and third statement, we find that 
           \[\Im \left( S_n \to \Hom_{\Cat[\stkBZ]_{/\cD}}\left( \Kbloc(\cC)^{\otimes n}, \Kbloc(\cC) \right) \right) \simeq \Im \left( S_n \to \Hom_{{\Catnk{1}{1}}_{/h_1\cD}}\left(h_1\cB^{\times n}, h_1 \Kbloc(\cC) \right) \right)
           \]
           which is --- as a mapping space of $\Catnk{1}{1}$ ---  a $1$-groupoid. 
\end{proof}

Now we prove the main theorem:
\begin{proof}[Proof of \cref{thm:main-theorem-last-sec}]
Given categories and functors as in part \eqref{item-num:main-theorem-last-sec-2} of \cref{thm:main-theorem-last-sec}, we will prove that the composite
\begin{align}
\label{eq:mainproof}
\Braid_{\Cat[\stkBZ]_{/\cD}}(\Kbloc(\cC)) 
&\to \PreBraid_{{\Catnk{1}{1}}_{/h_1\cD}}(h_1 \cC \to h_1\Kbloc(\cC)) 
\\
\nonumber
&\to \PreBraid_{{\Catnk{1}{1}}_{/h_1\cD}}(h_1 \cB \to h_1\Kbloc(\cC))
\end{align}
is an equivalence. Note that part \eqref{item-num:main-theorem-last-sec-1} of \cref{thm:main-theorem-last-sec} then follows by taking $\cB \to \cC$ to be the identity $\cC \to \cC$ (which clearly satisfies the required conditions). Then, the second statement follows since the first map and the composite in~\eqref{eq:mainproof} are equivalences, and hence so is the second map. 

To prove that~\eqref{eq:mainproof} is an equivalence,  note that it follows from \cref{lem:concrete-dominance} that the condition on $\cB \to \cC$ in the statement of \cref{thm:main-theorem-last-sec}.\eqref{item-num:main-theorem-last-sec-2} equivalently asserts that $\addkZloc{\cB} \to \cC$ is surjective on objects and dominant on $1$-morphisms. 

We now unpack~\eqref{eq:mainproof} as a sequence of equivalences of spaces of (pre-)braidings:

\begin{align*}
  \Braid&_{\Cat[\stkBZ]_{/\cD}} (\Kbloc(\cC)) \\
   &\simeq \PreBraid_{\Cat[\stkBZ]_{/\cD}} (\Kbloc(\cC))  \\  
  & \hspace*{1cm} \text{Cor.~\ref{cor:step1} applied via Cor.~\ref{cor:collection}.\eqref{itm:step1} to the $\EE_1$-algebra $\Kbloc(\cC)$ in $\Cat[\stkBZ]$} \\
    & \simeq \PreBraid_{\Cat[\addkBZ]_{/\cD}}(\cC \to \Kbloc(\cC))\\
    & \hspace*{1cm} \text{Cor.~\ref{cor:step0} applied to the adjunction}
     \begin{tikzcd}[column sep=1.5cm, ampersand replacement=\&]
\Cat[\addkBZ]
\arrow[yshift=0.9ex]{r}{ \Kbloc}
\arrow[leftarrow, yshift=-0.9ex]{r}[yshift=-0.2ex]{\bot}[swap]{\mathrm{forget}}
\&
\Cat[\stkBZ]
\end{tikzcd}\\
     & \simeq \PreBraid_{\Cat[\addkBZ]_{/\cD}}(\addkZloc{\cB} \to \Kbloc(\cC))\\
    &  \hspace*{1cm} \text{Cor.~\ref{cor:step3} applied via Cor.~\ref{cor:collection}.\eqref{itm:step2} to the $\EE_1$-morphism $\addkZloc{\cB} \to \cC$ in $\Cat[\addkBZ]$} \\
    & \simeq \PreBraid_{{\CatInfty{2}}_{/\cD}}(\cB \to \Kbloc(\cC))\\
    &  \hspace*{1cm} \text{Cor.~\ref{cor:step0} applied to the adjunction}
   \begin{tikzcd}[column sep=1.5cm, ampersand replacement=\&]
\CatInfty{2}
\arrow[yshift=0.9ex]{r}{ \addkZloc{-}}
\arrow[leftarrow, yshift=-0.9ex]{r}[yshift=-0.2ex]{\bot}[swap]{\mathrm{forget}}
\&
\Cat[\addkBZ]
\end{tikzcd}\\
    & \simeq \PreBraid_{{\Catnk{1}{1}}_{/h_1\cD}}\left(h_1 \cB \to h_1 \Kbloc(\cC)\right)\\
      & \hspace*{1cm} \text{Cor.~\ref{cor:step2} applied via Cor.~\ref{cor:collection}.\eqref{itm:step3} to the functor $h_1 \colon  \Cat[\addkBZ]_{/\cD} \to {\Catnk{1}{1}}_{/h_1\cD}$} \\
    \end{align*}
    This completes the proof of \cref{thm:main-theorem-last-sec}.
\end{proof}

\appendix

\section{Recollections, notation, and conventions regarding higher category theory}
\label{sec:appendix-recollections}
In this paper, we make essential use of higher category theory and higher algebra. Here, we give a rapid overview of the notions that are most important to this paper (with further references for the interested reader). 

\subsection{From \texorpdfstring{$n$}{n}-categories to \texorpdfstring{$(\infty,n)$}{(infinity,n)}-categories}
\label{subsec:appendix-infty-n-cats}

As a starting point, let us recall the distinction between \textit{strict} and \textit{weak} 2-categories \cite{benabou1967}. In a strict 2-category, one requires associativity and unitality of composition to hold up to \textit{equality}. By contrast, in a weak 2-category, one only requires these to hold up to \textit{natural isomorphism}. Moreover, these natural isomorphisms are then required to satisfy further coherence conditions.\footnote{For instance, the associativity isomorphisms must satisfy the \textit{pentagon axiom}, which articulates a coherence condition among the five possible ways of composing a sequence of four composable morphisms.} We can summarize the situation with the slogan that a weak 2-category is ``a category that is enriched in 1-categories up to coherent natural isomorphism''. More generally, one would like to define a weak $n$-category as ``a category that is enriched in weak $(n-1)$-categories up to coherent natural isomorphism''. However, making this notion rigorous for higher values of $n$ -- with all of the desired coherence conditions -- becomes increasingly infeasible as $n$ grows \cite{gordon1995}.

Homotopy theory provides a remarkable alternative perspective on this problem, which leads to a uniform and robust solution. To explain it, let us recall Grothendieck's \textit{homotopy hypothesis}: any appropriate definition of ``weak $n$-category'' should have that its weak $n$-groupoids are equivalent (in a suitably homotopical sense) to homotopy $n$-types (i.e.\! topological spaces with homotopy groups above dimension $n$ all vanishing, taken up to weak homotopy equivalence).\footnote{Indeed, the homotopy hypothesis can be taken as one motivation (of many) for a theory of weak $n$-categories in the first place: it can be seen directly that strict 3-groupoids do not model all homotopy 3-types \cite{simpson1998homotopy}.} Note that under the homotopy hypothesis, natural isomorphisms on the categorical side correspond to homotopies on the topological side. Hence, we arrive at an alternative proposed definition for ``weak $(n,1)$-categories'',\footnote{Recall that for $0 \leq k \leq n$, the term ``$(n,k)$-category'' refers to an $n$-category whose $i$-morphisms are all invertible for all $i > k$. So an $(n,n)$-category is just an $n$-category (without further conditions), while an $(n,0)$-category is an $n$-groupoid.} namely as categories that are enriched either in weak $(n-1)$-groupoids up to coherent natural isomorphism or in homotopy $n$-types up to coherent homotopy. In the limit, we find that ``weak $(\infty,1)$-categories'' should be categories that are enriched in (arbitrary) homotopy types up to coherent homotopy.

Before continuing our discussion, we pause to note a few conventions. First of all, just as we may refer to 1-categories simply as ``categories'', we will also refer to $(\infty,1)$-categories simply as ``$\infty$-categories''. Moreover, given that we will \textit{only} be interested in ``weak'' notions, we usually leave this term implicit henceforth.

Now, there exist a number of robust models for $\infty$-categories (i.e.\! categories enriched in homotopy types up to coherent homotopy), although all are known to be equivalent (in a suitably homotopical sense)
\cite{toen2005}. The most developed is that of \textit{quasicategories}, thanks to Lurie's foundational work \cite{HTT}, which we take as a primary reference. However, we stress that throughout this paper we work in an entirely \textit{model-independent} fashion: we only manipulate $\infty$-categories in a manner that makes no reference to a specific model (so e.g.\! we never make reference to the individual simplices of a quasicategory).

The theory of $\infty$-categories reifies the limiting case of Grothendieck's homotopy hypothesis: among $\infty$-categories, the $\infty$-groupoids are equivalent to homotopy types. We refer to such objects alternately as \bit{$\infty$-groupoids} or as \bit{spaces}, depending on the context. We write $\Cat_\infty$ for the $\infty$-category of (small) $\infty$-categories (see \Cref{subsec:set-theory-Grothendieck-universes} for a brief discussion of set-theoretic matters), and we write $\Spaces \subset \Cat_\infty$ for the full subcategory of spaces.

We can now return to the problem of defining weak $n$-categories. The essential observation is as follows: all of the desired coherence conditions articulate \textit{equivalences} between various composite operations. Thus, in order to obtain a robust theory of $(\infty,n)$-categories, it suffices to have a robust theory of $\infty$-categories enriched in a given one: then, we can recursively define $(\infty,n)$-categories to be $\infty$-categories that are enriched in the $(\infty,1)$-category of $(\infty,n-1)$-categories. As we explain further in \Cref{subsec:enriched-infty-cats}, such a robust formalism is provided by \cite{GH13}. Hence, writing $\Cat[\VV]$ for the $\infty$-category of $\VV$-enriched $\infty$-categories, we may recursively define \bit{the $(\infty,1)$-category of $(\infty,n)$-categories} as $\CatInfty{n} \coloneqq \Cat[\CatInfty{n-1}]$, the $\infty$-category of $\infty$-categories enriched in (small) $(\infty,n-1)$-categories; as a base case we define $\CatInfty{0} \coloneqq \Spaces$, and as a consistency check we have an equivalence $\CatInfty{1} \coloneqq \Cat[\CatInfty{0}] \simeq \Cat_\infty$ \cite[Thm. 5.4.6]{GH13}.
As explained in  \Cref{subsec:enriched-infty-cats}, for $k \geq 0$, $\CatInfty{k}$  is Cartesian presentably symmetric monoidal.
Among the $(\infty,n)$-categories, weak $(n,n)$-categories can then be defined simply as those satisfying certain discreteness conditions \cite[\S~6.1]{GH13}.

In fact, this definition ultimately affords an \textit{$(\infty,n+1)$-category} (as opposed to just an $(\infty,1)$-category) of $(\infty,n)$-categories, using the fact that $\CatInfty{n}$ is Cartesian closed \cite{rezk2010cartesian}: for any $(\infty,n)$-categories $\cC$ and $\cD$ we have an $(\infty,n)$-category $\Fun(\cC,\cD)$ of functors between them, which is uniquely characterized by the universal property that we have a natural equivalence \[\hom_{\CatInfty{n}}(\cE, \Fun(\cC,\cD)) \simeq \hom_{\CatInfty{n}}(\cE \times \cC , \cD)\] of hom-spaces for any $(\infty,n)$-category $\cE \in \CatInfty{n}$.

\subsection{Some basic notions in \texorpdfstring{$\infty$}{infinity}-category theory}

Here we highlight a few $\infty$-categorical notions that we use repeatedly throughout this paper. We make no effort to give a comprehensive account, and instead refer the interested reader to \cite{HTT} for a more thorough treatment. Indeed, a remarkable number of notions in ordinary category theory port over to $\infty$-category theory with minimal modification (though see \Cref{subsec:higher-coherence} for a prominent non-example, and see \Cref{subsubsection:connected-vs-weakly-ctrbl} for another non-example).

\subsubsection{Basic notions}\label{subsubsec:basic}

Broadly speaking, the fundamental role played by sets in ordinary category theory is played by spaces in $\infty$-category theory. In particular, as noted in \Cref{subsec:appendix-infty-n-cats}, an $\infty$-category $\cC$ is enriched in spaces (i.e.\! $\infty$-groupoids): for any pair of objects $c,d \in \cC$ we obtain a space $\hom_\cC(c,d) \in \Spaces$. These hom-spaces admit a composition law, which is associative and unital up to coherent homotopy. 

Any $\infty$-category has an associated ordinary $1$-category $h_1\cC$, called its \emph{homotopy category}, with the same objects  as $\cC$ and hom-sets $\hom_{h_1 \cC}(c,d) \coloneqq \pi_0 \hom_{\cC}(c,d)$,  i.e. identifying $1$-morphisms in $\cC$ if there is an invertible $2$-morphism between them. 

A \bit{presheaf} on an $\infty$-category $\cC$ is a functor $\cC^\op \ra \Spaces$. These assemble into the $\infty$-category $\cP(\cC) \coloneqq \Fun(\cC^\op,\Spaces)$, which receives a fully faithful Yoneda embedding $\cC \xra{\hom_\cC(=,-)} \cP(\cC)$ \cite[Prop. 5.1.3.1]{HTT}.

As a matter of terminology, we interchangeably use the terms ``isomorphism'', as in ordinary category theory and ``equivalence'' (in order to emphasize that one is working in a higher-categorical context).

In ordinary categories, objects characterized by universal properties (e.g.\! limits and colimits) are \textit{unique up to unique isomorphism} when they exist: said differently, the collection of objects satisfying the characterization assemble into an empty or contractible groupoid. In $\infty$-categories, objects characterized by a universal property instead assemble into an empty or contractible \textit{$\infty$-groupoid}. For instance, an object $c \in \cC$ is called \textit{initial} if for every $d \in \cC$ the space $\hom_\cC(c,d)$ is contractible, and the initial objects of $\cC$ assemble into an empty or contractible $\infty$-groupoid.

In classical category theory, the term ``unique up to unique isomorphism'' is sometimes replaced by the shorter term ``essentially unique''. The word ``essentially'' here is meant to indicate that object is not \textit{literally} unique (e.g.\! there exist many terminal objects in the category $\Set$ of sets (namely the singletons)), but rather that it is unique in the appropriate category-theoretic sense. However, in $\infty$-category theory one is emphatically \textit{never} interested in uniqueness beyond that in the $\infty$-categorical sense (i.e.\! parametrized by a contractible $\infty$-groupoid), and so we generally omit all technical uses of the word ``essentially''. Relatedly, we will refer to a functor $\cC \xra{F} \cD$ simply as \textit{surjective} (rather than ``essentially surjective'') if for every object $d \in \cD$ there exists an object $c \in \cC$ and an equivalence $F(c) \simeq d$.

\subsubsection{Monomorphisms and subcategories}
\label{subsubsection:monos-and-subcats}

A general pattern in higher category theory is that one must keep track of ``higher coherence data'' (see e.g.\! \Cref{subsec:higher-coherence}). Thus, it is notable when a given construction does \textit{not} require this. Given a construction that a priori might involve coherence data, we say that the data is in fact (\bit{merely}) \bit{a condition} in order to indicate that such data is unique if it exists (i.e.\! that the $\infty$-category of such assembles into an empty or contractible $\infty$-groupoid).

Most fundamentally, given a space $X$ and a subset of its path components, it is merely a condition for a point $x \in X$ to lie in one of these. In fact, the inclusions of path components are precisely the \textit{monomorphisms} in the $\infty$-category of spaces: given an inclusion of path components $Y \xhookra{i} X$, it is merely a condition for \textit{any} map $Z \ra X$ to factor through it.

This notion generalizes: we say that a morphism $c \ra d$ in an $\infty$-category $\cC$ is a \bit{monomorphism} if it is merely a condition for any morphism $e \ra d$ to factor through it.\footnote{This is equivalent to the condition that the commutative square
\[ \begin{tikzcd}[ampersand replacement=\&]
c
\arrow{r}{\id_c}
\arrow{d}[swap]{\id_c}
\&
c
\arrow{d}
\\
c
\arrow{r}
\&
d
\end{tikzcd} \]
is a pullback square. (In particular, if a functor commutes with pullbacks then it preserves monomorphisms.)} This is equivalent to the condition that the resulting morphism $\hom_\cC(-,c) \ra \hom_\cC(-,d)$ in $\cP(\cC)$ is a componentwise monomorphism.

As a notable example, the monomorphisms in $\Cat_\infty$ are precisely the functors that are fully faithful on equivalences and monomorphisms on all hom-spaces.\footnote{This claim is easy to check in the model for $\infty$-categories given by complete Segal spaces \cite{rezk2001model}.} We reserve the term \bit{subcategory} for (the image of) a monomorphism (in $\Cat_\infty$, or more generally in $\Cat[\VV]$ (again see \Cref{subsec:enriched-infty-cats})).

As another notable example, it is merely a condition for a morphism in an $\infty$-category to be an equivalence. Said differently, the functor $[1] \ra [1]^\gpd \simeq \pt$ is an epimorphism in $\Cat_\infty$ (see \S\S\ref{subsubsec:simplicial-objects}-\ref{subsubsec:localizations} for an explanation of the notation).

\subsubsection{Adjunctions}

It is merely a condition for a functor $\cC \xra{F} \cD$ to be a (say) left adjoint: its space of right adjoints is either empty or contractible. First of all, a \textit{pointwise right adjoint} to $F$ at an object $d \in \cD$ is a pair of an object $c \in \cC$ and a morphism $F(c) \xra{\varepsilon_d} d$ such that for every $c' \in \cC$ the composite $\hom_\cC(c',c) \xra{F} \hom_\cD(F(c'),F(c)) \xra{\varepsilon_d} \hom_\cD(F(c'),d)$ is an equivalence. Equivalently, this is the data of a representing object for the presheaf $\cC^\op \xra{\hom_\cD(F(-),d)} \Spaces$ (which by definition comes equipped with the data of a universal element $\varepsilon_d \in \hom_\cD(F(c),d)$ witnessing it as such). Then, a right adjoint exists if and only if a pointwise right adjoint exists at all objects of $\cD$; in this case, the right adjoint is the (necessarily unique) factorization of the functor $\cD \xra{\hom_\cD(F(=),-)} \cP(\cC)$ through the Yoneda embedding. (See \Cref{subsec:adjns-revisited} for an alternative description of $\infty$-categorical adjunctions.)

As a basic example, there exists a right adjoint $\Spaces \xleftarrow{\iota_0} \Cat_\infty$ to the inclusion, which carries an $\infty$-category $\cC$ to its maximal subgroupoid $\cC^\simeq$ (which is obtained by discarding all of its noninvertible morphisms).

\subsubsection{Simplicial objects}
\label{subsubsec:simplicial-objects}

We write $\Delta$ for the simplicial indexing category (the full subcategory of $\Cat_\infty$ on the finite nonempty totally ordered sets), and for any $n \geq 0$ we write $[n] \coloneqq \{ 0 < 1 < \cdots < n \} \in \Delta$ for the indicated standard object. A \textit{simplicial object} in an $\infty$-category $\cC$ is a functor $\Delta^\op \xra{X} \cC$; we use the term \textit{geometric realization} to refer to its colimit, and denote this by $|X| \coloneqq \colim_{\Delta^\op}(X) \in \cC$.\footnote{The classical notion of geometric realization that carries a simplicial set to a topological space is a \textit{homotopy} colimit (of the corresponding levelwise discrete simplicial topological space), which provides a concrete model for the colimit in $\Spaces$ of the corresponding levelwise discrete simplicial space. So, these two uses of the term are spiritually compatible.}

\subsubsection{Localizations}
\label{subsubsec:localizations}

Given an $\infty$-category $\cC$ and a collection $\bW$ of morphisms in $\cC$ (often assumed to be those defining a subcategory of $\cC$), the \bit{localization} of $\cC$ at $\bW$ is the target of the initial functor $\cC \ra \cC[\bW^{-1}]$ that carries all morphisms in $\bW$ to equivalences. As an extreme example, the \textit{$\infty$-groupoid completion} of $\cC$ is its localization at \textit{all} of its morphisms; this defines a left adjoint $\Cat_\infty \xra{(-)^\gpd} \Spaces$ to the inclusion. More generally, we can identify the localization at (the morphisms in) a subcategory $\bW \subseteq \cC$ as the pushout
\[ \begin{tikzcd}
\bW
\arrow[hook]{r}
\arrow{d}
&
\cC
\arrow{d}
\\
\bW^\gpd
\arrow{r}
&
\cC[\bW^{-1}]
\end{tikzcd}
~.
\]

A special case of localization is given by a \bit{reflective localization adjunction}, i.e.\! an adjunction
\begin{equation}
\label{eq:reflective-localization}
\begin{tikzcd}[column sep=1.5cm]
\cC
\arrow[yshift=0.9ex]{r}{L}
\arrow[hookleftarrow, yshift=-0.9ex]{r}[yshift=-0.2ex]{\bot}[swap]{R}
&
\cD
\end{tikzcd}
\end{equation}
in which the right adjoint is fully faithful. In this case, writing $\bW \subseteq \cC$ for the subcategory of morphisms in $\cC$ that are carried to equivalences in $\cD$, the left adjoint witnesses $\cD$ as the localization $\cC[\bW^{-1}]$. In this case, $R$ can be characterized as the inclusion of the full subcategory of objects of $\cC$ that are \bit{local} with respect to the morphisms in $\bW$, i.e.\! those $c \in \cC$ such that for every $d \ra e$ in $\bW$ the morphism $\hom_\cC(d,c) \leftarrow \hom_\cC(e,c)$ is an equivalence \cite[Prop. 5.5.4.2]{HTT}.\footnote{If $\cC$ admits a terminal object $\pt_\cC \in \cC$, then this locality condition is equivalent to the orthogonality relation $(d \ra e) \bot (c \ra \pt_\cC)$ (see \Cref{def:orthogonal-morphisms}).} Of course, dual remarks pertain to \textit{coreflective localization adjunctions}, i.e.\! adjunctions in which the left adjoint is fully faithful.

Note that we might obtain the same localization even if we change the collection $\bW$. For instance, it is unnecessary to invert equivalences (since they are already invertible), and for any pair of composable morphisms $f$ and $g$ inverting any two of $f$, $g$, and $gf$ automatically inverts the third (since equivalences have the two-out-of-three property). This observation plays a key role in the theory of accessible localizations of presentable $\infty$-categories (see \Cref{subsec:presentable-infty-cats}).

\subsubsection{Connected categories versus weakly contractible $\infty$-categories}
\label{subsubsection:connected-vs-weakly-ctrbl}

While certain results in ordinary category theory refer to \textit{connected} categories (i.e.\! those whose groupoid completions are connected), their $\infty$-categorical analogs generally instead refer to \textit{weakly contractible} $\infty$-categories (i.e.\! those whose $\infty$-groupoid completions are contractible). For instance, the forgetful functor $\cC_{c/} \ra \cC$ commutes with (and detects) weakly contractible colimits.\footnote{As a non-example, binary coproducts in $\cC_{c/}$ are computed by pushouts in $\cC$.}

\subsection{Higher coherence}
\label{subsec:higher-coherence}

Here we briefly illustrate the primary operational difference between working in ordinary categories and working in $\infty$-categories, namely that in the latter case one must keep track of \textit{higher coherence data}.

Let $M$ be a monoid (i.e.\! a set equipped with an associative and unital binary operation). Then, the data of $M$ is entirely recorded by its \textit{bar construction}, a simplicial set $\Bar(M)$ with $\Bar(M)_n \coloneqq M^{\times n}$ whose face and degeneracy maps respectively record the product and unit of $M$. Indeed, $M$ is already completely specified by the restriction $\Delta^\op_{\leq 3} \hookra \Delta^\op \xra{\Bar(M)} \Set$; note that the associativity of its multiplication is guaranteed by the commutativity of a certain square
\begin{equation}
\label{eq:square-for-associativity}
\begin{tikzcd}
{[3]}
\arrow{r}
\arrow{d}
&
{[2]}
\arrow{d}
\\
{[2]}
\arrow{r}
&
{[1]}
\end{tikzcd}
\end{equation}
of face maps in $\Delta^\op$ (whose morphisms all correspond to endpoint-preserving injections in $\Delta$). Altogether, we can identify monoids as a full subcategory either of $\Fun(\Delta^\op,\Set)$ or of $\Fun(\Delta^\op_{\leq 3},\Set)$.

By contrast, such a restriction -- or more generally, the restriction to $\Delta^\op_{\leq n} \subset \Delta^\op$ for any $n$ -- is \textit{not} possible in the context of $\infty$-category theory. As a fundamental example, an $\infty$-monoid $M$ (i.e.\! an $\infty$-categorical monoid object in $\Spaces$) is completely specified by its bar construction $\Delta^\op \xra{\Bar(M)} \Spaces$ (whose face and degeneracy maps likewise record its product and unit). Heuristically, we may think of relations among various morphisms and their composites in $\Delta^\op$ as recording coherence data for the muliplication of $M$, inasmuch as the functor $\Bar(M)$ carries these equalities in the hom-sets of $\Delta^\op$ only to ``homotopy-coherent equalities'' (i.e.\! higher equivalences) in $\Spaces$.\footnote{For instance, the commutative square \eqref{eq:square-for-associativity} selects an associator for $M$ (i.e.\! a path in $\hom_\Spaces(M^{\times 3} , M)$), and thereafter we can locate the pentagon axiom as the image in $\hom_\Spaces(M^{\times 4},M)$ of the (tautological and unique) nullhomotopy of a certain pentagon in $\hom_{\Delta^\op}([4],[1]) \in \Set \subset \Spaces$.} Because $\Spaces$ is an $\infty$-category (and not an $(n,1)$-category for any $n < \infty$, i.e.\! its hom-spaces can have homotopy groups in arbitrarily high dimensions), these coherence data never become unique or vacuous after some finite stage.

We note for future reference that $\infty$-monoids can be identified (via their bar constructions) as the full subcategory of $\Fun(\Delta^\op , \Spaces)$ on those simplicial spaces $X$ satisfying a \textit{Segal condition}, namely that for every $n \geq 0$ a certain natural morphism $X_n \ra (X_1)^{\times n}$ is an equivalence.

We generally suppress the modifier ``homotopy coherently'' (e.g.\! of the adjectives ``associative'' and ``unital''), unless we specifically mean to draw attention to it.

\subsection{Straightening and unstraightening}

A fundamental tool in $\infty$-category theory is the \textit{straightening} and \textit{unstraightening} equivalence, as we now briefly describe.
\footnote{See \cite{MR3999274} for a more leisurely description.}

Fix an $\infty$-category $\cB$ and a functor $\cB \xra{F} \Cat_\infty$. Then, the (\bit{coCartesian}) \bit{unstraightening} of $F$ (a.k.a.\! its (covariant) \textit{Grothendieck construction}) is an object $(\cE \xra{p} \cB) \in (\Cat_\infty)_{/\cB}$ that may be described heuristically as follows:
\begin{itemize}
\item an object of $\cE$ is given by a pair of an object $b \in \cB$ and an object $x \in F(b)$;
\item a morphism $(b,x) \ra (c,y)$ in $\cE$ is given by a morphism $b \xra{f} c$ in $\cB$ along with a morphism $F(f)(x) \xra{\alpha} y$ in $F(c)$.
\end{itemize}
(Of course, the images under $p$ of these data are simply $b$ and $f$, respectively.) Such a morphism $(f,\alpha)$ in $\cE$ is called (\bit{$p$-})\bit{coCartesian} if $\alpha$ is an equivalence. Observe that these satisfy a universal property: if $e \xra{\varphi} f$ in $\cE$ is $p$-coCartesian, then for any $g \in \cE_{p(f)}$ we have an equivalence $\hom_\cE(f,g) \simeq \hom_\cE(e,g) \times_{\hom_\cB(p(e),p(g))} \{p(\varphi) \}$.

Conversely, a functor $\cE \xra{p} \cB$ is called a \bit{coCartesian fibration} if for every pair of an object $e \in \cE$ and a morphism $p(e) \xra{f} b$ in $\cB$, the morphism $f$ admits a coCartesian lift with source $e$. In this case, $p$ is the unstraightening of a functor $\cB \xra{F} \Cat_\infty$, whose values are given by the fibers $F(b) \simeq \cE_b$ and whose functoriality is implicitly specified by the coCartesian morphisms (in essence because the Yoneda embedding is fully faithful). We refer to $F$ as the \bit{straightening} of $p$, and to its functoriality $F(b) \xra{F(f)} F(c)$ for a morphism $b \xra{f} c$ in $\cB$ as the \textit{coCartesian monodromy functor} of $\cE$ associated to $f$.

The coCartesian fibrations over $\cB$ define a (generally non-full) subcategory $\coCart_\cB \subseteq (\Cat_\infty)_{/\cB}$, whose morphisms are those functors over $\cB$ that preserve coCartesian morphisms. Altogether, by \cite[Thm. 3.2.0.1]{HTT}, straightening and unstraightening define inverse equivalences
\[
\Fun(\cB,\Cat_\infty)
\simeq
\coCart_\cB
~,
\]
under which precomposition with a functor $\cB' \ra \cB$ corresponds to pullback therealong.

A similar but dual story applies in the case of a functor $\cB^\op \xra{F} \Cat_\infty$: this now has a (\textit{Cartesian}) \textit{unstraightening} (a.k.a.\! its (contravariant) Grothendieck construction), giving an object $(\cE \ra \cB) \in (\Cat_\infty)_{/\cB}$ admitting a dual description. Altogether, we obtain an analogous equivalence
\[
\Fun(\cB^\op,\Cat_\infty)
\simeq
\Cart_\cB
~.
\]

\subsection{Adjunctions revisited}
\label{subsec:adjns-revisited}
An adjunction of $\infty$-categories can be defined as a functor $\cE \ra [1]$ that is both a coCartesian fibration and a Cartesian fibration. Its coCartesian unstraightening defines the left adjoint $\cE_0 \xra{L} \cE_1$, while its Cartesian unstraightening defines the right adjoint $\cE_0 \xleftarrow{R} \cE_1$, and the universal properties of coCartesian and Cartesian morphisms yield natural equivalences $\hom_{\cE_0}(e,R(f)) \simeq \hom_\cE(e,f) \simeq \hom_{\cE_1}(L(e),f)$ for any $e \in \cE_0$ and $f \in \cE_1$.

We define a \bit{morphism of adjunctions} to be a morphism in $\coCart_{[1]} \cap \Cart_{[1]}$.\footnote{These are also frequently referred to as ``Beck--Chevalley squares''.} In particular, a morphism of adjunctions determines a commutative square in $\Cat_\infty$ after omitting either both left adjoints or both right adjoints.

Given a commutative square in $\Cat_\infty$ in which two parallel functors are both (say) left adjoints, passing to their right adjoints we obtain a canonical laxly-commutative square (i.e.\! one that commutes up to a specified natural transformation), and it is merely a condition for this to be invertible so that the original square defines a morphism adjunctions \cite{haugseng2023lax}.\footnote{This is frequently referred to as a ``Beck--Chevalley condition'' on the original commutative square.} Of course, this is nothing but the condition that the morphism in $\coCart_{[1]}$ specified by the original square lies in the subcategory $\coCart_{[1]} \cap \Cart_{[1]}$. Dual remarks apply if the two parallel functors are instead both right adjoints.

\subsection{Set-theoretic considerations}
\label{subsec:set-theory-Grothendieck-universes}

In order to deal with set-theoretic issues, we systematically use the device of
\textit{Grothendieck universes} (see e.g.\! \cite[\S~1.2.15]{HTT}).
Specifically, we fix a triple of strongly inaccessible cardinals $\kappa_0 <
\kappa_1 < \kappa_2$. The sets of cardinality $<\kappa_i$ for $0\leq i \leq 2$ will be called
$\kappa_i$-small and they form Grothendieck universes $U_0\in U_1\in U_2$. Likewise, a category is called $\kappa_i$-small if the sets of
isomorphism classes of objects and the homotopy groups of morphisms spaces are of
cardinality $<\kappa_i$. We refer to $\kappa_0$-small objects as \textit{small},
to $\kappa_1$-small objects as \textit{large}, and to $\kappa_2$-small objects
as \textit{huge} (and the latter play almost no role in in our work). So for
instance, the $\infty$-category $\Cat_\infty$ of small $\infty$-categories is
large, as is the $\infty$-category $\Spaces$ of (small) spaces.

We occasionally write e.g.\! $\hatCat_\infty$ to refer to the huge $\infty$-category of large $\infty$-categories. Its main use is that it contains the $\infty$-category of \textit{presentable} $\infty$-categories (see \Cref{subsec:presentable-infty-cats}). We often prove results for $\Cat_\infty$ and then apply them to $\hatCat_\infty$ (which is easily justified by a change of Grothendieck universe) in order to discuss specializations to presentable $\infty$-categories.

We may sometimes emphasize smallness (e.g.\! of a set or of an $\infty$-category). On the other hand, we may also omit the word ``small'' for brevity; for instance, when we say that an $\infty$-category admits all colimits we certainly mean that it admits all \textit{small} colimits.

We generally refer to a large set as a ``class'' (and to a class that is not small as a ``proper class''). However, in related contexts we will have occasion to contemplate large \textit{spaces}, and rather than belabor the distinction we simply also refer to these as ``classes''.

Relatedly, in the most invariant terms, given an $\infty$-category $\cC$, ``a set of objects of $\cC$'' refers to a set $S$ equipped with a functor $S \xra{F} \cC$. We say that an object of $\cC$ lies in the set if it is in the image of $F$ (up to equivalence). Said differently, when we refer to a set of objects of $\cC$, we generally intend to implicitly refer to its image (a subgroupoid of $\cC$). Note that if $S$ is small and $\cC$ is locally small, then the image of $S$ in $\cC$ is also small; hence, in such cases this implicit passage to images does not change size.

\subsection{Presentable \texorpdfstring{$\infty$}{infinity}-categories}
\label{subsec:presentable-infty-cats}

Many ($\infty$-)categories of lasting interest are not small, but are
nevertheless ``controlled by small data'' -- namely, they are \bit{presentable}.
By definition, an $\infty$-category $\cC$ is presentable if it admits all small
colimits and moreover there exists some regular cardinal $\kappa$ such that
$\cC$ is the completion of its full subcategory $\cC^\kappa \subseteq \cC$ of
$\kappa$-compact objects under $\kappa$-filtered colimits. For this we recall
that a $\kappa$-filtered colimit means a colimit indexed by a $\kappa$-filtered
category, i.e. an 
$\infty$-category, in which every diagram of cardinality
$<\kappa$ has a cocone, and an object is called $\kappa$\emph{-compact} if the
associated representable functor preserves $\kappa$-filtered colimits. If we can
take $\kappa$ to be the cardinality $\omega$ of the natural numbers, we say that
$\cC$ is \bit{compactly generated} (as $\omega$-compact objects are generally
just called ``compact objects'').

An extremely convenient feature of presentable $\infty$-categories is their \bit{adjoint functor theorem} \cite[Cor. 5.5.2.9]{HTT}: a functor between presentable $\infty$-categories is a left adjoint if and only if it preserves small colimits, and it is a right adjoint if and only if it is accessible (i.e.\! preserves $\kappa$-filtered colimits for some $\kappa$) and preserves small limits. Presentable $\infty$-categories naturally define two subcategories
\[
{\Pr}^L
\subset
\hatCat_\infty
\supset
{\Pr}^R
\]
of the huge $\infty$-category of large $\infty$-categories, in which the morphisms are the left (resp.\! right) adjoint functors. Evidently, passing to adjoints defines an equivalence $\Pr^L \simeq (\Pr^R)^\op$. These actually define $(\infty,2)$-categories (by taking all natural transformations as 2-morphisms), and we write $\Fun^L(-,-)$ and $\Fun^R(-,-)$ for their respective hom-$(\infty,1)$-categories.

An \bit{accessible localization} is by definition a reflective localization among presentable $\infty$-categories. The left adjoint of an accessible localization is a localization not just in $\hatCat_\infty$ but also in $\Pr^L$ \cite[Prop. 5.5.4.20]{HTT}. Moreover, given any accessible localization
\[
\begin{tikzcd}[column sep=1.5cm]
\cC
\arrow[yshift=0.9ex]{r}{L}
\arrow[hookleftarrow, yshift=-0.9ex]{r}[yshift=-0.2ex]{\bot}[swap]{R}
&
\cD
\end{tikzcd}
~,
\]
the left adjoint $L$ witnesses $\cD$ as the localization $\cC[S^{-1}]$ for some small set $S$ of morphisms in $\cC$, and hence $R$ is the fully faithful inclusion of the subcategory of $S$-local objects \cite[Prop. 5.5.4.1]{HTT}.

Presentable $\infty$-categories admit presentations by generators and relations, in the following sense. First of all, for any small $\infty$-category $\cC \in \Cat_\infty$, its $\infty$-category $\cP(\cC)$ of presheaves is presentable. This is the free presentable $\infty$-category on $\cC$: for any $\cD \in \Pr^L$, restriction along the Yoneda embedding defines an equivalence $\Fun(\cC,\cD) \xleftarrow{\sim} \Fun^L(\cP(\cC),\cD)$ \cite[Thm. 5.1.5.6]{HTT}. And then, any presentable $\infty$-category is an accessible localization of $\cP(\cC)$ for some $\cC \in \Cat_\infty$ \cite[Thm. 5.5.1.1]{HTT}.

There exists a symmetric monoidal structure on $\Pr^L$, which is characterized by the fact that morphisms $\cC \otimes \cD \ra \cE$ in $\Pr^L$ (i.e.\! left adjoint functors) are equivalent to functors $\cC \times \cD \ra \cE$ that are bicocontinuous (i.e.\! cocontinuous (or equivalently, left adjoints) separately in each variable), whose unit object is $\Spaces \simeq \cP(\pt) \in \Pr^L$. A \bit{presentably} (\bit{symmetric}) \bit{monoidal $\infty$-category} is a (resp.\! commutative) algebra object in $(\Pr^L,\otimes)$, i.e.\! a presentable $\infty$-category equipped with a (resp.\! symmetric) monoidal structure that is cocontinuous separately in each variable.\footnote{In any symmetric monoidal $\infty$-category the unit object is canonically a commutative algebra object, and here this recovers the Cartesian symmetric monoidal structure on $\Spaces$ (which is a presentably symmetric monoidal structure).} Most (symmetric) monoidal presentable $\infty$-categories of lasting interest (e.g.\! $\CatInfty{n}$ (and in particular $\Spaces$ and $\Cat_\infty$) and $\Spectra$) are presentably (resp.\! symmetric) monoidal. 

\subsection{Some basics of \texorpdfstring{$\infty$}{infinity}-operads}
\label{appendix:operads}
Here we briefly discuss some relevant features of the theory of $\infty$-operads introduced in \cite[\S~2]{HA}.

\subsubsection{Basic notions}

The notion of an $\infty$-operad is an $\infty$-categorical version of the theory of \textit{colored operads}. A colored operad consists of a set $\iota_0\underline{\cO}$ of \textit{colors} along with for every finite set $\{X_i \in \iota_0 \underline{\cO} \}_{i \in I}$ of colors and every color $Y \in \iota_0 \underline{\cO}$ a set $\Mul_\cO(\{X_i\}_{i \in I},Y)$ of \textit{multimorphisms} from $\{X_i\}_{i \in I}$ to $Y$, which altogether must be equipped with a associative and unital composition law.\footnote{The usage of an abstract finite set $I$ here (as opposed to $\{1,\ldots,n\}$) is convenient since it naturally builds in the relevant symmetric group actions.} In particular, the unary multimorphisms (i.e.\! those with $|I| = 1$) define a category $\underline{\cO}$ of colors (whose set of objects is $\iota_0 \underline{\cO}$).

We now give a hint of the main definition. An \bit{$\infty$-operad} $\cO$ is an $\infty$-category $\cO^{\otimes}$ (called the \emph{$\infty$-category of operators} of $\cO$) equipped with a functor $\cO^{\otimes} \ra \Fin_*$ to the category of finite pointed sets satisfying certain conditions. We immediately introduce the notation $\underline{n}_+ \coloneqq \{ 1, 2, \ldots, n \}_+ \in \Fin_*$ for the indicated standard object, as well as the notation $\underline{\cO} \coloneqq \cO^{\otimes}_{\underline{1}_+}$ for the indicated fiber. We refer to $\underline{\cO}$ as the \textit{$\infty$-category of colors} of $\cO$ (or sometimes as its \textit{underlying $\infty$-category}, for reasons that will be explained shortly). We will sometimes abuse notation and denote the underlying $\infty$-category $\underline{\cO}$ of an $\infty$-operad $\cO$ simply also by $\cO$. The crux of the definition of an $\infty$-operad is that $\cO^{\otimes}$ satisfies a sort of ``fiberwise'' Segal condition which implies that for every $n \geq 0$ there is a natural equivalence $\cO^{\otimes}_{\underline{n}_+} \simeq \underline{\cO}^{\times n}$, as well as an ``internal'' Segal condition which implies that for every pair of objects $X \coloneqq (X_1,\ldots,X_m) \in \underline{\cO}^{\times m} \simeq \cO^{\otimes}_{\underline{m}_+}$ and $Y \coloneqq (Y_1,\ldots,Y_n) \in \underline{\cO}^{\times n} \simeq \cO^{\otimes}_{\underline{n}_+}$, we have a natural equivalence
\[
\hom_{\cO^{\otimes}}(X,Y)
\simeq
\bigsqcup_{f \in \hom_{\Fin_*}(\underline{m}_+,\underline{n}_+)}
\prod_{i = 1}^n \hom_{\cO^{\otimes}}(\{X_j\}_{j \in f^{-1}(i)}, Y_i)
~.
\]
An ordinary colored operad $\cO'$ defines an $\infty$-operad $\cO$ with $\hom_{\cO^{\otimes}}(\{X_i\}_{i \in I},Y) \coloneqq \Mul_{\cO'}(\{X_i\}_{i \in I},Y)$. As a result, we also write $\Mul_\cO(\{X_i\}_{i \in I},Y) \coloneqq \hom_{\cO^{\otimes}}(\{X_i\}_{i \in I},Y)$ for the hom-spaces in an $\infty$-operad $\cO$ whose targets lies in $\underline{\cO}$, and refer to their points as \textit{multimorphisms}. Altogether, $\infty$-operads assemble into a (non-full) subcategory $\Op \subset (\Cat_\infty)_{/\Fin_*}$: in essence, morphisms of $\infty$-operads are required to respect the Segal condition equivalences. In fact, allowing all 2-morphisms in $(\Cat_\infty)_{/\Fin_*}$ endows $\Op$ with the structure of an $(\infty,2)$-category, whose hom-$(\infty,1)$-categories we denote by $\eHom_\Op(-,-)$.

We say that an $\infty$-operad $\cO$ is \textit{single-colored} if its $\infty$-category of colors $\underline{\cO}$ is contractible. In this case, we may write $\ast \in \underline{\cO}$ for the unique point, and we write $\cO(n) \coloneqq \Mul_\cO(\{\ast\}_{i \in \{1,\ldots,n\}} , \ast)$ for the unique space of $n$-ary multimorphisms in $\cO$.

\subsubsection{Key examples}

Perhaps the most important family of examples of $\infty$-operads is the sequence $\EE_0 \ra \EE_1 \ra \cdots \ra \EE_\infty$. These are single-colored, with the space $\EE_k(n)$ of $n$-ary operations given by (the underlying space of) the topological space of configurations of $n$ disjoint points in $\RR^k$.\footnote{This topological space is homotopy equivalent to that of framed embeddings $(\RR^k)^{\sqcup n} \hookra \RR^k$, under which composition of multimorphisms in $\EE_k$ corresponds to composition of framed embeddings.} The above maps are induced by the standard embeddings $\RR^0 \hookra \RR^1 \hookra \cdots \hookra \RR^\infty$. We note that $\EE_1$ and $\EE_\infty$ are respectively the $\infty$-operads underlying the colored operads that parametrize associative and commutative algebras (and in particular, their spaces of multimorphisms are discrete). Hence, we also write $\Assoc \coloneqq \EE_1$ and $\Comm \coloneqq \EE_\infty$ and respectively refer to these as the \textit{associative} and \textit{commutative} $\infty$-operads. In fact, $\Comm$ is simply the identity functor $\Comm \coloneqq \EE_\infty \simeq \Fin_* \xra{\id} \Fin_*$, and defines a terminal object of $\Op$.

Another illustrative example is the $\infty$-operad  $\mathrm{LM}$ associated to the two-colored operad parametrizing pairs of an associative algebra object along with a left module over it. We will return to $\mathrm{LM}$ in \cref{subsec:module-cats}.

It will occasionally be useful for us to refer to the single-colored $\infty$-operad $\Triv$, which has no $n$-ary multimorphisms for $n \not= 1$ and the only $1$-ary morphism is the identity morphism. Given an $\infty$-operad $\cO$, we write $\cO_\Triv \coloneqq \cO \times_\Comm \Triv$.

\subsubsection{$\cO$-monoidal $\infty$-categories}
\label{subsubsec:O-monoidal-infty-cats}

Given an $\infty$-operad $\cO$, an \bit{$\cO$-monoidal $\infty$-category} $\cC$ is a coCartesian fibration $\cC^{\otimes} \ra \cO^{\otimes}$ satisfying analogous Segal conditions, which are equivalent to the condition that the composite $\cC^{\otimes} \ra \cO^{\otimes}\ra \Fin_*$ is also an $\infty$-operad. In particular, an $\cO$-monoidal $\infty$-category can be equivalently specified by the straightening $\cO^{\otimes} \ra \Cat_\infty$ of this coCartesian fibration.
\footnote{This latter perspective is effectively a generalization of the bar construction indicated in \Cref{subsec:higher-coherence}, and it further generalizes to $\cO$-algebra objects in any Cartesian symmetric monoidal $\infty$-category \cite[\S~2.4.2]{HA} (which notion is defined shortly).} We often abuse the notation by denoting an $\cO$-monoidal $\infty$-category by its source operad $\cC$.
Altogether, $\cO$-monoidal $\infty$-categories define a full subcategory $\Alg_\cO(\Cat_\infty) \subseteq \coCart_{\cO^{\otimes}} \simeq \Fun(\cO^{\otimes},\Cat_\infty)$. As special cases, we write $\Alg(\Cat_\infty) \coloneqq \Alg_\Assoc(\Cat_\infty)$ for the $\infty$-category of \bit{monoidal $\infty$-categories} and $\CAlg(\Cat_\infty) \coloneqq \Alg_\Comm(\Cat_\infty)$ for the $\infty$-category of \bit{symmetric monoidal $\infty$-categories}.
The restricted coCartesian fibration $\underline{\cC} \ra \underline{\cO}$ (or simply its source) may be thought of as the ``underlying $\infty$-category'' of $\cC$, although this is most immediately meaningful when $\cO$ is single-colored.

An $\infty$-category that admits finite products canonically upgrades to a \textit{Cartesian symmetric monoidal} $\infty$-category. We note that it is merely a condition for a symmetric monoidal $\infty$-category to be Cartesian symmetric monoidal. Dual remarks apply in the case of finite coproducts.

\subsubsection{$\cO$-algebra objects}
\label{appendix:CAlg}

Given an $\cO$-monoidal $\infty$-category $\cC$, an \bit{$\cO$-algebra object} in $\cC$ is a section of the structure map $\cC \ra \cO$ in $\Op$. These assemble into an $\infty$-category $\Alg_\cO(\cC) \coloneqq \eHom_{\Op_{/\cO}}(\cO,\cC)$. As special cases, we write $\Alg(\cC) \coloneqq \Alg_\Assoc(\cC)$ for the $\infty$-category of (\bit{associative}) \bit{algebra objects} in $\cC$ and $\CAlg(\cC) \coloneqq \Alg_\Comm(\cC)$ for the $\infty$-category of \bit{commutative algebra objects} in $\cC$.

More generally, given a morphism $\cP \xra{p} \cO$ in $\Op$, we analogously define the $\infty$-category $\Alg_{\cP/\cO}(\cC) \coloneqq \eHom_{\Op_{/\cO}}(\cP,\cC)$ of \bit{$\cP$-algebras} in $\cC$ (relative to $p$). Equivalently, the base change $p^*\cC \ra \cP$ defines the \textit{underlying $\cP$-monoidal $\infty$-category} of $\cC \in \Alg_\cO(\Cat_\infty)$, and we have $\Alg_{\cP/\cO}(\cC) \simeq \Alg_\cP(p^*\cC)$. 
For example, there is a natural morphism $\mathrm{LM} \ra \Assoc$, and so we can contemplate $\mathrm{LM}$-algebras in any monoidal $\infty$-category $\cC \in \Alg(\Cat_\infty)$.
Note that when $\cO = \Fin_*$ we also write this as $\Alg_{\cP}(\cC)$. 
\footnote{Hence, this framework adheres closely to the ``microcosm/macrocosm principle'': it is precisely an $\cO$-monoidal structure on an $\infty$-category that allows us to contemplate $\cO$-algebra objects therein. In particular, one can make sense of $\cO$-algebra objects for \textit{any} $\infty$-operad $\cO$ inside of a symmetric monoidal $\infty$-category (since $\Comm \in \Op$ is terminal).}  Altogether, for an $\cO$-monoidal $\infty$-category $\cC$ we obtain a functor
\[
(\Op_{/\cO})^\op
\xra{\Alg_{(-)/\cO}(\cC)}
\Cat_\infty
~,
\]
whose functoriality is given by precomposition.

As a matter of terminology, it is common to refer to $\cO$-algebra objects in a Cartesian symmetric monoidal $\infty$-category as \bit{$\cO$-monoids} (e.g.\! in $\Spaces$ or $\Cat_\infty$). In particular, $\cO$-monoidal $\infty$-categories are indeed $\cO$-monoids in $\Cat_\infty$. When referring to notions in spaces, one generally simply prepends ``$\infty$-'' to the classical terms, so e.g.\! the objects of $\Alg(\Spaces)$ may be referred to as ``$\infty$-monoids''.

Of particular relevance to this paper is the case $\cO = \EE_2$, and we generally use the term \bit{braided} in place of the prefix ``$\EE_2$-'': in particular, a braided monoidal $(\infty,2)$-category is an $\EE_2$-algebra in $\Cat_{(\infty,2)}$. Indeed, a braided monoidal $\infty$-category in the classical sense defines an $\EE_2$-monoid in $\Cat_\infty$.

\subsubsection{$\cO$-algebras of symmetric monoidal $\infty$-categories}\label{subsubsec:symmetric-monoidal-structure-on-algebras}
Let $\cO, \cC$ be $\infty$-operads, then the $\infty$-category $\Alg_{\cO}(\cC)$ has the structure of an $\infty$-operad \cite[Ex. 3.2.4.4]{HA}. From now on, we will denote the $\infty$-operad of $\cO$-algebras in $\cC$ by $\Alg_{\cO}(\cC)$, and the underlying $\infty$-category of $\cO$-algebra by $\underline{\Alg}_{\cO}(\cC)$, or just $\Alg_{\cO}(\cC)$ if clear from context. When the $\infty$-operad $\cC$ is in fact a symmetric monoidal $\infty$-category, i.e. $\cC^{\otimes} \to \Fin_*$ is a coCartesian fibration, then so is the $\infty$-operad $\Alg_{\cO}(\cC)$. Furthermore, let $X \in \underline{\cO}$ be a color,
the evaluation functor $e_X \colon \underline{\Alg}_{\cO}(\cC) \to \underline{\cC}$, which takes an $\cO$-algebra to its underlying $X$-object, is symmeric monoidal \cite[Prop. 3.2.4.3]{HA}.
More generally, for any map of $\infty$-operads $\cO' \to \cO$, the pullback functor on algebras $\Alg_{\cO}(\cC) \to \Alg_{\cO'}(\cC)$ is a symmetric monoidal functor. 

\subsubsection{Symmetric monoidal structure on overcategories}\label{subsubsec:overcat-sym-mon} 
Let $\cC$ be a symmetric monoidal $\infty$-category and $A \in \CAlg(\cC)$ be a commutative algebra object therein. Then, there exists a symmetric monoidal structure on the overcategory $\cC_{/A}$ (\cite[Thm. 2.2.2.4]{HA}), universally characterized (cf.~\cite[Def. 2.2.2.1]{HA})  by the following equivalence of $\infty$-categories for any $\infty$-operad $\cO$
\[
    \underline{\Alg}_{\cO}(\cC_{/A}) \simeq \underline{\Alg}_{\cO}(\cC)_{/A},
\]
where on the right hand side we $A$ is equipped with the $\cO$-algebra structure induced by the terminal map of operads $\cO^{\otimes} \to \mathrm{Comm}$.

\subsubsection{Boardman-Vogt tensor product and Dunn additivity}\label{subsubsec:BV-tensor}
The $\infty$-category $\Op$ of $\infty$-operads itself carries a symmetric monoidal structure, called the \emph{Boardman-Vogt tensor product} uniquely characterized\footnote{This follows from the explicit construction of the $\infty$-operad $\Alg_{\cO}(\cP)$ in \cite[Const. 3.2.4.1]{HA}}
 by giving rise to an equivalence of $\infty$-operads for all $\infty$-operads $\cO, \cO'$ and $\cP$:
\begin{equation}\label{eq:op-tensor-hom}
    \Alg_{\cO}(\Alg_{\cO'}(\cP)) \simeq \Alg_{\cO \otimes \cO'}(\cP).
\end{equation}
Equivalently, the Boardman-Vogt tensor product has $\Alg_{-}(-)$ as its internal hom. 

A  fundamental theorem in the theory of $\infty$-operad is Dunn's additivity theorem \cite[Thm. 5.1.2.2]{HA}: for $n, m \geq 0$, there is an equivalence of $\infty$-operads $\EE_n \otimes \EE_m \simeq \EE_{n+m}$. In particular, using \ref{eq:op-tensor-hom},  an $\EE_{n+m}$-algebra in an $\infty$-operad $\cO$ is equivalent to an $\EE_n$-algebra in the $\infty$-operad of $\EE_{m}$-algebras in $\cO$.

\subsubsection{Laxly $\cO$-monoidal functors} \label{subsubsection:laxOmonoidal}

If $\cC$ and $\cD$ are $\cO$-monoidal $\infty$-categories, then a morphism $\cC \ra \cD$ in $\Op_{/\cO}$ is called a \bit{laxly $\cO$-monoidal functor}.\footnote{This is also technically a $\cC$-algebra object in $\cD$ (relative to $\cO$), although we find the present terminology to be more illuminating.} 
Let us denote the coCartesian fibrations $\cC^{\otimes} \to \cO^{\otimes}$, $\cD^{\otimes} \to \cO^{\otimes}$ by $p$ and $q$, then an \bit{$\cO$-monoidal} functor is a laxly $\cO$-monoidal functor that takes $p$-coCartesian morphisms in $\cC^{\otimes}$ to $q$-coCartesian morphisms in $\cD^{\otimes}$.

Whereas an $\cO$-monoidal functor respects the $\cO$-monoidal structure up to coherent natural equivalence, a laxly $\cO$-monoidal functor $\cC \ra \cD$ respects it only up to certain (generally noninvertible) coherent natural transformations, which nevertheless suffices to obtain an induced functor $\Alg_\cO(\cC) \ra \Alg_\cO(\cD)$ on $\infty$-categories of $\cO$-algebra objects (simply by composition in $\Op_{/\cO}$). For instance, given a laxly monoidal functor $\cC \xra{F} \cD$ and an algebra object $A \in \Alg_{\EE_1}(\cC)$, we obtain structure maps $F(A) \otimes^\cD F(A) \ra F(A \otimes^\cC A) \xra{F(\mu_A)} F(A)$ and $\uno_\cD \ra F(\uno_\cC) \xra{F(\eta_A)} F(A)$ giving the multiplication and unit of $F(A) \in \Alg(\cD)$.

Furthermore, $\Alg_{\cO}(-)$ takes (laxly) symmetric monoidal functors between symmetric monoidal $\infty$-categories to (laxly) symmetric monoidal functors.

\subsubsection{Localizations of $\cO$-monoidal $\infty$-categories}
\label{subsubsection:loczns-of-O-monoidal-cats}

Given an $\cO$-monoidal $\infty$-category $\cC \in \Alg_\cO(\Cat_\infty)$ and a collection $\bW$ of morphisms in $\cC$, the \bit{$\cO$-monoidal localization} of $\cC$ at $\bW$ is (the target of) the initial object of $\Alg_\cO(\Cat_\infty)_{\cC/}$ in which the morphisms in $\bW$ are sent to equivalences. Of course, this generalizes the notion of localization of $\infty$-categories discussed in \Cref{subsubsec:localizations}.

As an important special case, we say that a reflective localization \eqref{eq:reflective-localization} is \bit{compatible} with a (symmetric) monoidal structure $\otimes \coloneqq \otimes^\cC$ on $\cC$ if for all objects $c,c' \in \cC$ the morphism $L(c \otimes c') \xra{L(\eta_c \otimes \eta_{c'})} L(RL(c) \otimes RL(c'))$ in $\cD$ is an equivalence.\footnote{Even if $\otimes^\cC$ is a symmetric monoidal structure, this compatibility only depends on its underlying monoidal structure.} In this case, $\cD$ inherits a (resp.\! symmetric) monoidal structure $\otimes^\cD$, defined by the formula $d \otimes^\cD d' \coloneqq L(R(d) \otimes^\cC R(d'))$ for any $d,d' \in \cD$ and with unit object $\uno_\cD \coloneqq L(\uno_\cC)$,\footnote{An illustrative example is the completed tensor product of modules over a topological commutative ring (e.g.\! a local commutative ring $(R,\mathfrak{m})$ equipped with the $\mathfrak{m}$-adic topology).} and the left adjoint $L$ is canonically (resp.\! symmetric) monoidal (so that the right adjoint $R$ is canonically laxly (resp.\! symmetric) monoidal). In this case, the left adjoint $L$ witnesses $\cD$ as not just a localization but also a (resp.\! symmetric) monoidal localization of $\cC$.

\subsubsection{Presentably $\cO$-monoidal $\infty$-categories}\label{subsubsec:presentably-O-monoidal}

Given an $\infty$-operad $\cO$, a \bit{presentably $\cO$-monoidal $\infty$-category} is an $\cO$-monoidal $\infty$-category $\cC^{\otimes} \ra \cO^{\otimes}$ such that for every color $X \in \underline{\cO}$ the $\infty$-category $\cC_X$ is presentable and moreover for every multimorphism $\{X_i\}_{i \in I} \ra Y$ in $\cO$ the corresponding multifunctor $\prod_{i \in I} \cC_{X_i} \ra \cC_Y$ is multi-cocontinuous (i.e.\! cocontinuous separately in each variable). This is equivalent to the condition that $\cC$ defines an $\cO$-algebra $(\Pr^L,\otimes)$, and we write $\Alg_\cO(\Pr^L) \subseteq \Alg_\cO(\hatCat_\infty)$ for the subcategory whose objects are the presentably $\cO$-monoidal $\infty$-categories whose morphisms are the $\cO$-monoidal left adjoints among them.

Given a presentably $\cO$-monoidal $\infty$-category $\cC \in \Alg_\cO(\Pr^L)$, the $\infty$-category $\Alg_\cO(\cC)$ is also presentable. Moreover, the functor $(\Op_{/\cO})^\op \xra{\Alg_{(-)/\cO}(\cC)} \hatCat_\infty$ factors through $\Pr^R$, i.e.\! for every morphism $\cA \ra \cB$ in $\Op_\cO$ there exists a left adjoint
\[ \begin{tikzcd}[column sep=1.5cm]
\Alg_{\cA/\cO}(\cC)
\arrow[dashed, yshift=0.9ex]{r}
\arrow[leftarrow, yshift=-0.9ex]{r}[yshift=-0.2ex]{\bot}
&
\Alg_{\cB/\cO}(\cC)
\end{tikzcd}
~,
\]
the ``free $\cB$-algebra on an $\cA$-algebra'' functor \cite[Cor. 3.1.3.5]{HA}. Of course, these left adjoints then assemble into a functor $\Op_{/\cO} \xra{\Alg_{(-)/\cO}} \Pr^L$.

\begin{warning}\label{warning:not-presentably-symmetric-monoidal}
Given a presentably symmetric monoidal $\infty$-category $\cC$ and a small $\infty$-operad $\cO$, the $\infty$-category $\Alg_{\cO}(\cC)$ is presentable and carries a symmetric monoidal structure. However, it is \emph{not} necessarily presentably symmetric monoidal: The symmetric monoidal structure on $\Alg_{\cO}(\cC)$  is not necessarily compatible with finite coproducts (though it is always compatible with sifted colimits). An easy counterexample is $\cC = \Set$ and $\cO=\EE_1$.  
\end{warning}

\subsubsection{Adjunctions of $\cO$-algebras}\label{subsubsec:adjunctions-of-cO-algebras} 
For an $\infty$-operad $\cO$, an \bit{$\cO$-monoidal left adjoint} is an $\cO$-monoidal functor $F \colon \cC \to \cD$ between $\cO$-monoidal $\infty$-categories such that for each color $X \in \cO$, the underlying functor $F_X \colon \cC \to \cD$ is a left adjoint. 

An important fact which we use repeatedly is that given an $\cO$-monoidal left adjoint $F$, its right adjoint $G$ is canonically laxly $\cO$-monoidal \cite[Cor. 7.3.2.7]{HA}.
Conversely, given a laxly $\cO$-monoidal right adjoint, it is merely a condition for its left adjoint to be $\cO$-monoidal \cite[Cor. 7.3.2.12]{HA}.
Moreover, such an adjunction determines an adjunction on $\cO$-algebra objects \cite[Rem 7.3.2.13]{HA}, whose adjoints both commute with the forgetful functors\footnote{Indeed, such data define an adjunction in the $(\infty,2)$-category $\Alg_\cO(\Cat_\infty)^\lax$, to which we may apply $\eHom_{\Alg_\cO(\Cat_\infty)}(\cO \leftarrow \cO_\Triv,-)$.}, i.e., defines a morphism of adjunction:
\begin{equation}\label{eq:adj-square}
    \begin{tikzcd}
        \Alg_{\cO}(\cC)\arrow[r,shift left=.5ex,"F_{\cO}"] \ar[d] 
        &
        \underline{}\Alg_{\cO}(\cD) \arrow[l,shift left=.5ex,"G_{\cO}"] \ar[d] \\ 
        \underline{\cC}\arrow[r,shift left=.5ex,"\underline{F}"]
        &
        \underline{\cD} \arrow[l,shift left=.5ex,"\underline{G}"].
    \end{tikzcd}
\end{equation}
If $\cC, \cD$ are symmetric monoidal $\infty$-categories, and $F \colon \cC \to \cD$ is a symmetric monoidal left adjoint, then the symmetric monoidal functor $\Alg_{\cO}(F) \colon \Alg_{\cO}(\cC) \to \Alg_{\cO}(\cD)$ is a symmetric monoidal left adjoint.

\subsection{Module \texorpdfstring{$\infty$}{infty}-categories}\label{subsec:module-cats}
\subsubsection{Left, right, and bimodules for associative algebras} \label{subsubsec:left-modules}

We briefly review the theory of module $\infty$-categories from \cite[\S~4]{HA}. 
There are $\infty$-operads  $\mathrm{LM}, \mathrm{RM}$ and $\mathrm{BM}$ parametrizing pairs $(a, {}_{a}m)$ of an associative algebra object $a$ along with a left module $m$ over it, pairs $(a, m_{a})$ of an associative algebra with a right module, and triples $(a, b, {}_{a}m_b)$ of associative algebras $a, b$ and a bimodule $m$ between them, respectively. Forgetting the module $m$ gives rise to operad maps $\EE_1 \to \mathrm{LM}$, $\EE_1 \to \mathrm{RM}$ and $\EE_1 \sqcup \EE_1 \to \mathrm{BM}$.  

An $\mathrm{LM}$-monoidal $\infty$-category $\cM$ amounts to a monoidal $\infty$-category $\cM_a$ together with a left module $\infty$-category $\cM_m$ over it. An $\mathrm{LM}$-algebra in such an $\mathrm{LM}$-monoidal $\infty$-category therefore consists of an $\EE_1$-algebra in $\cM_a$ together with a left module in $\cM_m$. 
We denote the $\infty$-category of \bit{$\mathrm{LM}$-algebras in $\cM$} by $\LMod(\cM)$. Furthermore, pre-composing with the map $\EE_1 \to \mathrm{LM}$ induces a functor  $\LMod(\cM) \to \Alg(\cM_a)$. For an algebra $A \in \Alg(\cM_a)$, we denote the fiber of this functor at $A$ by $\LMod_A(\cM_m)$, and call it the $\infty$-category of \emph{$A$-modules} in $\cM_m$. If $\cM$ is a \textit{presentably} $\mathrm{LM}$-monoidal $\infty$-category, then $\LMod_A(\cM_m)$ is also presentable \cite[Cor. 4.2.3.7]{HA}. 

Any monoidal $\infty$-category $\cC$ can be considered a $\mathrm{LM}$-monoidal $\infty$-category by setting $\cM_m = \cM_a$ with its canonical left module action. In this case, $\LMod_A(\cC)$ carries a canonical right action by $\cC$\cite[\S~4.3.2]{HA}, if $\cC$ is further presentably monoidal this exhibits $\LMod_A(\cC)$ as an object in $\RMod_\cC(\PrL)$.

We use analogous notation for $\mathrm{RM}$ and $\mathrm{BM}$-monoidal $\infty$-categories and algebras; 
for example, in a $\mathrm{BM}$-monoidal $\infty$-category consisting of two monoidal $\infty$-categories $\cM_a$ and $\cM_b$ and an $\cM_a$--$\cM_b$ bimodule $\infty$-category $\cM_m$, the fiber of $\BMod(\cM) \to \Alg_{\EE_1}(\cM_a) \times \Alg_{\EE_1}(\cM_b)$ at an algebra $A$ and $B$ is denoted ${}_{A}\BMod_{B}(\cM)$ and is presentable if $\cM$ is a presentably $\mathrm{BM}$-monoidal $\infty$-category. 

\subsubsection{The bar construction and relative tensor product of bimodules}
We describe the relative tensor product of bimodules in the presentably monoidal case, though the theory works much more generally.

Given bimodules ${}_{A}M_B$ and ${}_{B}N_C$ between algebras $A,B,C$ in a monoidal $\infty$-category $\cM$ the \emph{bar construction} defines a simplicial object $\Bar(M,B,N)\in {}_{A} \BMod_C(\cC)$ with $\Bar(M, B, N)_n \coloneqq M \otimes B^{\otimes n} N$ with face and degeneracy maps given by multiplication, actions and the unit. The \emph{relative tensor product} (\cite[Prop. 4.4.2.14]{HA}) $M\otimes_B N$ is defined as the geometric realization of $\Bar(M, B, N)$  in ${}_{A} \BMod_C(\cC)$. 
If $\cM$ is a presentably monoidal $\infty$-category, this defines a functor in $\PrL$: 
\[
   -\otimes_B - \colon {}_A\BMod_B(\cC) \otimes {}_B\Mod_C(\cC) \to {}_A\BMod_C(\cC).
\] 
For $A=B=C$, this induces a presentably monoidal structure on ${}_{A}\BMod_A(\cC)$.

\subsubsection{Module categories of commutative algebras}\label{subsubsec:modules-of-sym-monoidal-algebra}
If $\cC$ is a symmetric monoidal $\infty$-category and $A$ a commutative algebra in $\cC$, there is an equivalence $\LMod_A(\cC) \simeq \RMod_A(\cC)$ treating the given left action as a right action and vice versa. For this reason, we denote the $\infty$-category of modules of a commutative algebra simply by $\Mod_A(\cC)$ and refer to it as the $\infty$-category of \emph{$A$-modules}.
Moreover,  treating an $A$-module as a bimodule induces a functor  $\Mod_A(\cC) \to {}_A\BMod_A(\cC)$. If $\cC$ is presentably symmetric monoidal,  the relativ tensor product $-\otimes_A -$ defines a presentably monoidal structure on ${}_A\BMod_A(\cC)$. This lifts to a \emph{presentably symmetric monoidal} structure on $\Mod_A$ \cite[Thm. 4.5.2.1]{HA}.

\subsection{Enriched \texorpdfstring{$\infty$}{infinity}-categories}
\label{subsec:enriched-infty-cats}

Our work makes crucial use of the theory of enriched $\infty$-categories of \cite{GH13}, which we briefly review here. Given a monoidal $\infty$-category $\VV$, we write $\Cat[\VV]$ for the (large) $\infty$-category of (small) $\VV$-enriched $\infty$-categories. Similarly, let $\hatCat[\VV]$ be the (huge) $\infty$-category of $\VV$-enriched $\infty$-categories with large spaces of objects. 

We note from the outset that this formalism enjoys a convenient \textit{univalence} property: the equivalences in $\Cat[\VV]$ are precisely the (enrichedly) fully faithful and surjective functors. This may be contrasted with the classical notion of an ``equivalence of categories'', which is \textit{not} generally an isomorphism in the ordinary category of ordinary categories since it is not generally an isomorphism on objects. Of course, achieving this univalence requires an additional step, which is itself the imposition of a univalence condition.\footnote{This theory is effectively a generalization of the theory of \textit{complete Segal spaces} \cite{rezk2001model}, with the univalence being obtained by restricting to those from \textit{all} Segal spaces. In particular, using the terminology introduced just below, an ordinary category in the classical sense (defined in terms of a set of objects) is equivalently a categorical $\Set$-algebra whose space of objects is discrete (or equivalently a Segal set). By contrast, an object of $\Cat[\Set] \subset \Cat[\Spaces] \simeq \Cat_\infty$ -- equivalently, an $\infty$-category whose hom-spaces are all discrete -- merely has a 1-truncated $\infty$-groupoid of objects, a surjection to which from a set determines a presentation thereof as an ordinary category in the classical sense.}

Given a monoidal $\infty$-category $\VV$, a \bit{categorical $\VV$-algebra} $\cC$ with space of objects $X \in \Spaces$ heuristically consists of a functor $X^{\times 2} \xra{\eHom_\cC(-,-)} \VV$ specifying hom-objects as well as an associative and unital composition operation. 
These assemble into an $\infty$-category $\AlgCat[\VV]$. Given a categorical $\VV$-algebra $\cC \in \AlgCat[\VV]$ we generally write $\iota_0 \cC \in \Spaces$ for its space of objects, and the functor $\AlgCat[\VV] \xra{\iota_0} \Spaces$ is a Cartesian fibration (with Cartesian monodromy functors given by pulling back the hom-objects along a map of spaces). If $\VV$ is in fact symmetric monoidal, then $\AlgCat[\VV]$ admits a symmetric monoidal structure as well  \cite[Cor. 5.7.12]{GH13}, with $\iota_0 (\cC \otimes \cD) \simeq (\iota_0 \cC) \times (\iota_0 \cD)$ and $\eHom_{\cC \otimes \cD}((c,d) , (c',d')) \simeq \eHom_\cC(c,c') \otimes \eHom_\cD(d,d')$.

Now, given a categorical $\VV$-algebra $\cC \in \AlgCat[\VV]$ we can extract a space $\cC^\simeq \in \Spaces$ of equivalences (with respect to its internal category theory), and this comes equipped with a morphism $\iota_0 \cC \ra \cC^\simeq$ from its space of objects (which heuristically sends each object to its identity morphism). A morphism $\cC \to \cD$ in $\AlgCat[\VV]$ is \bit{surjective on objects} if the induced map $\cC^{\simeq} \to \cD^{\simeq}$ is surjective (i.e. surjective on $\pi_0$).  A morphism $F \colon \cC \to \cD$ in $\AlgCat[\VV]$ is \bit{fully faithful} if for any two objects $c, d \in \cC$, the induced map $\eHom_{\cC}(c,d) \to \eHom_{\cD}(Fc, Fd)$ in $\VV$  is an isomorphism in $\cD$.

We say that $\cC$ is \bit{univalent} if the morphism $\iota_0 \cC \ra \cC^\simeq$ is an equivalence; in essence, this is the condition that its internally- and externally-defined spaces of objects coincide. Finally, a \bit{$\VV$-enriched $\infty$-category} is a univalent categorical $\VV$-algebra. These define a full subcategory $\Cat[\VV] \subseteq \AlgCat[\VV]$.
Furthermore, the inclusion has a left adjoint (i.e. a reflective localization), which we may refer to as \bit{univalent completion}, which exhibits $\Cat[\VV]$ as the \emph{localization of $\AlgCat[\VV]$ with respect to fully faithful and essentially surjective functors} \cite[Thm. 2.4.11]{GH13}.

If we assume that $\VV$ is symmetric monoidal, then the reflective localization is compatible with the symmetric monoidal structure on $\AlgCat[\VV]$ in the sense of \Cref{subsubsection:loczns-of-O-monoidal-cats}.
It follows that $\Cat[\VV]$ also inherits a symmetric monoidal structure, given by taking the tensor product in categorical $\VV$-algebras and then univalently completing the result.\footnote{Beware that the functor $\Cat[\VV] \xra{\iota_0} \Spaces$ is \textit{not} generally symmetric monoidal: for $\cC,\cD \in \Cat[\VV]$, the morphism $\iota_0(\cC \otimes^{\AlgCat[\VV]} \cD) \ra \iota_0(L(\cC \otimes^{\AlgCat[\VV]} \cD)) \eqqcolon \iota_0(\cC \otimes^{\Cat[\VV]} \cD)$ is generally \textit{not} an equivalence. A simple example is given by taking $\VV = \Ab$ to be the category of abelian groups, and taking $\cC = B R$ and $\cD = B S$ to be the one-object categorical $\Ab$-algebras associated to associative rings $R$ and $S$. We have $\iota_0 ( L(B R)) \simeq B(R^\times)$, and $B R \otimes^{\AlgCat[\Ab]} B S \coloneqq B (R \otimes S)$, but the canonical map $R^\times \times S^\times \ra (R \otimes S)^\times$ is generally not an isomorphism of sets. (For instance, if $R$ and $S$ are fields of different characteristic, then $R \otimes S = 0$.)} By the same argument,  $\largecat[\VV]$ also inherits a symmetric monoidal structure.

If $\VV$ is presentably monoidal, then both $\infty$-categories $\Cat[\VV]$ and $\AlgCat[\VV]$ are presentable \cite[Prop. 5.7.8]{GH13}. 
If $\VV$ is furthermore presentably symmetric monoidal, then so is $\Cat[\VV]$~\cite[Prop. 5.7.16]{GH13}. This assembles into a functor $\Cat[-] \colon \CAlg(\PrL) \to \CAlg(\PrL)$.  

If $\VV$ is presentably monoidal, there also exists a  \textit{categorical suspension} functor $\VV \xra{\Sigma[-]} \AlgCat[\VV]$ \cite[Def. 4.3.21]{GH13}, which is characterized by the universal property that morphisms $\Sigma[V] \ra \cC$ are equivalent to a pair of objects $c,d \in \cC$ and a morphism $V \ra \eHom_\cC(c,d)$ in $\VV$.\footnote{In general, a categorical $\VV$-algebra with space of objects $X \in \Spaces$ has an underlying \emph{$\VV$-graph}, i.e.\! a functor $X^{\times 2} \ra \VV$ (which encodes the hom-objects but not composition), and this forgetful functor has a left adjoint free functor (using e.g.\! that $\VV$ is presentably monoidal). Then, given an object $V \in \VV$, the categorical algebra $\Sigma[V] \in \AlgCat[\VV]$ is free on the $\VV$-graph with space of objects $\{0,1\} \in \Set \subset \Spaces$ given by
\[
(i,j)
\longmapsto
\left\{
\begin{array}{ll}
V~,
&
(i,j) = (0,1)
\\
\emptyset_\VV~,
&
\text{otherwise}
\end{array}
\right.
~.
\]
} We also simply write $\Sigma[-]$ for the composite $\VV \to \AlgCat[\VV] \to \Cat[\VV]$, which has the same universal property in $\Cat[\VV]$.

\section{Factorization systems for enriched \texorpdfstring{$\infty$}{infinity}-categories}\label{app:f-s-for-enriched-cats}

As the terminology suggests, a \textit{factorization system} on an $\infty$-category gives a functorial way of factoring its morphisms. As a basic example, every morphism $X \xra{f} Y$ in the $\infty$-category of spaces (and in particular in the category of sets) admits a factorization
\[ \begin{tikzcd}
X
\arrow{rr}{f}
\arrow[dashed, two heads]{rd}
&
&
Y
\\
&
\Fact(f)
\arrow[dashed, hook]{ru}
\end{tikzcd} \]
as a surjection followed by a monomorphism. In fact, every morphism of spaces admits a \textit{unique} such factorization. This uniqueness persists in the case of a general factorization system, arising from a certain orthogonality relation that is required of the two factors.\footnote{As illustrated in \Cref{ex:fs-on-BNx}, this orthogonality relation may be seen as a sort of generalized coprimality requirement on the factors.}

In this appendix, given a presentably monoidal $\infty$-category $\VV$, we prove as \Cref{thm:fs_and_enriched_cat} that a factorization system on $\VV$ that is compatible with its monoidal structure determines a factorization system on the $\infty$-categories $\Cat[\VV]$ of $\VV$-enriched $\infty$-categories. In fact, we prove a more general result as \Cref{thm:fs-and-alg}: if $\VV$ is presentably $\cO$-monoidal, under mild hypotheses we obtain a factorization system on the $\infty$-category of algebras over \textit{any} $\infty$-operad $\cA$ equipped with a morphism $\cA \ra \cO$. We use \Cref{thm:fs_and_enriched_cat} to obtain factorization systems on $(\infty,k)$-categories (\Cref{thm:nsurj-nfaithful-fs-for-infty-k-cats}), on enriched $(\infty, 2)$-categories (\Cref{cor:surj-dominant-faithful-fs-on-Cat-add}), and on $\infty$-operads (\Cref{prop:fs-for-infty-opds}).
Along the way, we establish a number of useful results concerning factorization systems, some of which are also used in the main body of the paper.

We begin in \Cref{subsec:recollections-on-fs} by recalling some basic definitions and properties of factorization systems, including some convenient features that result from specializing to presentable $\infty$-categories. We then proceed in \Cref{subsec:induced-fs} to establish a number of ways of obtaining new factorization sytems from old ones. We then prove our two main results \Cref{thm:fs-and-alg} (concerning algebras over $\infty$-operads) in \Cref{subsec:fs-for-algebras-over-opds} and \Cref{thm:fs_and_enriched_cat} (concerning enriched $\infty$-categories) in \Cref{subsec:fs-for-Vcats}.

\subsection{Recollections on factorization systems}
\label{subsec:recollections-on-fs}

\subsubsection{Basics of factorization systems}

\begin{definition}[{\cite[Def. 5.2.8.1]{HTT}}]
\label{def:orthogonal-morphisms}
Given morphisms $a \xra{l} b$ and $c \xra{r} d$ in an $\infty$-category, we say that $l$ is \bit{left orthogonal} to $r$ or that $r$ is \bit{right orthogonal} to $l$ if for any solid commutative square
        \begin{equation}
        \label{eq:squarefs}
            \begin{tikzcd}
                a \arrow[r] \arrow{d}[swap]{l} & c \arrow[d, "r"] \\
                b \arrow[r] \arrow[ur, dashed] & d
            \end{tikzcd}
        \end{equation}
        the space of dashed lifts $b \to c$ is contractible. In this situation, we may write $l \bot r$. More broadly, given classes $\cL$ and $\cR$ of morphisms in an $\infty$-category, we write $\cL \bot \cR$ to indicate that $l \bot r$ for every $l \in \cL$ and every $r \in \cR$.
\end{definition}

\begin{example}
A morphism $f$ in an $\infty$-category satisfies the relation $f \bot f$ if and only if it is an equivalence.
\end{example}

\begin{observation}\label{obs:adjunctions-left-right}
Given an adjunction
\[
        \begin{tikzcd}[column sep=1.5cm]
        \cC
        \arrow[yshift=0.9ex]{r}{F}
        \arrow[leftarrow, yshift=-0.9ex]{r}[yshift=-0.2ex]{\bot}[swap]{G}
        &
        \cD
        \end{tikzcd}
    \]
and morphisms $f$ and $g$ in $\cC$ and $\cD$ respectively, the orthogonality relations $f \bot G(g)$ and $F(f) \bot g$ are equivalent. We use this fact without further comment.
\end{observation}

\begin{notation}
Given a class $S$ of morphisms in an $\infty$-category, we write $S^\perp$ (resp.\! $^\perp S$) for the class of morphisms that are right (resp.\! left) orthogonal to those in $S$.
\end{notation}

\begin{definition}[{\cite[Def. 5.2.8.8]{HTT}}]
\label{defn-fs}
A \bit{factorization system} on an $\infty$-category $\cC$ is a pair $(\cL, \cR)$ of classes of morphisms in $\cC$ satisfying the following conditions.
    \begin{enumerate}
    \item\label{item-defn-fs:retracts}
    
    The classes $\cL$ and $\cR$ are stable under the formation of retracts (in $\Fun([1],\cC)$).
    
    \item\label{item-defn-fs:orthogonality}
   
    We have the orthogonality relation $\cL \perp \cR$.
    
    \item\label{item-defn-fs:factorization}
        \label{defn-fs-factorization}
        Every morphism $c \xra{f} d$ in $\cC$ admits a factorization
        \[ \begin{tikzcd}
        c
        \arrow{rr}{f}
        \arrow[dashed]{rd}[sloped, swap]{l}
        &
        &
        d
        \\
        &
        e
        \arrow[dashed]{ru}[swap, sloped]{r}
        \end{tikzcd} \]
        with $l \in \cL$ and $r \in \cR$.
    \end{enumerate}
We respectively write $\Cat_\infty^{\fs,\cL}$, $\Cat_\infty^{\fs,\cR}$, and $\Cat_\infty^{\fs,\cL,\cR}$ for the $\infty$-categories of $\infty$-categories equipped with factorization systems, in which a morphism is a functor that respectively preserves the left class, the right class, or both classes.
\end{definition}

\begin{notation}
To simplify our notation, we take the following conventions when studying a class $S$ of morphisms in an $\infty$-category $\cC$.
\begin{enumerate}

\item

Assuming that $S$ consists of precisely the morphisms in a subcategory of $\cC$ (e.g.\! both classes in a factorization system on $\cC$), we simply write $S$ to denote this subcategory.

\item

Assuming that $S$ is stable under homotopy (e.g.\! both classes in a factorization system on $\cC$), we also simply write $S$ to denote the full subcategory of $\Fun([1],\cC)$ on the morphisms in $S$.

\item

We simply write $\cC^\simeq$ for the class of equivalences in $\cC$, and we simply write $\cC$ for the class of all morphisms in $\cC$.

\item
For any object $c \in \cC$, we write $\CRoverd{\cC}{S}{c} \subseteq \cC_{/c}$ for the full subcategory on those objects $(d \ra c) \in \cC_{/c}$ that lie in $S$ (when considered as morphisms in $\cC$). In the special case that $c \simeq \pt_\cC$ is terminal, we simply write $\cC^S \coloneqq \CRoverd{\cC}{S}{\pt_{\cC}}$.
\end{enumerate}

\end{notation}

\begin{example}
\label{ex:trivial-fs}
For any $\infty$-category $\cC$, the pairs $(\cC^\simeq,\cC)$ and $(\cC,\cC^\simeq)$ define factorization systems on $\cC$.
\end{example}

\begin{example}
\label{ex:fs-on-BNx}
Let $\NN^\times \coloneqq \{1, 2, 3, \ldots \}^\times$ denote the (commutative) monoid of natural numbers under multiplication. Given two elements $s,t \in \NN^\times$, their corresponding morphisms in $B \NN^\times$ satisfy $s \bot t$ (and thereafter $t \bot s$) if and only if $s$ and $t$ are coprime. From here, it is easy to check that e.g.\! the pairs (powers of 2, odds) and (odds, powers of 2) define factorization systems on $B \NN^\times$. More generally, if $\{2, 3, 5, \ldots \} = P_1 \sqcup P_2$ denotes a two-element partition of the set of prime numbers, then
\[
\text{(powers of elements of $P_1$, powers of elements of $P_2$)}
\]
determines a factorization system on $B \NN^\times$, and moreover every factorization system on $B \NN^\times$ arises in this way.
\end{example}

\begin{observation}
A factorization system $(\cL,\cR)$ on an $\infty$-category is completely determined by either $\cL$ or $\cR$ (since $\cR = \cL^\bot$ and $\cL = {}^\bot \cR$). We use this fact without further comment.
\end{observation}

\begin{observation}
\label{obs:factorizations-from-fs-are-unique}
By \cite[Prop. 5.2.8.17]{HTT}, the factorization in part \eqref{defn-fs-factorization} of \Cref{defn-fs} is unique. We often use this fact without further comment.
\end{observation}

\begin{notation}
\label{notn:Fact-for-factorizn-from-fs}
Justified by \Cref{obs:factorizations-from-fs-are-unique}, given a morphism $c \xra{f} d$ in an $\infty$-category $\cC$ equipped with a factorization system $(\cL,\cR)$, we write $\Fact(f) \coloneqq \Fact_{(\cL,\cR)}(f) \in \cC$ for the unique object through which $f$ factors via the factorization system.
\end{notation}

We introduce the following notion for future use.

\begin{definition}
\label{def:fs-compatible-with-O-monoidal-str}
Let $\cO$ be an $\infty$-operad and let $\cC$ be an $\cO$-monoidal $\infty$-category. Suppose that for every color $X \in \underline{\cO}$, the $\infty$-category $\cC_X$ of $X$-colored objects in $\cC$ is equipped with a factorization system $(\cL_X,\cR_X)$. We say that the $\cO$-monoidal structure of $\cC$ is \bit{compatible} with these factorization systems if for every $n \geq 0$ and every $n$-ary operation $(X_1,\ldots,X_n) \to X$ in $\cO$, the corresponding functor $\cC_{X_1} \times \cdots \times \cC_{X_n} \to \cC_X$ carries morphisms in $\cL_{X_1} \times \cdots \times \cL_{X_n}$ to morphisms in $\cL_X$.\footnote{The data of an $\cO$-monoidal $\infty$-category equipped with compatible factorization systems is equivalent to that of a functor $\cO \ra \Cat_\infty^{\fs,\cL}$ satisfying certain Segal conditions (as in \Cref{subsubsec:O-monoidal-infty-cats}). (Observe that $\Cat_\infty^{\fs,\cL}$ admits products, which are defined in the evident way.)} 
\end{definition}

\subsubsection{Factorization systems on presentable $\infty$-categories}

We now discuss factorization systems of small generation on presentable $\infty$-categories. We then proceed to make some further observations about factorization systems that admit specializations when applied to those of small generation.

For motivation, observe that both classes of a factorization system necessarily contain all equivalences. As a result, both classes of a factorization system on a large $\infty$-category must be large. However, on a \textit{presentable} $\infty$-category one can define a factorization system in terms of a small set of morphisms (which then generate the left class), as we now recall.

\begin{definition}[{\cite[Def. 5.5.5.1]{HTT}}]
    \label{def:saturated}
We say that a class of morphisms $S$ in an $\infty$-category $\cC$ is \bit{saturated} if it satisfies the following conditions.
    \begin{enumerate}
        
        \item\label{item-def:saturated-wide-subcat}
        
        The class $S$ contains all equivalences and is closed under composition.\footnote{Said differently, the morphisms in $S$ are precisely those that lie in a wide subcategory of $\cC$.}
        
        \item\label{item-def:saturated-colimits}
        
        The full subcategory $S \subseteq \Fun([1],\cC)$ is closed under (small) colimits.
        
        \item\label{item-def:saturated-cobase-change}
        
        The class $S$ is stable under cobase change.
        
    \end{enumerate}
\end{definition}

\begin{proposition}[{\cite[Prop. 5.5.5.7]{HTT}}]
\label{prop:fact-system-of-small-generation-on-presentable}
Fix a presentable $\infty$-category $\cC$ and a small set of morphisms $S$ in $\cC$. Then, there exists a factorization system $(\cL,\cR)$ on $\cC$ with $\cR = S^\perp$. Moreover, $\cL$ is the smallest saturated class of morphisms in $\cC$ that contains $S$. 
\qed 
\end{proposition}

\begin{definition}
\label{def:small-generation-and-Pr-fs}
In the context of \Cref{prop:fact-system-of-small-generation-on-presentable}, we say that the factorization system $(\cL,\cR)$ (or simply the left class $\cL$) is of \bit{small generation}, or more specifically that it is \bit{generated} by $S$. Moreover, we may write $\overline{S}$ for $\cL$. We define the subcategories
\[
{\Pr}^{L,\fs,\cL}
\subset
\hat{\Cat}_\infty^{\fs,\cL}
\qquad
\text{and}
\qquad
{\Pr}^{R,\fs,\cR}
\subset
\hat{\Cat}_\infty^{\fs,\cR}
\]
to be those on the presentable $\infty$-categories whose factorization systems are of small generation, whose morphisms are respectively required to be left or right adjoints (in addition to preserving the indicated class of the factorization system).
\end{definition}

    \begin{example}
    \label{ex:nconn-ntrunc-fs-on-Spaces}
    Fix any integer $n \geq -2$. By \cite[Ex. 5.2.8.16]{HTT}, the $\infty$-category $\Spaces$ of spaces admits a factorization system ($n$-connected, $n$-truncated),\footnote{These notions are recalled in \Cref{subsec:trun-and-conn}.} which is generated by the singleton $\{S^{n+1} \ra \pt\}$.\footnote{In the case that $n = -2$ this recovers $(\Spaces,\Spaces^\simeq)$, and in the case that $n = -1$ this recovers (surjections, monomorphisms).}
    \end{example}

\begin{observation}
Given an adjunction between $\infty$-categories equipped with factorization systems, the left adjoint preserves the left class if and only if the right adjoint preserves the right class. It follows that passing to adjoints determines an equivalence $\Pr^{L,\fs,\cL} \simeq (\Pr^{R,\fs,\cR})^\op$. We use these facts without further comment.
\end{observation}

\begin{lemma}
\label{lem:presentablecompatible}
 If $\cC$ is a presentably $\cO$-monoidal category for a small operad $\cO$ and suppose that for every color $X\in \underline{\cO}$, the presentable $\infty$-category $\cC_{X}$ is equipped with a factorization system $(\cL_X, \cR_X)$ generated by a set $S_X$. Then, the factorization systems are compatible with the $\cO$-monoidal structure if and only if for every operation $(X_1, \ldots, X_n) \to X$ in $\cO$, the corresponding functor $\cC_{X_1} \times \cdots \cC_{X_n} \to \cC_{X}$ carries morphisms in $S_{X_1} \cdots \times \cdots S_{X_n}$ to morphisms in $\cL_X$. 
\end{lemma}
\begin{proof}
Assume that  $\cC_{X_1} \times \cdots \cC_{X_n} \to \cC_{X}$ carries morphisms in $S_{X_1} \cdots \times \cdots S_{X_n}$ to morphisms in $\cL_X$.  By assumption, the functor $\cC_{X_1} \times \cdots \times \cC_{X_n} \to \cC_{X}$ preserves small colimits separately in all variables. Since for every $Y \in \underline{\cO}$, the class of morphisms $\cL_Y$ is by  \cref{prop:fact-system-of-small-generation-on-presentable} the smallest saturated class of morphisms in $\cC_Y$ that contains $S_Y$, the functor  $\cC_{X_1} \times \cdots \times \cC_{X_n} \to \cC_{X}$ therefore also carries morphisms in $\cL_{X_1} \times \cdots \times \cL_{X_n}$ to morphisms in $\cL_X$.
\end{proof}

\begin{observation}
\label{obs:fs-on-undercat-and-overcat}
Fix an $\infty$-category $\cC$ with a factorization system $(\cL,\cR)$.
\begin{enumerate}

\item\label{item-obs:fs-on-undercat-and-overcat-general-case}

For any object $c \in \cC$, we obtain factorization systems on both $\cC_{c/}$ and $\cC_{/c}$ in which both classes are pulled back from $\cC$ via the respective forgetful functors.

\item

Suppose that $\cC$ is presentable and that $(\cL,\cR)$ is of small generation. Then, the factorization systems of part \eqref{item-obs:fs-on-undercat-and-overcat-general-case} are both of small generation as well. Specifically, if $S$ denotes a set of morphisms in $\cC$ that generates $(\cL,\cR)$, then they are respectively generated by the evident (small) spaces of morphisms indexed by
\[
\bigsqcup_{(a \ra b) \in S} \hom_\cC(c,a)
\qquad
\text{and}
\qquad
\bigsqcup_{(a \ra b) \in S} \hom_\cC(b,c)
~.
\]

\end{enumerate}
\end{observation}

\begin{observation}
\label{obs:reflective-localizn-from-fs}
Fix an $\infty$-category $\cC$ with a factorization system $(\cL,\cR)$.
\begin{enumerate}

\item\label{item-obs:reflective-localizn-from-fs}

Assume that $\cC$ contains a terminal object. Then, there exists a left adjoint
\begin{equation}
\label{eq:reflective-localization-from-fs}
\begin{tikzcd}[column sep=1.5cm]
\cC
\arrow[yshift=0.9ex, dashed]{r}
\arrow[hookleftarrow, yshift=-0.9ex]{r}[yshift=-0.2ex]{\bot}
&
\cC^\cR
\end{tikzcd}
\end{equation}
to the fully faithful inclusion, which is given by the formula $c \mapsto \Fact(c \to \pt_\cC)$. Moreover, the right adjoint is the inclusion of the $\cL$-local objects, and hence the left adjoint exhibits $\cC^\cR$ as the localization $\cC[\cL^{-1}]$.\footnote{However, beware that a morphism in $\cC$ may be sent to an equivalence in $\cC^\cR$ even if it is not in $\cL$.}

\item\label{item-obs:reflective-localizn-from-fs-presentable}

Assume that $\cC$ is presentable and that $(\cL,\cR)$ is generated by a set $S$ of morphisms in $\cC$. Then, the reflective localization \eqref{eq:reflective-localization-from-fs} also identifies $\cC^\cR$ with the (accessible) localization $\cC[S^{-1}]$.

\item\label{item-obs:symmetric-monoidality-of-reflection-localizn}

Furthermore, if $\cC$ has a  symmetric monoidal structure compatible with the factorization system and which has the terminal object as monoidal unit, then it  induces a symmetric monoidal structure on $\cC^\cR$, for which the left adjoint $ \cC \to \cC^R$ is symmetric monoidal.
\item
Given a morphism $(\cC_0,(\cL_0,\cR_0)) \xra{F} (\cC_1,(\cL_1,\cR_1))$ in $\Cat_\infty^{\fs,\cL,\cR}$ in which both $\cC_0$ and $\cC_1$ admit terminal objects and $F(\pt_{\cC_0}) \simeq \pt_{\cC_1}$, the reflective localizations of part \eqref{item-obs:reflective-localizn-from-fs} assemble into a morphism
\[
\begin{tikzcd}[column sep=1.5cm]
\cC_0
\arrow[yshift=0.9ex]{r}
\arrow[hookleftarrow, yshift=-0.9ex]{r}[yshift=-0.2ex]{\bot}
\arrow{d}[swap]{F}
&
\cC_0^{\cR_0}
\arrow{d}{F}
\\
\cC_1
\arrow[yshift=0.9ex]{r}
\arrow[hookleftarrow, yshift=-0.9ex]{r}[yshift=-0.2ex]{\bot}
&
\cC_1^{\cR_1}
\end{tikzcd}
\]
of adjunctions.

\end{enumerate}
\end{observation}

\subsection{Induced factorization systems}
\label{subsec:induced-fs}

In this subsection, we record an assortment of useful results that allow us to obtain new factorization systems from old ones, in roughly increasing order of complexity. Specifically, we give sufficient conditions for inducing factorization systems on subcategories (\Cref{obs:restricted-fs-on-subcat}), on $\infty$-categories of functors (\Cref{lem:fs-and-functor-cat}), on the total $\infty$-categories of Cartesian fibrations (\Cref{lem:fs-and-Cart-fibn}), through reflective localizations (\Cref{lem:fs-on-reflective-loc}), and through monadic adjunctions (\Cref{lem:fs-and-monadic}).

We begin by recording necessary and sufficient conditions for a factorization system to restrict to a subcategory.

\begin{observation}
\label{obs:restricted-fs-on-subcat}
Let $\cC$ be an $\infty$-category equipped with a factorization system $(\cL,\cR)$, let $\cC_0 \subseteq \cC$ be a subcategory, and let us write $\cL_0 \coloneqq \cL \cap \cC_0$ and $\cR_0 \coloneqq \cR \cap \cC_0$. Then, the pair $(\cL_0,\cR_0)$ forms a factorization system on $\cC_0$ if and only if the following conditions are satisfied.
\begin{enumerate}

\item\label{item-obs:restricted-fs-on-subcat-lift-lies-in-subcat}

For any solid commutative square \eqref{eq:squarefs} in $\cC_0$ with $l \in \cL$ and $r \in \cR$, the unique lift in $\cC$ (guaranteed by the orthogonality relation $\cL \bot \cR$) also lies in $\cC_0$.

\item\label{item-obs:restricted-fs-on-subcat-factorization}

For any morphism $c \xra{f} d$ in $\cC_0$, both morphisms in the factorization $c \ra \Fact_{(\cL,\cR)}(f) \ra d$ also lie in $\cC_0$.

\end{enumerate}

In particular, if $\cC_0$ is a full subcategory of $\cC$, then $(\cL_0,\cR_0)$ forms a factorization system if and only if $\Fact_{(\cL,\cR)}(f)$ lies in $\cC_0$ for any morphism $f$ in $\cC_0$.

\end{observation}

\begin{lemma}
\label{lem:restricted-fs-on-subcat}
Let $\cC$ be an $\infty$-category equipped with a factorization system $(\cL,\cR)$, let $\cC_0 \subseteq \cC$ be a subcategory, and suppose that $\cL_0 \coloneqq \cL \cap \cC_0$ and $\cR_0 \coloneqq \cR \cap \cC_0$ fulfill conditions (1) and (2) of \cref{obs:restricted-fs-on-subcat}, so that they induce a factorization system $(\cL_0, \cR_0)$ on $\cC_0$. Then we have:
\begin{enumerate}
\item If $\cC_0$ and $\cC$ are presentable, $(\cL, \cR)$ is of small generation, and the inclusion $\cC_0 \to \cC$ admits a left adjoint $L\colon \cC \to \cC_0$, then $(\cL_0, \cR_0)$ is of small generation and $L(\cL) \subseteq \cL_0$.
\item If $\cC_0$ and $\cC$ are furthermore presentably $\cO$-monoidal for a small operad $\cO$, the left adjoint $L\colon \cC \to \cC_0$ is $\cO$-monoidal, and $(\cL, \cR)$ is of small generation and compatible with the $\cO$-monoidal structure, then so is $(\cL_0, \cR_0)$. 
\end{enumerate}
\end{lemma}
\begin{proof}
For (1), let $L\colon \cC \to \cC_0$ denote the left adjoint, let $S$ be a generating set for $\cL$ and define $S_0$ to be the set of $\cC_0$-morphisms $S_0 \coloneqq L(S)$. By adjunction, a morphism $f$ in $\cC_0$ is in $S_0^{\bot}$ iff it is in $S^{\bot} = \cR$, and hence that $S_0^{\bot} = \cR \cap \cC_0, $ proving that $(\cL_0, \cR_0)$ is generated by $S_0$. $L(\cL) \subseteq \cL_0$ follows from the adjunction.

For (2), it follows from the proof of (1) that for every $X\in \underline\cO$, the class $(\cL_0)_X$ is generated by $(S_0)_X \coloneqq L_X(S_X)$ where $S_X$ is a generating set for $\cL_X$. Hence, to check that $(\cL_0, \cR_0)$ is compatible with the $\cO$-monoidal structure, it suffices  by \cref{lem:presentablecompatible} to show that for every operation $(X_1, \ldots, X_n) \to X$ in $\cO$, the induced functor $ (\cC_0)_{X_1}\times \cdots \times (\cC_0)_{X_n} \to (\cC_0)_{X}$ carries morphisms in $L_{X_1}(S_{X_1}) \times \cdots \times L_{X_n}(S_{X_n})$ to a morphism in $(\cL_0)_X$. But since $L$ is $\cO$-monoidal, such a family of morphism is carried to the image under $L_X$ of their product in $\cC_X$. Since $\cL$ is compatible with the monoidal structure, that product is in $\cL_X$ and hence the morphisms are carried to a morphism in $L_X(\cL_X) \subseteq (\cL_0)_X$.  \end{proof}

We now show that an $\infty$-category of functors automatically inherits a factorization system from one on the target.

\begin{lemma}
\label{lem:fs-and-functor-cat}
Let $\cC$ be an $\infty$-category equipped with a factorization system $(\cL, \cR)$, and let $\cI$ be a small $\infty$-category.
\begin{enumerate}

\item\label{item-lem:fs-and-functor-cat-the-fs}

The $\infty$-category $\Fun(\cI, \cC)$ admits a factorization system $(\cL^\cI,\cR^\cI)$, in which (as the exponential notation suggests) a natural transformation between functors lies in $\cL^\cI$ (resp.\! $\cR^\cI$) if and only if its components all lie in $\cL$ (resp.\! $\cR$).

\item\label{item-lem:fs-and-functor-cat-generators}

If $\cC$ is presentable and $(\cL,\cR)$ is of small generation, then $\Fun(\cI,\cC)$ is presentable and $(\cL^\cI,\cR^\cI)$ is of small generation.

\item\label{item-lem:fs-and-functor-cat-compatible} If $\cO$ is a small operad, $I$ is $\cO$-monoidal,  $\cC$ is presentably $\cO$-monoidal and $(\cL, \cR)$ is compatible with the $\cO$-monoidal structure on $\cC$, then $(\cL^I, \cR^I)$ is compatible with the Day convolution $\cO$-monoidal structure on $\Fun(I, \cC)$.
\end{enumerate}
\end{lemma}

\begin{proof}
Part \eqref{item-lem:fs-and-functor-cat-the-fs} is a restatement of \cite[Cor. 5.2.8.18]{HTT}.
To prove part \eqref{item-lem:fs-and-functor-cat-generators}, assume that $\cC$ is presentable and let $S$ be a set of morphisms in $\cC$ that generates $\cL$. We note first that $\Fun(\cI,\cC)$ is presentable by \cite[Prop. 5.5.3.6]{HTT}. Now, for each functor $\pt \xra{i} \cI$ (selecting an object of $\cI$) we obtain an adjunction
 \[
\begin{tikzcd}[column sep=1.5cm]
\cC
\arrow[yshift=0.9ex]{r}{i_!}
\arrow[leftarrow, yshift=-0.9ex]{r}[yshift=-0.2ex]{\bot}[swap]{i^*}
&
\Fun(\cI,\cC)
\end{tikzcd}
~.
\]
It follows that $\cR^\cI$ is precisely the right orthogonal to the (small) space of morphisms
\[
S'
\coloneqq
\bigsqcup_{i \in \iota_0 \cI}
\bigsqcup_{f \in S}
i_!(f)
\]
in $\Fun(\cI,\cC)$. From here, \Cref{prop:fact-system-of-small-generation-on-presentable} implies that $(\cL^\cI,\cR^\cI)$ is generated by $S'$ (and in particular that $\cL^\cI$ is the smallest saturated class of morphisms containing $S'$).

For part~\eqref{item-lem:fs-and-functor-cat-compatible}, given an n-ary operation $(X_1, \ldots, X_n) \to X$ in $\cO$, the induced functor $\Fun(I_{X_1}, \cC_{X_1}) \times \cdots \times \Fun(I_{X_n}, \cC_{X_n}) \to \Fun(I_X, \cC_X)$ is computed as the left Kan extension of $I_{X_1} \times \cdots \times I_{X_n} \to \cC_{X_1} \times \cdots \times \cC_{X_n} \to \cC_X$ against $I_{X_1}\times \cdots \times I_{X_n} \to I_X$. As the left class of a factorization system, $\cL$ is closed under colimits in $\cC$. Therefore, the image of natural transformations  $f_i \in \Fun(I_{X_i}, \cC_{X_i})$ which are componentwise in $\cL$ will again be componentwise in $\cL$.
\end{proof}

A simple example of a factorization system arises from a Cartesian fibration $\cE \xra{p} \cB$: the total $\infty$-category $\cE$ admits a factorization system $(\cL,\cR)$ in which $\cL = p^{-1}(\cB^\simeq)$ and $\cR$ consists of the $p$-Cartesian morphisms. This can be generalized as follows.

\begin{lemma}
\label{lem:fs-and-Cart-fibn}
Fix an $\infty$-category $\cB$ and a functor $\cB^\op \xra{F} \hat{\Cat}_\infty^{\fs, \cR}$ (recall \Cref{defn-fs}). For each $b \in \cB$, let us write $(\cL_b,\cR_b)$ for the given factorization system on $F(b)$. Moreover, let us write $\cE \xra{p} \cB$ for the Cartesian fibration associated to $F$.
\begin{enumerate}

\item\label{item:fs-and-Cart-fibn-the-fs}

The $\infty$-category $\cE$ admits a factorization system $(\cL,\cR)$, described as follows.

\begin{enumerate}

\item A morphism $e \xra{\alpha} f$ lies in $\cL$ if and only if the morphism $p(e) \xra{p(\alpha)} p(f)$ in $\cB$ is an equivalence and moreover the morphism $e \ra p(\alpha)^*(f)$ in $\cE_{p(e)} \simeq F(p(e))$ lies in $\cL_{p(e)}$.

\item A morphism $e \xra{\alpha} f$ lies in $\cR$ if and only if the morphism $e \ra p(\alpha)^*(f)$ lies in $\cR_{p(e)}$.

\end{enumerate}

\item\label{item:fs-and-Cart-fibn-small-generation}

Suppose that $\cB$ is small and that $F$ factors through the subcategory $\Pr^{R,\fs,\cR} \subset \hat{\Cat}_\infty^{\fs,\cR}$ (recall \Cref{def:small-generation-and-Pr-fs}). Then, the factorization system of part \eqref{item:fs-and-Cart-fibn-the-fs} is also of small generation. More specifically, if for each $b \in \cB$ the set $S_b$ generates the class $\cL_b$, then the set $S \coloneqq \bigsqcup_{b \in \cB} S_b$ generates the class $\cL$ (considering each $S_b$ as defining a set of morphisms in the fiber $\cE_b \simeq F(b)$).

\end{enumerate}
\end{lemma}

\begin{proof}
Part \eqref{item:fs-and-Cart-fibn-the-fs} is straightforward. Thereafter, for part \eqref{item:fs-and-Cart-fibn-small-generation} it suffices to show that $\cR = S^\bot$ (so that $(\cL,\cR)$ is indeed the factorization system generated by $S$ via \cref{prop:fact-system-of-small-generation-on-presentable}). The containment $\cR \subseteq S^\bot$ follows from the fact that the Cartesian monodromy functors preserve the right classes, while the containment $\cR \supseteq S^\bot$ follows from the explicit description of $\cR$.
\end{proof}

We now provide sufficient conditions for descending a factorization system through a reflective localization.

\begin{lemma}
\label{lem:fs-on-reflective-loc}
Suppose that
\[
\begin{tikzcd}[column sep=1.5cm]
\cC
\arrow[yshift=0.9ex]{r}{L}
\arrow[hookleftarrow, yshift=-0.9ex]{r}[yshift=-0.2ex]{\bot}[swap]{R}
&
\cD
\end{tikzcd}
\]
is a reflective localization and that $(\cL,\cR)$ is a factorization system on $\cC$. Suppose further that $L(\cL)$ is stable under retracts and that $RL(\cR) \subseteq \cR$.
\begin{enumerate}

\item\label{item-lem:fs-on-reflective-loc-the-fs}

The pair $(L(\cL),L(\cR))$ forms a factorization system on $\cD$, in which the factorization of a morphism $d \xra{f} d'$ is given by the lower composite in the commutative diagram
\begin{equation}
\label{eq:factorizn-in-fs-on-reflective-localization}
\begin{tikzcd}
d
\arrow{rr}{f}
&
&
d'
\\
LR(d)
\arrow{u}[sloped, anchor=south]{\sim}
\arrow{rr}{LR(f)}
\arrow{rd}
&
&
LR(d')
\arrow{u}[sloped, anchor=north]{\sim}
\\
&
L(\Fact(R(f))
\arrow{ru}
\end{tikzcd}~.
\end{equation}
Moreover, considering $\cD$ as a full subcategory of $\cC$ (via $R$), the right class $L(\cR)$ is intersected from $\cR$ (i.e.\! we have $RL(\cR) = \cR \cap R(\cD)$).

\item\label{item-lem:fs-on-reflective-loc-generation}

Suppose further that $\cC$ and $\cD$ are presentable and that $(\cL,\cR)$ is generated by a set $S$ of morphisms in $\cC$. Then, $(L(\cL),L(\cR))$ is generated by the set $L(S)$ of morphisms in $\cD$.

\item\label{item-lem:fs-on-reflective-loc-monoidality}

Suppose that $\cC$ is (symmetric) monoidal compatible with both the reflective localization (recall \Cref{subsubsection:loczns-of-O-monoidal-cats}) and with the factorization system (recall \Cref{def:fs-compatible-with-O-monoidal-str}). Then, the induced (symmetric) monoidal structure on $\cD$ is compatible with the factorization system $(L(\cL),L(\cR))$.

\end{enumerate}
\end{lemma}

\begin{warning}
Although the right adjoint to a reflective localization is (by definition) the inclusion of a full subcategory, the factorization system given by \Cref{lem:fs-on-reflective-loc}.\eqref{item-lem:fs-on-reflective-loc-the-fs} is generally \textit{not} the same as that of \Cref{obs:restricted-fs-on-subcat}. More specifically, although its right class is simply the restriction of the larger right class, its left class is not generally the restriction of the larger left class: using the notation of \Cref{lem:fs-on-reflective-loc}, although we do have $RL(\cR) = \cR \cap R(\cD)$, in general $RL(\cL)$ and $\cL \cap R(\cD)$ are distinct. Indeed, by \Cref{obs:restricted-fs-on-subcat}, these factorization systems coincide if and only if for every morphism $f$ in $\cD$ the object $\Fact_{(\cL,\cR)}(R(f)) \in \cC$ lies in the image of $R$.\footnote{Notably, we do \textit{not} generally have this coincidence in a case of primary interest for us, namely the reflective localization from categorical $\VV$-algebras to $\VV$-enriched $\infty$-categories (see \Cref{thm:fs_and_enriched_cat} and \Cref{cor:fs-and-cat-alg}). Rather, given a morphism of $\VV$-enriched $\infty$-categories, if we consider it as a morphism of categorical $\VV$-algebras then its factorization will have the same underlying space as the source, but this will generally \textit{not} be univalent. (In this case, $RL(\cL)$ consists of the surjective functors between $\VV$-enriched $\infty$-categories that are homwise in $\cL$, whereas $\cL \cap R(\cD)$ consists of the $\iota_0$-equivalences between $\VV$-enriched $\infty$-categories that are homwise in $\cL$.)}
\end{warning}

\begin{proof}[Proof of \Cref{lem:fs-on-reflective-loc}]
We begin with part \eqref{item-lem:fs-on-reflective-loc-the-fs}.

Note first that in diagram \eqref{eq:factorizn-in-fs-on-reflective-localization}, the lower diagonal morphisms respectively lie in $L(\cL)$ and $L(\cR)$, so this is indeed a factorization of the desired type.

Next, $L(\cL)$ is stable under retracts by assumption. To see that $L(\cR)$ is also stable under retracts, consider a retract $g$ in $\Fun([1],\cD)$ of some $L(f) \in L(\cR)$. Applying $R$, we find that $R(g)$ is a retract in $\Fun([1],\cC)$ of $RL(f) \in RL(\cR)$. Since by assumption $RL(\cR) \subseteq \cR$, it follows that $RL(f) \in \cR$, and hence $R(g) \in \cR$ since $\cR$ is stable under retracts. It follows that $g \simeq LR(g) \in L(\cR)$, as desired.

We now verify the orthogonality relation $L(\cL) \bot L(\cR)$. For this, given any $f \in \cL$ and any $g \in \cR$, we must show that $L(f) \bot L(g)$. By adjunction, this is equivalent to showing that $f \bot RL(g)$. But by assumption we have $RL(g) \in RL(\cR) \subseteq \cR$, and so the claim follows from the fact that $\cL \bot \cR$.

We now verify both containments that together constitute the claim that $RL(\cR) = \cR \cap R(\cD)$. First of all, by assumption we have $RL(\cR) \subseteq \cR$, and moreover clearly $L(\cR) \subseteq \cD$ and hence $RL(\cR) \subseteq R(\cD)$. So indeed, we have $RL(\cR) \subseteq \cR \cap R(\cD)$. In the other direction, consider an arbitrary element $R(f) \in \cR \cap R(\cD)$. In particular we have $R(f) \in \cR$, so $LR(f) \in L(\cR)$, so $R(f) \simeq RLR(f) \in RL(\cR)$. So indeed, we have $RL(\cR) \supseteq \cR \cap R(\cD)$.

We now prove part \eqref{item-lem:fs-on-reflective-loc-generation}. By \Cref{prop:fact-system-of-small-generation-on-presentable}, it suffices to show that $L(\cR) = L(S)^\bot$. To verify the containment $L(\cR) \subseteq L(S)^\bot$, it is equivalent by adjunction to check that $RL(\cR) \subseteq S^\bot$, and this follows from the fact that $RL(\cR) \subseteq \cR = S^\bot$. To verify the containment $L(\cR) \supseteq L(S)^\bot$, we observe that for any $f \in L(S)^\bot$, by adjunction we have $R(f) \in S^\bot = \cR$, so indeed $f \simeq LR(f) \in L(\cR)$.

We conclude by proving part \eqref{item-lem:fs-on-reflective-loc-monoidality}. Note that it suffices to prove the claim in the monoidal case. And here, the claim follows from the observation that
\[
L(\cL) \otimes^\cD L(\cL)
\coloneqq
L(RL(\cL) \otimes^\cC RL(\cL))
\simeq
L(\cL \otimes^\cC \cL)
\subseteq
L(\cL)
~,
\]
in which the equivalence and the containment respectively follow from the compatibilities of the reflective localization with the monoidal structure $\otimes^\cC$ and with the factorization system $(\cL,\cR)$.
\end{proof}

We now turn to our final auxiliary result, which gives sufficient conditions for inducing a factorization system through a monadic adjunction.

\begin{lemma}
\label{lem:fs-and-monadic}
Fix a monadic adjunction
    \[
        \begin{tikzcd}[column sep=1.5cm]
        \cC
        \arrow[yshift=0.9ex]{r}{F}
        \arrow[leftarrow, yshift=-0.9ex]{r}[yshift=-0.2ex]{\bot}[swap]{U}
        &
        \cD
        &[-1.7cm]
        \coloneqq {}_T \Mod(\cC)
        \end{tikzcd}
    \]
between presentable $\infty$-categories (where $T \coloneqq UF$ denotes the underlying monad). Suppose that $(\cL,\cR)$ is a factorization system on $\cC$ of small generation, and suppose further that $T$ commutes with geometric realizations and preserves $\cL$. Then, $\cD$ admits a factorization system $(\cL',\cR')$ of small generation, where $\cL' = U^{-1}(L)$ and $\cR' = U^{-1}(\cR)$.

Furthermore, assume that $\cC$, $\cD$ are presentably $\cO$-monoidal for a small $\infty$-operad $\cO$, and the left adjoint $\cC \to \cD$ is $\cO$-monoidal. If the $\cO$-monoidal structure on $\cC$ is compatible with the factorization system, then the $\cO$-monoidal structure on $\cD$ is compatible with the induced factorization system on $\cD$.
\end{lemma}

\begin{proof}
We begin by fixing a small set $S$ of morphisms in $\cC$ that generates $\cL$. By \Cref{prop:fact-system-of-small-generation-on-presentable}, we obtain a factorization system $(\cL',\cR')$ on $\cD$ generated by its image $F(S)$. So, it remains to show that $\cL' = U^{-1}(L)$ and that $\cR' = U^{-1}(R)$.

We first show that $\cR' = U^{-1}(R)$. For this, note that by definition $\cR' = F(S)^\perp$. Hence, it suffices to show that a morphism lies in $F(S)^\perp$ precisely if its image under $U$ lies in $\cR$, which follows from the adjunction $F \adj U$ (and the fact that $\cR = S^\perp$).

We now show that $\cL' = U^{-1}(\cL)$. For this, let us write $\cL'' \coloneqq U^{-1}(\cL)$, so that our goal is to show that $\cL' = \cL''$. Since $\cL' = \overline{F(S)}$, it is equivalent to show that $\overline{F(S)} = \cL''$. In other words, it suffices to verify that $\cL''$ is the smallest saturated class of morphisms in $\cD$ that contains $F(S)$. 

We deduce this in steps. First of all, the fact that $\cL''$ contains $F(S)$ follows from the assumption that the monad $T$ preserves $\cL$ and the fact that $\cL$ contains $S$.

We now show that $\cL''$ is saturated by verifying the conditions of \Cref{def:saturated}.

Condition \eqref{item-def:saturated-wide-subcat} is clear: $\cL''$ defines a wide subcategory of $\cD$.

We now verify condition \eqref{item-def:saturated-colimits}, i.e.\! we show that the full subcategory $\cL'' \subseteq \Fun([1],\cD)$ is closed under small colimits. For this, fix a small $\infty$-category $\cI$ as well as a functor $\cI \xra{X} \Fun([1],\cD)$ that factors through $\cL''$. We wish to show that the colimit $\colim_\cI(X)$ (computed in $\Fun([1],\cD)$) also lies in $\cL''$, i.e.\! that $U(\colim_\cI(X)) \in \Fun([1],\cC)$ lies in $\cL$. Now, since the adjunction $F \adj U$ is monadic, every object $D \in \cD$ admits a functorial \textit{bar resolution}: it is the geometric realization of the levelwise free simplicial object $FT^\bullet U(D) \in \Fun(\Deltaop, \cD)$ \cite[Prop. 4.7.3.14]{HA}. Using this, we may compute $\colim_\cI(X)$ as the colimit of the functor
\[ \begin{tikzcd}[row sep=0cm]
\cI \times \Deltaop
\arrow{r}{X'}
&
\Fun([1],\cD)
\\
\rotatebox{90}{$\in$}
&
\rotatebox{90}{$\in$}
\\
(i,[n])
\arrow[maps to]{r}
&
FT^nU(X(i))
\end{tikzcd}
~.
\]
Namely, we obtain the string of equivalences
\begin{align*}
U(\colim_\cI(X))
& \simeq
U(\colim_{\cI \times \Deltaop}(X'))
\\
& \simeq
U( \colim_{[n] \in \Deltaop} ( \colim_{i \in \cI} ( X'(i,[n]))))
\\
& \simeq
\colim_{[n] \in \Deltaop} U(\colim_{i \in \cI} ( X'(i,[n])))
\\
& \eqqcolon
\colim_{[n] \in \Deltaop} U(\colim_{i \in \cI}(FT^nU(X(i))))
\\
& \simeq
\colim_{[n] \in \Deltaop} UF(\colim_{i \in \cI}(T^nU(X(i))))
\\
& \eqqcolon
\colim_{[n] \in \Deltaop} T(\colim_{i \in \cI}(T^n U(X(i))))
~,
\end{align*}
in which the third equivalence follows from the fact that $U$ commutes with geometric realizations since $T$ does by \cite[Cor. 4.2.3.5]{HA}. Now, by definition of $\cL'' \coloneqq U^{-1}(\cL)$, for each object $i \in \cI$ the object $U(X(i)) \in \Fun([1],\cC)$ lies in $\cL$. Using repeatedly both the fact that $T$ preserves $\cL$ and that $\cL \subseteq \Fun([1],\cC)$ is closed under colimits, we find that $U(\colim_\cI(X)) \in \Fun([1],\cC)$ lies in $\cL$, as desired. So indeed, $\cL'' \subseteq \Fun([1],\cD)$ is closed under colimits.

The verification of condition \eqref{item-def:saturated-cobase-change} (that $\cL'' \subseteq \Fun([1],\cD)$ is stable under cobase change) follows from an essentially identical argument (inasmuch as it involves the computation of a colimit (specifically a pushout) in $\Fun([1],\cD)$). So indeed, the class $\cL''$ of morphisms in $\cD$ is saturated.

In order to conclude that $\cL'' = \overline{F(S)}$, it therefore remains to show that any saturated class $\cL'''$ of morphisms in $\cD$ that contains $F(S)$ also contains $\cL''$. Since $F$ preserves colimits, certainly $F(\cL) \subseteq \cL'''$. From here, to show the containment $\cL'' \subseteq \cL'''$, choose any $f \in \cL'' \coloneqq U^{-1}(\cL)$. Recall that the aforementioned bar resolution yields an equivalence $|FT^\bullet U(f)| \simeq f$. Note that $U(f) \in \cL$, and since $T$ preserves $\cL$ then $T^nU(f) \in \cL$, and so all values of the simplicial object $FT^\bullet U(f)$ lie in $\F(\cL) \subseteq \Fun([1],\cD)$. Hence, its geometric realization -- namely, $f$ -- must lie in $\cL'''$. So indeed, $\cL''$ is the smallest saturated class of morphisms in $\cD$ containing $F(S)$.

It remains to prove the compatibility with $\cO$-monoidal structure.  Given an operation $(a_1, ..., a_m) \to b$ in $\cO$, we want to show that the induced functor $\mu \colon \cD_{a_1} \times \cdots \times \cD_{a_m} \to \cD_b$ carries $\cL'_{a_1} \times \cdots \times \cL'_{a_m}$ to $\cL'_b$. Explicitly, given morphisms $X_i  \colon [1] \to \cD_{a_i}$ in $\cL'_{a_i}$, we would like to show that $U(\mu(X_1, \cdots, X_m)) \in \cL$. As above, the bar resolution gives us a simplicial object $X'_i \colon \Deltaop \to \Fun([1], \cD_{a_i})$ for each $1 \leq i \leq n$, with $X'_i([n]) = FT^nU(X_i)$.

We have a string of equivalences:
\begin{equation}\label{eq:string-of-equivalence}
\begin{aligned}
U(\mu(X_1, \cdots, X_n))
& \simeq
U(\mu( \colim_{[n_1] \in \Deltaop} X'_1([n_1]), \cdots,  \colim_{[n_m] \in \Deltaop} X'_m([n_m])))
\\
& \simeq
U(\colim_{([n_1], \cdots, [n_m]) \in {(\Deltaop)}^m} \mu(X'_1([n_1]), \cdots,  X'_m([n_m])))\\
& \simeq 
U(\colim_{[n]\in \Deltaop} \mu(X'_1([n]), \cdots, X'_m([n])))\\
& \simeq 
\colim_{[n]\in \Deltaop} U(\mu(X'_1([n]), \cdots, X'_m([n])))\\ 
& \simeq 
\colim_{[n]\in \Deltaop} U(\mu(FT^nU(X_1), \cdots, FT^nU(X_m)))\\ 
& \simeq 
\colim_{[n]\in \Deltaop} UF(\mu(T^nU(X_1), \cdots, T^nU(X_m)))\\ 
& \simeq 
\colim_{[n]\in \Deltaop} T(\mu(T^nU(X_1), \cdots, T^nU(X_m)))
~,
\end{aligned}
\end{equation}
in which the third line uses the fact that the diagonal $\Deltaop$ in ${(\Deltaop)}^n$ is cofinal, which is equivalent to the statement that $\Deltaop$ is sifted \cite[Def. 5.5.8.1, Lem. 5.5.8.4]{HA}. For  $1 \leq i \leq n$,  $U(X_i) \in \cL$ by asssumption. It follows that $T^nU(X_i)$ is also in $\cL$. Since the $\cO$-monoidal structure on $\cC$ is compatible with the factorization system and $T$ preserves $\cL$, we see that $T(\mu(T^nU(X_1), \cdots, T^nU(X_m)))$ is in $\cL$. The result now follows from \eqref{eq:string-of-equivalence} and 
the fact that $\cL$ is closed under colimits.
\end{proof}

\subsection{Factorization systems for algebras over \texorpdfstring{$\infty$}{infinity}-operads}
\label{subsec:fs-for-algebras-over-opds}

We now prove the first main theorem of this appendix, which gives factorization systems for algebras over $\infty$-operads.

\begin{theorem}\label{thm:fs-and-alg}
Fix a small $\infty$-operad $\cO$ such that $\underline{\cO} \simeq \pt$ and a presentably $\cO$-monoidal $\infty$-category $\cC$ equipped with a compatible factorization system $(\cL,\cR)$ of small generation.
\begin{enumerate}

\item\label{item:fs-and-alg-single-cat-of-algebras}

For any $\cA \in \Op_{/\cO}$, the presentable $\infty$-category $\Alg_{\cA/\cO}(\cC)$ admits a factorization system $(\cL_\cA,\cR_\cA)$ of small generation with $\cL_\cA = U^{-1}(\cL^{\underline{\cA}})$ and $\cR_\cA = U^{-1}(\cR^{\underline{\cA}})$, where we write $\Alg_{\cA/\cO}(\cC) \xra{U} \Fun(\underline{\cA},\underline{\cC})$ for the forgetful functor and (as the exponential notation suggests (and as in \Cref{lem:fs-and-functor-cat})) a morphism in $\Fun(\underline{\cA},\underline{\cC})$ lies in $\cL^{\underline{\cA}}$ (resp.\! $\cR^{\underline{\cA}}$) if and only if its components all lie in $\cL$ (resp.\! $\cR$).

\item\label{item:fs-and-alg-adjts-preserve-classes}

Fix a morphism $\cA \to \cB$ in $\Op_{/\cO}$. Then, in the adjunction
\[
\begin{tikzcd}[column sep=1.5cm]
\Alg_{\cA / \cO}(\cC)
\arrow[yshift=0.9ex]{r}{F_\cA^\cB}
\arrow[leftarrow, yshift=-0.9ex]{r}[yshift=-0.2ex]{\bot}[swap]{U_\cA^\cB}
&
\Alg_{\cB / \cO}(\cC)
\end{tikzcd}
\]
we have $F_\cA^\cB(\cL_\cA) \subseteq \cL_\cB$ and $U_\cA^\cB(\cR_\cB) \subseteq \cR_\cA$ (using the notation of part \eqref{item:fs-and-alg-single-cat-of-algebras}). In particular, the factorization systems of part \eqref{item:fs-and-alg-single-cat-of-algebras} determine a lift
\begin{equation}
\label{diagram:lift-OpoverOop-to-PrRfsR}
\begin{tikzcd}[column sep=1.5cm]
&
{\Pr}^{R,\fs,\cR}
\arrow{d}{\fgt}
\\
(\Op_{/\cO})^\op
\arrow[dashed]{ru}
\arrow{r}[swap]{\Alg_{(-)/\cO}(\cC)}
&
{\Pr}^R
\end{tikzcd}
\end{equation}
through the indicated forgetful functor.

\item\label{item:fs-on-big-alg}

The total $\infty$-category of the Cartesian unstraightening of the horizontal functor in diagram \eqref{diagram:lift-OpoverOop-to-PrRfsR} admits a factorization system $(\cL_\Alg,\cR_\Alg)$ of small generation, described as follows: an arbitrary morphism $(A \in \Alg_{\cA/\cO}(\cC)) \xra{\tilde{\alpha}} (B \in \Alg_{\cB/\cO}(\cC))$ therein is specified by its image $\cA \xra{\alpha} \cB$ in $\Op_{/\cO}$ along with a morphism $A \ra \alpha^* B$ in $\Alg_{\cA/\cO}(\cC)$, and
\begin{enumerate}

\item it lies in $\cL_\Alg$ if and only if $\alpha$ is an equivalence and moreover for every color $X \in \underline{\cA}$ the morphism $A_X \ra (\alpha^* B)_X$ in $\underline{\cC}$ lies in $\cL$, and

\item it lies in $\cR_\Alg$ if and only if for every color $X \in \underline{\cA}$ the morphism $A_X \ra (\alpha^* B)_X$ in $\underline{\cC}$ lies in $\cR$.
\end{enumerate}

\item\label{item:fs-and-alg-symmetric-monoidal-compatibility}
In the case that $\cO = \EE_\infty$ and  $\cC$ is a presentably symmetric monoidal $\infty$-category with a compatible factorization system, $\Alg_{\cA}(\cC)$ has a canonical symmetric monoidal structure \footnote{Which may not be presentably symmetric monoidal, see \cref{warning:not-presentably-symmetric-monoidal}.} (see \cref{subsubsec:symmetric-monoidal-structure-on-algebras}). The factorization system $(\cL_\cA, \cR_\cA)$ defined above is compatible with the symmetric monoidal structure.
\end{enumerate}
\end{theorem}

\begin{proof}
We begin with part \eqref{item:fs-and-alg-single-cat-of-algebras}.

First of all, observe the equivalences and the adjunction
\begin{equation}
\label{adj:sections-on-cats-of-colors-vs-relative-algebras}
\begin{tikzcd}[column sep=2cm]
\Fun(\underline{\cA}, \underline{\cC} )
\simeq
\Fun_{/\underline{\cO}}(\underline{\cA} , \underline{\cC} )
\simeq
&[-2.2cm]
\Alg_{\cA_\Triv / \cO}(\cC)
\arrow[yshift=0.9ex]{r}{F \coloneqq F_{\cA_\Triv}^\cA}
\arrow[leftarrow, yshift=-0.9ex]{r}[yshift=-0.2ex]{\bot}[swap]{U \coloneqq U_{\cA_\Triv}^\cA}
&
\Alg_{\cA/\cO}(\cC).
\end{tikzcd}
~
\end{equation}
\Cref{lem:fs-and-functor-cat} furnishes the factorization system $(\cL^{\underline{\cA}} , \cR^{\underline{\cA}})$ of small generation on $\Fun(\underline{\cA}, \underline{\cC})$. Hence, we prove part \eqref{item:fs-and-alg-single-cat-of-algebras} by applying \Cref{lem:fs-and-monadic}, whose hypotheses it remains to show are satisfied.

We first show that the adjunction \eqref{adj:sections-on-cats-of-colors-vs-relative-algebras} is monadic and that its underlying monad preserves geometric realizations. For monadicity, by \cite[Thm. 4.7.0.3]{HA} it suffices to show that $U$ is conservative and preserves sifted colimits. The former follows from \cite[Lem. 3.2.2.6]{HA}, while the latter follows from \cite[Prop. 3.2.3.1]{HA}. Of course, $F$ preserves geometric realizations (being a left adjoint), and so the monad $T \coloneqq UF$ preserves geometric realizations as well.

We now claim that this monad $T$ preserves $\cL^{\underline{\cA}}$. This follows from the explicit description of the free algebra functor as an operadic left Kan extension (see particularly \cite[Props. 3.1.1.15, 3.1.1.16, and 3.1.1.20]{HA}). So indeed, the hypotheses of \Cref{lem:fs-and-monadic} are satisfied, and we obtain a factorization system $(\cL_\cA , \cR_\cA)$ of small generation on $\Alg_{\cA/\cO}(\cC)$ as asserted.

For part \eqref{item:fs-and-alg-adjts-preserve-classes}, it suffices to note that the containment $U_\cA^\cB(\cR_\cB) \subseteq \cR_\cA$ follows directly from the commutative square
\[ \begin{tikzcd}
\cA_\Triv
\arrow{r}
\arrow{d}
&
\cB_\Triv
\arrow{d}
\\
\cA
\arrow{r}
&
\cB
\end{tikzcd} \]
in $\Op_{/\cO}$.

With parts \eqref{item:fs-and-alg-single-cat-of-algebras} and \eqref{item:fs-and-alg-adjts-preserve-classes} in hand, part \eqref{item:fs-on-big-alg} follows from \Cref{lem:fs-and-Cart-fibn}.

Lastly, part \eqref{item:fs-and-alg-symmetric-monoidal-compatibility} follows from part \eqref{item:fs-and-alg-single-cat-of-algebras} and the fact that the forgetful functor $U \colon \Alg_{\cA}(\cC) \to \Alg_{\cA_{\Triv}}(\cC) = \Fun(\underline{\cA}, \underline{\cC})$ is symmetric monoidal (see \cref{subsubsec:symmetric-monoidal-structure-on-algebras}).
\end{proof}

\subsection{Factorization systems for enriched \texorpdfstring{$\infty$}{infinity}-categories}
\label{subsec:fs-for-Vcats}

We now prove the second main theorem of the appendix, which gives factorization systems for enriched $\infty$-categories (using those for algebras over $\infty$-operads).

\begin{theorem}\label{thm:fs_and_enriched_cat}
Let $\VV$ be a presentably monoidal $\infty$-category equipped with a compatible factorization system $(\cL, \cR)$.
\begin{enumerate}

\item\label{thm-item:fs_and_enriched_cat_the_fs}

The $\infty$-category $\Cat[\VV]$ of $\VV$-enriched $\infty$-categories admits a factorization system $(\cL_\Cat,\cR_\Cat)$, described as follows.
    \begin{enumerate}

    \item A morphism lies in $\cL_\Cat$ if and only if it is surjective on objects (i.e. $\iota_0$-surjective) and lies in $\cL$ homwise.

    \item A morphism lies in $\cR_\Cat$ if and only if it lies in $\cR$ homwise.

    \end{enumerate}

\item\label{thm-item:fs_and_enriched_cat_generators}

If $S$ is a set of generators for $\cL$, then the localization of $\Sigma[S]$ is a set of generators for $\cL_\Cat$.

\item\label{thm-item:fs_and_enriched_cat_symmetric_monoidal}

If $ \VV$ is symmetric monoidal, then this factorization system is compatible with the resulting symmetric monoidal structure on $\Cat[\VV]$.

\end{enumerate}
\end{theorem}

\begin{remark}
\Cref{thm:fs_and_enriched_cat} generalizes the (fully faithful, essentially surjective) factorization system on enriched $\infty$-categories established by Haugseng in the recent work \cite{haugseng2023tensor} (without any presentability assumptions).
\end{remark}

Before we prove \Cref{thm:fs_and_enriched_cat}, let us first consider factorization systems on categorical algebras:
\begin{lemma}
\label{cor:fs-and-cat-alg}
Fix a presentably monoidal $\infty$-category $\VV$ equipped with a compatible factorization system $(\cL,\cR)$.
\begin{enumerate}

\item
\label{cor-item:fs-and-cat-alg-the-fs}

The $\infty$-category $\AlgCat[\VV]$ of categorical $\VV$-algebras admits a factorization system $(\cL_{\AlgCat},\cR_{\AlgCat})$, described as follows.
\begin{enumerate}

\item A morphism lies in $\cL_{\AlgCat}$ if and only if it is an $\iota_0$-equivalence and it lies in $\cL$ homwise.

\item A morphism lies in $\cR_{\AlgCat}$ if and only if it lies in $\cR$ homwise.

\end{enumerate}

\item
\label{cor-item:fs-and-cat-alg-generators}

If $S$ is a set of generators for $\cL$, then $\Sigma[S] \coloneqq \{ \Sigma(s)\}_{s \in S}$ is a set of generators for $\cL_{\AlgCat}$.

\item
\label{cor-item:fs-and-cat-alg-symm-mon}

If $\VV$ is symmetric monoidal, then the factorization system $(\cL_{\AlgCat},\cR_{\AlgCat})$ is compatible with the resulting symmetric monoidal structure on $\AlgCat[\VV]$.

\end{enumerate}
\end{lemma}

\begin{proof}
By definition, the Cartesian fibration $\AlgCat[\VV] \xra{\iota_0} \Spaces$ is the unstraightening of a composite functor
\[
\Spaces^{\op}
\xra{\codisc}
(\Op_{/\EE_1})^{\op}
\xra{\Alg_{(-)/\EE_1}(\VV)}
{\Pr}^R
~.\footnote{See \cite[Def. 4.3.1]{GH13}, and note that that nonsymmetric (a.k.a.\! planar) $\infty$-operads are equivalent to $\infty$-operads over $\EE_1$ by \cite[Thm. 4.1.3.14]{HA}.}
\]
For a space $X \in \Spaces$, its corresponding $\infty$-operad $\codisc(X) \in \Op_{/\EE_1}$ has space of colors given by pairs of points $x,y \in X$ (up to a symmetrization (i.e.\! the quotient by the $\fS_2$-action) coming from \cite[Thm. 4.1.3.14]{HA}), and a categorical $\VV$-algebra $\cC$ with space of objects $X$ assigns to these the hom-object $\hom_\cC(x,y) \in \VV$. Hence, checking conditions on morphisms in $\underline{\VV}$ colorwise over $\codisc(X)$ indeed corresponds to checking conditions on morphisms homwise, and thereafter part \eqref{cor-item:fs-and-cat-alg-the-fs} follows by combining \Cref{thm:fs-and-alg}.\eqref{item:fs-on-big-alg} and \Cref{lem:fs-and-Cart-fibn}.

Thereafter, part \eqref{cor-item:fs-and-cat-alg-generators} follows from the observation that $(\Sigma[S])^\bot = \cR_{\AlgCat}$, which is immediate from the universal property of $\Sigma[-]$.

Lastly, part \eqref{cor-item:fs-and-cat-alg-symm-mon} follows from the assumption that $(\cL,\cR)$ is compatible with the monoidal structure of $\VV$.
\end{proof}

\begin{proof}[Proof of \Cref{thm:fs_and_enriched_cat}]
Given the factorization system on $\AlgCat[\VV]$ of \Cref{cor:fs-and-cat-alg}, we wish to apply \Cref{lem:fs-on-reflective-loc} to the reflective localization
\[
\begin{tikzcd}[column sep=1.5cm]
\AlgCat[\VV]
\arrow[yshift=0.9ex]{r}{L}
\arrow[hookleftarrow, yshift=-0.9ex]{r}[yshift=-0.2ex]{\bot}[swap]{U}
&
\Cat[\VV]
\end{tikzcd}
~.
\]
We note preliminarily that the morphisms in $\Cat[\VV]$ that are localizations of $\iota_0$-equivalences in $\AlgCat[\VV]$ are precisely the $\iota_0$-surjections.

We first show that the hypotheses of \Cref{lem:fs-on-reflective-loc} are satisfied. To show that $RL(\cR_{\AlgCat}) \subseteq \cR_{\AlgCat}$, we simply observe that if a morphism $F$ in $\AlgCat[\VV]$ is homwise in $\cR$ then so is its localization $L(F)$ and hence so is $RL(F)$. To show that $L(\cL_{\AlgCat})$ is stable under retracts, it suffices to observe that $\cL$ is stable under retracts (by definition of a factorization system) and that surjections in $\Spaces$ are stable under retracts (since surjections in $\Set$ are).

From here, the three parts of \Cref{lem:fs-on-reflective-loc} respectively imply the three parts of the present result.
\end{proof}

\renewcommand*{\bibfont}{\small}
\setlength{\bibitemsep}{0pt}
\raggedright
\printbibliography
\end{document}